\documentclass[a4paper, 11pt, twoside]{article}
\usepackage{fullpage}
\pdfoutput=1
\usepackage{amsmath,amsfonts,amssymb,amsthm,array}
\usepackage{amsmath}
\usepackage{algorithm}
\usepackage{algpseudocode,algorithm,algorithmicx}
\usepackage{caption}
\usepackage{subcaption}
\usepackage{algpseudocode}
\usepackage{bm}
\usepackage{color}
\usepackage{graphicx}
\usepackage{multirow}

\algrenewcommand\algorithmicrequire{\textbf{Input:}}
\algrenewcommand\algorithmicensure{\textbf{Initialize:}}

\newcommand{\Exp}{\mathbb{E}}

\newcommand{\E}[1]{{\mathbb{E}\left[#1\right] }}    

\newcommand{\R}{\mathbb{R}}

\usepackage{tikz}

\newcommand\encircle[1]{%
  \tikz[baseline=(X.base)] 
    \node (X) [draw, shape=circle, inner sep=0] {\strut #1};}

\newcommand{\bA}{\mathbf{A}}
\newcommand{\bB}{\mathbf{B}}

\newcommand{\bI}{\mathbf{I}}
\newcommand{\bL}{\mathbf{L}}

\newcommand{\bP}{\mathbf{P}}

\newcommand{\bS}{\mathbf{S}}

\newcommand{\bW}{\mathbf{W}}
\newcommand{\bZ}{\mathbf{Z}}
\newcommand{\bU}{\mathbf{U}}

\newcommand{\eqdef}{:=}

\newcommand{\cD}{{\cal D}}
\newcommand{\cE}{{\cal E}}
\newcommand{\cF}{{\cal F}}
\newcommand{\cG}{{\cal G}}

\newcommand{\cL}{{\cal L}}

\newcommand{\cO}{{\cal O}}

\newcommand{\cV}{{\cal V}}

\newcommand{\mA}{{\bf A}}
\newcommand{\mB}{{\bf B}}

\newcommand{\mH}{{\bf H}}
\newcommand{\mI}{{\bf I}}

\newcommand{\mS}{{\bf S}}

\newcommand{\mU}{{\bf U}}

\newcommand{\mW}{{\bf W}}

\newcommand{\mZ}{{\bf Z}}

\theoremstyle{plain}
\newtheorem{thm}{Theorem}[]
\newtheorem{lem}[thm]{Lemma}

\newtheorem{rem}{Remark}[]

\theoremstyle{remark}

\usepackage[colorinlistoftodos,bordercolor=orange,backgroundcolor=orange!20,linecolor=orange,textsize=scriptsize]{todonotes}

\usepackage[colorlinks=true,linkcolor=blue]{hyperref} 

\begin{document}
\title{Momentum and Stochastic Momentum for Stochastic  Gradient, Newton, Proximal Point and Subspace Descent Methods}
\author{Nicolas Loizou \thanks{ School of Mathematics, The University of Edinburgh. ---  E-mail: n.loizou@sms.ed.ac.uk}
 \and Peter Richt\'{a}rik\thanks{CEMSE, King Abdullah University of Science and Technology (KAUST), Thuwal, Saudi Arabia ---  School of Mathematics, The University of Edinburgh, United Kingdom --- Moscow Institute of Physics and Technology (MIPT), Dolgoprudny, Moscow, Russia.--- E-mail: peter.richtarik@ed.ac.uk}}
\date{December 22, 2017\footnote{A short version of this paper (5 pages) was posted on arXiv on 30 Oct 2017 \cite{loizou2017linearly}. The paper was accepted for presentation at the 2017 NIPS Optimization for Machine Learning workshop in a peer reviewed process. The accepted papers are listed on the website of the workshop, but are not published in any proceedings volume.}}
\maketitle
\providecommand{\keywords}[1]{\textbf{Keywords} #1} 
\providecommand{\ams}[1]{\textbf{Mathematical Subject Classifications } #1}
\begin{abstract} 
In this paper we study several classes of stochastic optimization algorithms enriched with  {\em heavy ball momentum}. Among the methods studied are: stochastic gradient descent, stochastic Newton, stochastic proximal point  and stochastic dual subspace ascent. This is the first time momentum variants of several of these methods are studied. We choose to perform our analysis in a setting in which all of the above methods are equivalent. We prove global nonassymptotic linear convergence rates for all methods and various measures of success, including primal function values, primal iterates (in L2 sense), and dual function values. We also show that the primal iterates converge at an accelerated linear rate in the L1 sense. This is the first time a linear rate is shown for the stochastic heavy ball method (i.e., stochastic gradient descent method with momentum). Under somewhat weaker conditions,  we establish a sublinear convergence rate for Cesaro averages of primal iterates.  Moreover, we propose a novel concept, which we call {\em stochastic momentum}, aimed at decreasing the cost of performing the momentum step. We prove linear convergence of several stochastic methods with stochastic momentum, and show that in some sparse data regimes and for sufficiently small momentum parameters, these methods enjoy better overall complexity than methods with deterministic momentum. Finally, we perform extensive numerical testing on artificial and real datasets, including data coming from average consensus problems. 
\end{abstract}

\noindent \keywords{ stochastic methods $\cdot$ heavy ball momentum $\cdot$ linear systems $\cdot$ randomized coordinate descent $\cdot$ randomized Kaczmarz $\cdot$  stochastic gradient descent $\cdot$ stochastic Newton $\cdot$ quadratic optimization $\cdot$ convex optimization}

\noindent \ams{68Q25 $\cdot$ 68W20 $\cdot$ 68W40 $\cdot$ 65Y20 $\cdot$ 90C15 $\cdot$ 90C20 $\cdot$ 90C25 $\cdot$ 15A06 $\cdot$ 15B52 $\cdot$ 65F10 }

\section{Introduction} \label{sec:intro}

Two of the most popular algorithmic ideas  for solving optimization problems involving big volumes of data are {\em stochastic approximation}  and {\em momentum}. By stochastic approximation we refer to the practice pioneered by Robins and Monro \cite{robbins1951stochastic} of replacement of costly-to-compute quantities (e.g.,  gradient of the objective function) by cheaply-to-compute {\em stochastic} approximations thereof (e.g., unbiased estimate of the gradient).  By momentum we refer to the {\em heavy ball} technique originally developed by Polyak \cite{polyak1964some} to accelerate the convergence rate of gradient-type methods.  

While much is known about the effects of {\em stochastic approximation} and {\em momentum} in isolation, surprisingly little is known about the {\em combined effect} of these two popular algorithmic techniques. For instance, to the best of our knowledge, there is no context in which a method combining stochastic approximation with momentum  is known to have a linear convergence rate. One of the  contributions of this work is to show that there are important problem classes for which a linear rate can indeed be established for a range of stepsize and momentum parameters.

\subsection{Setting}

In this paper we study three closely related problems: 
\begin{enumerate}
\item[(i)] stochastic optimization, 
\item[(ii)] best approximation, and 
\item[(iii)] (bounded) concave quadratic maximization.
\end{enumerate} 
These problems and the relationships between them are described in detail in Section~\ref{sec:prelim}. Here we only briefly outline some of the key relationships. By stochastic optimization we refer to the problem of the form
\begin{equation}\label{eq:stoch_reform_intro} \min_{x\in \R^n} f(x)\eqdef \E{f_{\mS}(x)},\end{equation}
where the expectation is over random matrices $\mS$ drawn from an arbitrary  distribution $\cD$, and $f_{\mS}$ is a stochastic convex quadratic function of a  least-squares type, depending on $\mS$, and in addition on a matrix $\mA\in \R^{m\times n}$, vector $b\in \R^m$,  and an $n\times n$ positive definite matrix $\mB$ (see Section~\ref{sec:prelim} for full details). The problem is constructed in such a way that the set of minimizers of $f$ is identical to the set of solutions of a given (consistent) linear system \begin{equation} \label{eq:lin-sys-intro}\mA x =b,\end{equation}
where $\mA\in \R^{m\times n}$ and $b \in \R^m$. In this sense, \eqref{eq:stoch_reform_intro} can  be seen as the reformulation of the linear system \eqref{eq:lin-sys-intro} into a stochastic optimization problem. Such reformulations  provide an explicit connection between the fields of  linear algebra and stochastic optimization, which may inspire future research by enabling the transfer of knowledge, techniques, and algorithms from one field to another. For instance, the randomized Kaczmarz method of Strohmer and Vershynin \cite{RK} for solving \eqref{eq:lin-sys-intro} is equivalent to the stochastic gradient descent method applied to \eqref{eq:stoch_reform_intro}, with $\cD$ corresponding to a discrete distribution over unit coordinate vectors in $\R^m$. However, the flexibility of being able to choose $\cD$ arbitrarily allows for numerous generalizations of the randomized Kaczmarz method \cite{ASDA}. Likewise, provably faster variants of the randomized Kaczmarz method (for instance, by utilizing importance sampling) can be designed using the connection. 

\subsection{Three stochastic methods}
Problem  \eqref{eq:stoch_reform_intro}  has several peculiar characteristics which are of key importance to this paper. For instance, the Hessian of $f_{\mS}$ is a (random) projection matrix, which can be used to show that $f_{\mS}(x) = \tfrac{1}{2}\|\nabla f_{\mS}(x)\|_\mB^2$ . Moreover, it follows that the Hessian of $f$ has all eigenvalues bounded by 1, and so on. These characteristics can be used to show that several otherwise {\em distinct}  stochastic algorithms for solving the stochastic optimization problem \eqref{eq:stoch_reform_intro} are {\em identical} \cite{ASDA}. In particular, the following optimization methods for solving \eqref{eq:stoch_reform_intro} are identical\footnote{In addition, these three methods are identical to a stochastic fixed point method (with relaxation) for solving the fixed point problem $x = \E{\Pi_{\cL_\mS}(x)}$, where $\cL_{\mS}$ is the set of solutions of $\mS^\top \mA x = \mS^\top b$, which is a  {\em sketched} version of the  linear system \eqref{eq:lin-sys-intro}, and can be seen as a stochastic approximation  of the set $\cL\eqdef \{x\;:\; \mA x = b\}$}: 
\begin{itemize}
\item {\em Method 1:} stochastic gradient descent (SGD),
\item {\em Method 2:} stochastic Newton method (SN), and
\item {\em Method 3:} stochastic proximal point method (SPP);
\end{itemize}
all with a fixed stepsize $\omega>0$.  The methods will be described in detail in Section~\ref{sec:prelim}; see also Table~\ref{tbl:all_methods} for a quick summary.

The equivalence of these methods is useful for the purposes of this paper as it allows us to study their variants  {\em with momentum} by studying a single algorithm only. We are not aware of any successful attempts to analyze momentum variants of SN and SPP and  as we said before, there are no linearly convergent variants of SGD with momentum in {\em any setting}.

\subsection{Best approximation, duality and  stochastic dual subspace ascent} It was shown in \cite{gower2015stochastic} in the $\omega=1$ case and in \cite{ASDA} in the general $\omega>0$ case  that SGD, SN and SPP converge to a very particular minimizer of $f$: the projection of the starting point $x_0$ onto the solution set of the linear system \eqref{eq:lin-sys-intro}. This naturally leads  to the best approximation problem, which is the problem of projecting\footnote{In the rest of the paper we consider projection with respect to an arbitrary Euclidean norm.} a given vector onto the solution space of the linear system \eqref{eq:lin-sys-intro}:
\begin{equation}\label{eq:best_approx_intro}\min_{x\in \R^n} \tfrac{1}{2}\|x-x_0\|_{\mB}^2 \quad \text{subject to}\quad \mA x = b .  \end{equation}

The dual of the best approximation problem is an unconstrained concave quadratic maximization problem \cite{gower2015stochastic}.  Consistency of $\mA x = b$ implies that the dual is bounded.  It follows from the results of \cite{gower2015stochastic} that for $\omega=1$, the random iterates of SGD, SN and SPP arise as affine images of  the random iterates produced by an algorithm for solving the dual of the best approximation  problem \eqref{eq:best_approx_intro}, known as
\begin{itemize}
\item  {\em Method 4:} stochastic dual subspace ascent (SDSA).
\end{itemize}
In this paper we show that this equivalence extends beyond the $\omega=1$ case, specifically for $0<\omega<2$, and further study {\em SDSA with momentum}. We then show that SGD, SN and SPP with momentum arise as affine images of SDSA with momentum. SDSA proceeds by taking steps in a random subspace spanned by the columns of $\mS$ randomly drawn in each iteration from $\cD$. In this subspace, the method moves to the point which maximizes the dual objective, $D(y)$. Since $\cD$ is an arbitrary distribution of random matrices, SDSA moves in arbitrary random subspaces, and as such, can be seen as a vast generalization of randomized coordinate descent methods  and their minibatch variants \cite{fercoq2015accelerated, qu2016coordinate}.

\subsection{Structure of the paper}

The remainder of this work is organized as follows. In Section~\ref{sec:contributions} we summarize our contributions in the context of existing literature. In Section~\ref{sec:prelim} we provide a detailed account of the stochastic optimization problem, the best approximation and its dual. Here we also describe the SGD, SN and SPP methods. In  Section~\ref{sec:primal} we describe and analyze  primal  methods  with momentum (mSGD, mSN and mSPP), and in Section~\ref{sec:dual} we describe and analyze the dual method with momentum (mSDSA). In Section~\ref{sec:sm} we describe and analyze primal methods with stochastic momentum (smSGD, smSN and smSPP). Numerical experiments are presented in Section~\ref{sec:experiments}.  Proofs of all key results can be found in the appendix.

 \subsection{Notation}

The following notational conventions are used in this paper. Boldface upper-case letters denote matrices; $\bI$ is the identity matrix. By $\cL$ we denote the solution set of the linear system $\bA x=b$. By $\cL_{\bS}$, where $\mS$ is a random matrix, we denote the solution set of the {\em sketched} linear system $\bS^\top \bA x= \bS^\top b$.  Throughout the paper, $\mB$ is an $n\times n$ positive definite matrix giving rise to an inner product and norm on $\R^n$. Unless stated otherwise, throughout the paper, $x_*$ is the projection of $x_0$ onto $\cL$ in the $\mB$-norm: $x_*=\Pi_{\cL}^\bB(x_0)$. We write $[n]\eqdef \{1,2, \dots ,n\}$.

\section{Momentum Methods and Our Contributions} \label{sec:contributions}

In this section we give a brief review of the relevant literature, and provide a  summary of our contributions.

\subsection{Heavy ball method}

The baseline first-order method for minimizing a differentiable function $f$ is the gradient descent (GD) method, \[x_{k+1} = x_k - \omega_k \nabla f(x_k),\] where $\omega_k>0$ is a stepsize. For convex functions with $L$-Lipschitz gradient (function class $\cF^{1,1}_{0,L}$),  GD converges at at the rate of $\cO(L/\epsilon)$. When, in addition, $f$ is $\mu$-strongly convex (function class $\cF^{1,1}_{\mu,L}$), the rate is linear: $\cO((L /\mu) \log(1/\epsilon))$ \cite{nesterov2013introductory}. To improve the convergence behavior of the method, Polyak proposed to modify GD by the   introduction of a (heavy ball) momentum term\footnote{Arguably a much more popular, certainly theoretically much better understood alternative to Polyak's momentum is the momentum introduced by Nesterov \cite{nesterov1983method, nesterov2013introductory}, leading to the famous {\em accelerated gradient descent (AGD)} method. This method converges nonassymptotically and globally; with optimal sublinear rate $\cO(\sqrt{L/\epsilon})$  \cite{nemirovskii1983problem} when  applied to minimizing a smooth convex objective function (class $\cF^{1,1}_{0,L}$), and with the optimal linear  rate $\cO(\sqrt{L/\mu} \log(1/\epsilon))$ when minimizing smooth strongly convex functions (class $\cF^{1,1}_{\mu,L}$). Both Nesterov's and Polyak's update rules are known in the literature as ``momentum'' methods. In this paper, however, we focus exclusively on Polyak's heavy ball momentum.}, $\beta(x_k-x_{k-1})$. This leads to the gradient descent method with  momentum (mGD), popularly  known as the {\em heavy ball method:} \[x_{k+1} = x_k - \omega_k \nabla f(x_k) + \beta (x_k-x_{k-1}).\]  More specifically, Polyak proved that with the correct choice of the stepsize parameters $\omega_k$ and momentum parameter $\beta$, a \textit{local} accelerated linear convergence rate of  $\cO(\sqrt{L/\mu}\log(1/\epsilon))$  can be achieved in the case of twice continuously differentiable, $\mu$-strongly convex objective functions with $L$-Lipschitz gradient (function class $\mathcal{F}_{\mu, L}^{2,1}$).  See the first line of Table~\ref{ComparisonWIthHeavy}.

Recently, Ghadimi et al.~\cite{ghadimi2015global} performed a \textit{global} convergence analysis for the heavy ball method. In particular, the authors showed that for a certain combination of the stepsize and momentum parameter, the method  converges sublinearly to the optimum when the objective function is convex  and has Lipschitz gradient ($f\in \mathcal{F}_{0, L}^{1,1}$), and  linearly when the function is also strongly convex ($f\in \mathcal{F}_{\mu, L}^{1,1}$). A particular, selection of the parameters $\omega$ and $\beta$ that gives the desired accelerated linear rate was not provided.
 
To the best of our knowledge, despite considerable amount of work on the on heavy ball method, there is still no  global convergence analysis which would guarantee an accelerated linear rate for $f\in  \mathcal{F}_{\mu, L}^{1,1}$. However, in the special case of a strongly convex quadratic, an elegant proof was recently proposed in \cite{lessard2016analysis}. Using the notion of integral quadratic constraints from robust control theory, the authors proved that by choosing $\omega_k = \omega=4 / (\sqrt{L}+\sqrt{\mu})^2$ and $\beta=(\sqrt{L/\mu}-1)^2/(\sqrt{L/\mu}+1)^2$, the heavy ball method enjoys a global \textit{asymptotic} accelerated convergence rate of $\cO(\sqrt{L/\mu} \log(1/\epsilon))$. The aforementioned results are summarized in the first part of Table~\ref{ComparisonWIthHeavy}.

Extensions of the heavy ball method have been recently proposed in the proximal setting \cite{ochs2015ipiasco}, non-convex setting \cite{ochs2014ipiano, zavriev1993heavy} and for distributed optimization \cite{ghadimi2013multi}.

\subsection{Stochastic heavy ball method}

In contrast to the recent advances in our theoretical understanding of the (classical) heavy ball method, there has been less progress in understanding the convergence behavior of  {\em stochastic} variants of the heavy ball method. The key method in this category is stochastic gradient descent with momentum (mSGD; aka: stochastic heavy ball method):
\[x_{k+1} = x_k -\omega_k g (x_k)  + \beta (x_k-x_{k-1}),\]
where $g_k$ is an unbiased estimator of the true gradient $\nabla f(x_k)$. While mSGD is used extensively in practice, especially in deep learning \cite{sutskever2013importance, szegedy2015going, krizhevsky2012imagenet, wilson2017marginal}, its convergence behavior is not very well understood.

In fact, we are aware of only two  papers, both recent, which set out to study the complexity of mSGD: the work of Yang et al.\ \cite{yang2016unified}, and the work of Gadat et al.\ \cite{gadat2016stochastic}. In the former paper, a unified convergence analysis for stochastic gradient methods with momentum  (heavy ball and Nesterov's momentum) was proposed; and an analysis for  both convex and non convex functions was performed. For a general Lipschitz continuous convex objective function with bounded variance, a rate of $O(1/\sqrt{\epsilon})$ was proved. For this, the authors employed a decreasing stepsize strategy: $\omega_k=\omega_0/\sqrt{k+1}$, where $\omega_0$ is a positive constant.  In \cite{gadat2016stochastic}, the authors first describe several almost sure convergence results in the case of general non-convex coercive functions, and then provided a complexity analysis for the case of quadratic strongly convex function. However, the established rate is slow. More precisely, for strongly convex quadratic  and coercive functions, mSGD  with diminishing stepsizes $\omega_k=\omega_0 / k^\beta$ was shown to convergence as $\cO(1/k^\beta)$ when the momentum parameter is $\beta <1$, and with the rate $O(1/\log k)$ when $\beta=1$. The convergence rates established in  both of these papers  are sublinear. In particular, no insight is provided into whether the inclusion of the momentum term provides what is was aimed to provide: acceleration.

The above results are summarized in the second part of Table~\ref{ComparisonWIthHeavy}. From this perspective, our contribution lies in providing an in-depth analysis of mSGD (and, additionally, of SGD with stochastic momentum).  Our contributions are discussed next.
 
\begin{table}
\begin{center}
\scalebox{0.8}{
\begin{tabular}{ |c | c| c| c | c| }
 \hline
 Method & Paper & Rate & Assumptions on $f$ & Convergence\\
 \hline \hline
\multirow{4}{*}{\begin{tabular}{c}Heavy Ball\\ (mGD) \end{tabular}} & Polyak, 1964 \cite{polyak1964some} & accelerated linear & $\mathcal{F}_{\mu, L}^{2,1}$ & local   \\ 
& Ghadimi et al, 2014 \cite{ghadimi2015global} & sublinear & $ \mathcal{F}_{ L}^{1,1}$ & global \\
  & Ghadimi et al, 2014 \cite{ghadimi2015global} & linear  & $ \mathcal{F}_{\mu, L}^{1,1}$ & global \\
 & Lessard et al, 2016 \cite{lessard2016analysis} &  accelerated linear & $ \mathcal{F}_{\mu, L}^{1,1}$ + quadratic 
& global, asymptotic \\
 \hline
\multirow{3}{*}{\begin{tabular}{c}Stochastic \\ Heavy Ball  \\ (mSGD)\end{tabular}}  & Yang et al. 2016 \cite{yang2016unified} & sublinear & $ \mathcal{F}_{0, L}^{1,1}$ + bounded variance & global, non-asymptotic \\
& Gadat et al, 2016 \cite{gadat2016stochastic} & sublinear & $ \mathcal{F}_{\mu, L}^{1,1}$ + other assumptions & global, non-asymptotic \\
& \textbf{THIS PAPER} & {\bf see Table~\ref{OurResults}}  & $ \mathcal{F}_{0, L}^{1,1}$ + quadratic & global, non-asymptotic \\
 \hline
\end{tabular}}
\end{center}
\caption{Known complexity results for gradient descent with momentum (mGD, aka: heavy ball method),  and stochastic gradient descent with momentum (mSGD, aka: stochastic heavy ball method). We give the first linear and accelerated rates for mSGD. For full details on iteration complexity results we obtain, refer to Table~\ref{OurResults}.}
\label{ComparisonWIthHeavy}
\end{table}
 
\subsection{Connection to incremental gradient methods}

Assuming $\cD$ is discrete distribution (i.e., we sample from $M$ matrices, $\mS^1,\dots,\mS^{M}$, where $\mS^i$ is chosen with probability $p_i>0$), we can write the stochastic optimization problem \eqref{eq:stoch_reform_intro} in the {\em finite-sum} form
\begin{equation}\label{eq:finite-sum}\min_{x\in \R^n} f(x) = \sum_{i=1}^{M} p_i f_{\mS^i}(x).\end{equation}
Choosing $x_0=x_1$, mSGD with fixed stepsize $\omega_k=\omega$ applied to \eqref{eq:finite-sum} can be written in the form
\begin{equation}
x_{k+1} =  x_k-\omega \sum_{t=1}^k\beta^{k-t} \nabla f_{\bS_t} (x_t)+\beta^k(x_1-x_0)
= x_k-\omega \sum_{t=1}^k\beta^{k-t} \nabla f_{\bS_t} (x_t),
\label{ExpansionOfUpdate}
\end{equation} 
where $\mS_t = \mS^i$ with probability $p_i$. Problem~\eqref{eq:finite-sum} can be also solved using incremental  average/aggregate gradient methods, such as the IAG method of Blatt et al.\ \cite{blatt2007convergent}. These methods have a similar form to \eqref{ExpansionOfUpdate}, with the main difference being in the way the past gradients are aggregated. While \eqref{ExpansionOfUpdate} uses a geometric weighting of the gradients, the incremental average gradient methods use a uniform/arithmetic weighting. The  stochastic average gradient (SAG) method of Schmidt et al.\ \cite{schmidt2017minimizing} can be also written in a similar form. Note that mSGD  uses a geometric weighting of previous gradients, while the the incremental and stochastic average gradient methods use an arithmetic weighting. Incremental and incremental average gradient methods are widely studied algorithms for minimizing objective functions which can expressed as a sum of finite convex functions. For a review of key works on incremental methods and a detailed presentation of the connections with stochastic gradient descent, we refer the interested reader to the excellent survey of Bertsekas~ \cite{bertsekas2011incremental}; see also the work of Tseng~ \cite{tseng1998incremental}. 

In \cite{gurbuzbalaban2017convergence}, an incremental average gradient method with momentum was proposed for minimizing strongly convex functions. It was proved that the method converges to the optimum with linear rate. The rate is always worse than that of the no-momentum variant. However, it was  shown experimentally that in practice the method is faster, especially in problems with high condition number.  In our setting, the objective function has a very specifc structure \eqref{eq:stoch_reform_intro}. It is not a finite sum problem as the distribution $\cD$ could be continous; and we also do not assume strong convexity. Thus, the convergence analysis of \cite{gurbuzbalaban2017convergence} can not be directly applied to our problem.

\subsection{Summary of contributions}

We now summarize the contributions of this paper.

{\bf New momentum methods. } We study several  classes of stochastic optimization algorithms (SGD, SN, SPP and SDSA) {\em with momentum}, which we call mSGD, mSN, mSPP and mSDSA, respectively (see the first and second columns of Table~\ref{tbl:all_methods}). We do this in a simplified  setting with quadratic objectives where all of these algorithms are  equivalent. These methods can be seen as solving three related optimization problems: the stochastic optimization problem \eqref{eq:stoch_reform_intro}, the best approximation problem \eqref{eq:best_approx_intro} and its dual.  To the best of our knowledge, momentum variants of SN, SPP and SDSA were not analyzed before. 

{\footnotesize
\begin{table}[t!]
\begin{center}
\scalebox{0.8}{
\begin{tabular}{|c|c|c|}
\hline
&& \\
 \begin{tabular}{ccc}no momentum \\ ($\beta=0$) \end{tabular} & \begin{tabular}{cc} momentum \\ ($\beta\geq 0$) \end{tabular}  & \begin{tabular}{cc} stochastic momentum \\ ($\beta\geq 0$) \end{tabular} \\
&& \\
\hline
\hline
&& \\
 \begin{tabular}{c} SGD \cite[$\omega=1$]{gower2015randomized},  \cite[$\omega>0$]{ASDA}  \\ \\
  $ x_{k+1}= x_k - \omega \nabla f_{\mS_k}(x_k)$  
   \end{tabular}  & \begin{tabular}{ccc} {\bf mSGD} [Sec~\ref{sec:primal}]\\ \\
$ + \beta (x_k-x_{k-1})$ \end{tabular} & 
 \begin{tabular}{c}
 {\bf smSGD} [Sec~\ref{sec:sm}] \\ \\
 $ + n \beta e_{i_k}^\top (x_k-x_{k-1}) e_{i_k}$
 \end{tabular} 
  \\
&& \\
\hline
&& \\
  \begin{tabular}{c} 
 SN \cite{ASDA} \\
 \\
 $  x_{k+1}= x_k - \omega (\nabla^2 f_{\mS_k}(x_k))^{\dagger_\mB} \nabla f_{\mS_k}(x_k)$
  \end{tabular}
 & \begin{tabular}{c}{\bf mSN} [Sec~\ref{sec:primal}] \\\\
 $ + \beta (x_k-x_{k-1})$
\end{tabular} 
 & 
 \begin{tabular}{c}
 {\bf smSN} [Sec~\ref{sec:sm}]\\ \\
 $ + n \beta e_{i_k}^\top (x_k-x_{k-1}) e_{i_k}$
 \end{tabular} 
 \\
&& \\
\hline
&& \\
  \begin{tabular}{c}  
 SPP \cite{ASDA} \\ \\
$ x_{k+1}=  \arg\min_x \left\{ f_{\mS_k}(x) + \frac{1-\omega}{2\omega}\|x-x_k\|_{\mB}^2\right\}
$
  \end{tabular}
&   \begin{tabular}{c}{\bf mSPP} [Sec~\ref{sec:primal}] \\ \\
 $ + \beta (x_k-x_{k-1})$
\end{tabular} 
&  \begin{tabular}{c}
 {\bf smSPP} [Sec~\ref{sec:sm}]\\ \\
 $ + n \beta e_{i_k}^\top (x_k-x_{k-1}) e_{i_k}$
 \end{tabular} 
\\
&& \\
\hline
&& \\
  \begin{tabular}{ccc}SDSA \cite[$\omega=1$]{gower2015stochastic} \\ \\ 
  $y_{k+1} = y_k + \mS_k \lambda_k$ \end{tabular}  
  &  \begin{tabular}{c}{\bf mSDSA} [Sec~\ref{sec:dual}] \\ \\
 $ + \beta (y_k-y_{k-1})$
\end{tabular} 
  &   \\
&& \\
\hline
\end{tabular}
}
\end{center}

\caption{All methods analyzed in this paper. The methods highlighted in bold (with momentum and stochastic momentum) are new. SGD = Stochastic Gradient Descent,  SN = Stochastic Newton, SPP = Stochastic Proximal Point, SDSA = Stochastic Dual Subspace Ascent. At iteration $k$, matrix $\mS_k$ is drawn in an i.i.d.\ fashion from distribution $\cD$, and a stochastic step is performed.}
\label{tbl:all_methods}
\end{table}
}

{\bf Linear rate.} We prove several (global and non-asymptotic) linear convergence results for our primal momentum methods mSGD/mSN/mSPP. First, we establish a linear rate for the decay of $\E{\|x_k-x_*\|_\mB^2}$ to zero (i.e., $L2$ convergence), for a range of stepsizes $\omega> 0$ and momentum parameters $\beta\geq 0$. We show that the same rate holds for the decay of the expected function values $\E{f(x_k)-f(x_*)}$ of \eqref{eq:stoch_reform_intro} to zero. Further, the same rate holds for  mSDSA, in particular, this is for the convergence of the dual objective to the optimum. For a summary of these results, and pointers to the relevat theorems, refer to lines 1, 2 and 6 of Table~\ref{OurResults}. Unfortunately, the theoretical rate for all our momentum methods is optimized for $\beta = 0$, and  gets worse as the momentum parameter increases. However, no prior linear rate for any of these methods with momentum are known. We give the first linear convergence rate for SGD with momentum (i.e., for the stochastic heavy ball method). 

\begin{table}[t!]
\begin{center}
\scalebox{0.8}{
\begin{tabular}{ |c|c|c|c|c|c| }
 \hline
 Algorithm & $\omega$& \begin{tabular}{c} momentum \\ $\beta$ \end{tabular} & \begin{tabular}{c} Quantity \\ converging to 0 \end{tabular} & \begin{tabular}{c}Rate\\(all: global, non-asymptotic)\end{tabular} & Theorem \\
 \hline
 \hline
 mSGD/mSN/mSPP & $(0,2)$ &  $ \geq 0$ &   $\Exp[\|x_k-x_*\|^2_{\bB}]$ & linear &\ref{L2}\\
 \hline
mSGD/mSN/mSPP & $(0,2)$ &  $\geq 0$ &   $\Exp[f(x_k) - f(x_*)]$  & linear &\ref{L2}\\
 \hline
mSGD/mSN/mSPP & $(0,2)$ & $\geq 0$ & $\E{f(\hat{x}_k)} - f(x_*)$  & sublinear: $O(1/k)$& \ref{cesaro} \\
 \hline 
mSGD/mSN/mSPP  & 1 & $\left(1-\sqrt{0.99 \lambda_{\min}^+}\right)^2 $ & $\|\Exp[x_k-x_*]\|^2_{\bB}$ &  accelerated linear & \ref{theoremheavyball}\\
 \hline
  mSGD/mSN/mSPP  & $\frac{1}{\lambda_{\max}}$ & $ \displaystyle  \left(1-\sqrt{0.99 \tfrac{\lambda_{\min}^+}{\lambda_{\max}}}\right)^2 $ & $\|\Exp[x_k-x_*]\|^2_{\bB}$   &  \begin{tabular}{c}accelerated linear\\ (better than for $\omega=1$) \end{tabular} & \ref{theoremheavyball}\\
  \hline
  mSDSA & $(0,2)$ & $\geq 0$ & $\E{D(y_*) - D(y_0)}$ & linear & \ref{thm:dual-conv}  \\
  \hline
 smSGD/smSN/smSPP & $(0,2)$ & $\geq 0$ &   $\Exp[\|x_k-x_*\|^2_{\bB}]$ & linear &  \ref{thm:DSHB-L2}\\
 \hline
  smSGD/smSN/smSPP  & $(0,2)$ & $\geq 0$ &   $\Exp[f(x_k)-f(x_*)]$  & linear & \ref{thm:DSHB-L2}\\
 \hline
\end{tabular}}
\end{center}
\caption{Summary of the iteration complexity results obtained in this paper. Parameters of the methods: $\omega$ (stepsize) and $\beta$ (momentum term). In all cases, $x_*=\Pi_{\cL}^\bB(x_0)$ is the solution of the best approximation problem.  Theorem~\ref{cesaro} refers to Cesaro averages:  $\hat{x}_k = \frac{1}{k}\sum_{t=0}^{k-1}x_t$. Theorem~\ref{thm:dual-conv} refers to suboptimality in dual function values ($D$ is the dual function).}
\label{OurResults}
\end{table}

{\bf Accelerated linear rate.} We then study the decay of the larger quantity $\|\E{x_k}-x_*\|_\mB^2$  to zero (i.e., L1 convergence).  In this case, we establish  an {\em accelerated} linear rate, which depends on the square root of the condition number (of the Hessian of $f$). This is a quadratic speedup when compared to  the no-momentum methods as these depend on the condition number. See lines 4 and 5 of Table~\ref{OurResults}. To the best of our knowledge, this is the first time an  accelerated rate is obtained for the stochastic heavy ball method (mSGD).  Note that there are no global non-asymptotic accelerated linear rates proved even in the non-stochastic setting (i.e., for the heavy ball method). Moreover, we are not aware of any accelerated linear convergence results for  the stochastic proximal point method.

{\bf Sublinear rate for Cesaro averages.} We show that the Cesaro averages, $\hat{x}_k = \frac{1}{k}\sum_{t=0}^{k-1}x_t$, of all primal momentum methods enjoy a sublinear $O(1/k)$ rate (see line 3 of Table~\ref{OurResults}). This holds under weaker assumptions than those which lead to the linear convergence rate. 

{\bf Primal-dual correspondence.} We show that SGD, SN and SPP with momentum arise as affine images of SDSA with momentum (see Theorem~\ref{thm:dual_corresp}). This extends the result of \cite{gower2015stochastic} where this was shown for the no-momentum methods ($\beta=0$) and in the special case of the unit stepsize ($\omega=1$).

{\bf Stochastic momentum.} We propose a new momentum strategy, which we call {\em stochastic momentum}.   Stochastic momentum is a stochastic (coordinate-wise) approximation of the deterministic momentum, and hence is much less costly, which in some situations leads to computational savings in each iteration. On the other hand, the additional noise introduced this way increases the number of iterations needed for convergence. We analyze the SGD, SN and SPP methods with stochastic momentum, and prove linear convergence rates. We prove that in some settings  the overall complexity of SGD with stochastic momentum is better than the overall complexity of SGD with momentum. For instance, this is the case if we consider the randomized Kaczmarz (RK) method as a special case of SGD, and if $\mA$ is sparse.

{\bf Space for generalizations.} We hope that the present work can serve as a starting point for the development of SN, SPP and SDSA methods with momentum for more general classes (beyond special quadratics) of convex and perhaps also nonconvex optimization problems. In such more general settings, however, the symmetry which implies equivalence of these algorithms will break, and hence a different analysis will be needed for each method.

\subsection{No need for variance reduction}

 SGD  is arguably  one of the most popular algorithms in machine learning. Unfortunately, SGD suffers from slow convergence, which is due to the fact that the variance of the stochastic gradient as an estimator of the gradient does not naturally diminish. For this reason, SGD is typically used with a decreasing stepsize rule, which ensures that the variance converges to zero. However, this has an adverse effect on the convergence rate. For instance, SGD has a sublinear rate even if the function to be minimized is strongly convex.  To overcome this problem, a new class of so-called  {\em variance-reduced} methods was developed over the last 2-5 years, including SAG \cite{schmidt2017minimizing}, SDCA \cite{SDCA, richtarik2014iteration}, SVRG/S2GD  \cite{johnson2013accelerating, S2GD}, minibatch SVRG/S2GD \cite{mS2GD}, and SAGA \cite{defazio2014saga, defazio2016simple}.
 
Since we assume that the linear system \eqref{eq:lin-sys-intro} is feasible, it follows that the stochastic gradient vanishes at the optimal point (i.e., $\nabla f_{\bS}(x_*)=0$ for any $\mS$). This suggests that additional variance reduction techniques are not necessary since the variance of the stochastic gradient drops to zero as we approach the optimal point $x_*$. In particular, in our context, SGD with fixed stepsize enjoys linear rate without any variance reduction strategy \cite{needell2014stochastic, gower2015randomized, ASDA}. Hence, in this paper we can bypass the development of variance reduction techniques, which allows us to focus on the momentum term.

\section{Technical Preliminaries} \label{sec:prelim}

A general framework for studying consistent linear systems via carefully designed {\em stochastic reformulations} was recently proposed by Richt\'{a}rik and Tak\'{a}\v{c} \cite{ASDA}. In particular, given the consistent linear system \eqref{eq:lin-sys-intro}, they provide four reformulations in the form of a stochastic optimization problem,  stochastic linear system,  stochastic fixed point problem and a stochastic intersection problem.  These reformulations are equivalent in the sense that their solutions sets are identical. That is, the set of minimizers of  the stochastic optimization problem is equal to the set of solutions of the stochastic linear system and so on. Under a certain assumption,  for which the term {\em exactness} was coined in \cite{ASDA}, the solution sets  of these reformulations are equal to the solution set of the linear system.

\subsection{Stochastic optimization}

 Stochasticity enters the reformulations via a user defined distribution $\cD$ of matrices (all with $m$ rows). In addition, the reformulations  utilize a positive definite matrix $\mB\in \R^{n\times n}$ as a parameter,  used to define an inner product in $\R^n$ via $\langle x,z \rangle_\mB \eqdef \langle \mB x, z\rangle$ and the induced norm
$\|x\|_\mB\eqdef (x^\top \mB x)^{1/2}$.   In particular, the stochastic optimization reformulation \eqref{eq:stoch_reform_intro}, i.e., 
$\min_{x\in \R^n} f(x) \eqdef \E{f_\mS(x)},$
is defined by
setting  \begin{equation}\label{eq:f_s}f_{\mS}(x) \eqdef \frac{1}{2}\|\mA x - b\|_{\mH}^2 = \frac{1}{2}(\mA x - b)^\top \mH (\mA x - b),\end{equation} where $\mH$ is a random symmetric positive semidefinite matrix defined as $\mH \eqdef  \mS (\mS^\top \mA \mB^{-1} \mA^\top \mS)^\dagger \mS^\top.$ By $\dagger$ we denote the Moore-Penrose pseudoinverse.  

\paragraph{Hessian and its eigenvalues.} Note that the Hessian\footnote{While the Hessian is not self-adjoint with respect to the standard inner product, it is self-adjoint with respect to the inner product $\langle \mB x, y\rangle$ which we use as the canonical inner product in $\R^n$.} of $f = \E{f_\mS}$ is given by
$\nabla^2 f = \mB^{-1}\E{\mZ},$
where 
\begin{equation}\label{eq:Z} \mZ \eqdef \mA^\top \mH \mA.\end{equation}
 Note that $\nabla^2 f$ and \begin{equation}\label{eq:W-def}\mW \eqdef \mB^{-1/2} \E{\mZ}\mB^{-1/2}\end{equation} have the same spectrum. Matrix $\mB$ is symmetric and positive semidefinite (with respect to the standard inner product). Let 
\[\mW = \mU \Lambda \mU^\top = \sum_{i=1}^n \lambda_i u_i u_i^\top\]
be the eigenvalue decomposition of $\mW$, where $\mU = [u_1,\dots,u_n]$ is an orthonormal matrix of eigenvectors, and $\lambda_1\leq \lambda_2 \leq \cdots \leq \lambda_{n}$ are the corresponding eigenvalues. Let $\lambda_{\min}^+$ be the smallest nonzero eigenvalue, and $\lambda_{\max}  = \lambda_n$ be the largest eigenvalue. It was shown in \cite{ASDA} that $0\leq \lambda_i \leq 1$ for all $i\in [n]$.

\paragraph{Exactness.} Note that $f_\mS$ is a convex quadratic, and that $f_\mS(x) = 0$ whenever $x\in \cL\eqdef \{x\;:\; \mA x = b\}$. However, $f_\mS$ can be zero also for points $x$ outside of $\cL$. Clearly, $f(x)$ is nonnegative, and $f(x)=0$ for $x\in \cL$. However, without further assumptions, the set of minimizers of $f$ can be larger than $\cL$. The exactness assumption mentioned above ensures that this does not happen. For necessary and sufficient conditions for exactness, we refer the reader to \cite{ASDA}. Here it suffices to remark that  a sufficient condition for exactness is to require $\E{\mH}$ to be positive definite. This is easy to see by observing that
\[f(x) = \E{f_\mS(x)} =\tfrac{1}{2}\|\mA x - b\|^2_{\E{\mH}}. \]

\subsection{Three algorithms for solving the stochastic optimization problem}

The authors of \cite{ASDA} consider solving the stochastic optimization problem \eqref{eq:stoch_reform_intro}  via stochastic gradient descent (SGD)\footnote{The gradient is computed with respect to the inner product $\langle \mB x, y\rangle $.}
\begin{equation}\label{eq:SGD}x_{k+1} = x_k - \omega \nabla f_{\mS_k}(x_k),\end{equation}
where $\omega>0$ is a fixed stepsize and  $\mS_k$ is sampled afresh in each iteration from $\cD$. Note that the gradient of $f_\mS$ with respect to the $\mB$ inner product is equal to
\begin{equation}\label{eq:grad_f_S}\nabla f_\mS(x) \overset{\eqref{eq:f_s}}{=} \mB^{-1} \mA^\top \mH (\mA x - b)  = \mB^{-1} \mA^\top \mH \mA (x -x_*)  = \mB^{-1} \mZ (x -x_*),\end{equation}
where
$\mZ \eqdef \mA^\top \mH \mA$, and $x_*$ is any vector in $ \cL$.

They observe that, surprisingly,  SGD is in this setting equivalent to several other methods; in particular, to the {\em stochastic Newton method}\footnote{In this method we take the $\mB$-pseudoinverse of the Hessian of $f_{\mS_k}$ instead of the classical inverse, as the inverse does not exist. When $\mB=\mI$, the $\mB$ pseudoinverse specializes to the standard Moore-Penrose pseudoinverse.},
\begin{equation}\label{alg:SNM}x_{k+1} = x_k - \omega (\nabla^2 f_{\mS_k}(x_k))^{\dagger_\mB} \nabla f_{\mS_k}(x_k),\end{equation}
and to the {\em stochastic proximal point method}\footnote{In this case, the equivalence only works for $0<\omega\leq 1$.}
\begin{equation}\label{alg:SPPM}x_{k+1} = \arg\min_{x\in \R^n} \left\{ f_{\mS_k}(x) + \frac{1-\omega}{2\omega}\|x-x_k\|_{\mB}^2\right\}.\end{equation}

\subsection{Stochastic fixed point problem}

The stochastic fixed point problem considered in \cite{ASDA} as one of the four stochastic reformulations has the form
\begin{equation} \label{eq:SFP} x = \E{\Pi^\mB_{\cL_\mS}(x)},\end{equation}
where the expectation is taken with respect to $\mS\sim \cD$, and where $\Pi^\mB_{\cL_{\mS}}(x)$ is the projection of $x$, in the $\mB$ norm, onto  the sketched system $\cL_{\mS}= \{x \in \R^n\;:\; \mS^\top \mA x = \mS^\top b\}$. An explicit formula for the projection onto $\cL$ is given by
\begin{equation}
\Pi_{\cL}^\bB(x)\eqdef \arg\min_{x' \in \cL} \|x'-x\|_{\bB} =x-\bB^{-1}\bA^\top (\bA\bB^{-1}\bA ^ \top )^\dagger (\bA x-b);
\end{equation}
a formula for $\cL_\mS$ is obtained by replacing $\mA$ with $\mS^\top \mA$ everywhere.

The {\em stochastic fixed point method} (with relaxation parameter $\omega>0$) for solving 
\eqref{eq:SFP} is defined by \begin{equation}\label{alg:SPM}x_{k+1} = \omega \Pi^\mB_{\cL_{\mS_k}}(x_k) + (1-\omega) x_k.\end{equation}

\subsection{Best approximation problem, its dual and SDSA}

It was shown in \cite{ASDA} that the above methods converge linearly to $x_*=\Pi^{\mB}_{\cL}(x_0)$; the projection of the initial iterate onto the solution set of the linear system. Hence, besides solving problem \eqref{eq:stoch_reform_intro}, they solve the {\em best approximation problem}
\begin{equation}\label{eq:primal}\min_{x\in \R^n} P(x) \eqdef \tfrac{1}{2}\|x-x_0\|_\mB^2 \quad \text{subject to} \quad \mA x = b .\end{equation}
The Fenchel dual of \eqref{eq:primal} is
the (bounded) unconstrained concave quadratic maximization problem
\begin{equation}\label{eq:Dual}
\max_{y\in \R^m} D(y) \eqdef (b-\bA x_0)^\top y - \tfrac{1}{2}\|\bA^\top y\|^2_{\bB^{-1}}.
\end{equation}
Boundedness follows from consistency. It turns out that by varying $\mA, \mB$ and $b$ (but keeping consistency of the linear system), the dual problem in fact captures {\em all} bounded unconstrained concave quadratic maximization problems.

In the special case of unit stepsize, method \eqref{alg:SPM} was first proposed by Gower and Richt\'{a}rik \cite{gower2015randomized} under the name ``sketch-and-project method'', motivated by the iteration structure which proceeds in two steps: i) replace the set $\cL \eqdef \{x\in \R^n \;:\; \mA x = b\}$ by its {\em sketched} variant $\cL_{\mS_k}$, and then project the last iterate $x_k$ onto $\cL_{\mS_k}$. Analysis in \cite{gower2015randomized} was done under the assumption that $\mA$ be of full column rank. This assumption was lifted in \cite{gower2015stochastic}, and a {\em duality} theory for the method developed. In particular, for $\omega=1$, the iterates $\{x_k\}$ arise as  images of the iterates $\{y_k\}$ produced by  a specific {\em dual method} for solving \eqref{eq:Dual} under the mapping $\phi:\R^m \mapsto \R^n$ given by
\begin{equation}\label{eq:phi}\phi (y) \eqdef x_0 + \mB^{-1}\mA^\top y. \end{equation}
The dual method---{\em stochastic dual subspace ascent (SDSA)}---has the form
\begin{equation}\label{alg:SDA}y_{k+1} = y_k + \mS_k \lambda_k,\end{equation}
where $\mS_k$ is in each iteration sampled from $\cD$, and $\lambda_k$ is chosen greedily, maximizing the dual objective $D$: $\lambda_k \in \arg\max_\lambda D(y_k + \mS_k \lambda)$. Such a $\lambda$ might not be unique, however. SDSA is defined by picking the solution with the smallest (standard Euclidean) norm. This leads to the formula:
\[ \lambda_k =  \left(\mS_k^\top \bA \bB^{-1}\bA^\top \mS_k \right)^\dagger\bS_k^\top \left(b-\bA(x_0 + \bB^{-1}\bA^\top y_k) \right).\]

SDSA proceeds by moving in random subspaces spanned by the random columns of $\mS_k$. In the special case when $\omega=1$ and  $y_0=0$,  Gower and Richt\'{a}rik \cite{gower2015stochastic} established the following relationship between the iterates $\{x_k\}$ produced by the primal methods \eqref{eq:SGD}, \eqref{alg:SNM},  \eqref{alg:SPPM},  \eqref{alg:SPM} (which are equivalent), and the dual method \eqref{alg:SDA}:
\begin{equation}\label{eq:dual-corresp}x_{k} = \phi(y_k) \overset{\eqref{eq:phi}}{=}  x_0 + \mB^{-1} \mA^\top y_k.\end{equation}

\subsection{Other related work} Variants of the sketch-and-project methods have been recently proposed for solving several other problems. Xiang and Zhang \cite{double} show that the sketch-and-project framework is capable of expressing, as special cases, randomized  variants of  16 classical  algorithms for solving linear systems.  Gower and Richt\'{a}rik \cite{gower2016randomized, gower2016linearly} use similar ideas to develop of linearly convergent randomized iterative methods for computing/estimating the inverse and the pseudoinverse of a large matrix, respectively. A limited memory variant of the stochastic block BFGS method  for solving the empirical risk minimization problem arising in machine learning  was proposed by Gower et al.\ \cite{gower2016stochastic}. Tu et al.\ \cite{tu2017breaking}  utilize the sketch-and-project framework to show that breaking block locality can accelerate block Gauss-Seidel methods. In addition, they develop an accelerated variant of the method for a specific distribution $\cD$.  Loizou and Richt\'{a}rik~\cite{LoizouRichtarik} use the sketch-and-project method to solve the average consensus problem; and  Hanzely et al.\ \cite{hanzely2017privacy} design new variants of sketch and project methods for the average consensus problem with  privacy considerations (see Section~\ref{consensus} for more details regarding the average consensus problem).

\section{Primal Methods with Momentum} \label{sec:primal}

 Applied to problem \eqref{eq:stoch_reform_intro}, i.e., 
$\min_{x\in \R^n} f(x) = \E{f_\mS(x)},$
  the gradient descent method with momentum (also known as the  heavy ball method) of Polyak \cite{polyak1964some, polyak1987introduction} takes the form
\begin{equation}
\label{HB}
x_{k+1} = x_k - \omega \nabla f(x_k) + \beta(x_k - x_{k-1}),
\end{equation}
where $\omega>0$ is a stepsize and $\beta\geq 0$ is a momentum parameter. Instead of marrying the momentum term with gradient descent, we can marry it with SGD. This leads to SGD with momentum (mSGD), also known as the {\em stochastic heavy ball method}:
\begin{equation}\label{eq:SHB-intro} x_{k+1} = x_k - \omega \nabla f_{\mS_k}(x_k) + \beta(x_k-x_{k-1}).\end{equation}

Since SGD is equivalent to SN and SPP, this way we obtain momentum variants of the stochastic Newton  (mSN) and  stochastic proximal point (mSPP) methods.  The method is formally described below:

\begin{center}
\boxed{
\begin{minipage}[!h][][b]{0,9\textwidth}
\noindent \textbf{mSGD / mSN / mSPP}

\bigskip

{\bf Parameters:} 
Distribution $\mathcal{D}$ from which method samples matrices; positive definite matrix $\bB \in \R^{n\times n}$; stepsize/relaxation parameter $\omega \in \R$ the heavy ball/momentum parameter $\beta$.

\bigskip
{\bf Initialize:} Choose initial points $x_0,x_1\in \R^n$

\bigskip

For $k\geq 1$ do
\begin{enumerate}
\item Draw a fresh $\bS_k \sim \cD$
\item Set $$x_{k+1}=x_k -\omega \nabla f_{\bS_k}(x_k) + \beta(x_k - x_{k-1}) $$
\end{enumerate}
{\bf Output:} last iterate $x_k$
\end{minipage}
}
\end{center}

To the best of our knowledge, momentum variants of SN and SPP were not considered  in the literature before. Moreover, as far as we know, there are no momentum variants of even deterministic variants of \eqref{alg:SNM}, \eqref{alg:SPPM} and \eqref{alg:SPM}, such as incremental or batch Newton method,  incremental or batch proximal point method and  incremental or batch projection method; not even for a problem formulated differently.

In the rest of this section we state our convergence results for mSGD/mSN/mSPP. 

\subsection{$L2$ convergence and function values: linear rate}
In this section we study L2 convergence of mSGD/mSN/mSPP; that is, we study the convergence of the quantity $\Exp[\|x_k-x_*\|_{\mB}^2]$ to zero. We show that for a range  of stepsize parameters $\omega > 0$ and  momentum terms $\beta \geq 0$ the method enjoys global  linear convergence rate. To the best of our knowledge, these results are the first of their kind for the stochastic heavy ball method.
As a corollary of  L2 convergence, we  obtain convergence of the expected function values.

\begin{thm}
\label{L2}
Choose $x_0= x_1\in \R^n$.  Assume exactness. Let $\{x_k\}_{k=0}^\infty$ be the sequence of random iterates produced by mSGD/mSN/mSPP.  Assume $0< \omega < 2$ and $\beta \geq 0$ and that the expressions
\[a_1 \eqdef 1+3\beta+2\beta^2 - (\omega(2-\omega) +\omega\beta)\lambda_{\min}^+, \qquad \text{and}\qquad
a_2 \eqdef \beta +2\beta^2 + \omega \beta \lambda_{\max}\]
satisfy $a_1+a_2<1$. Let $x_* = \Pi_{\mathcal{L}}^{\bB}(x_0)$. Then 
\begin{equation}\label{eq:nfiug582}\Exp[\|x_{k}-x_*\|^2_{\bB}] \leq q^k (1+\delta)  \|x_{0}-x_*\|^2_{\bB}\end{equation}
and 
$$\Exp[f(x_k)] \leq q^k  \frac{\lambda_{\max}}{2} (1+\delta) \|x_{0}-x_*\|^2_{\bB},$$
where  $q=\frac{a_1+\sqrt{a_1^2+4a_2}}{2}$ and $\delta=q-a_1$. Moreover, $a_1+a_2 \leq q <1$.
\end{thm}
\begin{proof} See Appendix~\ref{app:1}.
\end{proof}

In the above theorem we obtain a global linear rate. To the best of our knowledge, this is the first time that linear rate is established for a stochastic variant of the heavy ball method (mSGD) in any setting. All existing results are sublinear. These seem to be  the first momentum variants of SN and SPP methods.

If we choose $\omega \in (0,2)$, then the condition $a_1+a_2<1$ is satisfied for all   \begin{equation}\label{rangesSHB} 0\leq \beta< \tfrac{1}{8} \left( -4+\omega \lambda_{\min}^+-\omega \lambda_{\max} +\sqrt{(4-\omega \lambda_{\min}^++\omega \lambda_{\max})^2+16\omega (2-\omega) \lambda_{\min}^+ }\right).\end{equation}

If $\beta=0$, mSGD reduces to  SGD analyzed in \cite{ASDA}. In this special case, $q = 1-\omega(2-\omega)\lambda_{\min}^+$, which is the rate established in \cite{ASDA}. Hence, our result is more general.

Let $q(\beta)$ be the rate as a function of $\beta$. Note that since $\beta\geq 0$, we have
\begin{eqnarray} q(\beta) &\geq & a_1 + a_2 \notag \\
&=& 1 + 4\beta + 4\beta^2 + \omega\beta(\lambda_{\max}-\lambda_{\min}^+) - \omega(2-\omega)\lambda_{\min}^+ \notag \\
&\geq & 1-\omega(2-\omega)\lambda_{\min}^+ = q(0).\label{eq:qbeta}\end{eqnarray}
Clearly, the lower bound on $q$ is an increasing function of $\beta$. 
Also, for any $\beta$ the rate is always inferior to that of SGD ($\beta=0$). It is an open problem whether one can prove a strictly better rate for mSGD than for SGD.

Our next theorem states that $\Pi_\cL^\mB(x_k) = x_*$ for all iterations $k$ of mSGD. This invariance is important, as it allows the algorithm to converge to $x_*$.

\begin{thm}\label{prop:projections}
Let $x_0=x_1\in \R^n$ be the starting points of the mSGD method and let $\{x_k\}$ be the random iterates generated by mSGD. Then $\Pi_\cL^\bB(x_k)=\Pi_\cL^\bB(x_0)$ for all $k\geq 0$.
\end{thm}
\begin{proof} Note that in view of \eqref{eq:f_s}, $ \nabla f_{\mS}(x) = \mB^{-1}\mA^\top \mH (\mA x - b) \in {\rm Range}(\mB^{-1}\mA^\top)$. Since 
\[x_{k+1} = x_k -\omega \nabla f_{\bS_k}(x_k) + \beta(x_k - x_{k-1}),\]
and since $x_0=x_1$, it can shown by induction that $x_k \in x_0 + {\rm Range}(\mB^{-1}\mA^\top)$ for all $k$. However, ${\rm Range}(\mB^{-1}\mA^\top)$ is the orthogonal complement to ${\rm Null}(\mA)$ in the $\mB$-inner product. Since $\cL$ is parallel to ${\rm Null}(\mA)$, vectors $x_k$ must have the same $\mB$-projection onto $\cL$ for all $k$: $\Pi^\mB_\cL(x_0) = x_*$.
\end{proof}

\subsection{Cesaro average: sublinear rate without exactness assumption}

In this section we present the convergence analysis of the function values computed on the Cesaro average. Again our results are global in nature. To the best of our knowledge are the first results that show $O(1/k)$ convergence of the stochastic heavy ball method. Existing results apply in more general settings at the expense of slower rates. In particular, \cite{yang2016unified} and \cite{gadat2016stochastic} get $O(1/\sqrt{k})$ and $O(1/k^{\beta})$ convergence, respectively. When $\beta=1$, \cite{gadat2016stochastic} gets $O(1/\log(k))$ rate.

\begin{thm}
\label{cesaro}
Choose $x_0=x_1$ and let $\{x_k\}_{k=0}^\infty$ be the random iterates produced by mSGD/mSN/mSPP, where the momentum parameter $0\leq \beta <1$ and relaxation parameter (stepsize) $\omega > 0$ satisfy $\omega + 2\beta <2$. Let $x_*$ be any vector satisfying $f(x_*)=0$. If we let $\hat{x}_k=\frac{1}{k}\sum_{t=1}^{k}x_t$, then
$$\Exp[f(\hat{x}_k)] \leq \frac{(1-\beta)^2\|x_0-x_*\|_{\mB}^2 + 2\omega \beta f(x_0)}{2\omega(2-2\beta-\omega) k}.$$
\end{thm}

\begin{proof} See Appendix~\ref{app:acc212}.
\end{proof}

In the special case of $\beta=0$, the above theorem gives the rate
$$\Exp[f(\hat{x}_k)] \leq \frac{\|x_0-x_*\|^2_{\bB}}{2\omega (2 -\omega) k} .$$
This is the convergence rate for Cesaro averges of the ``basic method'' (i.e., SGD) established in \cite{ASDA}.

Our proof strategy  is similar to  \cite{ghadimi2015global} in which the first global convergence analysis of the (deterministic) heavy ball method was presented. There it was shown that when the objective function has a Lipschitz continuous gradient, the Cesaro averages of the iterates converge to the optimum at a rate of $O(1/k)$. To the best of our knowledge, there are no results in the literature that prove the same rate of convergence in the stochastic case for any class of objective functions.  

In \cite{yang2016unified} the authors analyzed  mSGD for general Lipshitz continuous convex objective functions (with bounded variance) and proved the {\em sublinear} rate $O(1/\sqrt{k})$. In \cite{gadat2016stochastic}, a complexity analysis is provided for the case of quadratic strongly convex smooth coercive functions. A  sublinear convergence  rate of $O(1/k^\beta)$, where $\beta \in (0,1)$, was proved. In contrast to our results, where we assume fixed stepsize $\omega$, both papers analyze mSGD with diminishing stepsizes.

\subsection{$L1$ convergence: accelerated linear rate}

In this section we show that by a proper combination of the relaxation (stepsize) parameter  $\omega$ and the momentum parameter $\beta$, mSGD/mSN/mSPP enjoy an {\em accelerated} linear convergence rate in mean.

\begin{thm}
\label{theoremheavyball}
Assume exactness. Let $\{x_k\}_{k=0}^{\infty}$ be the sequence of random iterates produced by mSGD / mSN / mSPP, started with $x_0, x_1 \in \R^n$ satisfying the relation $x_0-x_1 \in {\rm Range}(\bB^{-1} \bA^ \top)$, with relaxation parameter (stepsize)  $0<\omega \leq1/\lambda_{\max}$ and momentum parameter  $(1-\sqrt{\omega \lambda_{\min}^+})^2 < \beta <1$. Let $x_* = \Pi^\mB_{\cL}(x_0)$. Then there exists constant $C >0$ such that for all $k\geq0$ we have 
$$\|\Exp[x_{k} -x_*]\|_{\bB}^2  \leq \beta^k C.$$

\begin{itemize}
\item[(i)] If we choose $ \omega= 1$ and $\beta= \left(1- \sqrt{0.99 \lambda_{\min}^+}\right) ^2$ then
 $\|\Exp[x_{k} -x_*]\|_{\bB}^2  \leq \beta^k C$
and the iteration complexity becomes
$ \tilde{O}\left(\sqrt{1/ \lambda_{\min}^+}\right)$.
\item[(ii)] If we choose $ \omega= 1/\lambda_{\max}$ and $\beta= \left(1- \sqrt{ \frac{0.99\lambda_{\min}^+}{\lambda_{\max}}}\right)^2$ then
$\|\Exp[x_{k} -x_*]\|_{\bB}^2  \leq \beta^k C$
and the iteration complexity becomes
$\tilde{O}\left(\sqrt{\lambda_{\max}/ \lambda_{\min}^+}\right)$.

\end{itemize}
\end{thm}
\begin{proof} See Appendix~\ref{app:acc}.
\end{proof}

Note that the convergence factor is precisely equal to the value of the momentum parameter $\beta$. Let $x$ be any random vector in $\R^n$ with finite mean $\Exp[x]$, and $x_*\in \R^n$ is any reference vector (for instance, any solution of $\mA x = b$). Then we have the identity (see, for instance \cite{gower2015randomized})
\begin{equation}
\label{weakstrong}
\E{\|x-x_*\|_{\bB}^2} = \|\E{x-x_*}\|_{\bB}^2 + \E{\|x-\Exp[x]\|^2_{\bB}}.
\end{equation}
This means that the quantity $\E{\|x-x_*\|_{\bB}^2}$ appearing in our L2 convergence result (Theorem~\ref{L2}) is larger than 
$\|\E{x-x_*}\|_{\bB}^2$ appearing in the L1 convergence result (Theorem~\ref{theoremheavyball}), and hence harder to push to zero. As a corollary, L2 convergence implies L1 convergence. However, note that in Theorem~\ref{theoremheavyball} we have established an {\em accelerated} rate.  A similar theorem, also obtaining an accelerated rate in the L1 sense, was established in \cite{ASDA} for an accelerated variant of SGD in the sense of Nesterov.

\section{Dual Methods with Momentum} \label{sec:dual}

In the previous sections we focused on methods for solving the stochastic optimization problem \eqref{eq:stoch_reform_intro} and the best approximation problem~\eqref{eq:best_approx_intro}. In this section we focus on the dual of the best approximation problem, and propose a momentum variant of SDSA, which we call mSDSA. 

\begin{center}
\boxed{
\begin{minipage}[!h][][b]{0,9\textwidth}
\noindent \textbf{Stochastic Dual Subspace Ascent with Momentum (mSDSA)}

\bigskip

{\bf Parameters:} 
Distribution $\mathcal{D}$ from which method samples matrices; positive definite matrix $\bB \in \R^{n\times n}$; stepsize/relaxation parameter $\omega \in \R$ the heavy ball/momentum parameter $\beta$. SDSA is obtained as a special case of mSDSA for $\beta=0$. 

\bigskip
{\bf Initialize:} Choose initial points $y_0 =  y_1 = 0 \in \R^m$

\bigskip

For $k\geq 1$ do
\begin{enumerate}
\item Draw a fresh $\bS_k \sim \cD$
\item Set $\lambda_k = \left(\mS_k^\top \bA \bB^{-1}\bA^\top \mS_k \right)^\dagger\bS_k^\top \left(b-\bA(x_0 + \bB^{-1}\bA^\top y_k) \right)$
\item Set $y_{k+1}=y_k+\omega \mS_k \lambda_k +\beta(y_k-y_{k-1})$
\end{enumerate}
{\bf Output:} last iterate $y_k$
\end{minipage}
}
\end{center}

\subsection{Correspondence between primal and dual methods}

In our first result we show that the random iterates of the mSGD/mSN/mSPP methods arise as an affine image of mSDSA under the mapping $\phi$ defined in \eqref{eq:phi}.

\begin{thm}[Correspondence Between Primal and Dual Methods]  \label{thm:dual_corresp} Let $x_0=x_1$ and let $\{x_k\}$ be the iterates of mSGD/mSN/mSPP. Let $y_0=y_1=0$, and let $\{y_k\}$ be the iterates of mSDSA.   Assume that the  methods use the same stepsize $\omega > 0$, momentum parameter $\beta\geq 0$, and the same sequence of random matrices $\mS_k$.  Then  
\[x_k = \phi(y_k) = x_0 + \mB^{-1} \mA^\top y_k\]
for all $k$. That is, the primal iterates arise as affine images of the dual iterates.
\end{thm}
\begin{proof}
First note that
\begin{eqnarray*}\nabla f_{\mS_k}(\phi(y_k)) &\overset{\eqref{eq:grad_f_S}}{=}& \mB^{-1}\mA^\top  \mS_k ( \mS_k^\top \mA \mB^{-1} \mA^\top \mS_k)^\dagger \mS_k^\top (\mA \phi(y_k) - b) \; =\;  - \mB^{-1}\mA^\top  \mS_k \lambda_k.
\end{eqnarray*}
We now use this to show that
\begin{eqnarray*}
\phi(y_{k+1}) &\overset{\eqref{eq:phi}}{=}& x_0  + \bB^{-1}\bA^\top y_{k+1}\\
&=& x_0+\bB^{-1}\bA^\top\left[y_k+\omega \bS_k \lambda_k +\beta(y_k-y_{k-1})\right]\\
&=& \underbrace{x_0+\bB^{-1}\bA^\top y_k}_{\phi(y_k)}+\omega \underbrace{\bB^{-1}\bA^\top \bS_k \lambda_k}_{-\nabla f_{\bS_k}(\phi(y_k))} +\beta \bB^{-1}\bA^\top (y_k-y_{k-1})\\
&=& \phi(y_k)- \omega \nabla f_{\bS_k}(\phi(y_k)) +\beta  (\bB^{-1}\bA^\top y_k-\bB^{-1}\bA^\top y_{k-1})\\
&\overset{\eqref{eq:phi}}{=}& \phi(y_k) -\omega \nabla f_{\bS_k}(\phi(y_k)) +\beta  (\phi(y_k) - \phi(y_{k-1})).
\end{eqnarray*}
So, the sequence of vectors $\{\phi(y_k)\}$ mSDSA satisfies the same recursion of degree as the sequence $\{x_k\}$ defined by mSGD.  It remains to check that the first two elements of both recursions coincide. Indeed, since $y_0=y_1=0$ and $x_0=x_1$, we have $x_0 = \phi(0) =  \phi(y_0)$, and $x_1 = x_0= \phi(0) =  \phi(y_1)$.
\end{proof}

\subsection{Convergence}

We are now ready to state a linear convergence convergence result describing the behavior of mSDSA in terms of the dual function values $D(y_k)$.

\begin{thm}[Convergence of dual objective] 
\label{thm:dual-conv}
Choose $y_0= y_1\in \R^n$.  Assume exactness. Let $\{y_k\}_{k=0}^\infty$ be the sequence of random iterates produced by mSDSA.  Assume $0\leq \omega \leq 2$ and $\beta \geq 0$ and that the expressions
\[a_1 \eqdef 1+3\beta+2\beta^2 - (\omega(2-\omega) +\omega\beta)\lambda_{\min}^+, \qquad \text{and}\qquad
a_2 \eqdef \beta +2\beta^2 + \omega \beta \lambda_{\max}\]
satisfy $a_1+a_2<1$. Let $x_* = \Pi_{\mathcal{L}}^{\bB}(x_0)$ and let $y_*$ be any dual optimal solution. Then 
\begin{equation}
\label{eq:nfiug582}
\Exp[D(y_*)-D(y_k)] \leq q^k (1+\delta) \left[D(y_*)-D(y_0)\right]
\end{equation}
where  $q=\frac{a_1+\sqrt{a_1^2+4a_2}}{2}$ and $\delta=q-a_1$. Moreover, $a_1+a_2 \leq q <1$.
\end{thm}
\begin{proof} This follows by applying Theorem~\ref{L2} together with Theorem \ref{thm:dual_corresp} and the identity $\tfrac{1}{2}\|x_k-x_0\|^2_\mB = D(y_*) - D(y_k)$.
\end{proof}

Note that for $\beta=0$, mSDSA simplifies to SDSA. Also recall that for unit stepsize ($\omega=1$), SDSA was analyzed in  \cite{gower2015randomized}. In the $\omega=1$ and $\beta=0$ case, our result specializes to that established in  \cite{gower2015randomized}. Following similar arguments to those in  \cite{gower2015randomized},  the same rate of convergence can be proved for the duality gap $\Exp[P(x_k)-D(y_k)]$.

\section{Methods  with Stochastic Momentum} \label{sec:sm}

To motivate {\em stochastic momentum}, for simplicity fix $\mB=\mI$, and assume that $\mS_k$ is chosen as the $j$th random unit coordinate vector  of $\R^m$ with probability $p_j>0$. In this case, SGD \eqref{eq:SGD} reduces to the randomized Kaczmarz method for solving the linear system $\mA x = b$, first analyzed for $p_j\sim \|\mA_{j:}\|^2$ by Strohmer and Vershynin \cite{RK}. 

In this case, mSGD becomes the {\em randomized Kaczmarz method with momentum} (mRK), and the iteration \eqref{eq:SHB-intro}  takes the explicit form
\[x_{k+1} = x_k - \omega \frac{\bA_{j:} x_k -b_j}{\|\bA_{j:}\|^2} \bA_{j:}^ \top + \beta(x_k-x_{k-1}).\]
Note that the cost of one iteration of this method is $\cO(\|\mA_{j:}\|_0 + n)$, where the cardinality term $\|\mA_{j:}\|_0$ comes from the stochastic gradient part, and $n$ comes from the momentum part.
When $\mA$ is sparse, the second term will dominate. Similar considerations apply for many other (but clearly not all) distributions~$\cD$. 

In such circumstances, we propose to replace the expensive-to-compute momentum term by a cheap-to-compute stochastic approximation thereof. In particular, we let $i_k$ be chosen from $[n]$ uniformly at random,  and replace $x_k-x_{k-1}$ with $v_{i_k} \eqdef e_{i_k}^\top (x_k-x_{k-1})e_{i_k}^\top$,  where $e_{i_k} \in \R^n$ is the $i_k$th unit basis vector in $\R^n$, and $\beta$ with $n\beta$. Note that $v_{i_k}$ can be computed in $\cO(1)$ time. Moreover, \[\Exp_{i_k}[n\beta v_{i_k}] =\beta( x_k-x_{k-1}).\] Hence,  we replace the momentum term by  an unbiased estimator, which allows us to cut the cost to $\cO(\|\mA_{j:}\|_0)$.

\subsection{Primal methods with stochastic momentum}

We now propose a  variant of the SGD/SN/SPP methods employing stochastic momentum  (smSGD/smSN/smSPP). Since SGD, SN and SPP are equivalent, we will describe the development from the perspective of SGD.
In particular, we propose the following method:
\begin{equation}\label{eq:DSHB-intro} x_{k+1} = x_k - \omega \nabla f_{\mS_k}(x_k) + n\beta e_{i_k}^\top (x_k-x_{k-1})e_{i_k}.\end{equation}

The method is formalize below:

\begin{center}
\boxed{
\begin{minipage}[!h][][b]{0,9\textwidth}
\noindent \textbf{smSGD/smSN/smSPP}

\bigskip

{\bf Parameters:} 
Distribution $\mathcal{D}$ from which the method samples matrices;  stepsize/relaxation parameter $\omega \in \R$ the heavy ball/momentum parameter $\beta$.

\bigskip
{\bf Initialize:} Choose initial points $x_1=x_0  \in \R^n$; set $\bB  = \mI\in \R^{n\times n}$

\bigskip

For $k\geq 1$ do
\begin{enumerate}
\item Draw a fresh $\bS_k \sim \cD$
\item Pick $i_k\in [n]$ uniformly at random 
\item Set $$x_{k+1}=x_k -\omega \nabla f_{\bS_k}(x_k) + \beta e_{i_k}^\top (x_k - x_{k-1})e_{i_k} $$
\end{enumerate}
{\bf Output:} last iterate $x_k$
\end{minipage}
}
\end{center}

\subsection{Convergence}

In the next result we establish L2 linear convergence of smSGD/smSN/smSPP. For this we will require the matrix $\mB$ to be equal to the identity matrix.

\begin{thm}\label{thm:DSHB-L2} Choose $x_0= x_1\in \R^n$.  Assume exactness. Let $\mB=\mI$.  Let $\{x_k\}_{k=0}^\infty$ be the sequence of random iterates produced by smSGD/smSN/smSPP.  Assume $0< \omega < 2$ and $\beta \geq 0$ and that the expressions
\begin{equation}\label{eq:98ys8h89dh} a_1 \eqdef 1+3\tfrac{\beta}{n}+2\tfrac{\beta^2}{n} - \left(\omega(2-\omega) +\omega\tfrac{\beta}{n}\right)\lambda_{\min}^+, \qquad \text{and}\qquad
a_2 \eqdef \tfrac{1}{n}(\beta +2\beta^2 + \omega \beta \lambda_{\max}) 
\end{equation}
satisfy $a_1+a_2<1$. Let $x_* = \Pi_{\mathcal{L}}^{\bI}(x_0)$. Then 
\begin{equation}\label{eq:nfiug582X}\Exp[\|x_{k+1}-x_*\|^2] \leq q^k (1+\delta)  \|x_{0}-x_*\|^2
\end{equation}
and 
$\Exp[f(x_k)] \leq q^k  \frac{\lambda_{\max}}{2} (1+\delta) \|x_{0}-x_*\|^2,$
where  $q\eqdef \frac{a_1+\sqrt{a_1^2+4a_2}}{2}$ and $\delta\eqdef q-a_1$. Moreover, $ a_1+a_2 \leq q < 1$.
\end{thm}

\begin{proof} See Appendix~\ref{app:7}.
\end{proof}

It is straightforward to see that if we choose $\omega \in (0,2)$, then the condition $a_1+a_2<1$ is satisfied for all $\beta$ belonging to the interval   \[ 0\leq \beta < \tfrac{1}{8}\left (-4+\omega \lambda_{\min}^+-\omega \lambda_{\max}+\sqrt{(4-\omega \lambda_{\min}^++\omega \lambda_{\max})^2+16n\omega (2-\omega)  \lambda_{\min}^+}\right).\]
The  upper bound  is similar to that for mSGD/mSN/mSPP; the only difference is an extra factor of $n$ next to the constant~16.

\subsection{Momentum versus stochastic momentum}
\label{comparison}

As indicated in the introduction, if we wish to compare mSGD to smSGD used with momentum parameter $\beta$, it makes sense to use momentum parameter $\beta n$ in smSGD. This is because  the momentum term in smSGD will then be an unbiased estimator of the deterministic momentum term used in mSGD.

Let $q(\beta)$ be the convergence constant for mSGD with stepsize $\omega=1$ and an admissible momentum parameter $\beta\geq 0$. Further, let $\bar{a}_1(\beta),\bar{a}_2(\beta),\bar{q}(\beta)$ be the convergence constants for smSGD with stepsize $\omega=1$ and momentum parameter $\beta$.  We have
\begin{eqnarray*}\bar{q}(\beta n) \geq  \bar{a}_1(\beta n) + \bar{a}_2(\beta n) &\overset{\eqref{eq:98ys8h89dh}}{=}& 1 + 4\beta + 4\beta^2n + \beta(\lambda_{\max}-\lambda_{\min}^+) - \lambda_{\min}^+\\
&\overset{\eqref{eq:qbeta}}{=}& a_1(\beta) + a_2(\beta) + 4\beta^2 (n-1)\\
&\geq & a_1(\beta) + a_2(\beta).
 \end{eqnarray*}
 Hence, the lower bound on the rate for smSGD is worse than the lower bound for mSGD. 
 
The same conclusion holds for the convergence rates themselves. Indeed, note that since $\bar{a}_1(\beta n) - a_1(\beta) = 2\beta^2(n-1) \geq 0$ and $\bar{a}_2(\beta n) - a_2(\beta) = 2\beta^2 (n-1) \geq 0$, we have
\[\bar{q}(\beta n) = \frac{\bar{a}_1(\beta n) + \sqrt{\bar{a}^2_1(\beta n) +4 \bar{a}_2(\beta n)}}{2} \geq \frac{a_1(\beta ) + \sqrt{a_1^2(\beta ) +4 a_2(\beta )}}{2} = q(\beta),\]
and hence the rate of mSGD is always better than that of smSGD.

However, the expected cost of a single iteration of mSGD may be significantly larger than that of smSGD. Indeed, let $g$ be the expected cost of evaluating a stochastic gradient. Then we need to compare $\cO(g+n)$ (mSGD) against $\cO(g)$ (smSGD). If $g\ll n$, then one iteration of smSGD is significantly cheaper than one iteration of mSGD. Let us now compare the total complexity to investigate the trade-off between the rate and cost of stochastic gradient evaluation. Ignoring constants, the total cost of the two methods (cost of a single iteration multiplied by the number of iterations) is:
\begin{equation} \label{eq:C-SHB}C_{\text{mSGD}}(\beta) \eqdef \frac{g+n}{1-q(\beta)} = \frac{g+n}{1-\frac{a_1(\beta) + \sqrt{a_1^2(\beta) + 4 a_2(\beta)}}{2}}, \end{equation}
and
\begin{equation} \label{eq:C-DSHB}C_{\text{smSGD}}(\beta n) \eqdef \frac{g}{1-\bar{q}(\beta n)} = \frac{g}{1-\frac{\bar{a}_1(\beta n) + \sqrt{\bar{a}_1^2(\beta n) + 4 \bar{a}_2(\beta n)}}{2}}. \end{equation}

Since  \begin{equation}\label{eq:iod886562}q(0) = \bar{q}(0 n),\end{equation} and since $q(\beta)$ and $\bar{q}(\beta n)$ are continuous functions of $\beta$, then because $g+n > g$, for small enough $\beta$ we will have 
$C_{\text{mSGD}}(\beta) > C_{\text{smSGD}}(\beta n).$ In particular, the speedup of smSGD compared to mSGD for $\beta \approx 0$ will be close to
\[\frac{C_{\text{mSGD}}(\beta)}{C_{\text{smSGD}}(\beta n)}  \approx \lim_{\beta' \to_+ 0} \frac{C_{\text{mSGD}}(\beta')}{C_{\text{smSGD}}(\beta' n)} \overset{\eqref{eq:C-SHB}+\eqref{eq:C-DSHB}+\eqref{eq:iod886562}}{=} \frac{g+n}{g} = 1 + \frac{n}{g}.\]

Thus, we have shown the following statement.

\begin{thm} \label{thm:DSHBspeedup} For small $\beta$, the total complexity of smSGD is approximately $1+ n/g$ times smaller than the total complexity of mSGD, where $n$ is the number of columns of $\mA$, and $g$ is the expected cost of evaluating a stochastic gradient $\nabla f_{\mS}(x)$.  
\end{thm}

\section{Special Cases: Randomized Kaczmarz with Momentum and Randomized Coordinate Descent with Momentum} \label{sec:cases}

In Table~\ref{SpecialCasesAlgorithms} we specify several special instances of mSGD  by choosing distinct combinations of the parameters $\cD$ and $\bB$.  We use $e_i$ to denote the $i$th unit coordinate vector in $\R^m$, and $\bI_{:C}$ for the column submatrix of the $m\times m$ identity matrix indexed by (a random) set $C$.

\begin{table}[t!]
\begin{center}
\scalebox{0.9}{
\begin{tabular}{ | p{4.3cm} | p{1.1 cm} | p{1.2cm} |  p{8.7cm} | }
  \hline
 \multicolumn{4}{| c |}{\textbf{Variants of mSGD}}\\
 \hline
 \begin{center}Variant of mSGD \end{center}& \begin{center}$\mS$\end{center} & \begin{center}$\bB$\end{center}  & \begin{center}$x_{k+1}$\end{center}\\
 \hline
  \hline
\begin{center}{\bf mRK}: randomized Kaczmarz with momentum \end{center} & $$e_{i}$$ & $$\bI$$ & $$x_k -\omega \frac{\bA_{i :} x_k -b_{i}}{\|\bA_{i :}\|_2^2} \bA_{i :}^ \top + \beta(x_k - x_{k-1}) $$\\
  \hline
 \begin{center} {\bf mRCD = mSDSA}: randomized coordinate desc. with momentum \end{center} & $$e_{i}$$ &  $$\bA\succ 0$$  & $$ x_k -\omega \frac{(\bA_{i:})^\top x_k -b_i}{\bA_{ii}} e_i  + \beta(x_k - x_{k-1}) $$\\
 \hline
 \begin{center} {\bf mRBK}: randomized block Kaczmarz with momentum \end{center} & $$\bI_{:C}$$ & $$\bI$$ &  $$x_k -\omega \bA_{C:}^\top (\bA_{C:}\bA_{C:}^\top)^\dagger (\bA_{C:}x_k-b_C) + \beta(x_k - x_{k-1}) $$\\
  \hline  
\begin{center}{\bf mRCN = mSDSA}:  randomized coordinate Newton descent with momentum  \end{center}& $$\bI_{:C}$$ & $$\bA\succ0$$& $$x_k -\omega \bI_{:C} (\bI_{:C}^\top\bA \bI_{:C})^\dagger \bI_{:C}^\top (\bA x_k-b) + \beta(x_k - x_{k-1}) $$\\
  \hline
\begin{center}{\bf mRGK}: randomized Gaussian Kaczmarz \end{center}   &$$ N(0,\bI)$$ & $$\bI$$ & $$x_k -\omega \frac{\mS^\top (\bA x_k -b)}{\|\bA^\top \mS \|_2^2} \bA^\top \mS + \beta(x_k - x_{k-1}) $$\\
  \hline 
 \begin{center} {\bf mRCD}: randomized coord. descent (least squares)\end{center} & $$ \bA_{:i}$$ & $$\bA^\top \bA$$ & $$x_k -\omega \frac{(\bA_{:i})^\top (\bA x_k -b)}{\|\bA_{:i}\|_2^2} e_i + \beta(x_k - x_{k-1}) $$ \\
 \hline
\end{tabular}
}
\end{center}
\caption{Selected special cases of mSGD.  In the special case of $\mB=\mA$, mSDSA is directly equivalent to mSGD (this is due to the primal-dual relationship \eqref{eq:dual-corresp}; see also Theorem~\ref{thm:dual_corresp}).  Randomized coordinate Newton  (RCN) method was first proposed in \cite{qu2015sdna}; mRCN is its momentum variant. Randomized Gaussian Kaczmarz (RGK) method was first proposed in \cite{gower2015randomized}; mRGK is its momentum variant.
}
\label{SpecialCasesAlgorithms}
\end{table}

The updates for smSGD can be derived by substituting the momentum term $\beta(x_k-x_{k-1})$ with its stochastic variant $n\beta e_{i_k}^\top (x_k-x_{k-1})e_{i_k}$. We do not aim to be comprehensive. For more details on the possible combinations of the parameters $\bS$ and $\bB$  we refer the interested reader to Section~3 of \cite{gower2015randomized}. 

In the rest of this section we present in detail two special cases: the randomized Kaczmarz method with momentum (mRK) and the randomized coordinate descent method with momentum (mRCD). Further, we compare the $L1$ convergence rates (i.e., bounds on $\|\E{x_k}-x_*\|_\mB^2$) obtained in this paper  with rates that can be inferred from known results for their no-momentum variants.

\subsection{mRK: randomized Kaczmarz with momentum} We now provide a discussion on mRCD (the method in the first row of Table~\ref{SpecialCasesAlgorithms}).  Let $\bB= \bI$ and let pick in each iteration the random matrix $\bS=e_i$ with probability $p_i=\|\bA_{i:}\|^2 / \|\bA\|_F^2$. In this setup the update rule of the mSGD simplifies to $$x_{k+1}=x_k -\omega \frac{\bA_{i:} x_k -b_i}{\|\bA_{i:}\|_2^2} \bA_{i:}^ \top + \beta(x_k - x_{k-1}) $$
and 
\begin{eqnarray}
\bW &\overset{\eqref{eq:W-def}}{=}& \bB^{-1/2} \mA^\top \Exp_{\mS\sim \cD}[\mH] \mA \bB^{-1/2} \; = \; \Exp[\mA ^\top \mH \mA] \notag \\
&=& \sum_{i=1}^{m}p_i\frac{\bA_{i:}^\top\bA_{i:}}{\|\bA_{i:}\|^2_2} \; = \; \frac{1}{\|\bA\|^2_{F}} \sum_{i=1}^{m}\bA_{i:}^\top\bA_{i:} \; =\; \frac{\bA^\top \bA}{\|\bA\|^2_F}.\label{matrixW}
\end{eqnarray}

The objective function takes the following form:
\begin{equation}
\label{functionRK}
f(x)=\Exp_{\bS \sim \mathcal{D}} [f_{\bS}(x)]= \sum_{i=1}^m p_i f_{\bS_i}(x)=\frac{\|\bA x -b\|^2_2}{2\|\bA\|^2_{F}}.
\end{equation}

For $\beta=0$, this method reduces to the \textit{randomized Kaczmarz method} with relaxation, first analyzed in \cite{ASDA}. If we also have $ \omega=1$, this is equivalent with the \textit{randomized Kaczmarz method} of Strohmer and Vershynin \cite{RK}.  RK without momentum ($\beta=0$)  and without relaxation ($\omega=1$) 
converges  with iteration complexity \cite{RK, gower2015randomized, gower2015stochastic} of \begin{equation}\label{eq:biugs897*9h8}\tilde{O}(1/\lambda_{\min}^+(\bW))=\tilde{O}\left(\frac{\|\bA\|^2_F}{\lambda_{\min}^+(\bA^\top \bA)}\right).\end{equation}

In contrast, based on Theorem~\ref{theoremheavyball} we have
\begin{itemize}
\item For $ \omega= 1$ and $\beta= \left(1- \sqrt{0.99 \lambda_{\min}^+}\right)^2=\left(1- \sqrt{\frac{0.99}{\|\bA\|^2_F} \lambda_{\min}^+(\bA^\top \bA)}\right)^2$, the iteration complexity of the mRK is: \[ \tilde{O}\left(\sqrt{\frac{\|\bA\|^2_F}{\lambda_{\min}^+(\bA^\top \bA)}}\right).\]
\item For $ \omega= \|\bA\|^2_F/\lambda_{\max}(\bA^\top \bA)$ and $\beta= \left(1- \sqrt{ \frac{0.99 \lambda_{\min}^+(\bA^\top \bA)}{\lambda_{\max}(\bA^\top \bA)}}\right)^2$
the iteration complexity becomes:
\[\tilde{O}\left(\sqrt{\frac{ \lambda_{\max}(\bA^\top \bA)}{\lambda_{\min}^+(\bA^\top \bA)}}\right).\]
\end{itemize}

 This is quadratic improvement on the  previous best result \eqref{eq:biugs897*9h8}.

\paragraph{Related Work.} 
The Kaczmarz method for solving consistent linear systems was originally introduced by Kaczmarz in 1937 \cite{kaczmarz1937angenaherte}. This classical method selects the rows to project onto in a cyclic manner.  In practice, many different selection rules can be adopted. For non-random selection rules (cyclic, greedy, etc) we refer the interested reader to \cite{popa1995least,byrne2008applied, nutini2016convergence, popa2017convergence, CsibaPL}. In this work we are interested in {\em randomized} variants of the Kaczmarz method, first analyzed by Strohmer and Vershynin \cite{RK}. In \cite{RK} it was shown that RK converges with a linear convergence rate to the unique solution of a full-rank consistent linear system. This result sparked renewed interest in design of randomized methods for solving linear systems \cite{needell2010randomized, RBK, eldar2011acceleration, MaConvergence15, zouzias2013randomized, l2015randomized, schopfer2016linear}.  All existing results on accelerated variants of RK use the Nesterov's approach of acceleration \cite{lee2013efficient, liu2016accelerated, tu2017breaking, ASDA}. To the best of our knowledge, no convergence analysis of mRK exists in the literature (Polyak's momentum). Our work fills this gap.

\subsection{mRCD: randomized coordinate descent with momentum} 

We now provide a discussion on the mRCD method (the method in the second row of Table~\ref{SpecialCasesAlgorithms}). If the matrix $\bA$ is positive definite, then we can choose $\bB= \bA$  and $\bS=e_i$ with probability $p_i=\frac{\bA_{ii}}{{\rm Trace}(\bA)}$. It is easy to see that $\bW=\frac{\bA}{{\rm Trace}(\bA)}$. In this case,  $\bW$ is positive definite and as a result, $\lambda_{\min}^+(\bW)=\lambda_{\min}(\bW)$. Moreover, we have
\begin{equation}
\label{functionRCD}
f(x)=\Exp_{\bS \sim \mathcal{D}} [f_{\bS}(x)]= \sum_{i=1}^m p_i f_{\bS_i}(x)=\frac{\|\bA x -b\|^2_2}{2{\rm Trace}(\bA)}.
\end{equation}

For $\beta=0$ and $ \omega=1$ the method is equivalent with  \emph{randomized coordinate descent} of Leventhal and Lewis \cite{leventhal2010randomized}, which was shown to converge with iteration complexity  \begin{equation}\label{eq:hbis8690d9}\text{Previous  best result:} \qquad  \tilde{O}\left(\frac{{\rm Trace}(\bA)}{\lambda_{\min}(\bA)}\right).\end{equation}

In contrast, following  Theorem~\ref{theoremheavyball}, we can obtain the following  $L_1$ iteration complexity results for mRCD:
 \begin{itemize}
  \item For $ \omega= 1$ and $\beta= \left(1- \sqrt{\frac{0.99}{{\rm Trace}(\bA)} \lambda_{\min}(\bA)}\right)^2$, the iteration complexity is \[ \tilde{O}\left(\sqrt{\frac{{\rm Trace}(\bA)}{\lambda_{\min}(\bA)}}\right).\]
\item For  $ \omega= {\rm Trace}(\bA)/\lambda_{\max}(\bA)$ and $\beta= \left(1- \sqrt{ \frac{0.99 \lambda_{\min}(\bA)}{\lambda_{\max}(\bA)}}\right)^2 $  the iteration complexity becomes 
\[\tilde{O}\left(\sqrt{\frac{ \lambda_{\max}(\bA)}{\lambda_{\min}(\bA)}}\right).\]
 \end{itemize}
 
 This is quadratic improvement on the previous best result \eqref{eq:hbis8690d9}.

\paragraph{Related Work.}  It is known that if $\mA$ is positive definite, the popular {\em randomized Gauss-Seidel} method can be interpreted as randomized coordinate descent (RCD). RCD methods were first analyzed by Lewis and Leventhal in the context of linear systems and least-squares problems \cite{leventhal2010randomized}, and later extended by several authors to more general settings, including smooth convex optimization \cite{nesterov2012efficiency}, composite convex optimization \cite{richtarik2014iteration}, and parallel/subspace descent variants  \cite{richtarik2016parallel}. These results were later further extended to handle  arbitrary sampling distributions  \cite{qu2016coordinate,qu2016coordinate2, qu2015quartz, SCP} . Accelerated variants of RCD were studied in \cite{lee2013efficient, fercoq2015accelerated, allen2016even}. For other non-randomized coordinate descent variants and their convergence analysis, we refer the reader to \cite{wright2015coordinate,nutini2015coordinate, CsibaPL}. To the best of our knowledge, mRCD and smRCD have never  been analyzed before in any setting.

\subsection{Visualizing the acceleration mechanism}
\label{graphica}

We devote this section to the graphical illustration of the acceleration mechanism behind momentum. Our goal is to shed more light on how the proposed algorithm works in practice. For simplicity, we illustrate this by comparing RK and mRK.

\begin{figure}[h]
\centering
\begin{subfigure}{.5\textwidth}
  \centering
  \includegraphics[width=1\linewidth]{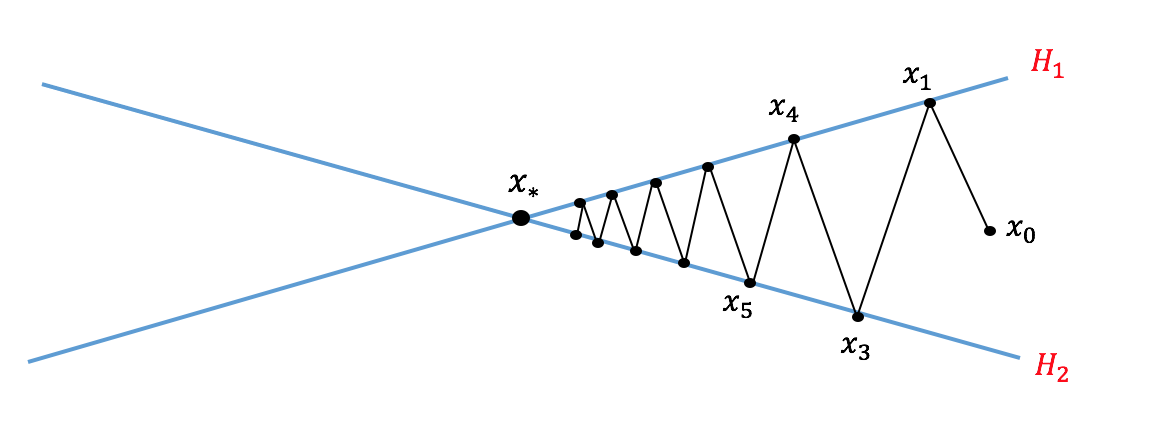}
  \caption{\footnotesize Randomized Kaczmarz Method \cite{RK}}
\end{subfigure}%
\begin{subfigure}{.5\textwidth}
  \centering
  \includegraphics[width=1\linewidth]{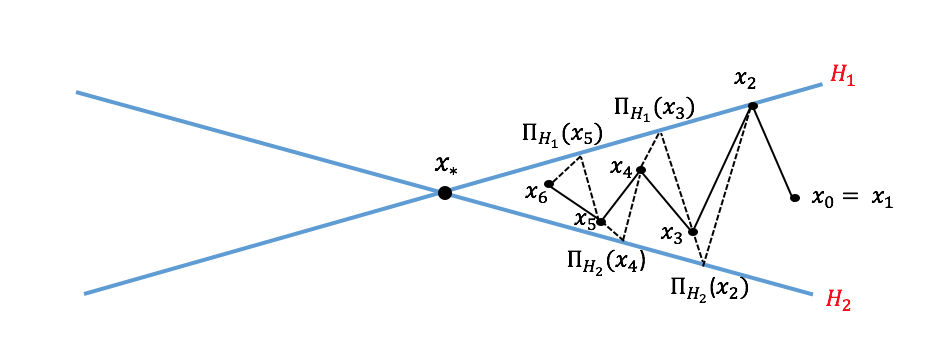}
  \caption{\footnotesize Randomized Kaczmarz Method with Momentum}
\end{subfigure}
\caption{Graphical interpretation of the {\em randomized Kaczmarz method} and the {\em randomized Kaczmarz method with momentum} in a simple example with only two  hyperplanes $H_i=\{x\;:\; \bA_{i:} x=b_i\}$ where $i=1,2$ and a unique solution $x_*$.}
\label{graphicalrepresentation}
\end{figure}

In Figure~\ref{graphicalrepresentation} we present in a simple $\R^2$ illustration of the difference between the workings of RK and mRK. Our goal is to show graphically how the addition of momentum leads to acceleration.  Given iterate $x_k$, one can think of the update rule of the mRK \eqref{eq:SHB-intro} in two steps:
\begin{enumerate}
\item \emph{The Projection:} The projection step corresponds to the first part $x_k - \omega  \nabla f_{\bS_k}(x_k)$ of the  mRK update \eqref{eq:SHB-intro} and it means that the current iterate $x_k$ is projected onto a randomly chosen hyperplane $H_i$\footnote{In the plots of Figure ~\ref{graphicalrepresentation}, the hyperplane of each update is chosen in an alternating fashion for illustration purposes}. The value of the stepsize $\omega \in (0,2)$ defines whether the projection is exact or not. When $\omega=1$ (no relaxation) the projection is exact, that is the point  $\Pi_{H_i}(x_k)$ belongs in the hyperplane $H_i$.  In Figure~\ref{graphicalrepresentation} all projections are exact.
\item \emph{Addition of the momentum term}: The momentum term (right part of the update rule) $\beta (x_k-x_{k-1})$ forces the next iterate $x_{k+1}$ to be closer to the solution $x_*$ than the corresponding point $\Pi_{H_i}(x_k)$. Note also that the vector $x_{k+1}-\Pi_{H_i}(x_k)$ is always parallel to $x_k-x_{k-1}$ for all $k\geq0$.
\end{enumerate}

\begin{rem}
In the example of Figure~\ref{graphicalrepresentation}, the performance of mRK is similar to the performance of RK until iterate $x_3$. After this point, the momentum parameter becomes more effective and the mRK method accelerates. This behavior appears also in our experiments in the next section where we work with  matrices with many rows. There we can notice that the momentum parameter seems to become more effective after the first $m+1$ iterations.
\end{rem}

\section{Numerical Experiments} \label{sec:experiments}

In this section we study the computational behavior of the two proposed algorithms, mSGD and smSGD. In particular, we focus mostly on the evaluation of the performance of mSGD. To highlight the usefulness of smSGD, an empitical verification of Theorem~\ref{thm:DSHBspeedup} is presented in subsection~\ref{DSHB comp}. As we have already mentioned, both mSGD and smSGD can be interpreted as sketch-and-project methods (with relaxation), and as a result a comprehensive array of well-known algorithms can be recovered as special cases by varying the main parameters of the methods (check Section~\ref{sec:cases}). In our experiments we focus on the popular special cases of randomized Kaczmarz method (RK) and the randomized coordinate descent method (RCD) without relaxation ($\omega=1$), and  show the practical benefits of the addition of the momentum term\footnote{The experiments were repeated with various values of the main parameters and initializations, and similar results were obtained in all cases.}. The choice of the stepsize $\omega=1$ is not arbitrary. Recently, in \cite{ASDA} both relaxed RK and relaxed RCD were analyzed, and it was proved that the quantity $\E{\|x_k-x_*\|^2_{\bB}}$ converges linearly to zero for $\omega \in (0,2)$, and that the best convergence rate is obtained precisely for $\omega=1$. Thus the comparison is with the best-in-theory  no-momentum variants.

Note that, convergence analysis of the error $\E{\|x_k-x_*\|^2_{\bB}}$ (L2 convergence) and of the expected function values $\E{f(x_k)}$ in Theorem~\ref{L2} shows that mSGD enjoys global non-asymptotic linear convergence rate but not faster than the no-momentum method. 
The accelerated linear convergence rate has been obtained only in the weak sense (Theorem \ref{theoremheavyball}). Nevertheless, in practice as indicated from our experiments, mSGD is faster than its no momentum variant. Note also that in all of the presented experiments the momentum parameters $\beta$ of the methods are chosen to be positive constants that do not depend on parameters that are not known to the users such as $\lambda_{\min}^+$ and $\lambda_{\max}$.

In comparing the methods with their momentum variants we use both the relative error measure $\|x_k-x_*\|^2_\bB / \|x_0-x_*\|^2_\bB $ and the function values $f(x_k)$\footnote{Remember that in our setting we have $f(x_*)=0$ for the optimal solution $x_*$ of the best approximation problem; thus $f(x)-f(x_*)=f(x)$. The function values $f(x_k)$ refer to function~\eqref{functionRK} in the case of RK and to function~\eqref{functionRCD} for the RCD.  For block variants the objective function of problem \eqref{eq:stoch_reform_intro} has also closed form expression but it can be very difficult to compute. In these cases one can instead evaluate the quantity $\|\bA x-b\|^2_{\bB}$.}. In all implementations, except for the experiments on average consensus (Section~\ref{consensus}), the starting point is chosen to be $x_0=0$. In the case of average consensus the starting point must be the vector with the initial private values of the nodes of the network. 
All the code for the experiments is written in the Julia programming language.
For the horizontal axis we use either the number of iterations or the wall-clock time measured using the tic-toc Julia function.

This section is divided in three main experiments. In the first one we evaluate the performance of the mSGD method in the special cases of mRK and mRCD for solving both synthetic consistent Gaussian systems and consistent linear systems with real matrices. In the second experiment we computationally verify Theorem \ref{thm:DSHBspeedup} (comparison between the mSGD and smSGD methods). In the last experiment building upon the recent results of \cite{LoizouRichtarik} we show how the addition of the momentum accelerates the pairwise randomized gossip (PRG) algorithm for solving the average consensus problem.

\begin{table}[H]
\begin{center}
{\footnotesize
\begin{tabular}{| p{4cm} | p{3cm} | p{3cm} | p{4cm} |  }
 \hline
Assumptions & No-momentum, & Momentum, &  Stochastic Momentum,\\
& $\beta=0$ & $\beta\geq0$ &$\beta\geq0$\\
   \hline
$\bA$ general, $\bB=\bI$ & RK & mRK & smRK\\
 \hline
$\bA\succ 0$, $\bB=\bA$ & RCD & mRCD & smRCD\\
  \hline
$\bA$ incidence matrix, $\bB=\bI$ & PRG &  mPRG & smPRG \\
  \hline
\end{tabular}
}
\end{center}
\caption{Abbreviations of the algorithms (special cases of general framework) that we use in the numerical evaluation section. In all methods the random matrices are chosen to be unit coordinate vectors in $\R^m$  ($\bS=e_i$). With PRG we denote the Pairwise Randomized Gossip algorithm for solving the average consensus problem first proposed in \cite{boyd2006randomized}. Following similar notation with the rest of the paper with mPRG and smPRG we indicate its momentum and stochastic momentum variants respectively.}
\label{tableSHBandDSHB}
\end{table}

\subsection{Evaluation of mSGD}
\label{subsectionEvaluation}
In this subsection we study the computational behavior of mRK and mRCD when they compared with their no momentum variants for both synthetic and real data.

\subsubsection{Synthetic Data}
\label{gaussiansyntheric}
The synthetic data for this comparison is generated as follows\footnote{Note that in the first experiment we use Gaussian matrices which by construction are full rank matrices with probability 1 and as a result the consistent linear systems have unique solution. Thus, for any starting point $x_0$, the vector $z$ that used to create the linear system is the solution mSGD converges to.  This is not true for general consistent linear systems, with no full-rank matrix. In this case, the solution   $x_{*}=\Pi_{\cL}^{\bB}(x_0)$ that mSGD converges to is not necessarily equal to $z$. For this reason, in the evaluation of the relative error measure $\|x_k-x_*\|^2_\bB / \|x_0-x_*\|^2_\bB$, one should be careful and use the value $x_*=x_0+\bA^\dagger (b- \bA x_0)\overset{x_0=0}{=} \bA^\dagger b$.}.

\textbf{For mRK:}
All elements of matrix $\bA \in \R^{m \times n}$ and of vector $z \in \R^n$ are chosen to be i.i.d $\mathcal{N}(0,1)$. Then the right hand side of the linear system is set to $b=\bA z$. With this way the consistency of the linear system with matrix $\bA$ and right hand side $b$ is ensured. 

\textbf{For mRCD:}  A Gaussian matrix $\bP \in \R^{m \times n}$ is generated and then matrix $\bA = \bP^\top \bP \in \R^{n \times n}$ is used in the linear system. The vector $z\in \R^n$ is chosen to be i.i.d $\mathcal{N}(0,1)$ and again to ensure consistency of the linear system, the right hand side is set to $b=\bA z$.

In particular for the evaluation of mRK we generate Gaussian matrices with $m=300$ rows and several columns while for the case of mRCD the matrix $\bP$ is chosen to be Gaussian with $m=500$ rows and several columns\footnote{RCD converge to the optimal solution only in the case of positive definite matrices. For this reason $\bA = \bP^\top \bP \in \R^{n \times n}$ is used which with probability $1$ is a full rank matrix}. 
Linear systems of these forms were extensively studied \cite{RK, geman1980limit} and it was shown that the quantity $1/\lambda_{\min}^+$(condition number) can be easily controlled. 

For each linear system we run mRK (Figure~\ref{RKperformace1})  and mRCD (Figure~\ref{RCDperformance1}) for several values of momentum parameters $\beta$ and fixed stepsize $\omega=1$ and we plot the performance of the methods (average after 10 trials) for both the relative error measure and the function values. Note that for $\beta=0$ the methods are equivalent with their no-momentum variants RK and RCD respectively.

From Figures~\ref{RKperformace1} and \ref{RCDperformance1} it is clear that the addition of momentum term leads to an improvement in the performance of RK and RCD, respectively. More specifically, from the two figures we observe the following:

\begin{itemize}
\item For the well conditioned linear systems ($1/\lambda_{\min}^+$ small) it is known that even the no-momentum variant converges rapidly to the optimal solution. In these cases the benefits of the addition of momentum are not obvious. The momentum term is beneficial for the case where the no-momentum variant ($\beta=0$) converges slowly, that is when  $1/\lambda_{\min}^+$ is large (ill-conditioned linear systems). 
\item For the case of fixed stepsize $\omega=1$, the problems with small condition number require smaller momentum parameter $\beta$ to have faster convergence. Note the first two rows of Figures~\ref{RKperformace1} and \ref{RCDperformance1}, where $\beta=0.3$ or $\beta =0.4$, are good options. 
\item For large values of $1/\lambda_{\min}^+$, it seems that the choice of $\beta=0.5$ is the best. As an example for matrix $\bA\in \R^{300 \times 280}$ in Figure~\ref{RKperformace1}, (where $1/\lambda_{\min}^+=208,730$), note that to reach relative error $10^{-10}$, RK needs around 2 million iterations, while mRK with momentum parameter $\beta=0.5$ requires only half that many iterations. The acceleration is obvious also in terms of time where in 12 seconds the mRK with momentum parameter $\beta=0.5$ achieves relative error of the order $10^{-9}$ and RK requires more than 25 seconds to obtain the same accuracy. 
\item We observe that both mRK and mRCD, with appropriately chosen momentum parameters $0< \beta \leq 0.5$, always converge  faster than their no-momentum variants, RK and RCD, respectively. This is a smaller momentum parameter than $\beta \approx 0.9$ which is being used extensively with mSGD for training deep neural networks \cite{zhang2017yellowfin, wilson2017marginal, sutskever2013importance}. 
\item In \cite{de2017accelerated} a stochastic power iteration with momentum is proposed for principal component analysis (PCA). There it was demonstrated empirically that a naive application of momentum to the stochastic power iteration does not result in a faster method. To achieve faster convergence, the authors proposed mini-batch and variance-reduction techniques on top of the addition of momentum. In our setting, mere addition of the momentum term  to SGD (same is true for  special cases such as RK and RCD) leads to empirically faster methods.
\end{itemize}

\begin{figure}[!]
\centering
\begin{subfigure}{.23\textwidth}
  \centering
  \includegraphics[width=1\linewidth]{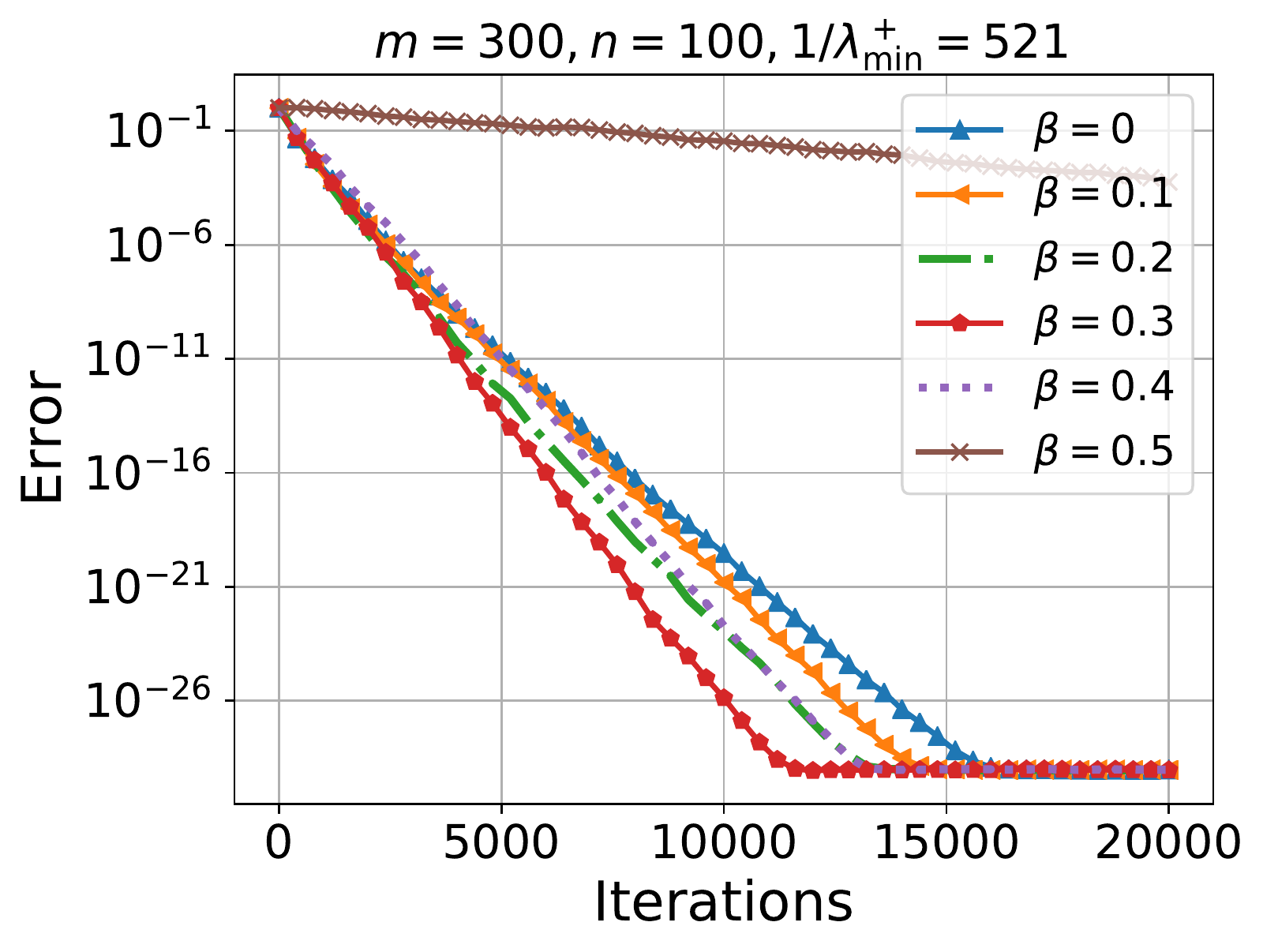}
\end{subfigure}%
\begin{subfigure}{.23\textwidth}
  \centering
  \includegraphics[width=1\linewidth]{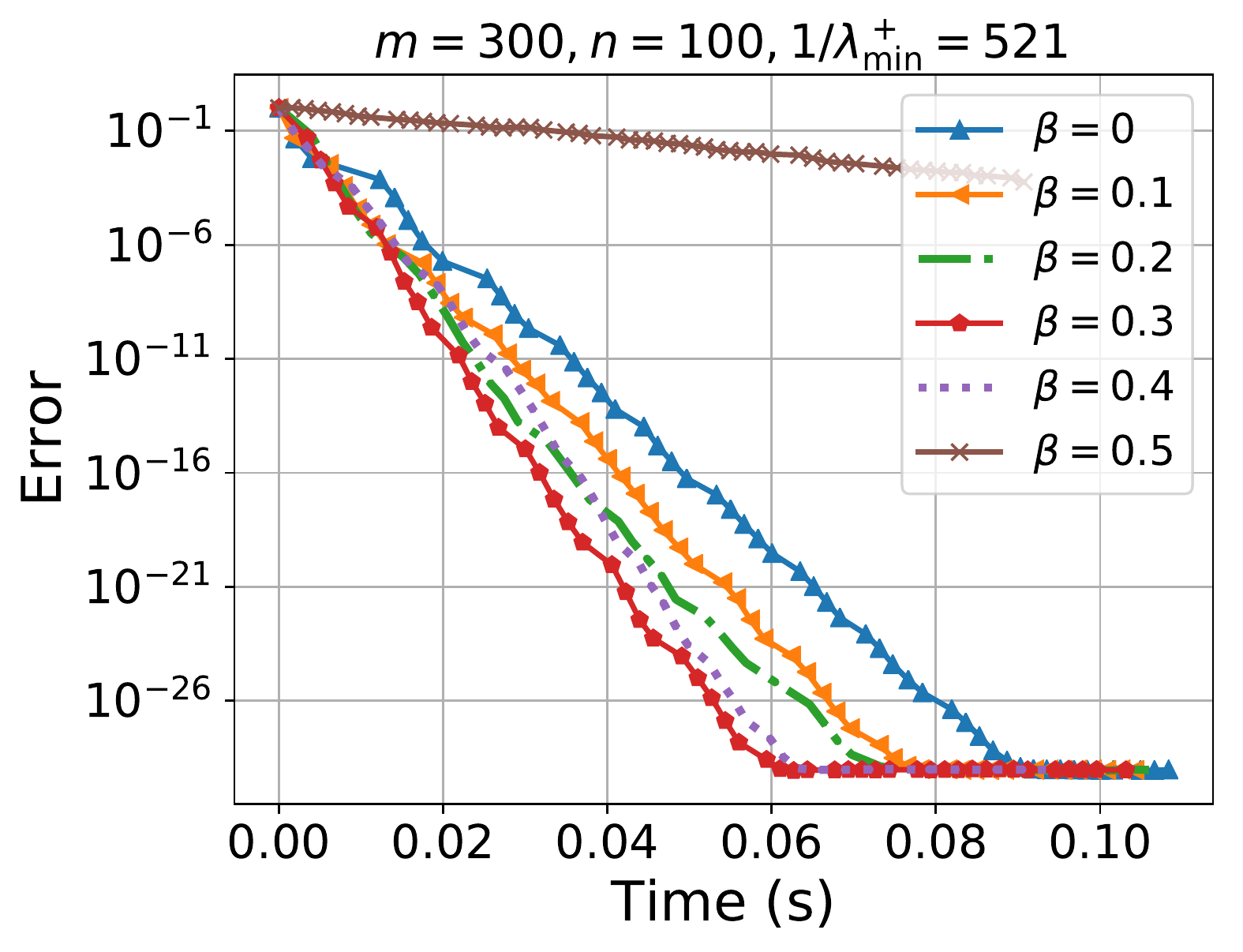}
\end{subfigure}
\begin{subfigure}{.23\textwidth}
  \centering
  \includegraphics[width=\linewidth]{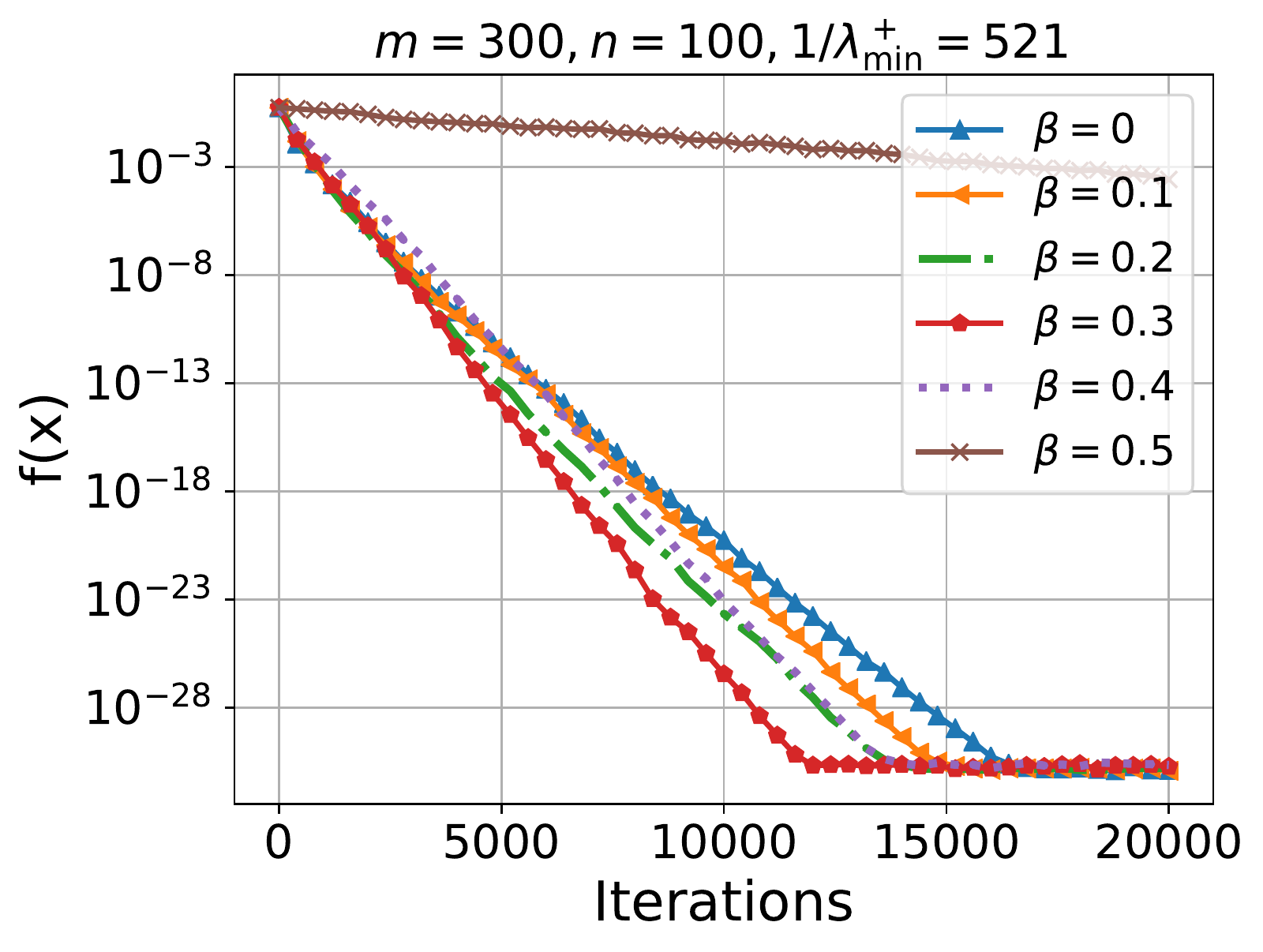}
\end{subfigure}
\begin{subfigure}{.23\textwidth}
  \centering
  \includegraphics[width=1\linewidth]{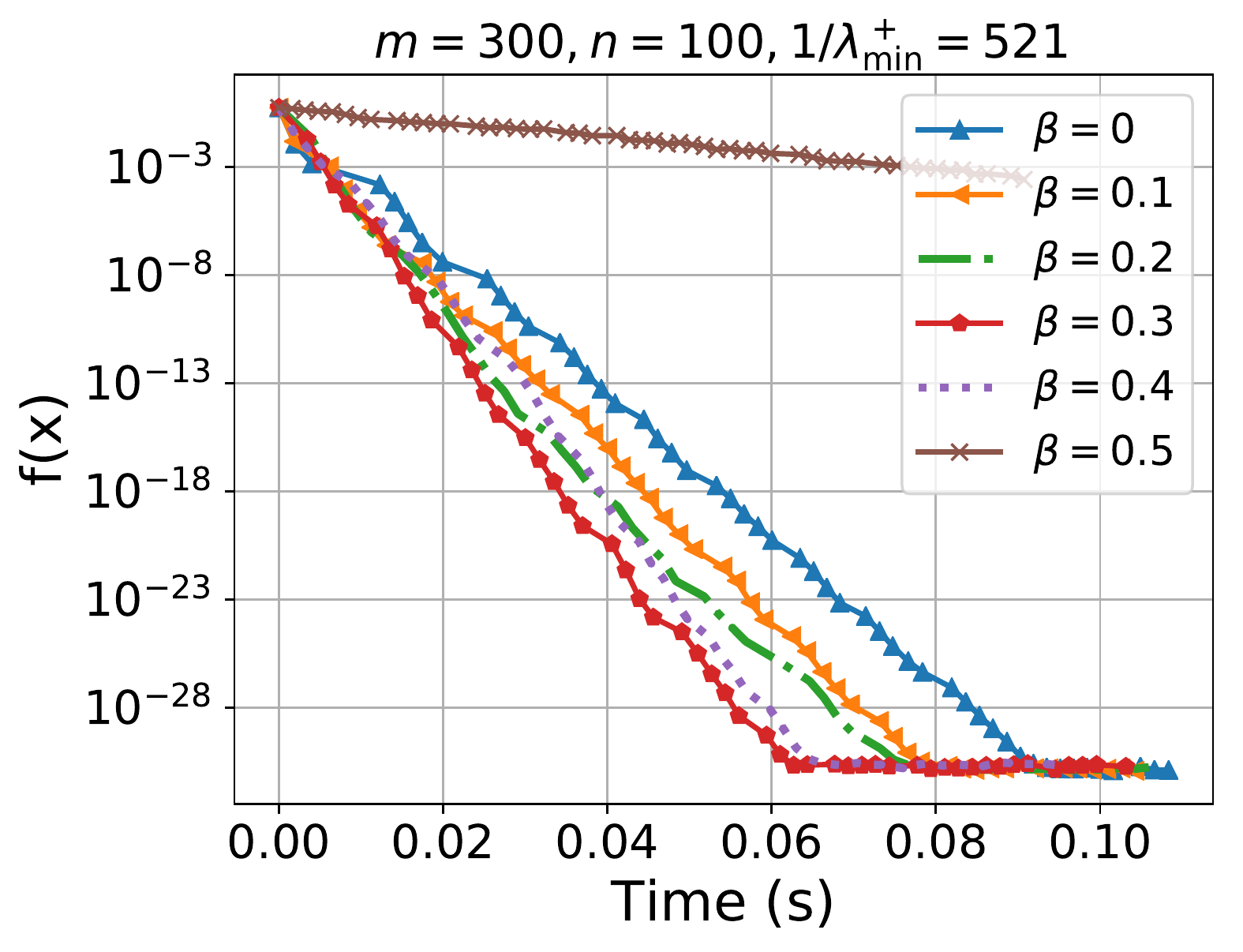}
\end{subfigure}\\
\begin{subfigure}{.23\textwidth}
  \centering
  \includegraphics[width=1\linewidth]{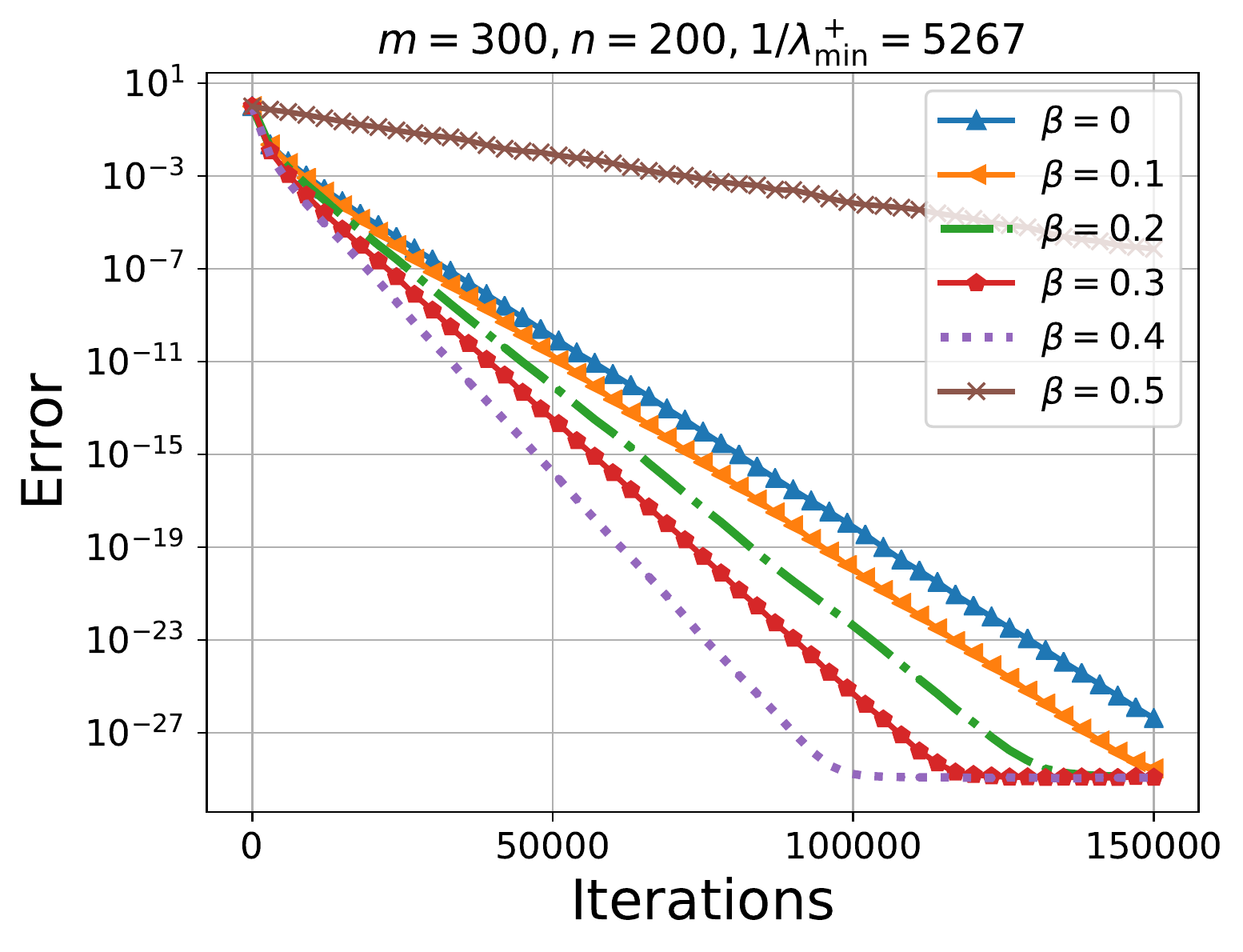}
\end{subfigure}%
\begin{subfigure}{.23\textwidth}
  \centering
  \includegraphics[width=1\linewidth]{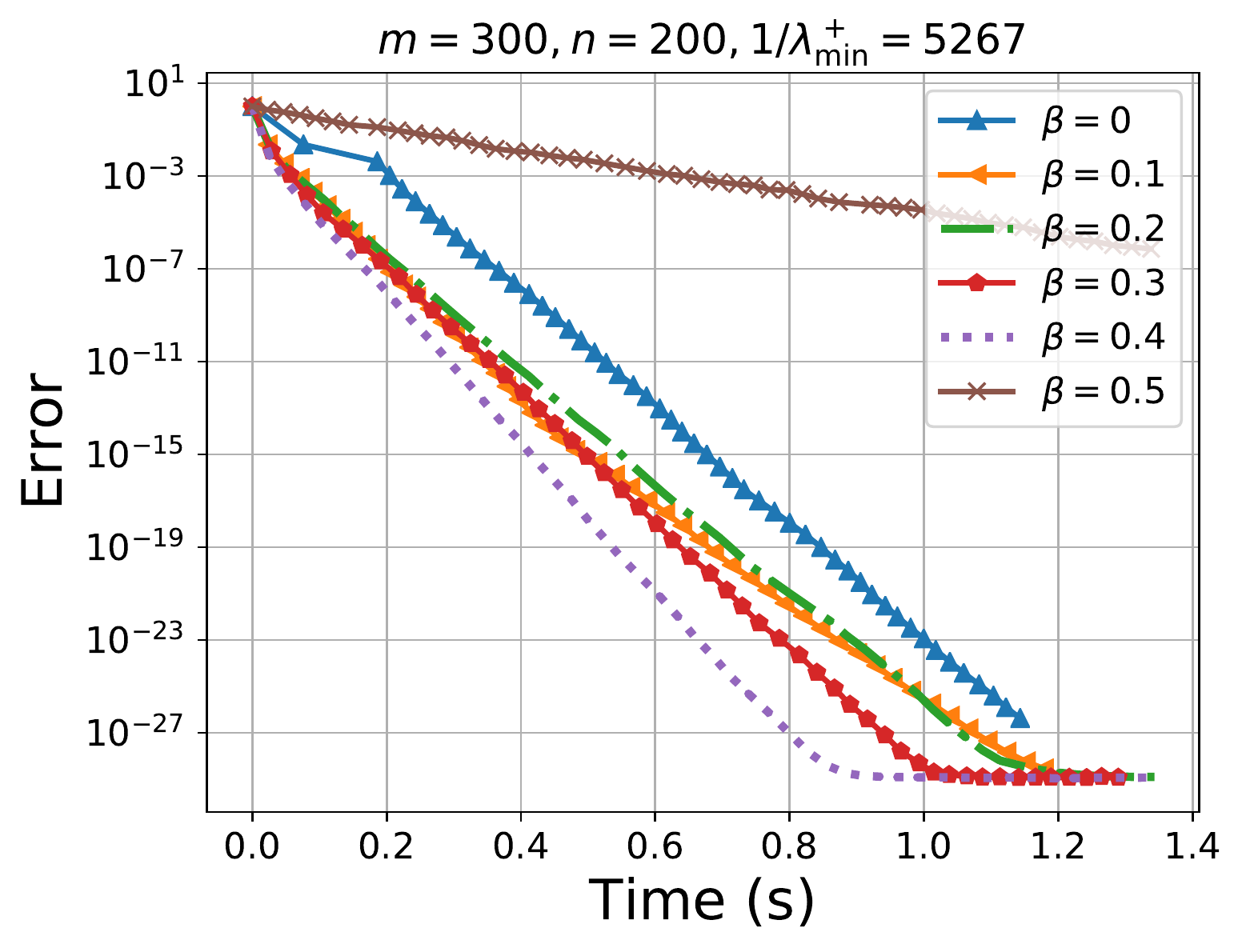}
\end{subfigure}
\begin{subfigure}{.23\textwidth}
  \centering
  \includegraphics[width=\linewidth]{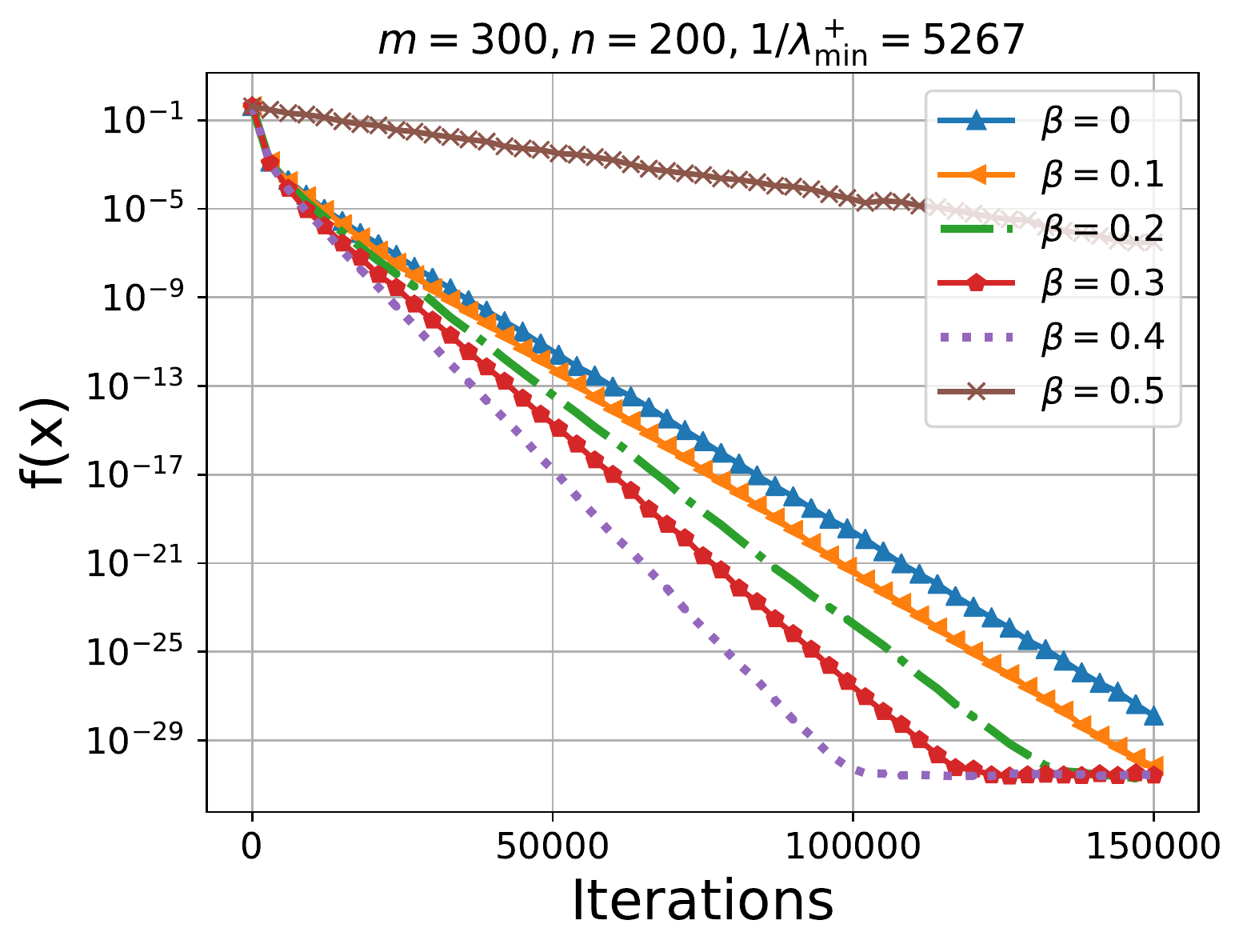}
\end{subfigure}
\begin{subfigure}{.23\textwidth}
  \centering
  \includegraphics[width=1\linewidth]{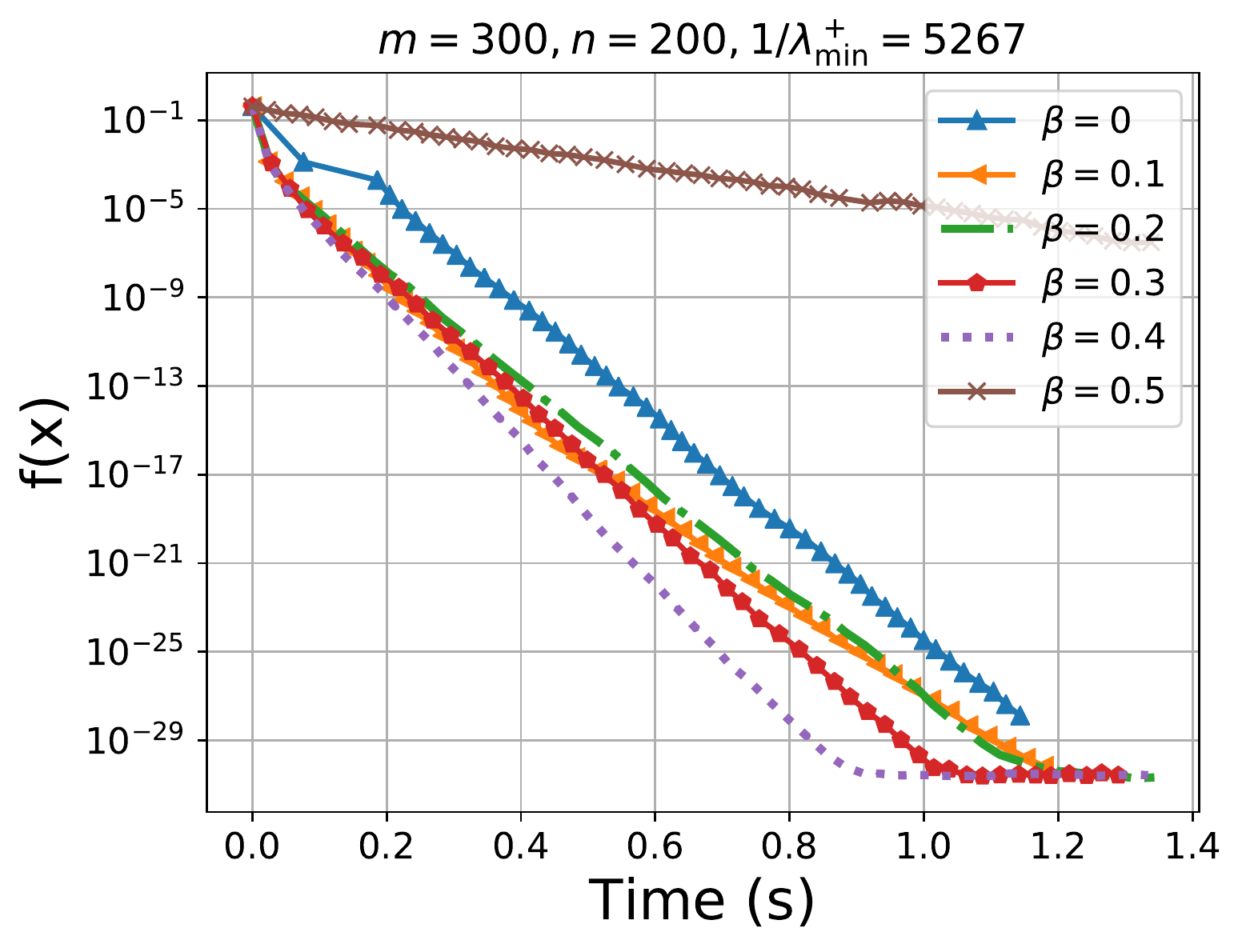}
\end{subfigure}\\
\begin{subfigure}{.23\textwidth}
  \centering
 \includegraphics[width=1\linewidth]{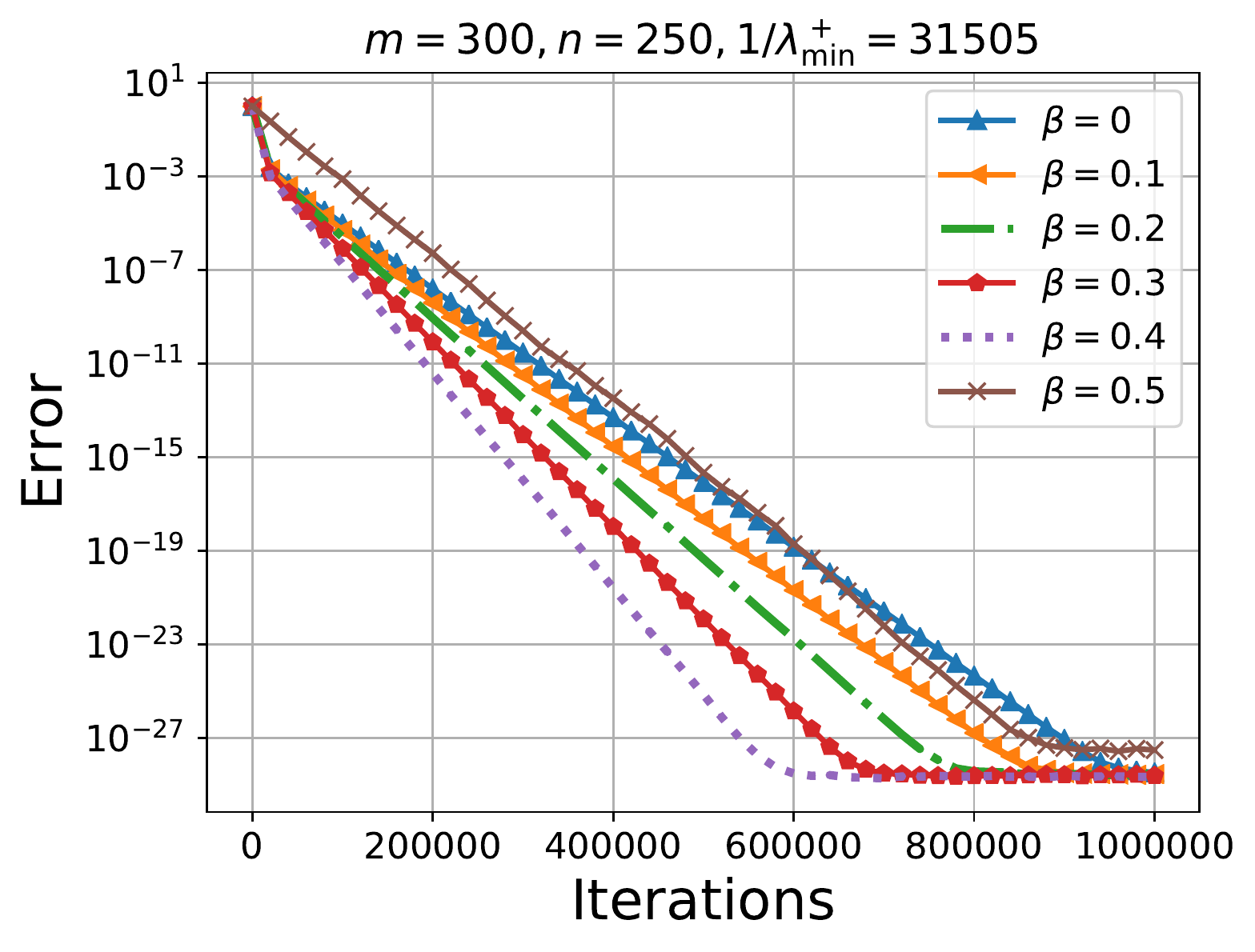}
\end{subfigure}%
\begin{subfigure}{.23\textwidth}
  \centering
\includegraphics[width=1\linewidth]{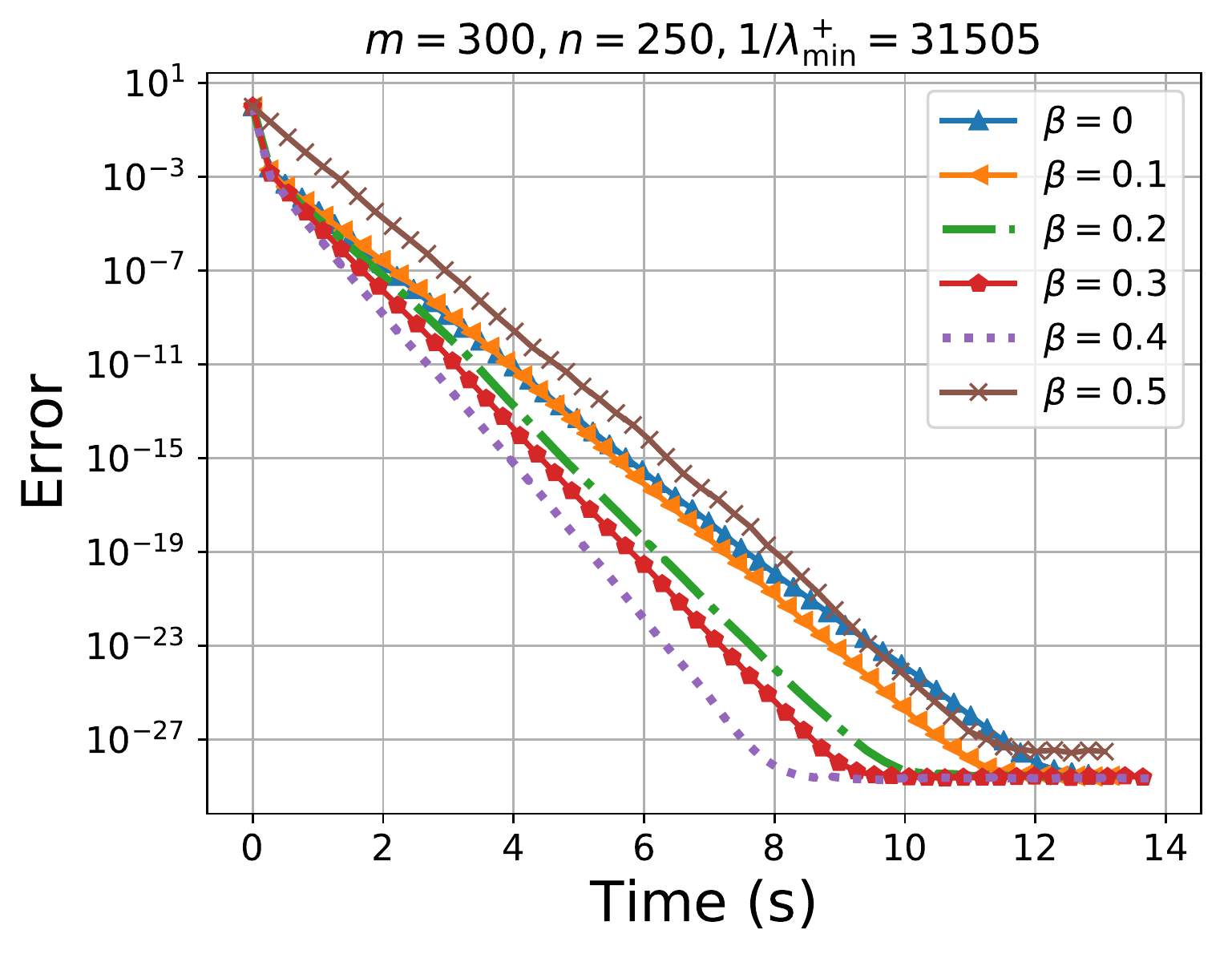}
\end{subfigure}
\begin{subfigure}{.23\textwidth}
  \centering
    \includegraphics[width=1\linewidth]{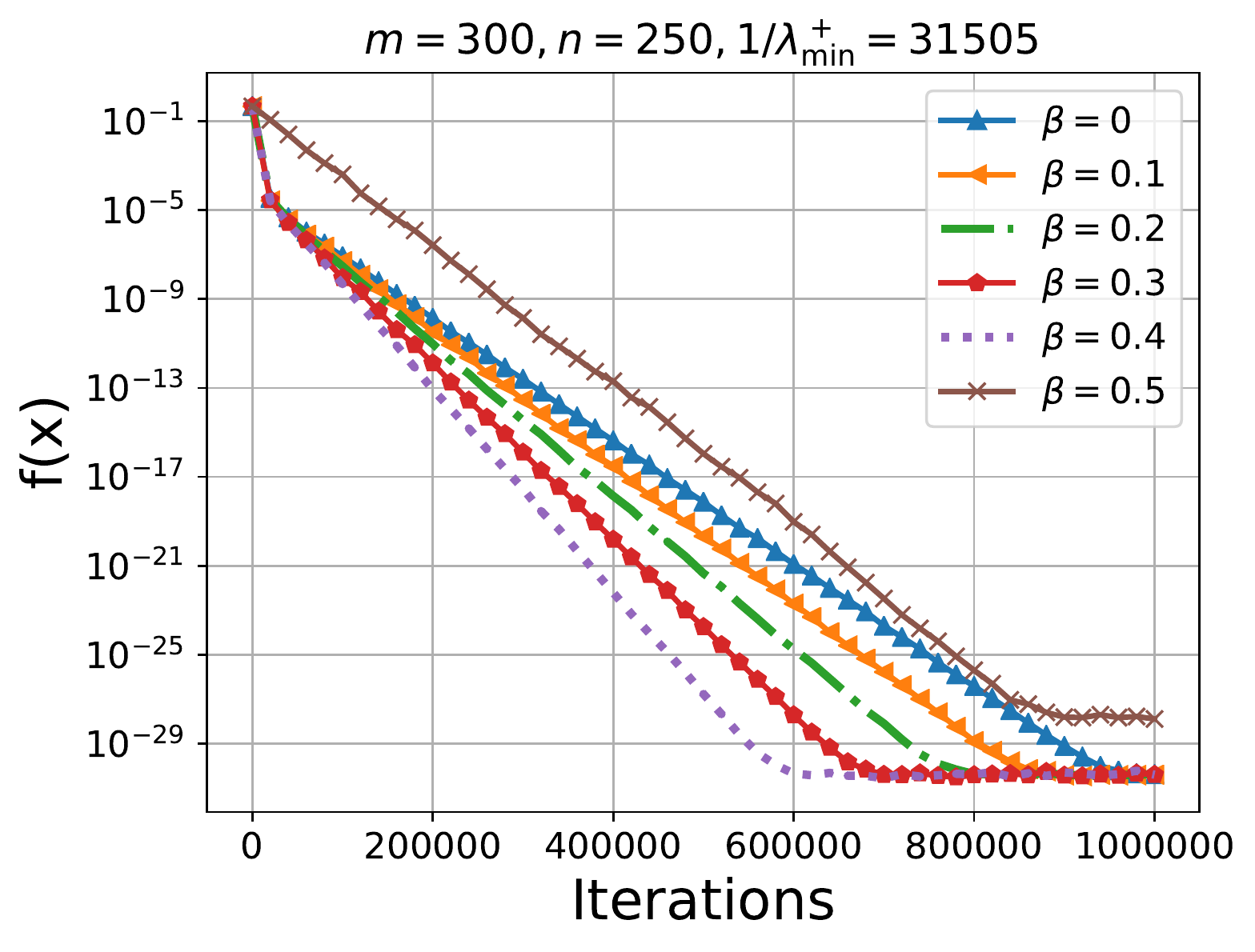}
\end{subfigure}
\begin{subfigure}{.23\textwidth}
  \centering
  \includegraphics [width=1\linewidth] {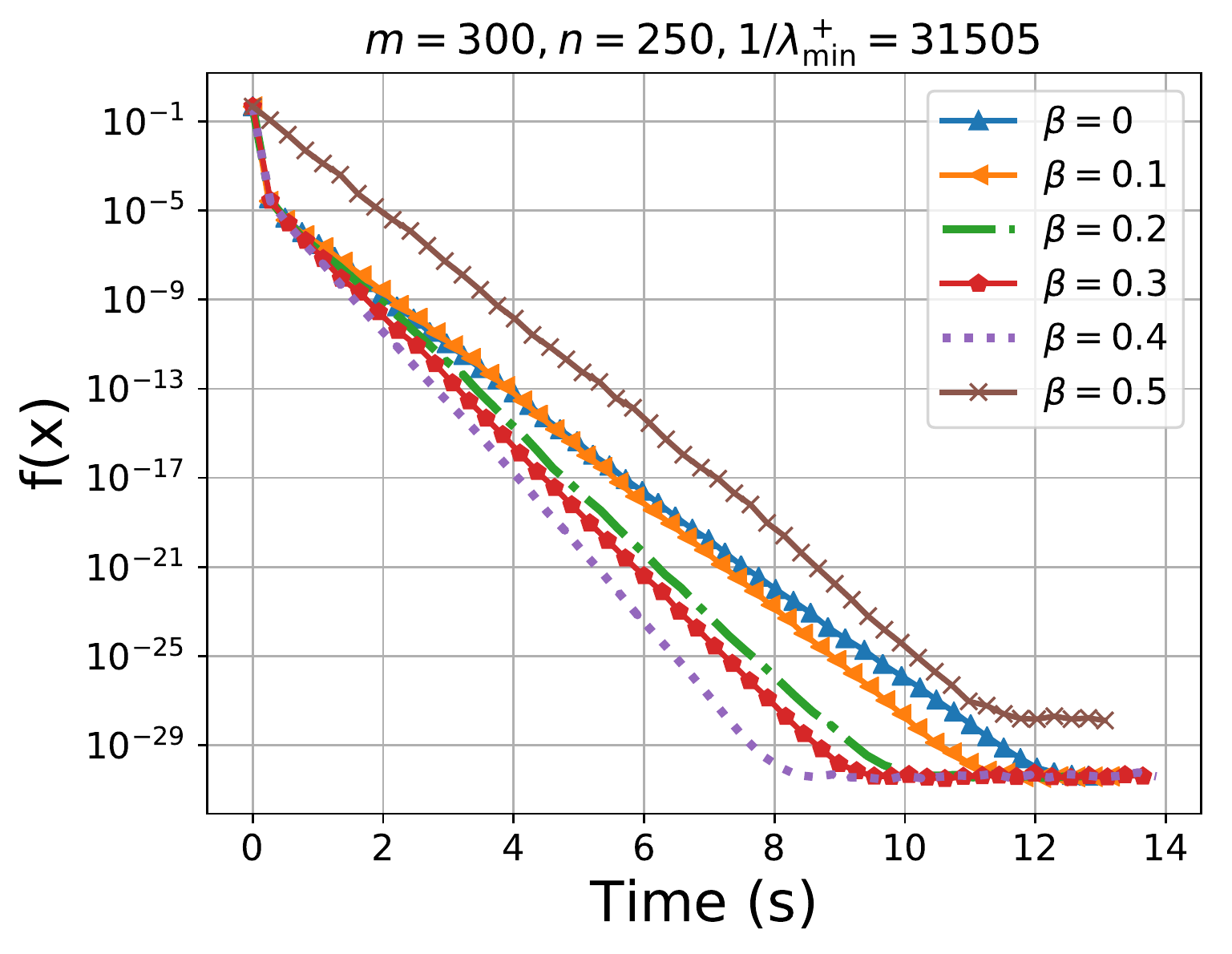} 
\end{subfigure}\\
\begin{subfigure}{.23\textwidth}
  \centering
  \includegraphics[width=1\linewidth]{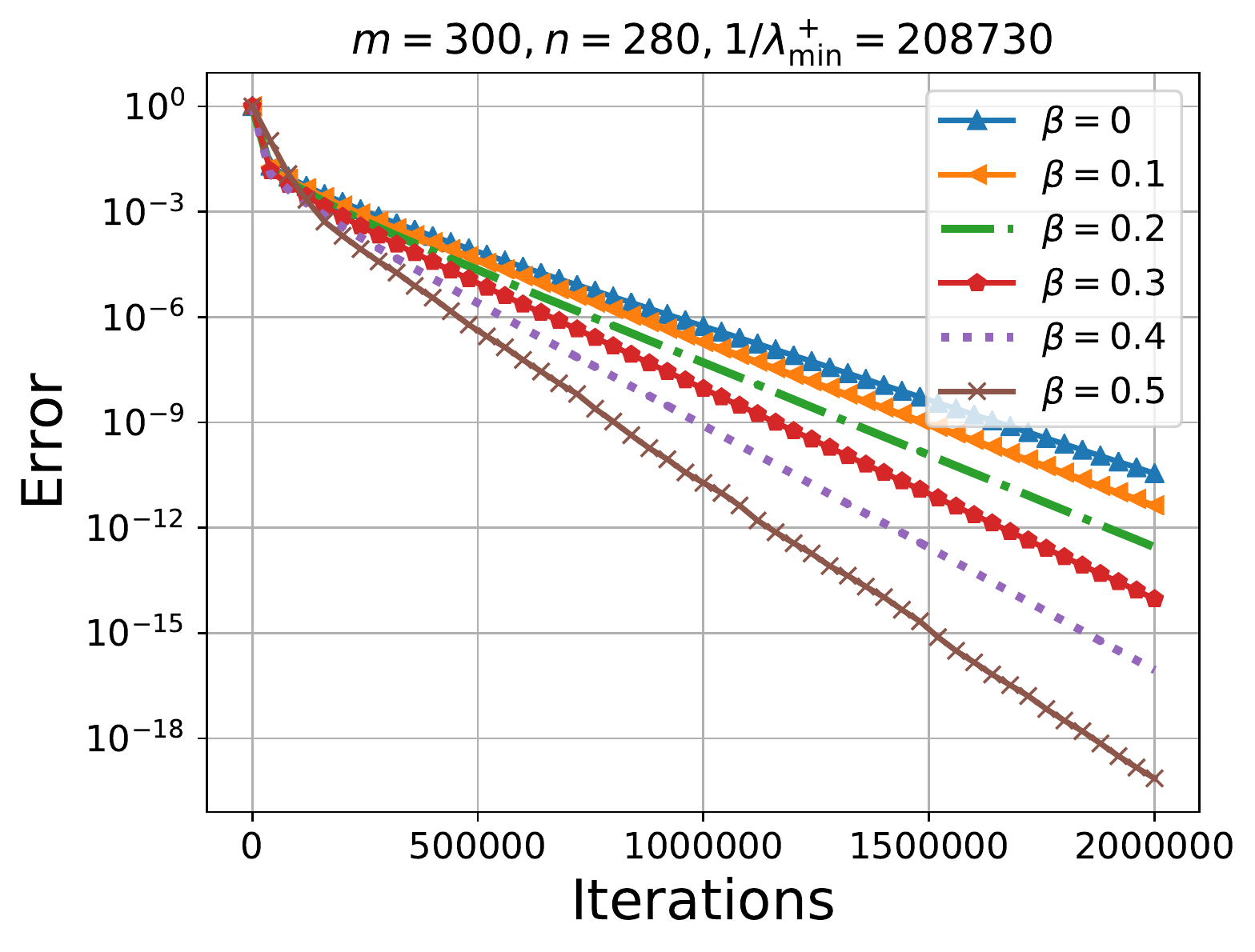}
\end{subfigure}%
\begin{subfigure}{.23\textwidth}
  \centering
  \includegraphics[width=1\linewidth]{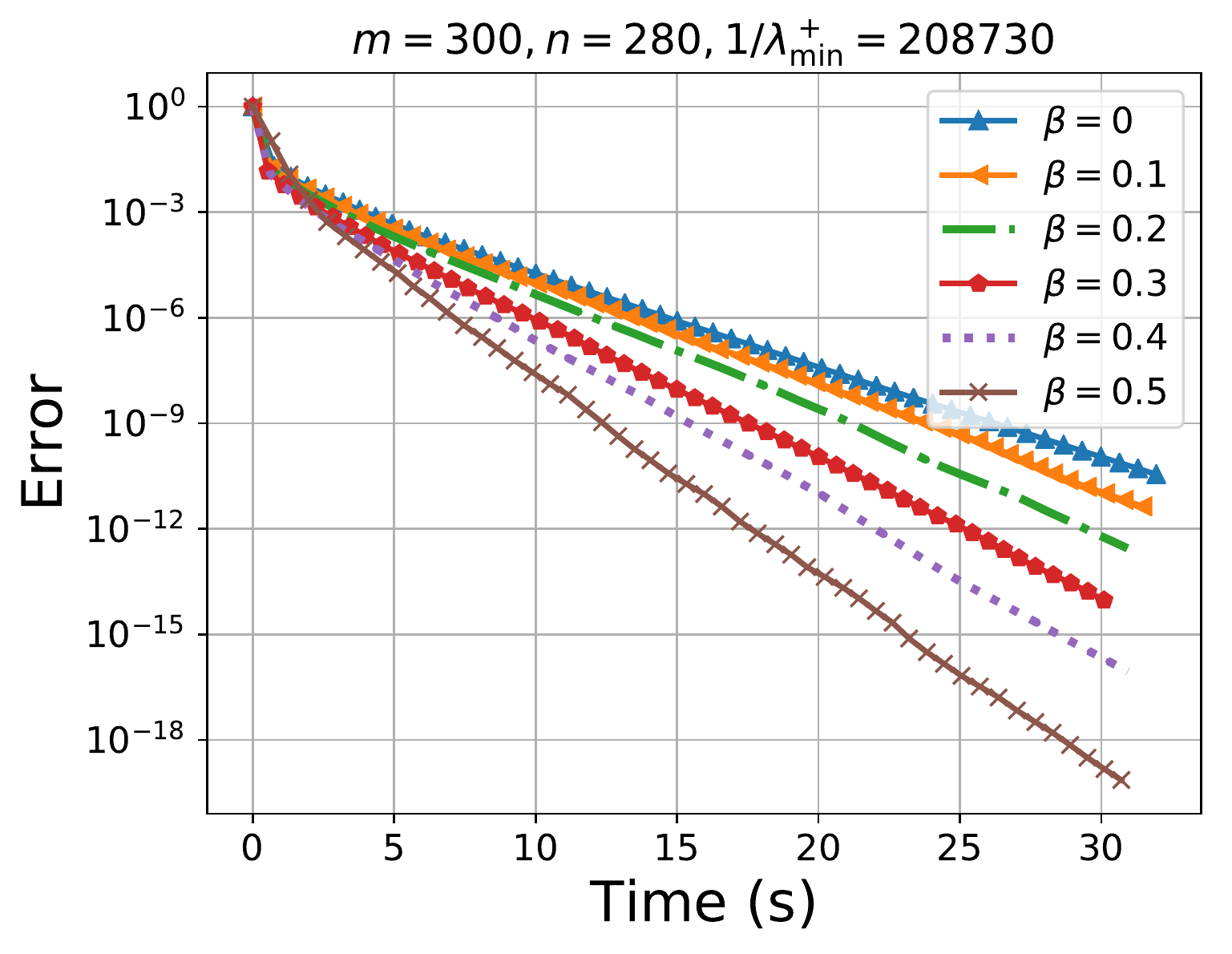}
\end{subfigure}
\begin{subfigure}{.23\textwidth}
  \centering
  \includegraphics[width=1\linewidth]{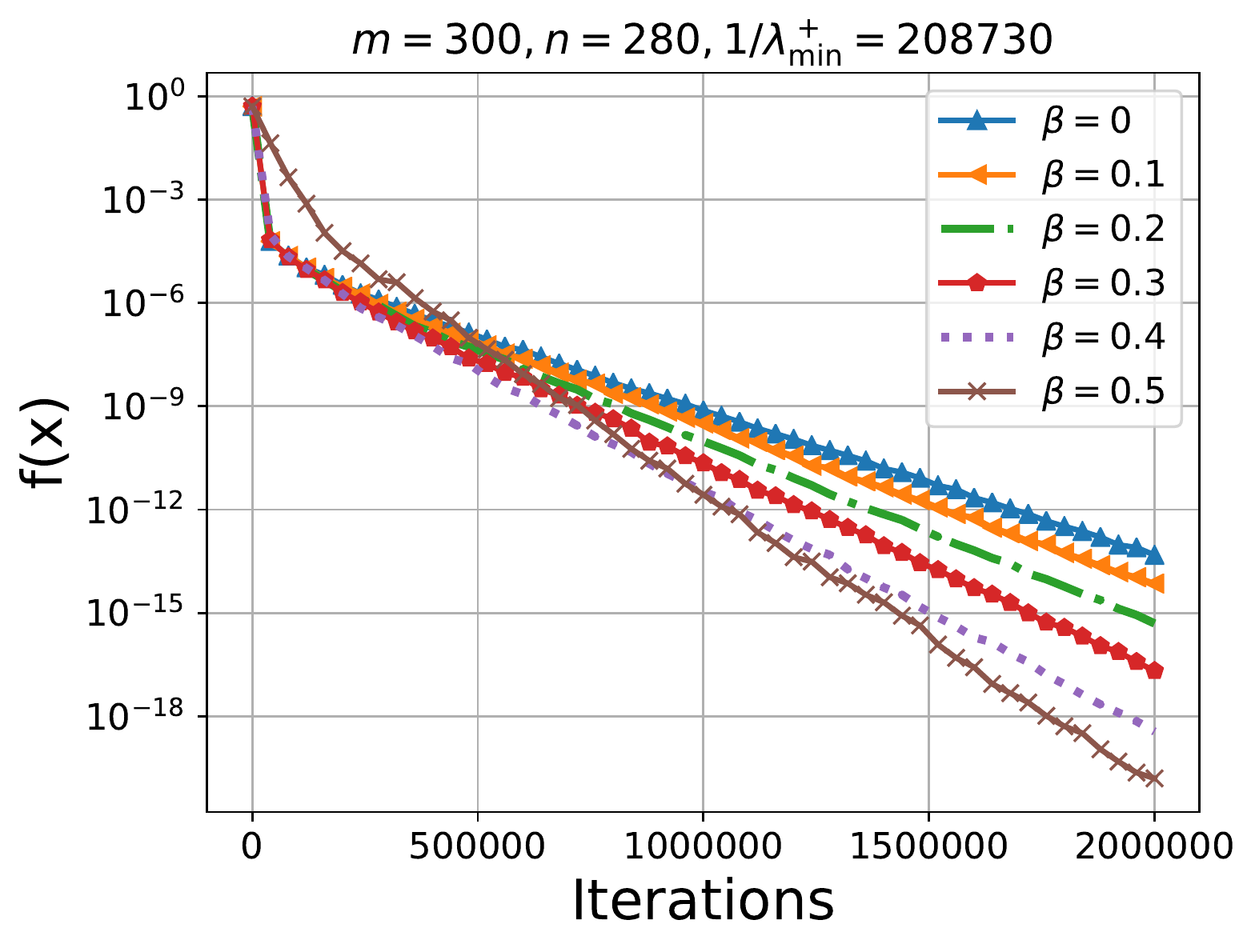}
\end{subfigure}
\begin{subfigure}{.23\textwidth}
  \centering
  \includegraphics[width=1\linewidth]{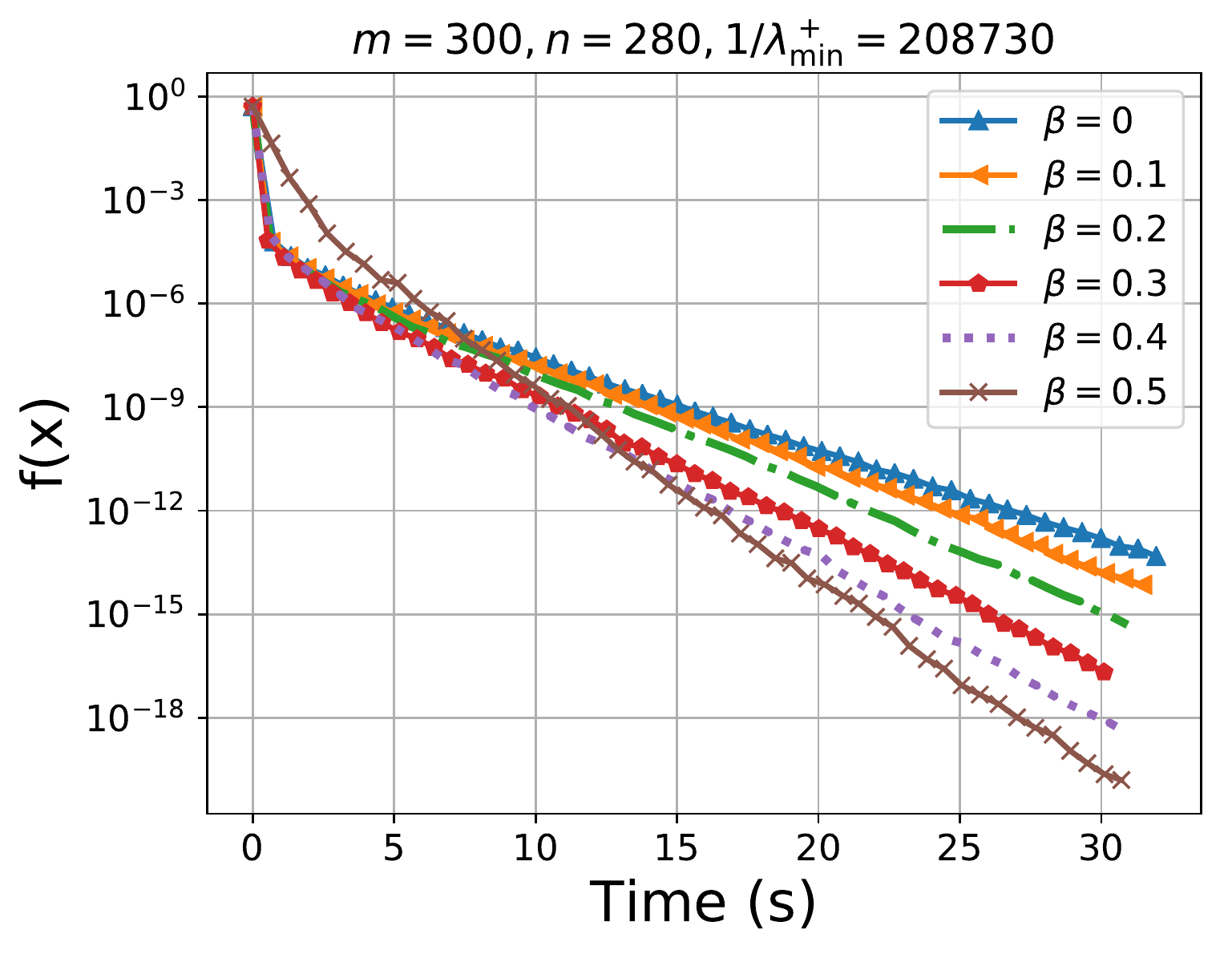}
\end{subfigure}\\
\begin{subfigure}{.23\textwidth}
  \centering
  \includegraphics[width=1\linewidth]{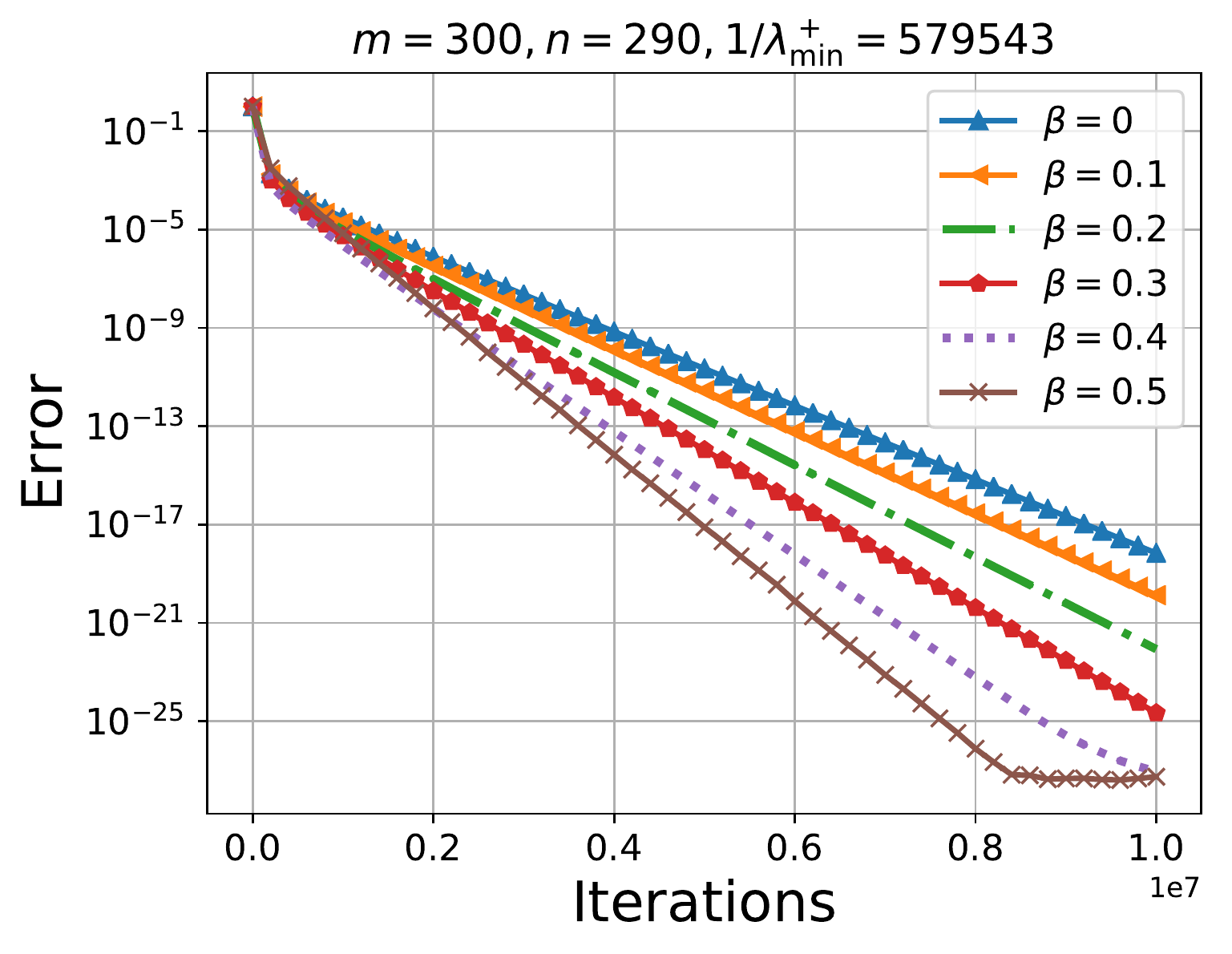}
\end{subfigure}%
\begin{subfigure}{.23\textwidth}
  \centering
  \includegraphics[width=1\linewidth]{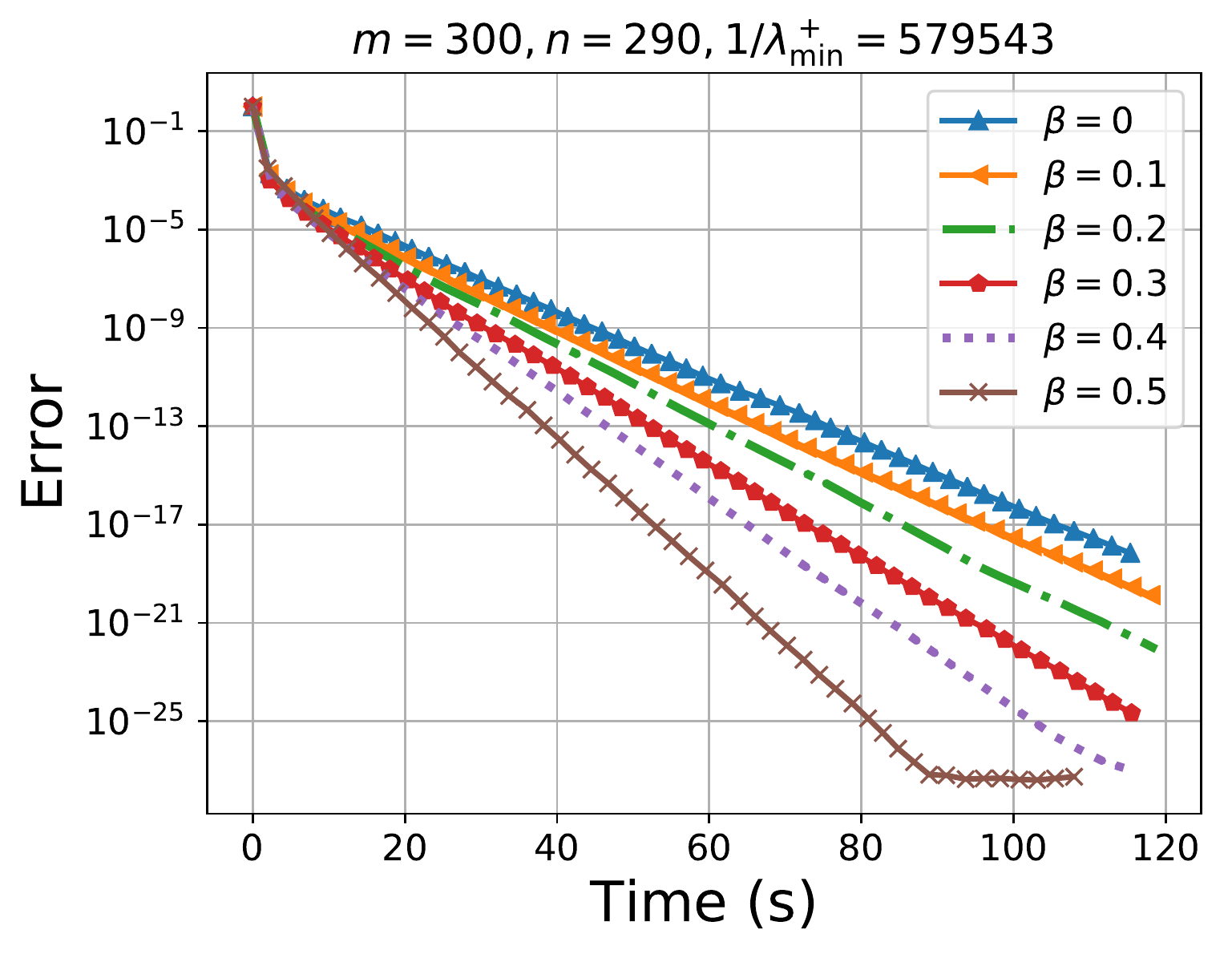}
\end{subfigure}
\begin{subfigure}{.23\textwidth}
  \centering
  \includegraphics[width=1\linewidth]{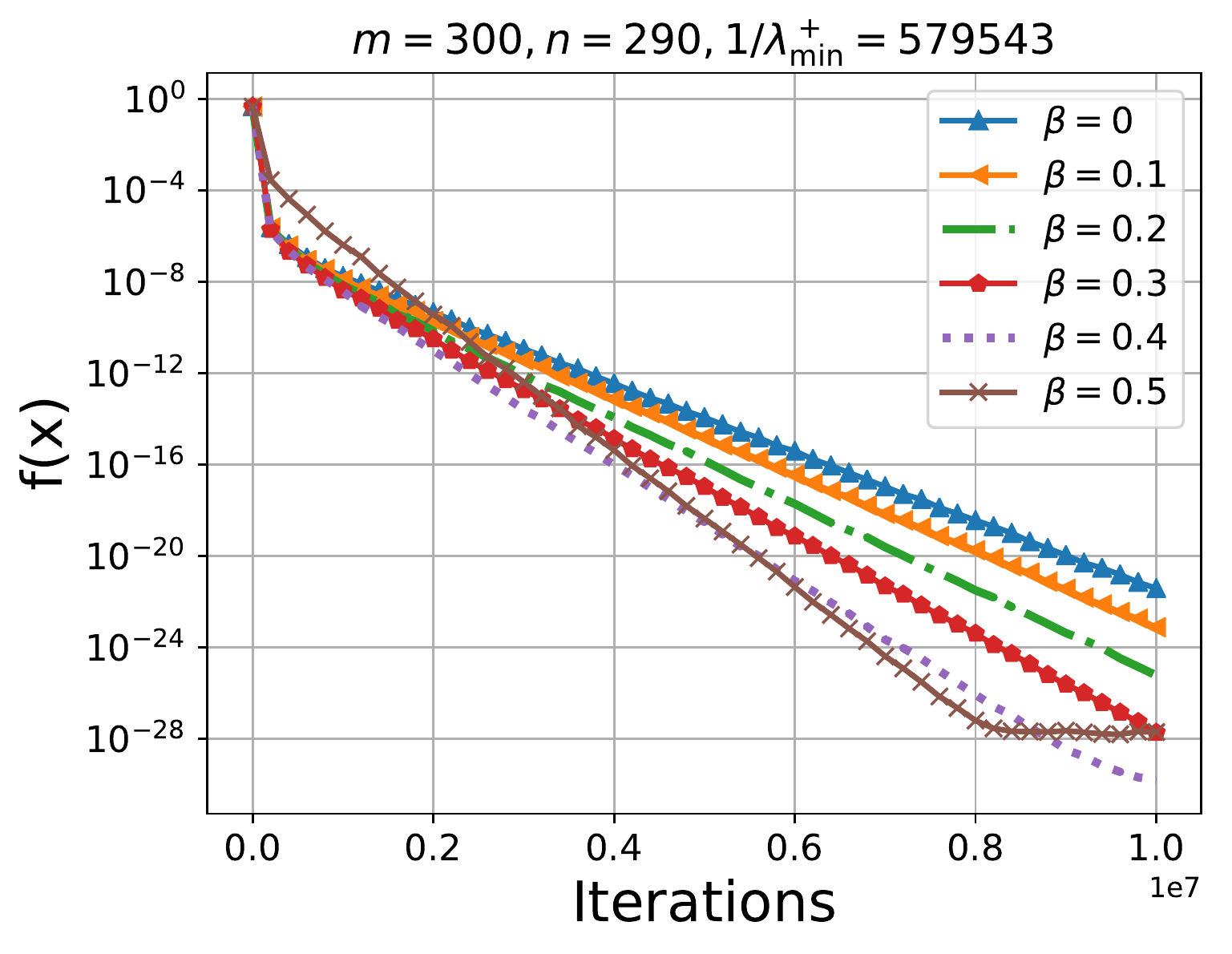}
\end{subfigure}
\begin{subfigure}{.23\textwidth}
  \centering
  \includegraphics[width=1\linewidth]{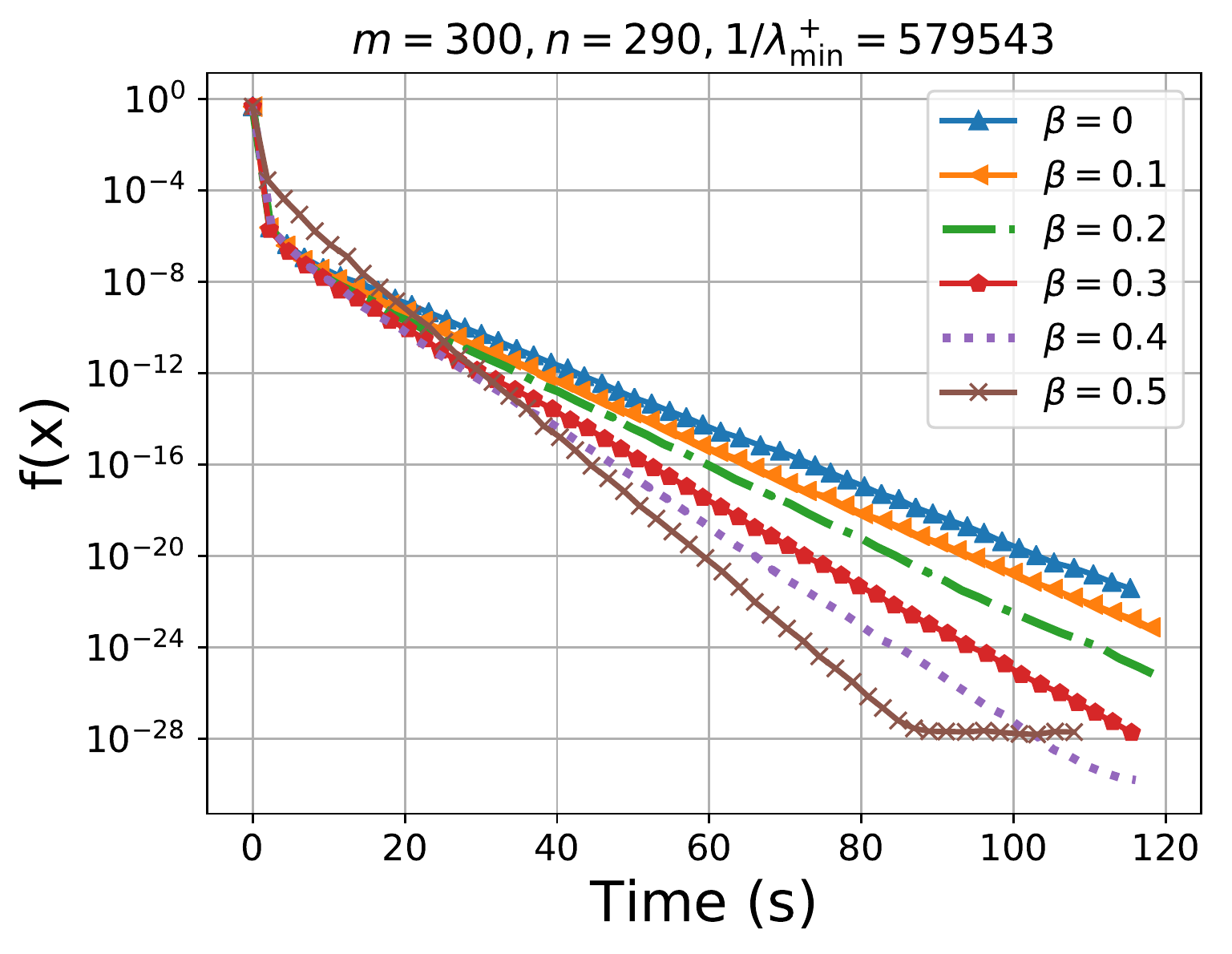}
\end{subfigure}\\
\caption{Performance of mRK for fixed stepsize $\omega=1$ and several momentum parameters $\beta$ for consistent linear systems with Gaussian matrix $\bA$ with $m=300$ rows and $n=100,200,250,280,290$ columns. The graphs in the first (second) column plot iterations (time) against residual error while those in the third (forth) column plot iterations (time) against function values. All plots are averaged over 10 trials. The title of each plot indicates the dimensions of the matrix $\bA$ and the value of $1/\lambda_{\min}^+$. The ``Error" on the vertical axis represents the relative error $\|x_k-x_*\|^2_\bB / \|x_0-x_*\|^2_\bB \overset{\bB=\bI, x_0=0}{=}\|x_k-x_*\|^2 / \|x_*\|^2_\bB$ and the function values $f(x_k)$ refer to function~\eqref{functionRK}.}

\label{RKperformace1}
\end{figure}

\begin{figure}[!]
\centering
\begin{subfigure}{.23\textwidth}
  \centering
  \includegraphics[width=1\linewidth]{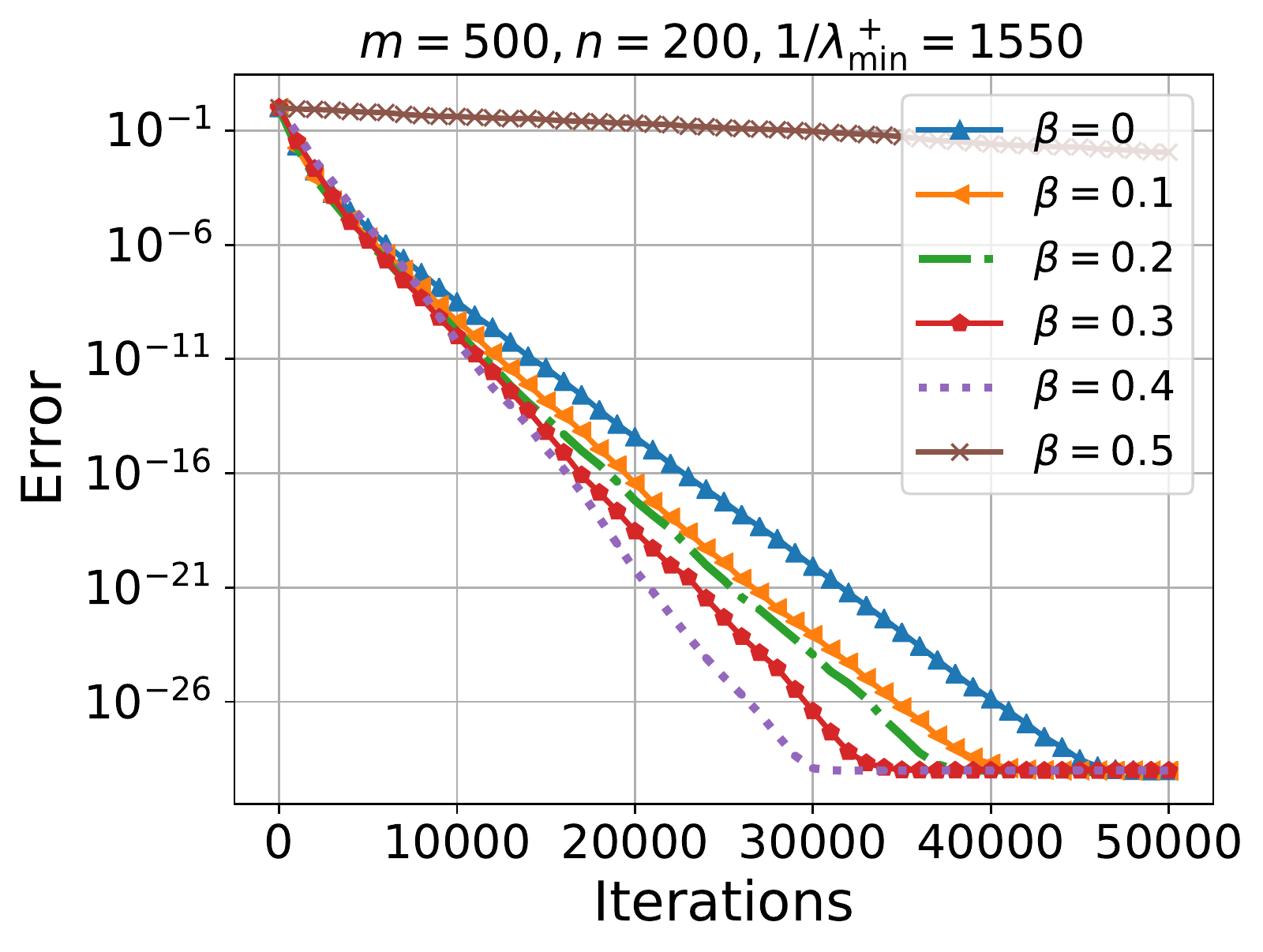}
\end{subfigure}%
\begin{subfigure}{.23\textwidth}
  \centering
  \includegraphics[width=1\linewidth]{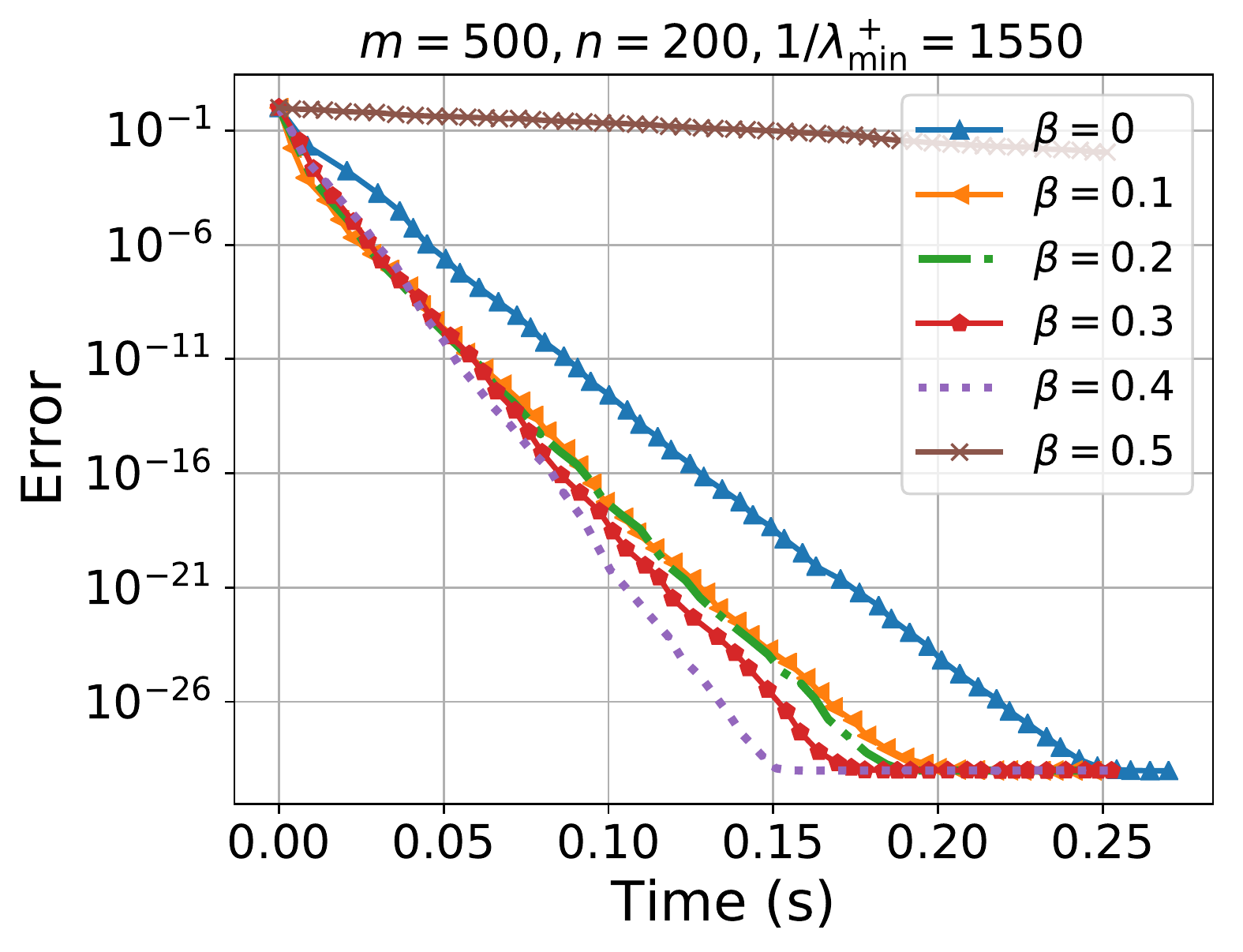}
\end{subfigure}
\begin{subfigure}{.23\textwidth}
  \centering
  \includegraphics[width=1\linewidth]{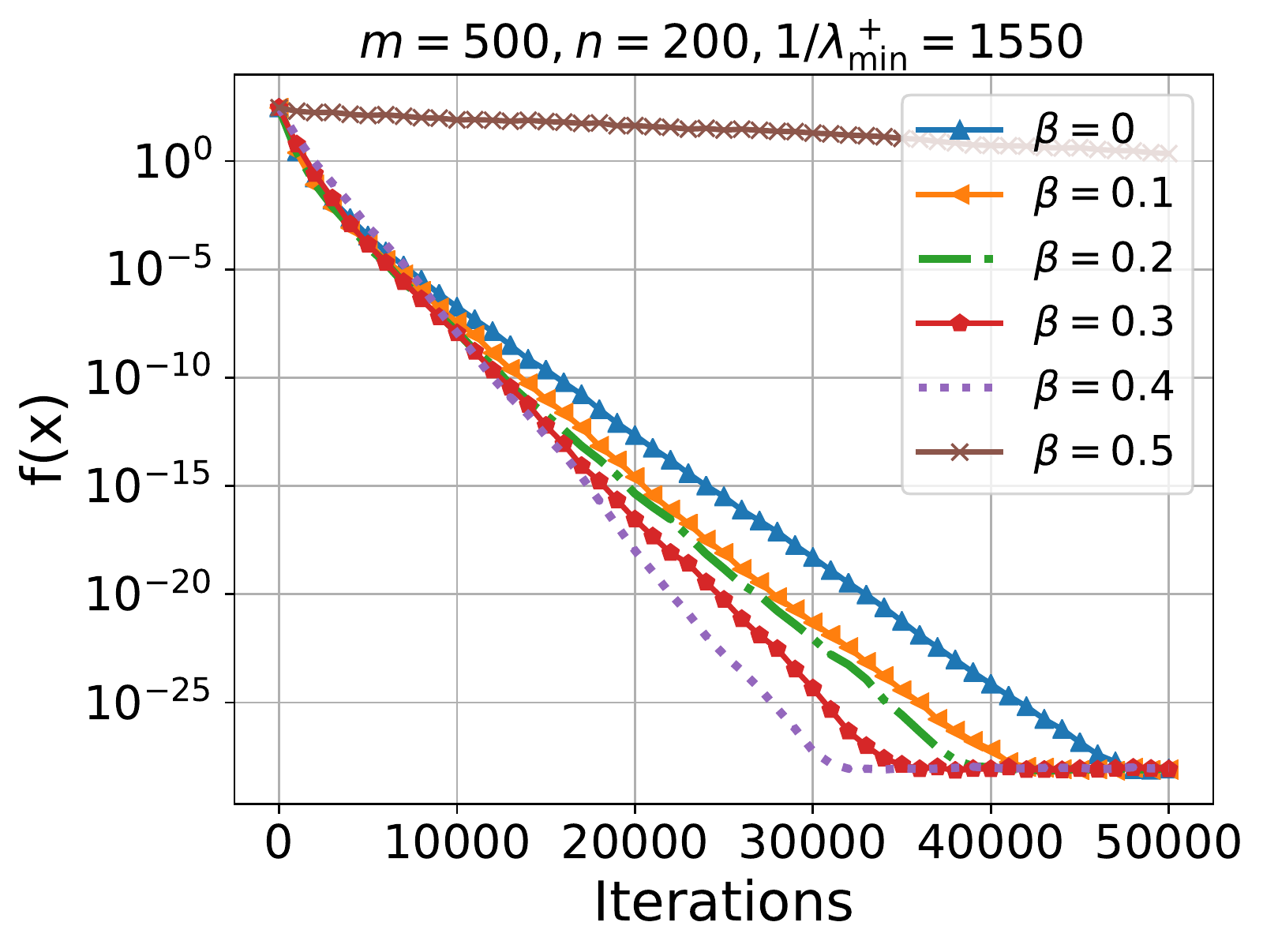}
\end{subfigure}
\begin{subfigure}{.23\textwidth}
  \centering
  \includegraphics[width=1\linewidth]{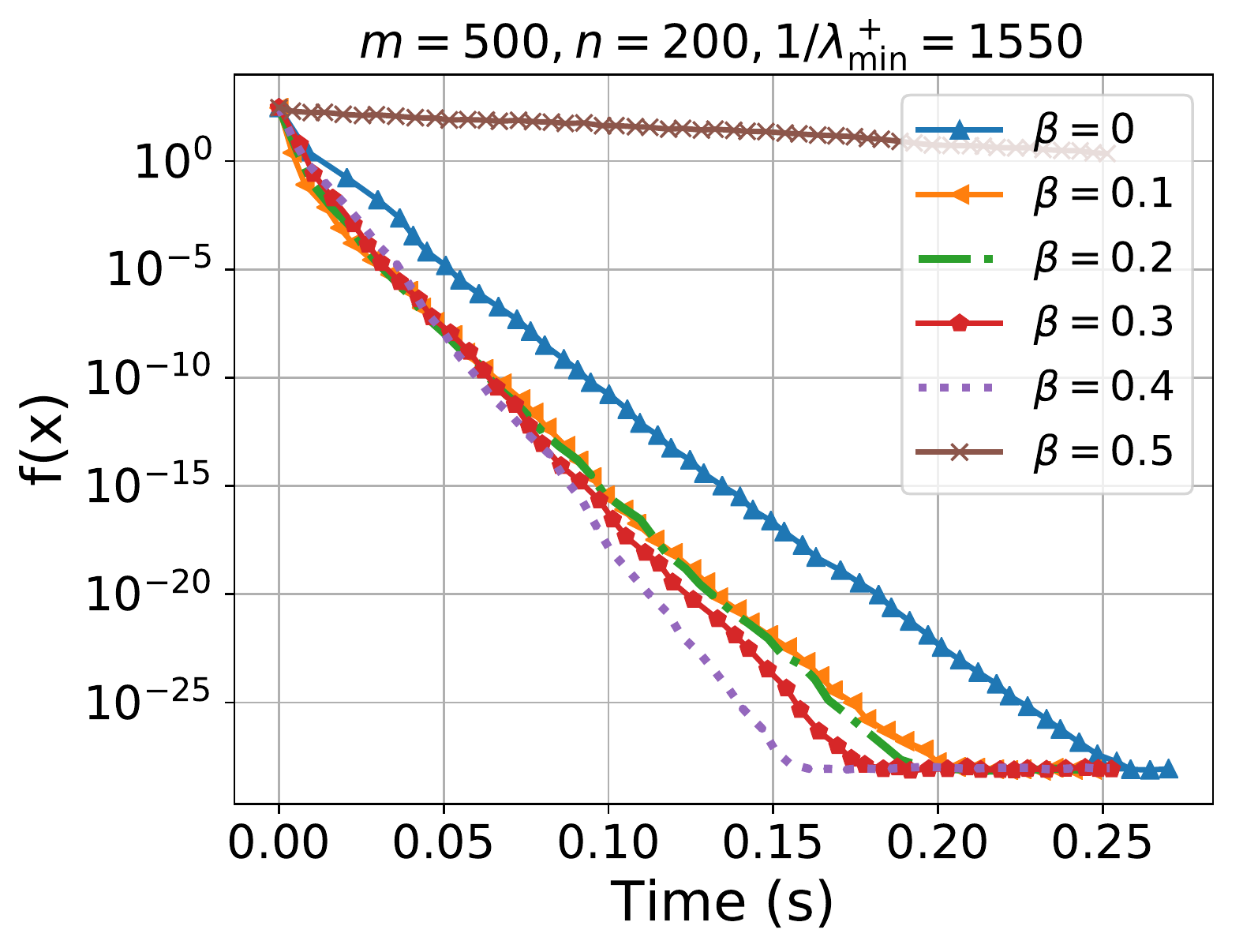}
\end{subfigure}\\
\begin{subfigure}{.23\textwidth}
  \centering
  \includegraphics[width=1\linewidth]{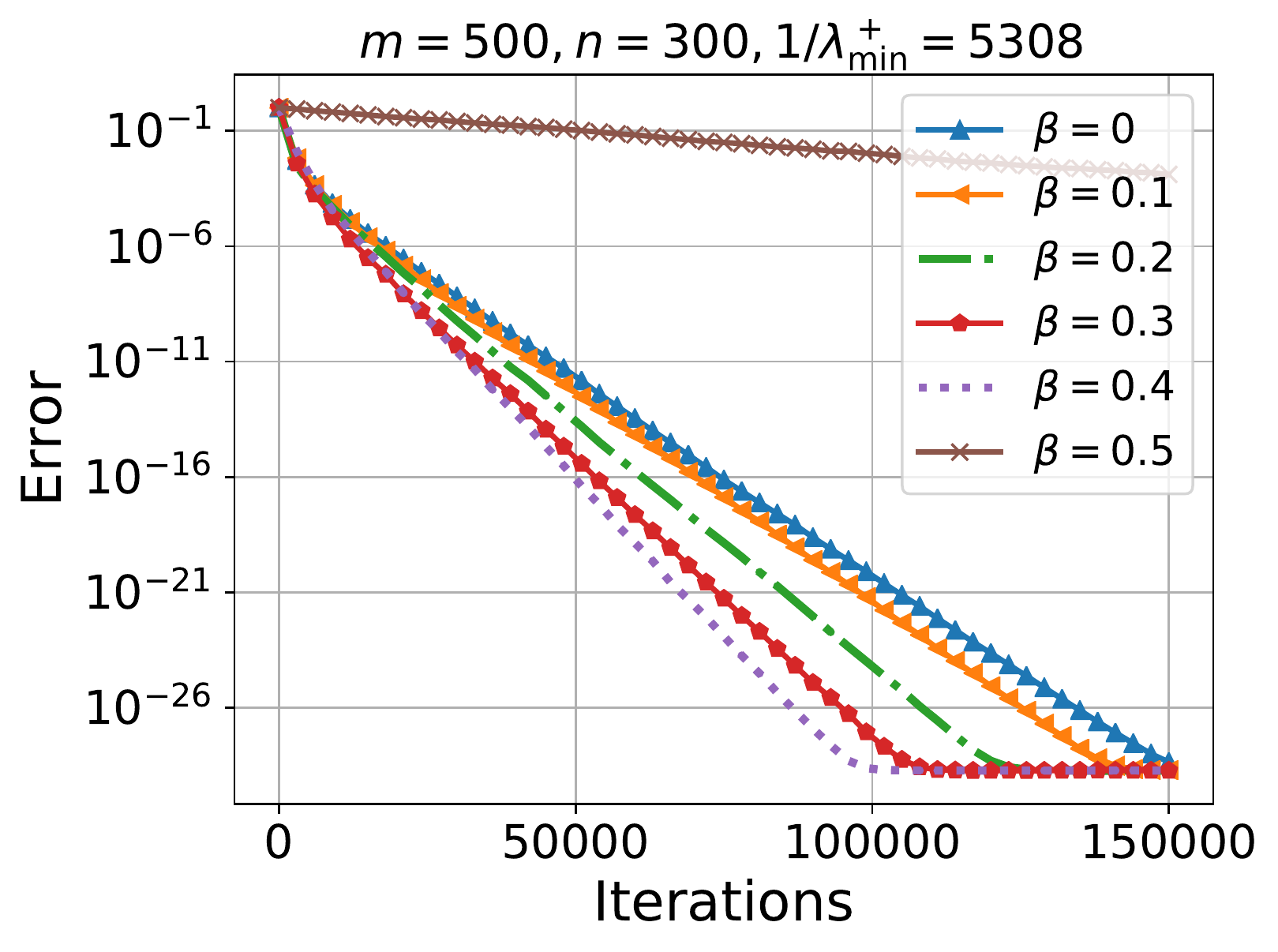}
\end{subfigure}%
\begin{subfigure}{.23\textwidth}
  \centering
  \includegraphics[width=1\linewidth]{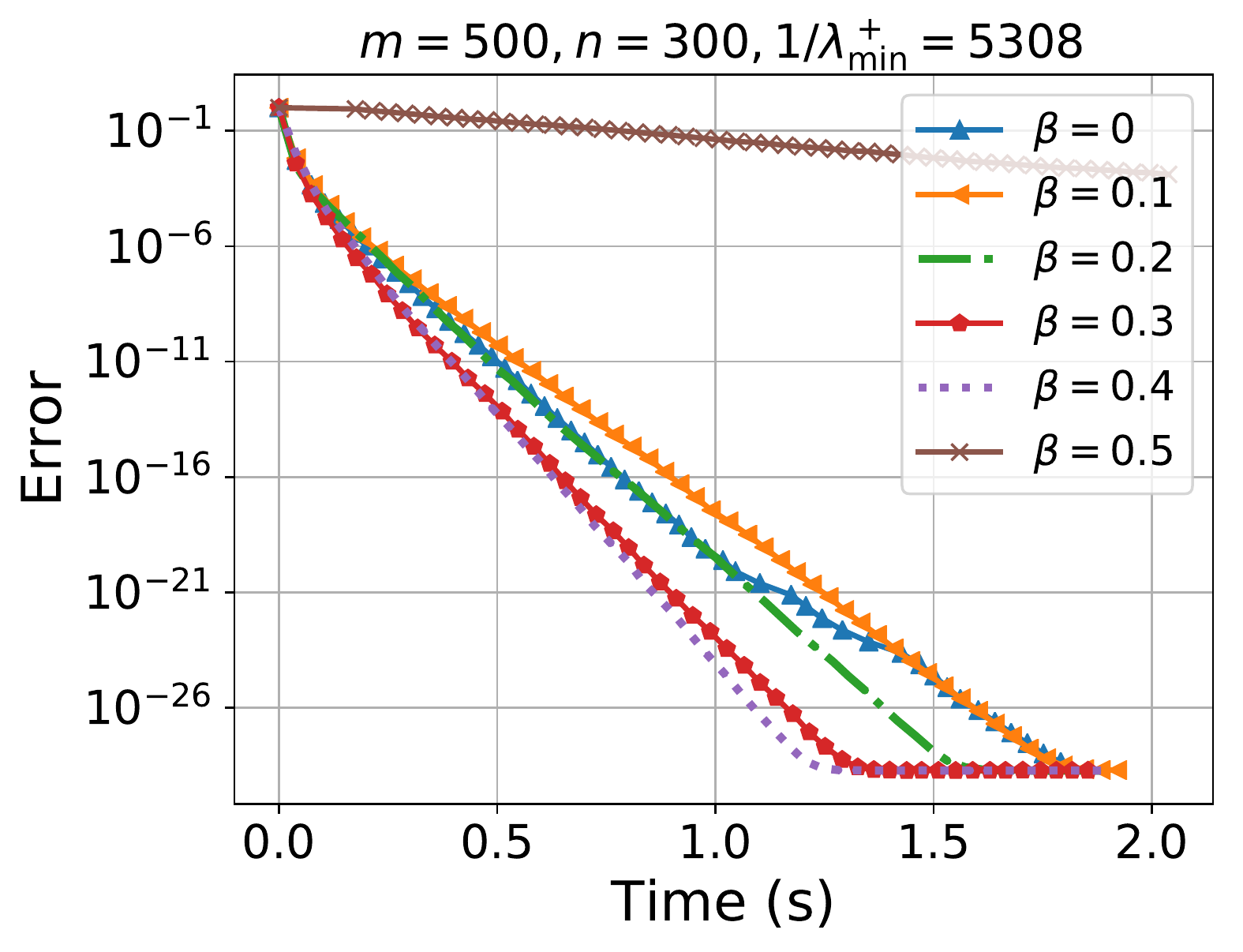}
\end{subfigure}
\begin{subfigure}{.23\textwidth}
  \centering
  \includegraphics[width=1\linewidth]{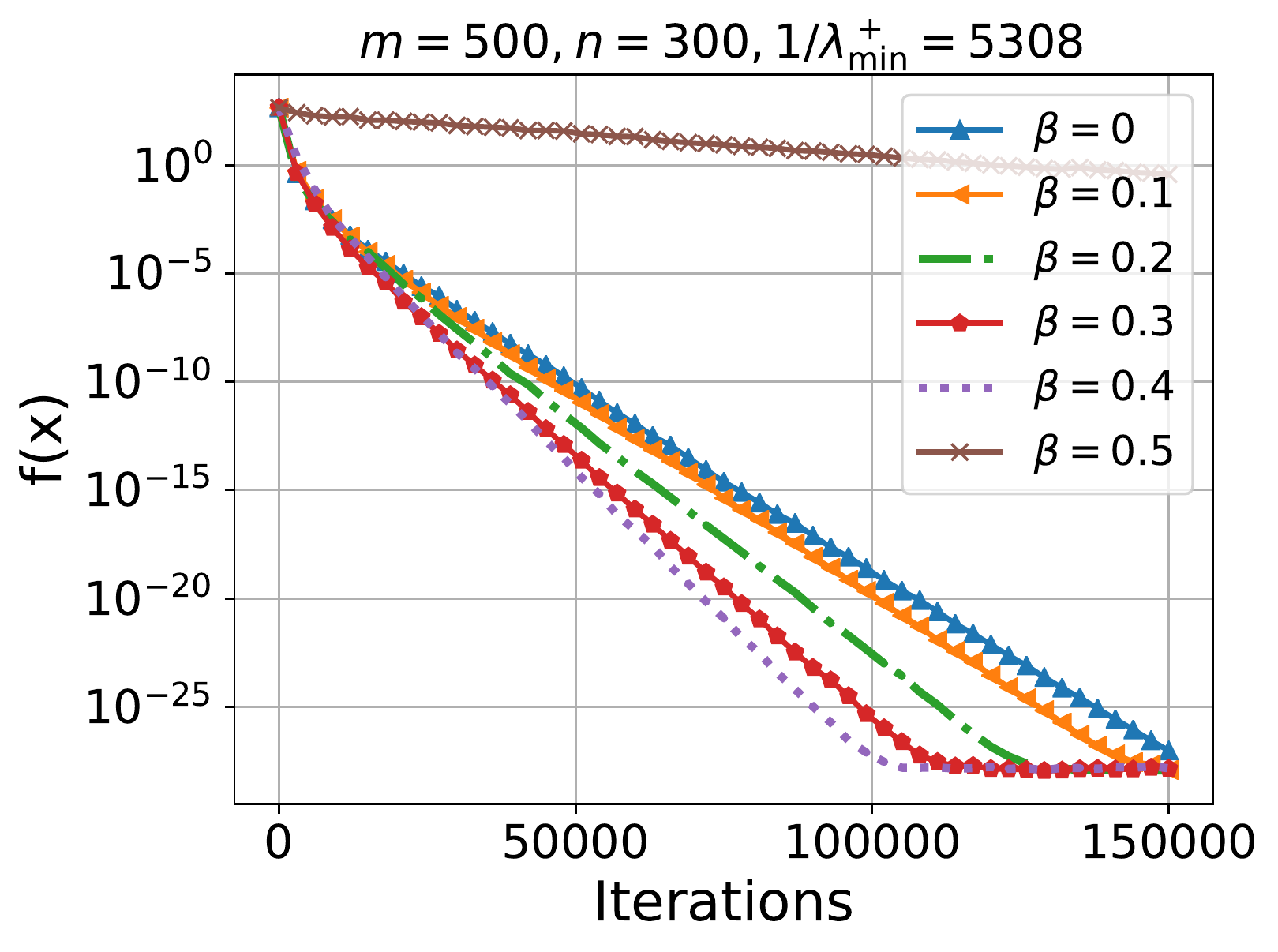}
\end{subfigure}
\begin{subfigure}{.23\textwidth}
  \centering
  \includegraphics[width=1\linewidth]{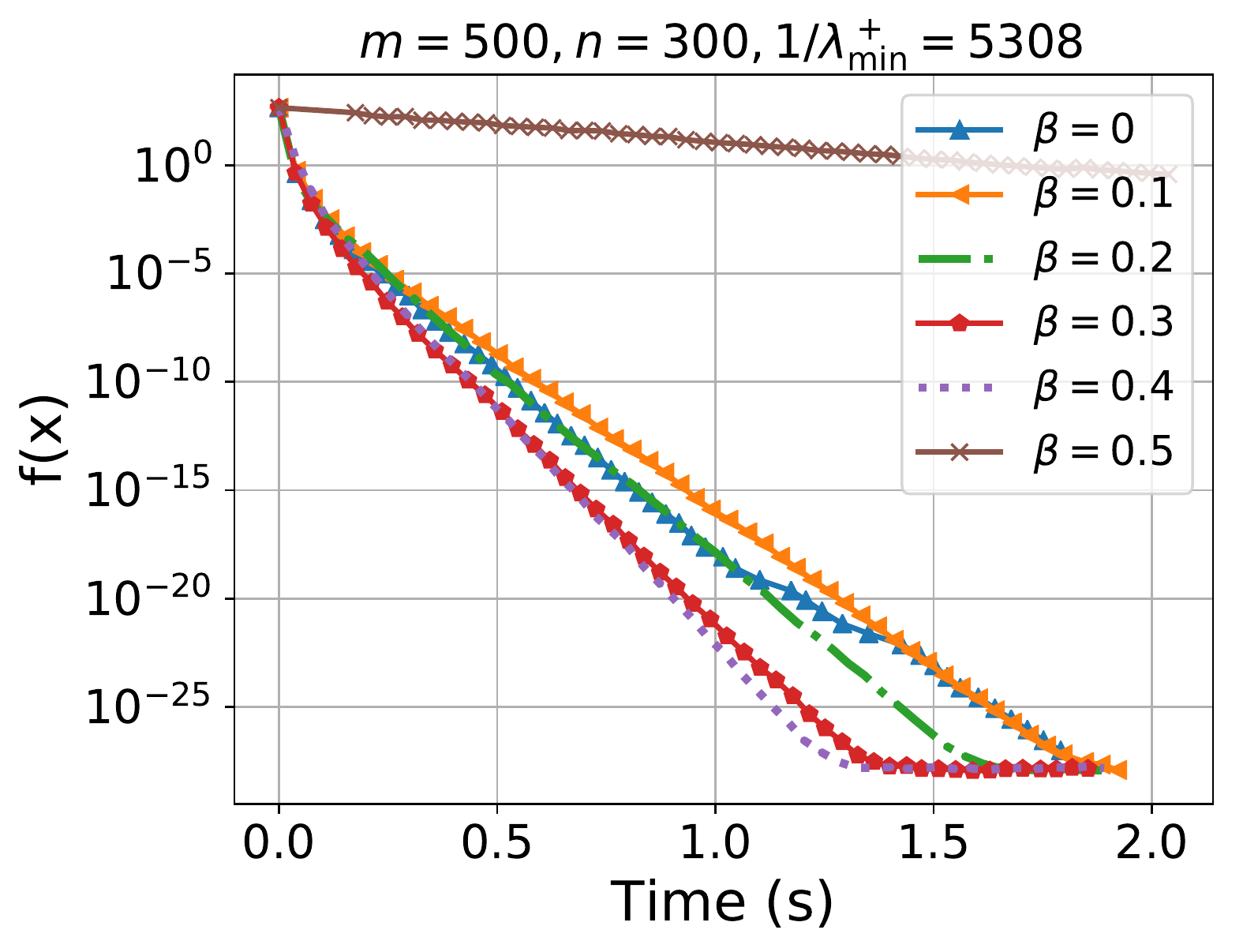}
\end{subfigure}\\
\begin{subfigure}{.23\textwidth}
  \centering
  \includegraphics[width=1\linewidth]{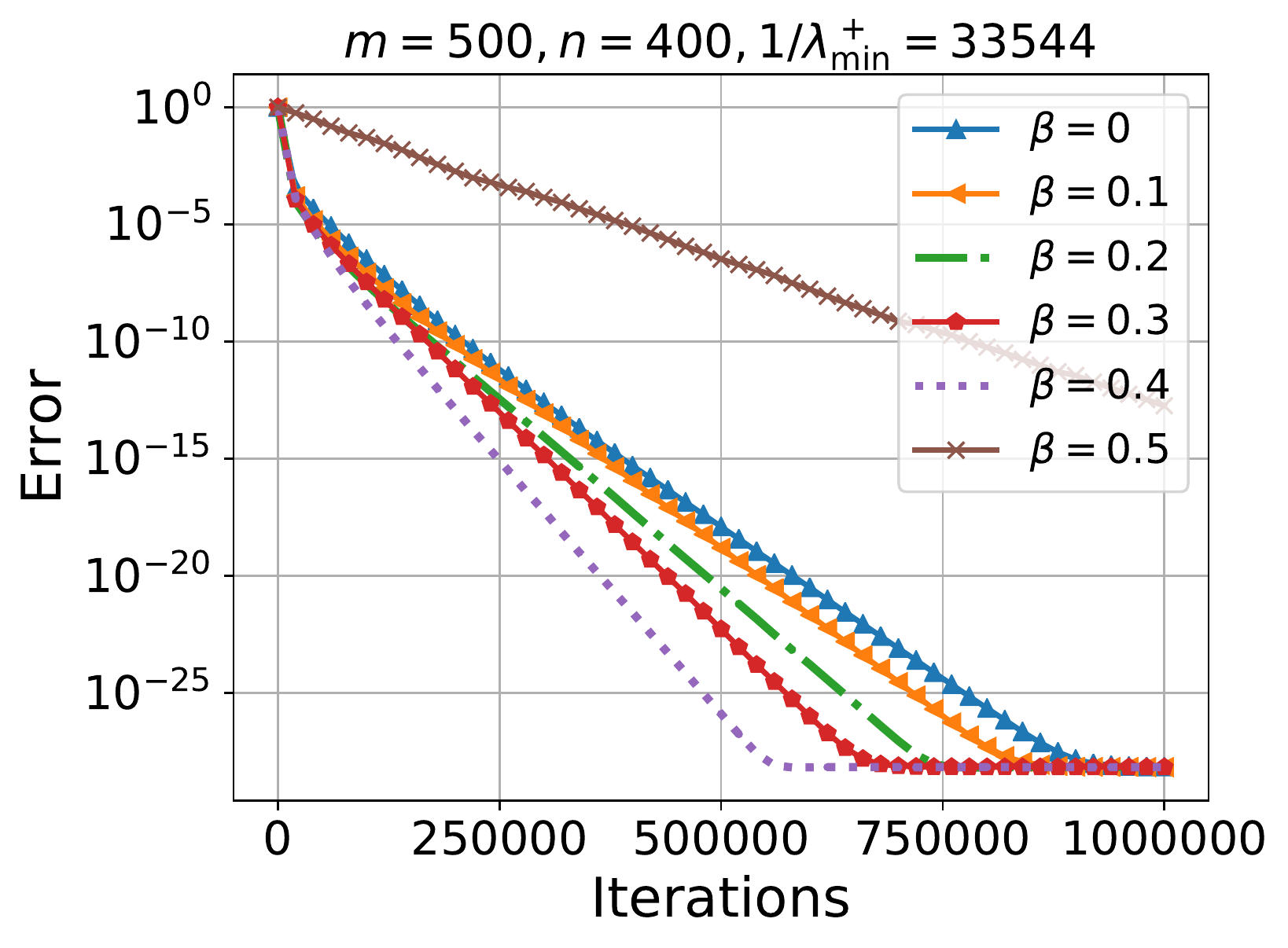}
\end{subfigure}%
\begin{subfigure}{.23\textwidth}
  \centering
  \includegraphics[width=1\linewidth]{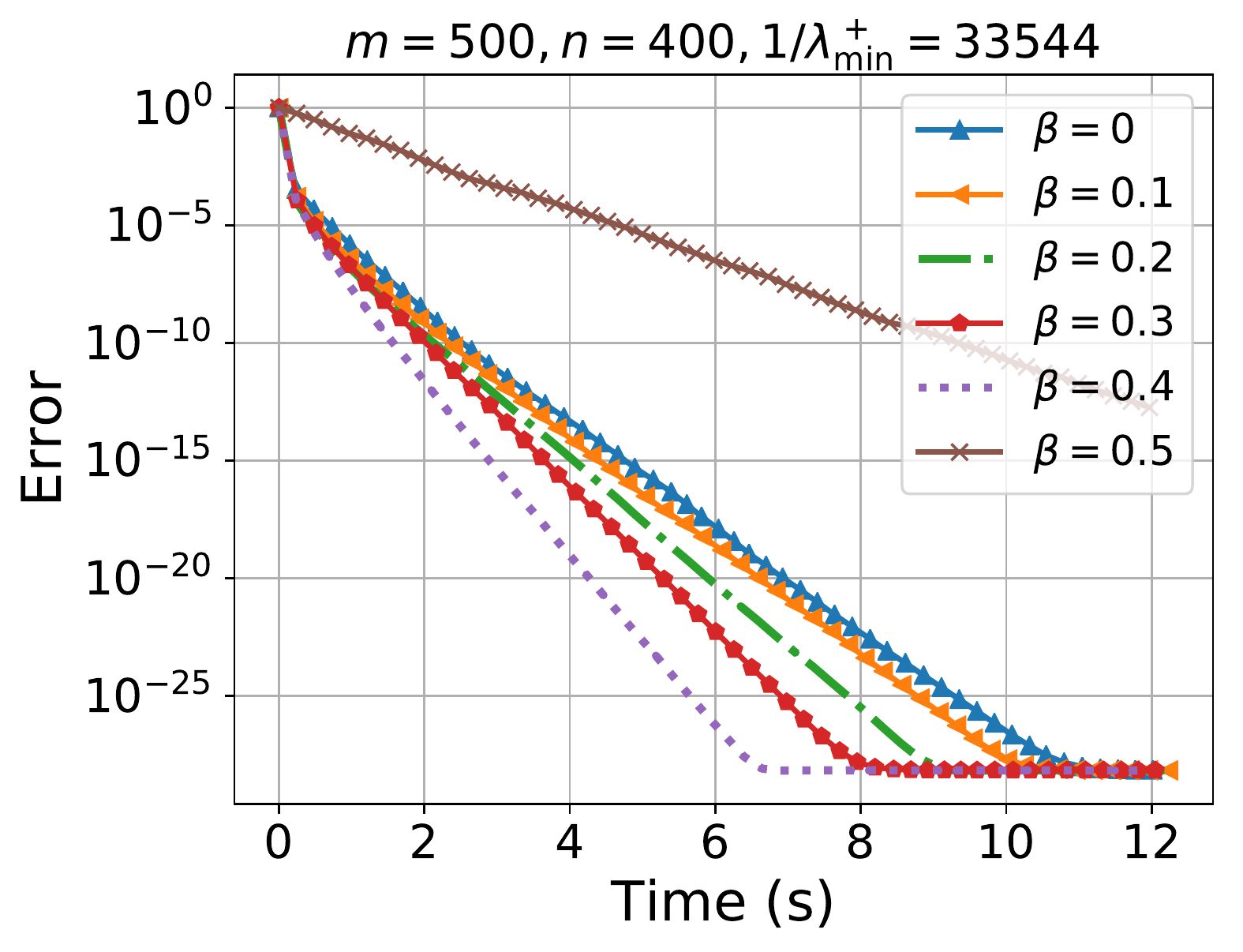}
\end{subfigure}
\begin{subfigure}{.23\textwidth}
  \centering
  \includegraphics[width=1\linewidth]{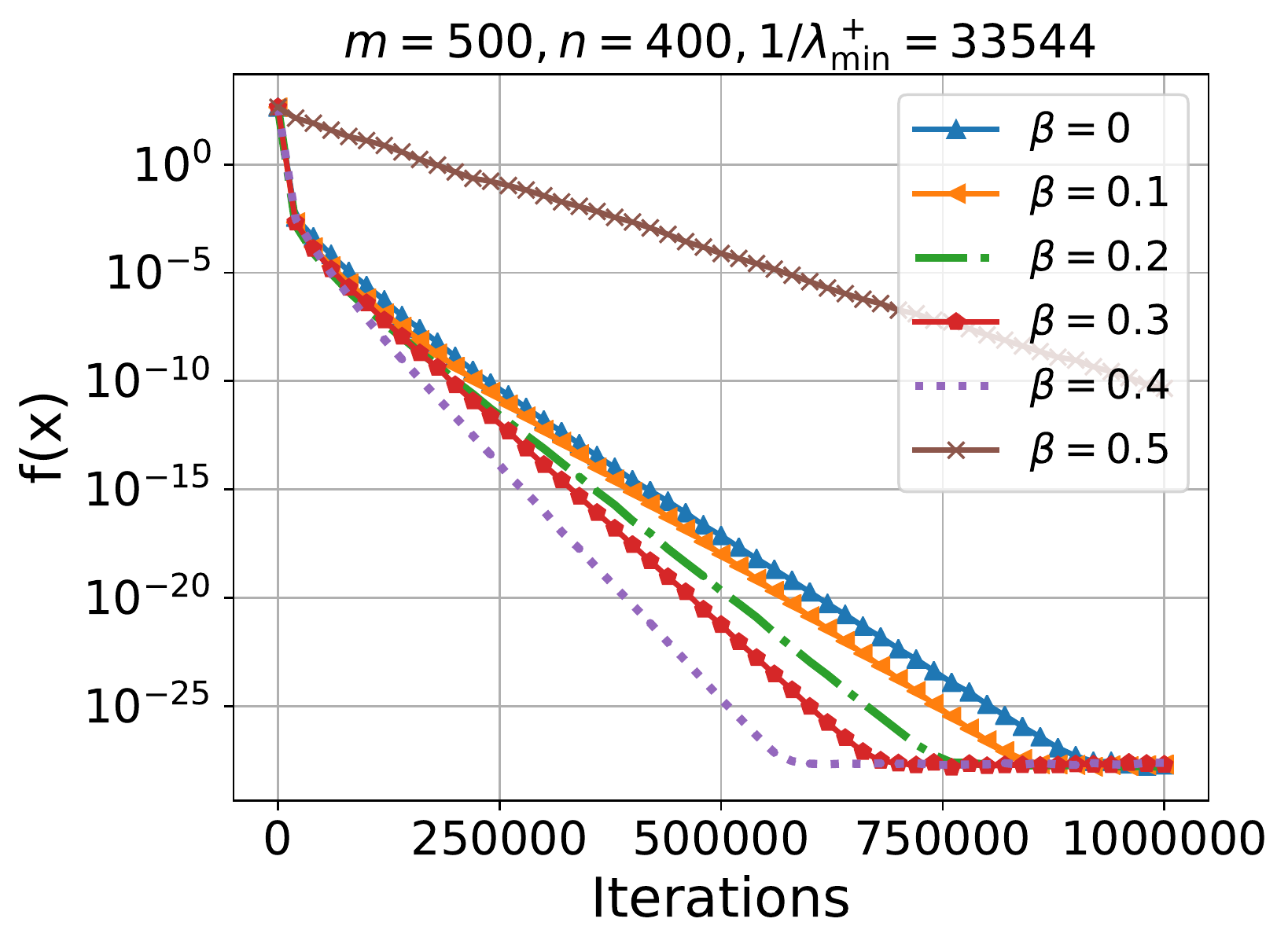}
\end{subfigure}
\begin{subfigure}{.23\textwidth}
  \centering
  \includegraphics[width=1\linewidth]{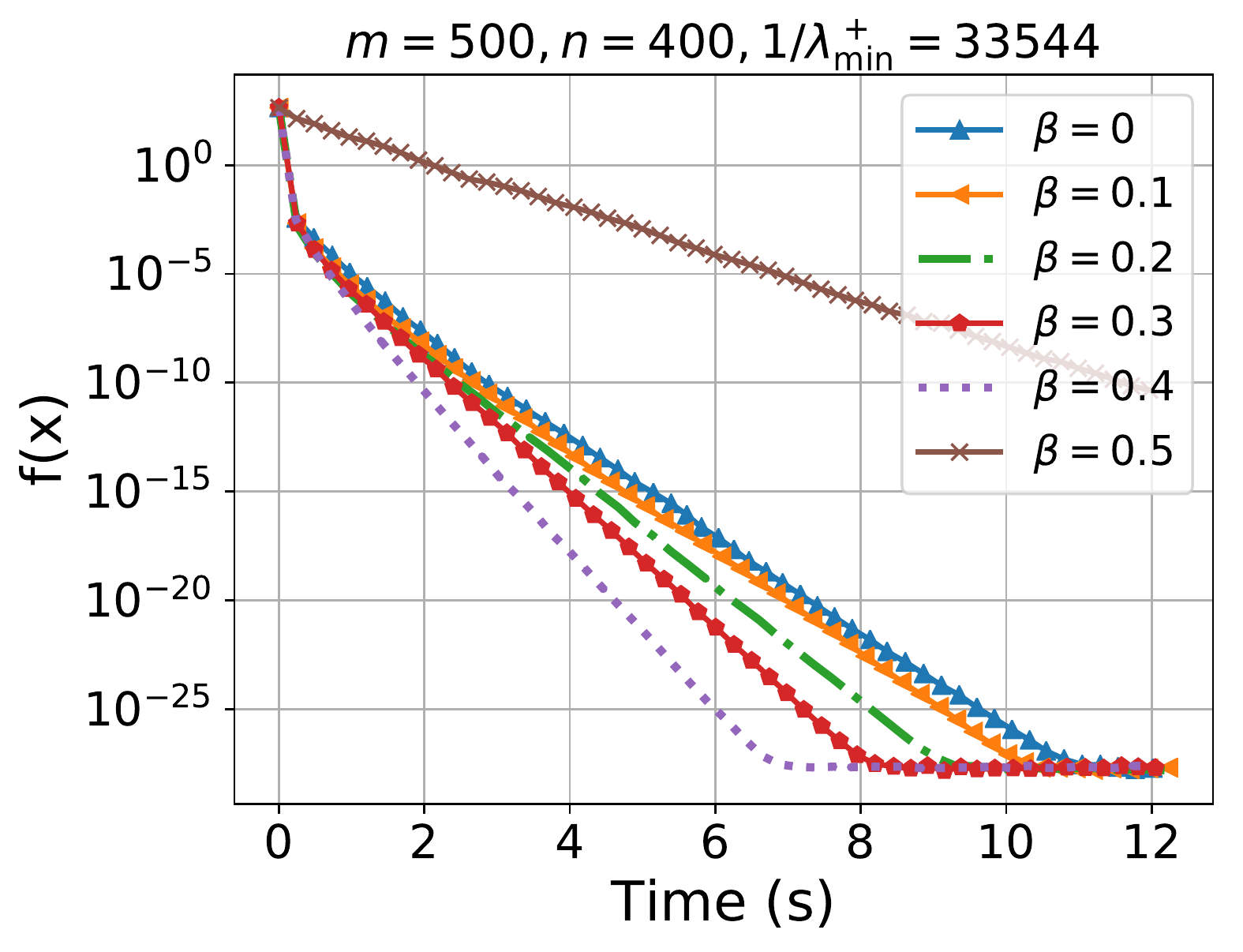}
\end{subfigure}\\
\begin{subfigure}{.23\textwidth}
  \centering
  \includegraphics[width=1\linewidth]{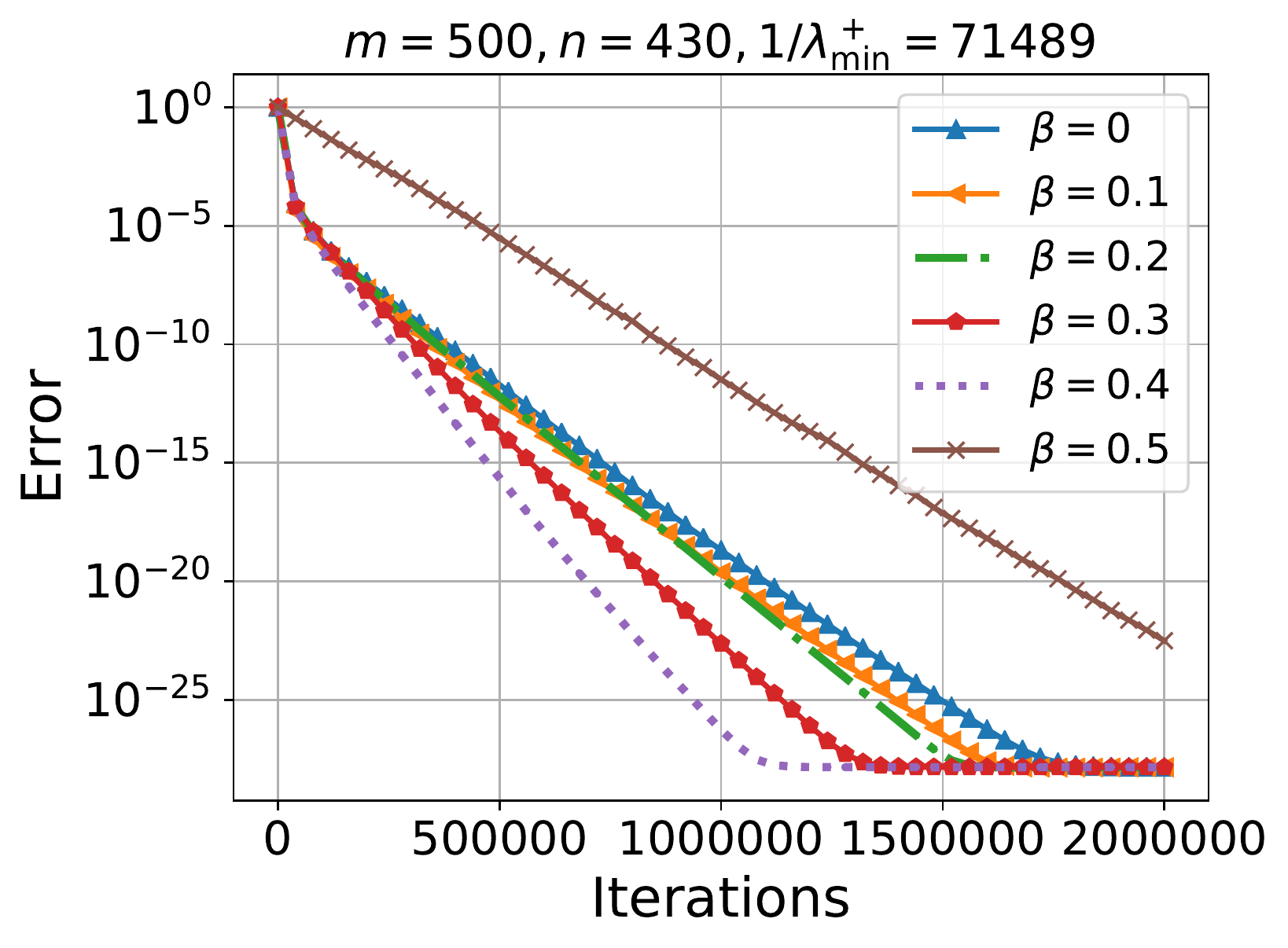}
\end{subfigure}%
\begin{subfigure}{.23\textwidth}
  \centering
  \includegraphics[width=1\linewidth]{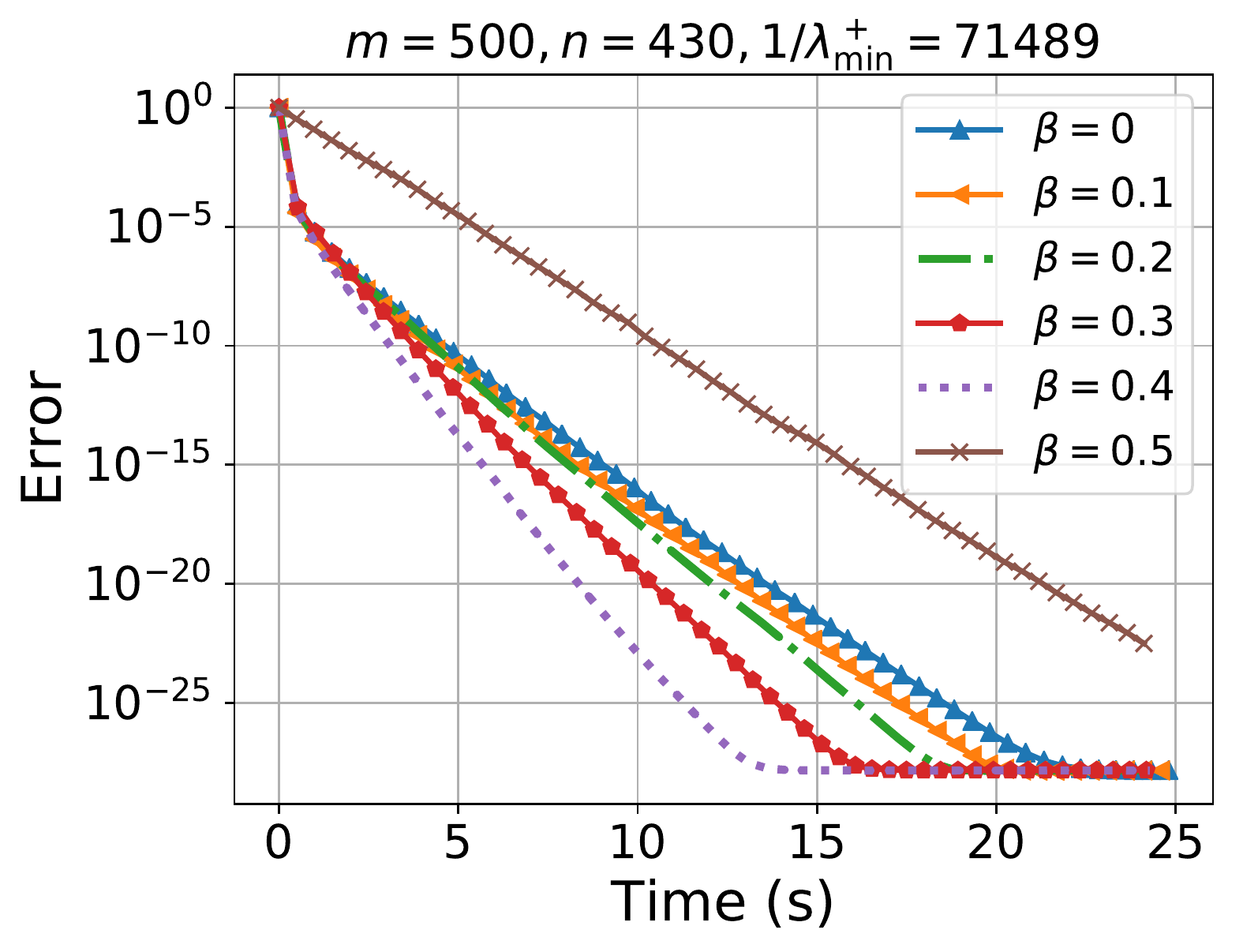}
\end{subfigure}
\begin{subfigure}{.23\textwidth}
  \centering
  \includegraphics[width=1\linewidth]{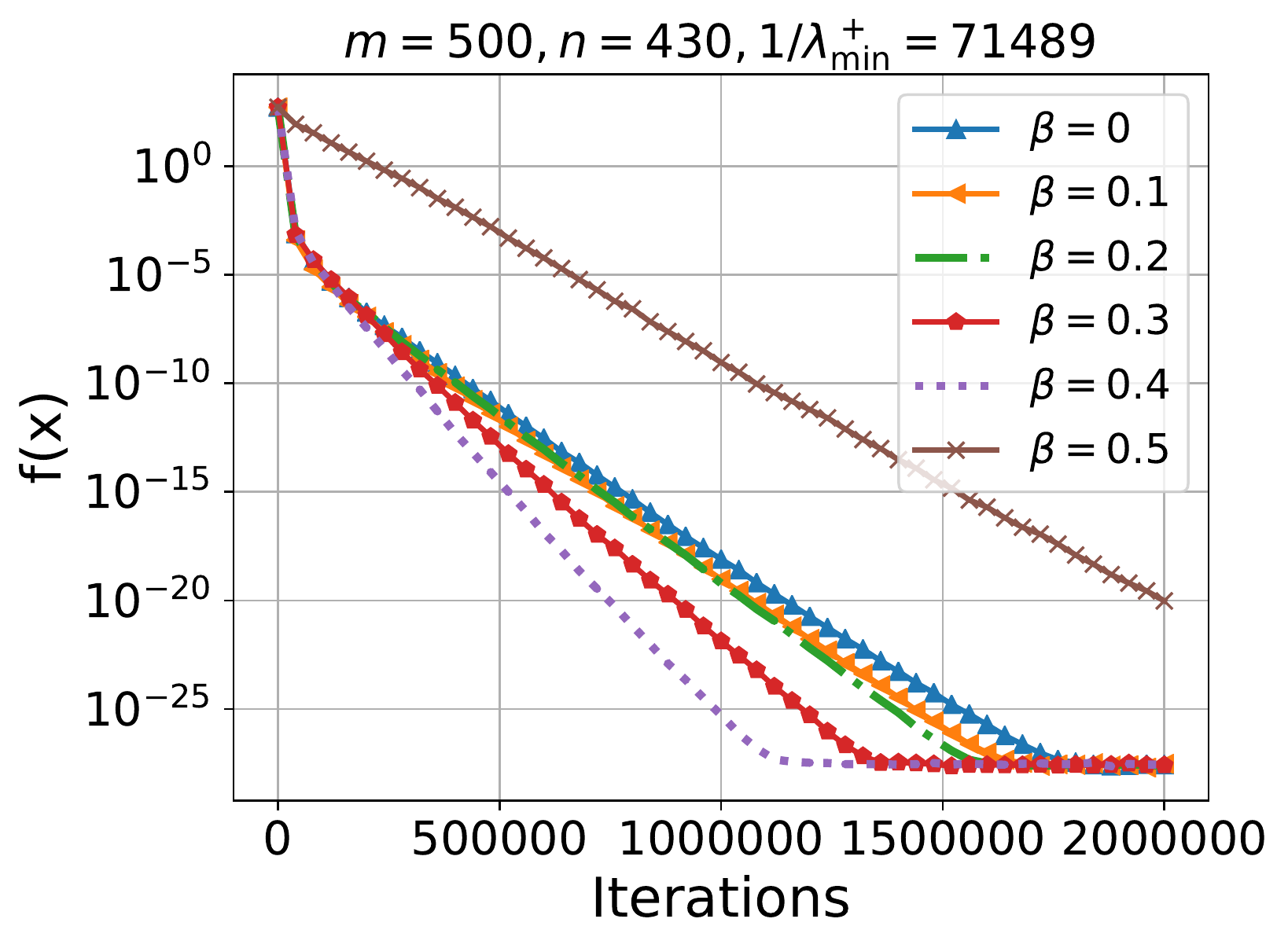}
\end{subfigure}
\begin{subfigure}{.23\textwidth}
  \centering
  \includegraphics[width=1\linewidth]{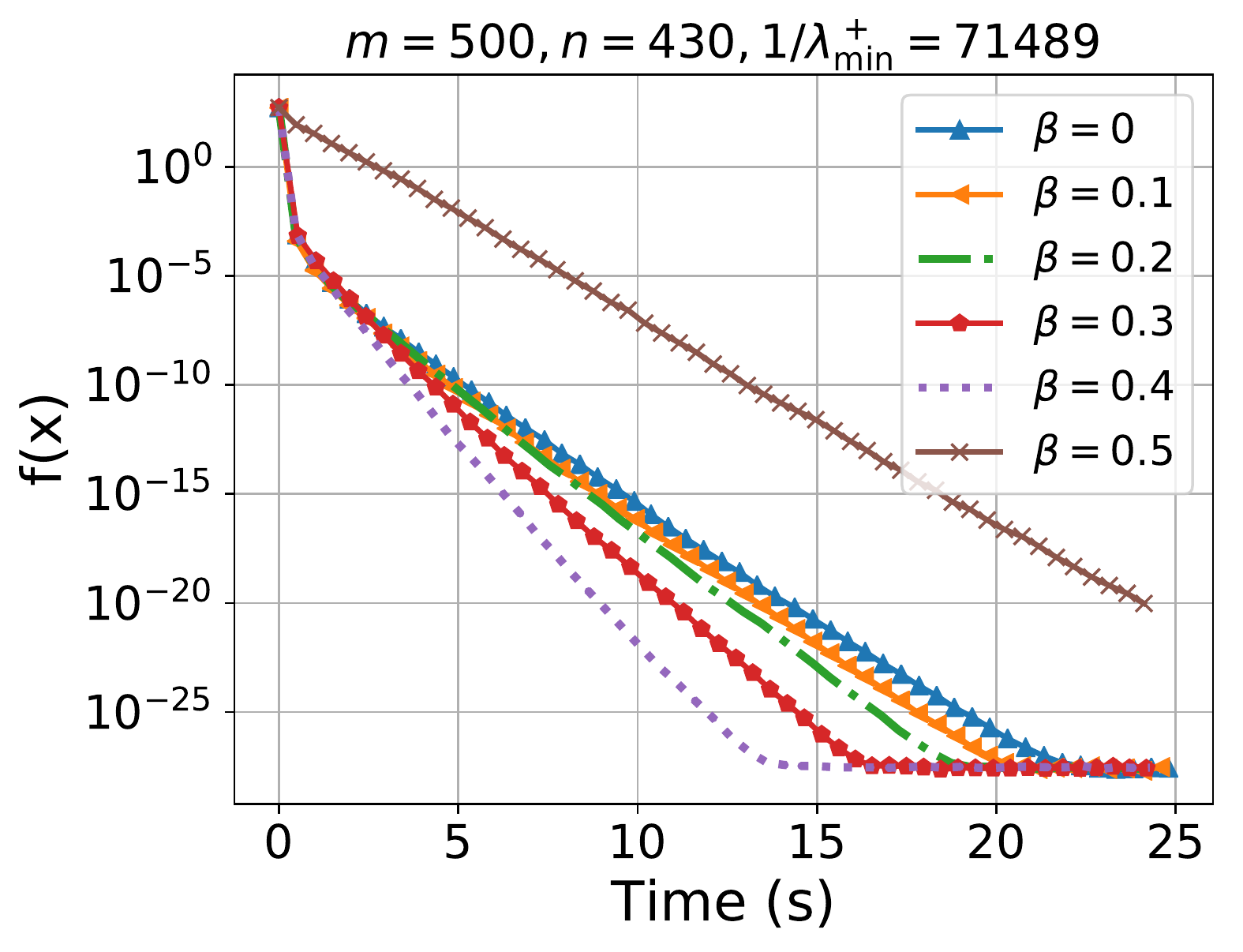}
\end{subfigure}\\
\begin{subfigure}{.23\textwidth}
  \centering
  \includegraphics[width=1\linewidth]{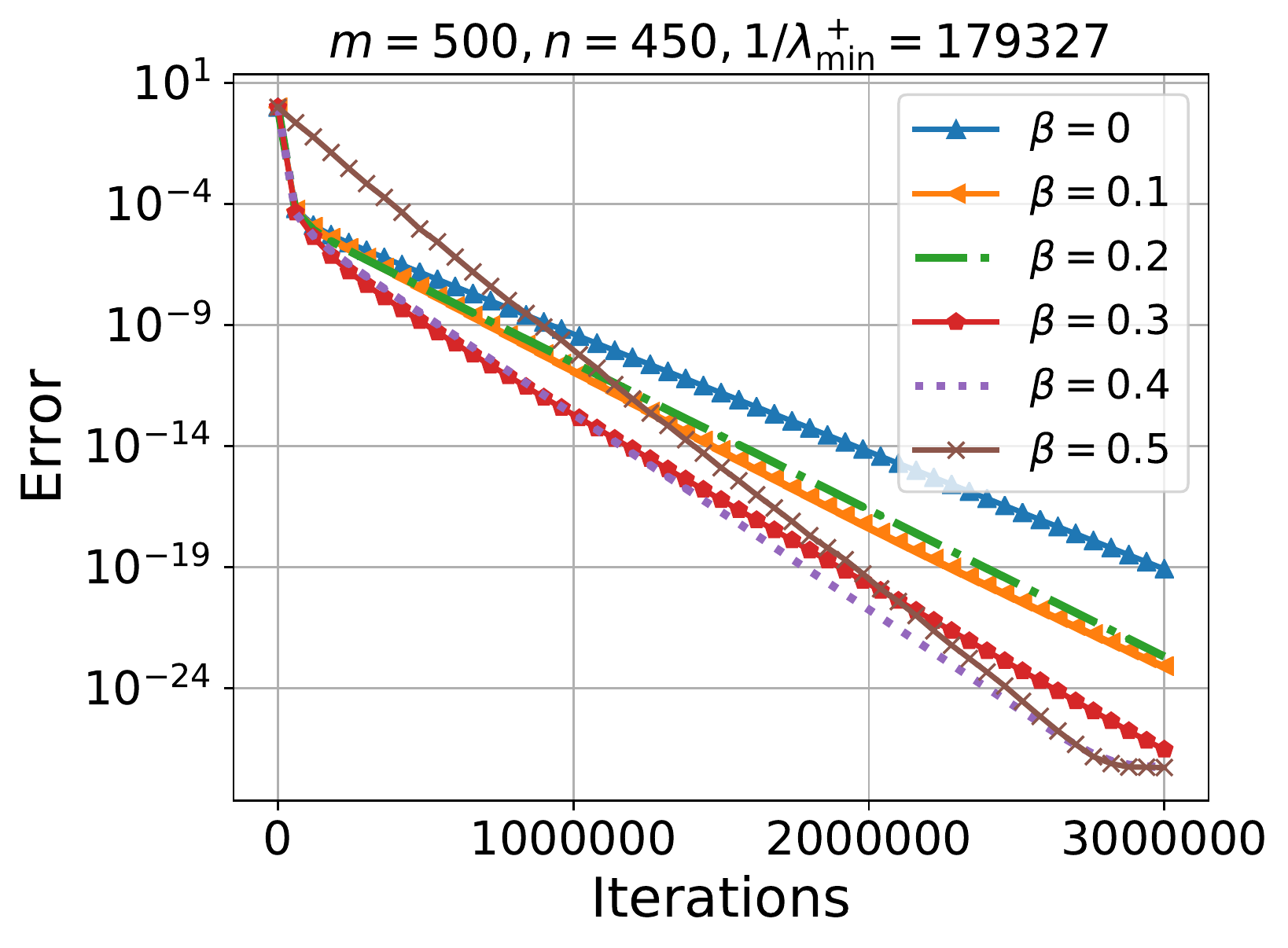}
\end{subfigure}%
\begin{subfigure}{.23\textwidth}
  \centering
  \includegraphics[width=1\linewidth]{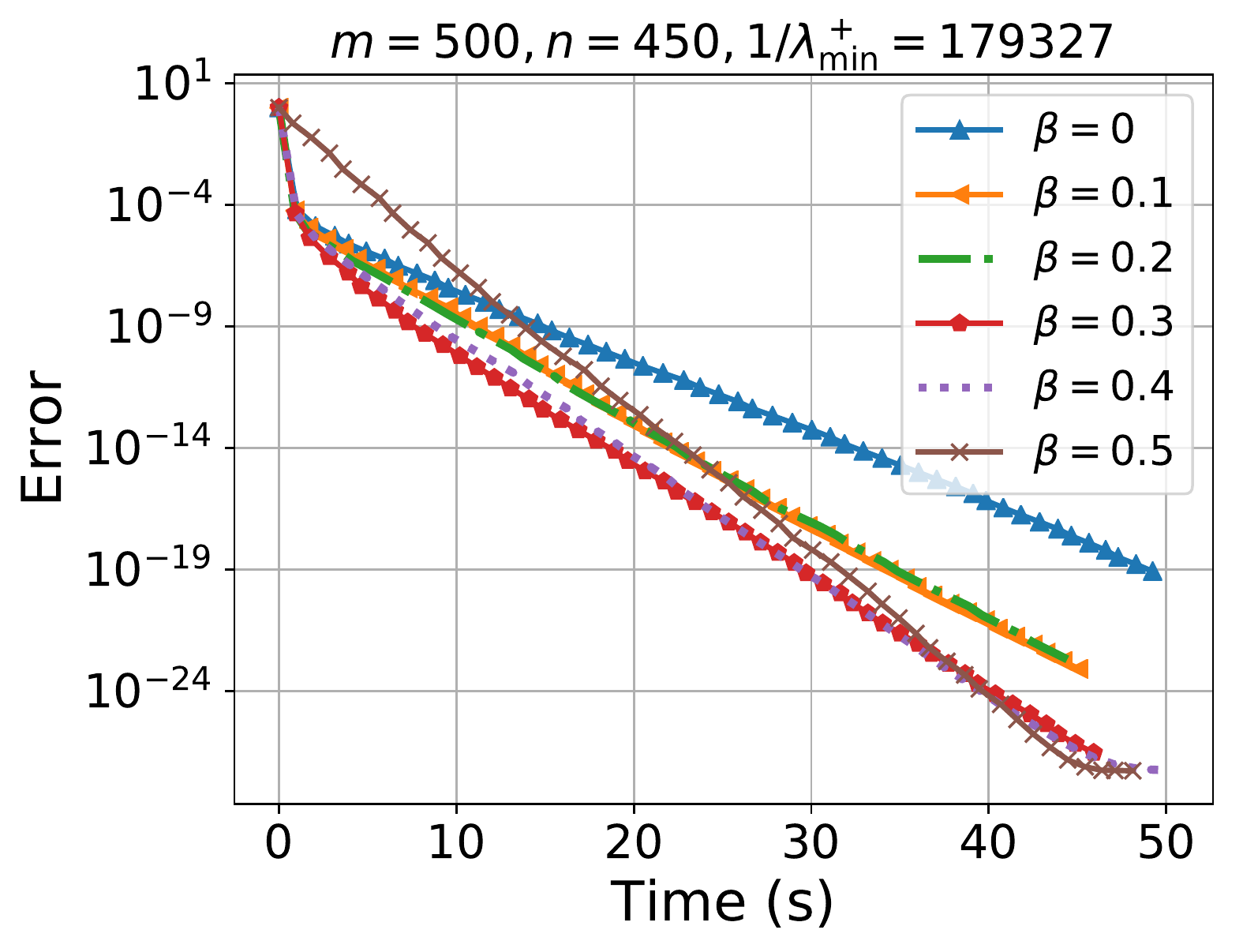}
\end{subfigure}
\begin{subfigure}{.23\textwidth}
  \centering
  \includegraphics[width=1\linewidth]{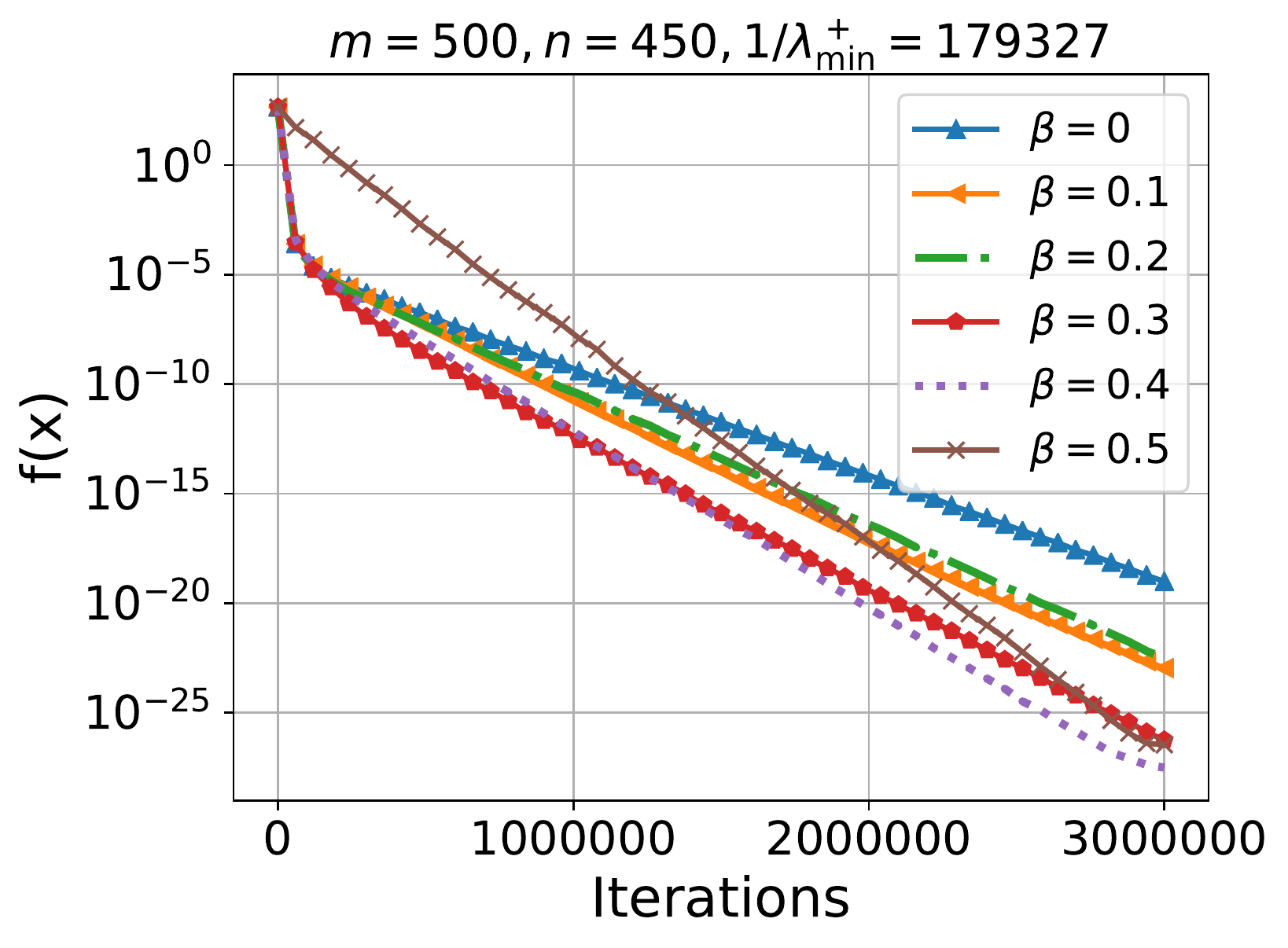}
\end{subfigure}
\begin{subfigure}{.23\textwidth}
  \centering
  \includegraphics[width=1\linewidth]{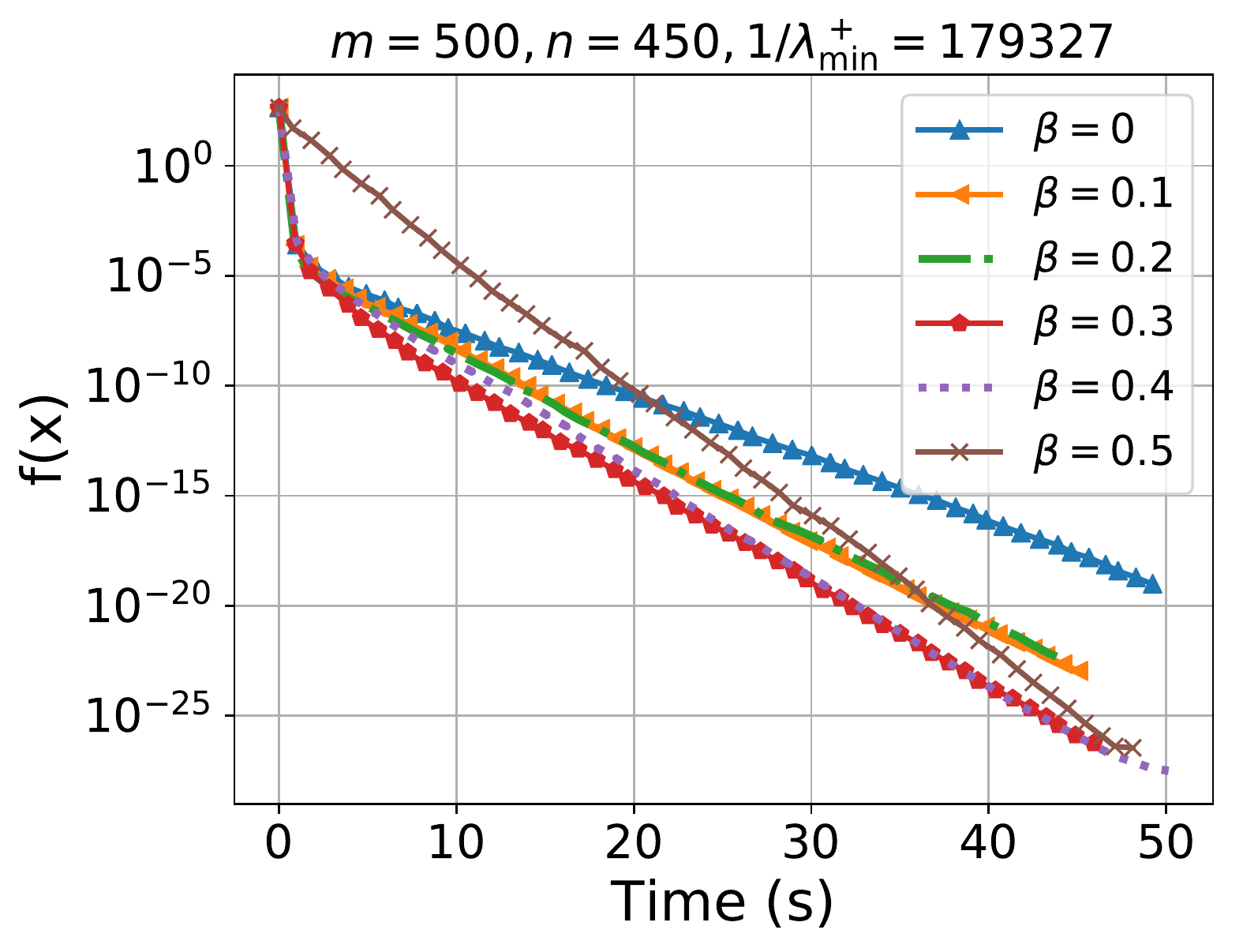}
\end{subfigure}
\caption{Performance of mRCD for fixed stepsize $\omega=1$ and several momentum parameters $\beta$ for consistent linear systems with positive definite matrices $\bA=\bP^\top \bP$ where $\bP \in \R^{m \times n}$ is Gaussian matrix with $m=500$ rows and $n=200,300,400,430,450$. The graphs in the first (second) column plot iterations (time) against residual error while those in the third (forth) column plot iterations (time) against function values. All plots are averaged over 10 trials. The title of each plot indicates the dimensions of the matrix $\bP$ and the value of $1/\lambda_{\min}^+$. The ``Error" on the vertical axis represents the relative error $\|x_k-x_*\|^2_\bB / \|x_0-x_*\|^2_\bB \overset{\bB=\bA, x_0=0}{=}\|x_k-x_*\|^2_\bA / \|x_*\|^2_\bA$ and the function values $f(x_k)$ refer to function~\eqref{functionRCD}.}
\label{RCDperformance1}
\end{figure}

\subsubsection{Real Data}
In the following experiments we test the performance of mRK using real matrices (datasets) from the library of support vector machine problems LIBSVM \cite{chang2011libsvm}. Each dataset consists of a matrix $\bA \in \R^{m \times n}$  ($m$ features and $n$ characteristics) and a vector of labels $b \in \R^m$.  In our experiments we choose to use only the matrices of the datasets and ignore the label vector. As before, to ensure consistency of the linear system,  we choose a Gaussian vector $z\in \R^n$ and the right hand side of the linear system is set to $b=\bA z$. 
Similarly as in  the case of synthetic data, mRK is tested for several values of momentum parameters $\beta$ and fixed stepsize $\omega=1$. 

In Figure~\ref{RealDataplots} the performance of all methods for both relative error measure $\|x_k-x_*\|^2 / \|x_*\|^2_\bB$ and function values $f(x_k)$ is presented. Note again that $\beta=0$ represents the baseline RK method. The addition of momentum parameter is again often beneficial and leads to faster convergence. As an example, inspect the plots for the \textit{mushrooms} dataset in Figure~\ref{RealDataplots}, where mRK with $\beta=0.5$ is much faster than the simple RK method in all presented plots, both in terms of iterations and time. In particular, the addition of a momentum parameter leads to visible speedup for the datasets \textit{mushrooms}, \textit{splice}, \textit{a9a} and \textit{ionosphere}. For these datasets the acceleration is obvious in all plots both in terms of relative error and function values. For the datasets  \textit{australian}, \textit{gisette} and \textit{madelon} the speedup is less obvious in the plots of the relative error, while for the plots of function values it is not present at all.

\begin{figure}[!]
\centering
\begin{subfigure}{.23\textwidth}
  \centering
  \includegraphics[width=1\linewidth]{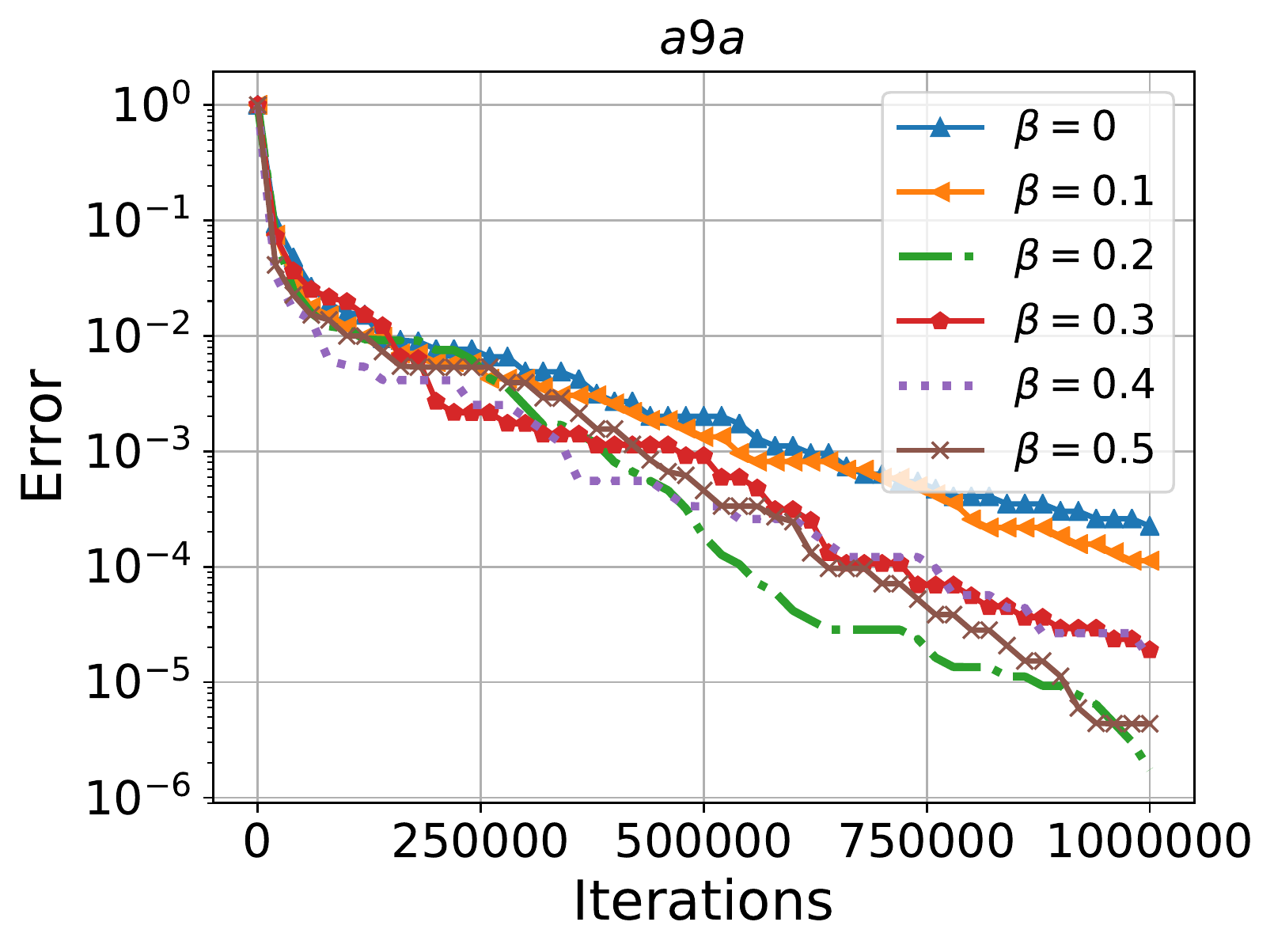}
\end{subfigure}%
\begin{subfigure}{.23\textwidth}
  \centering
  \includegraphics[width=1\linewidth]{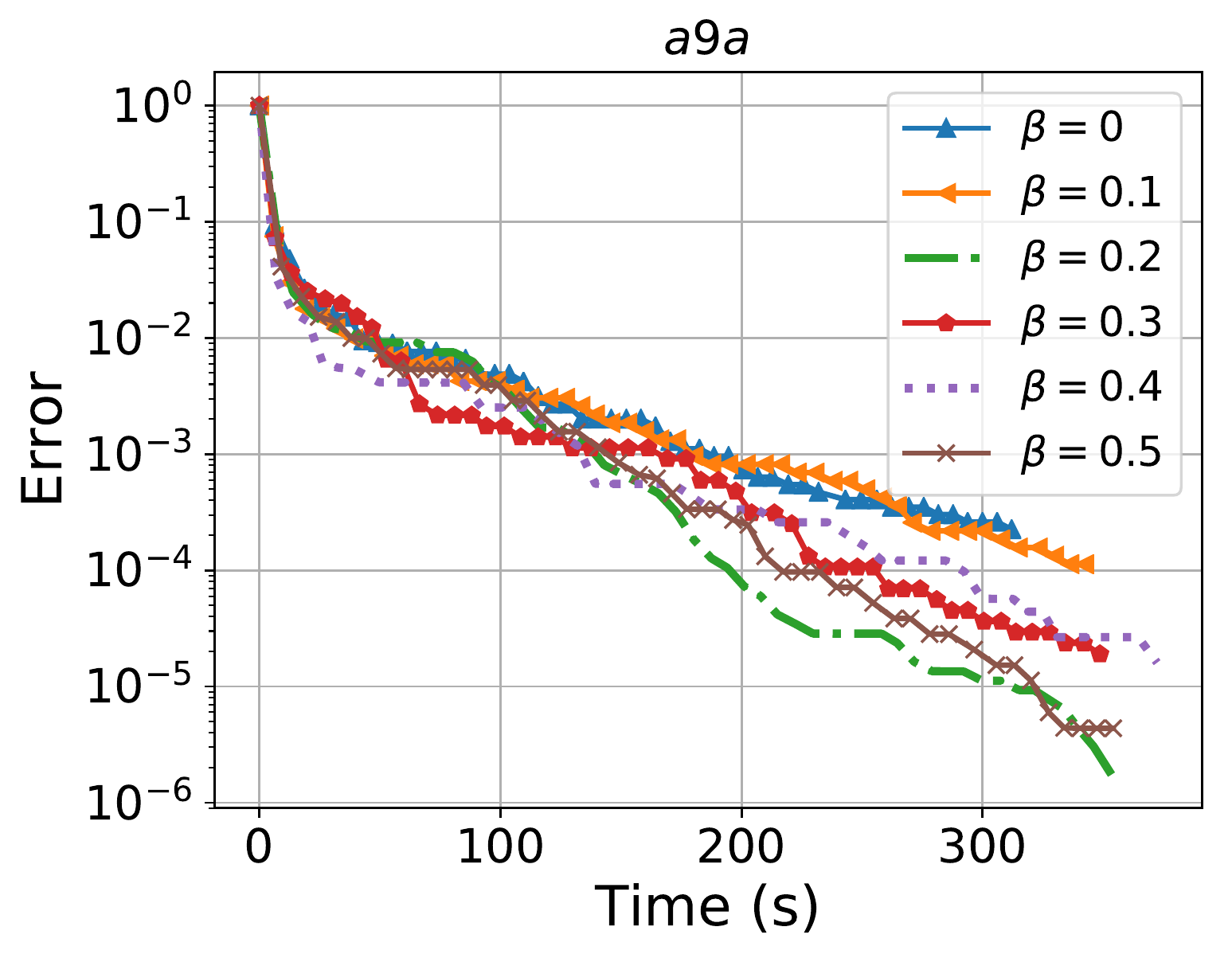}
\end{subfigure}
\begin{subfigure}{.23\textwidth}
  \centering
  \includegraphics[width=1\linewidth]{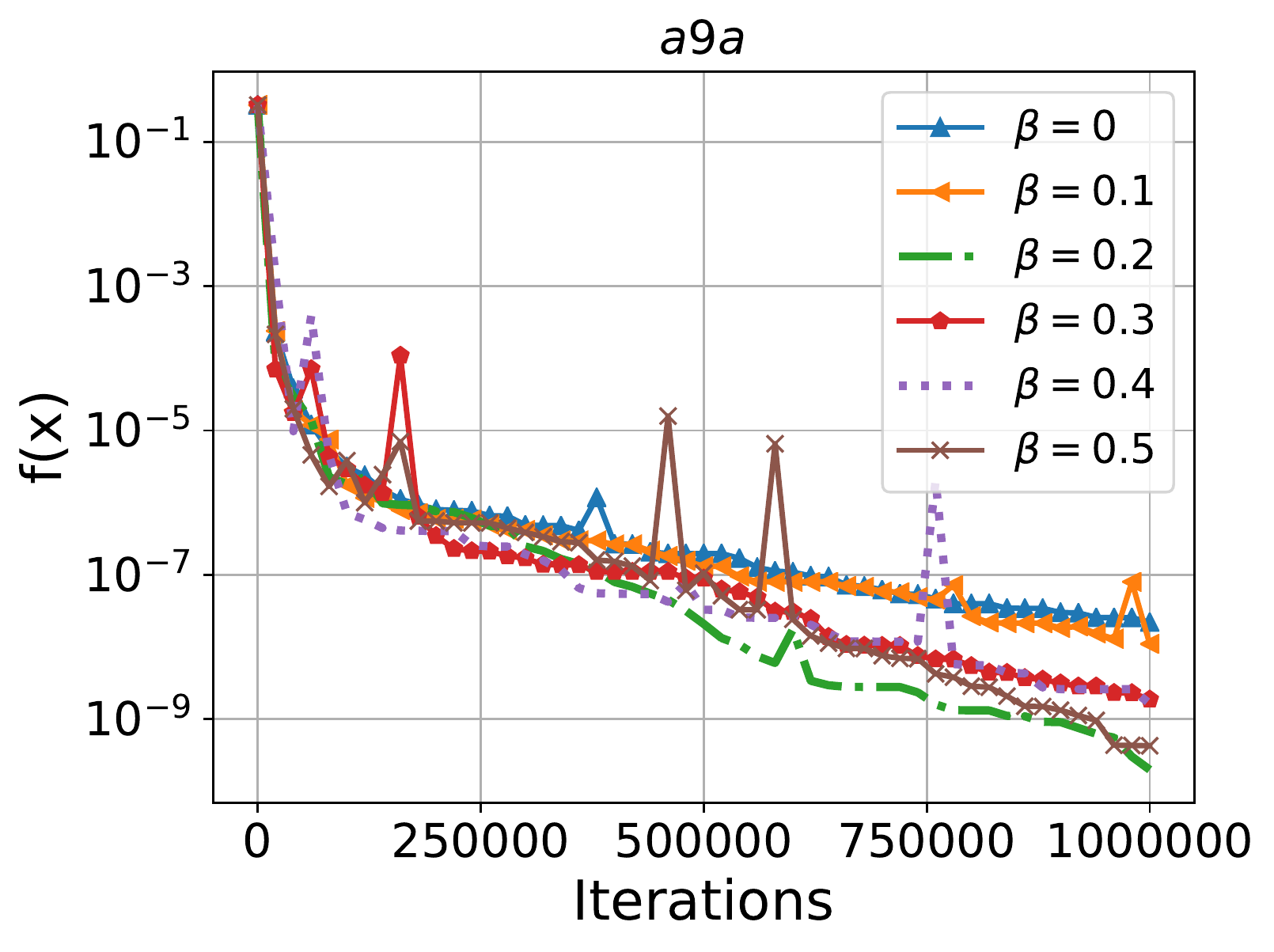}
\end{subfigure}
\begin{subfigure}{.23\textwidth}
  \centering
  \includegraphics[width=1\linewidth]{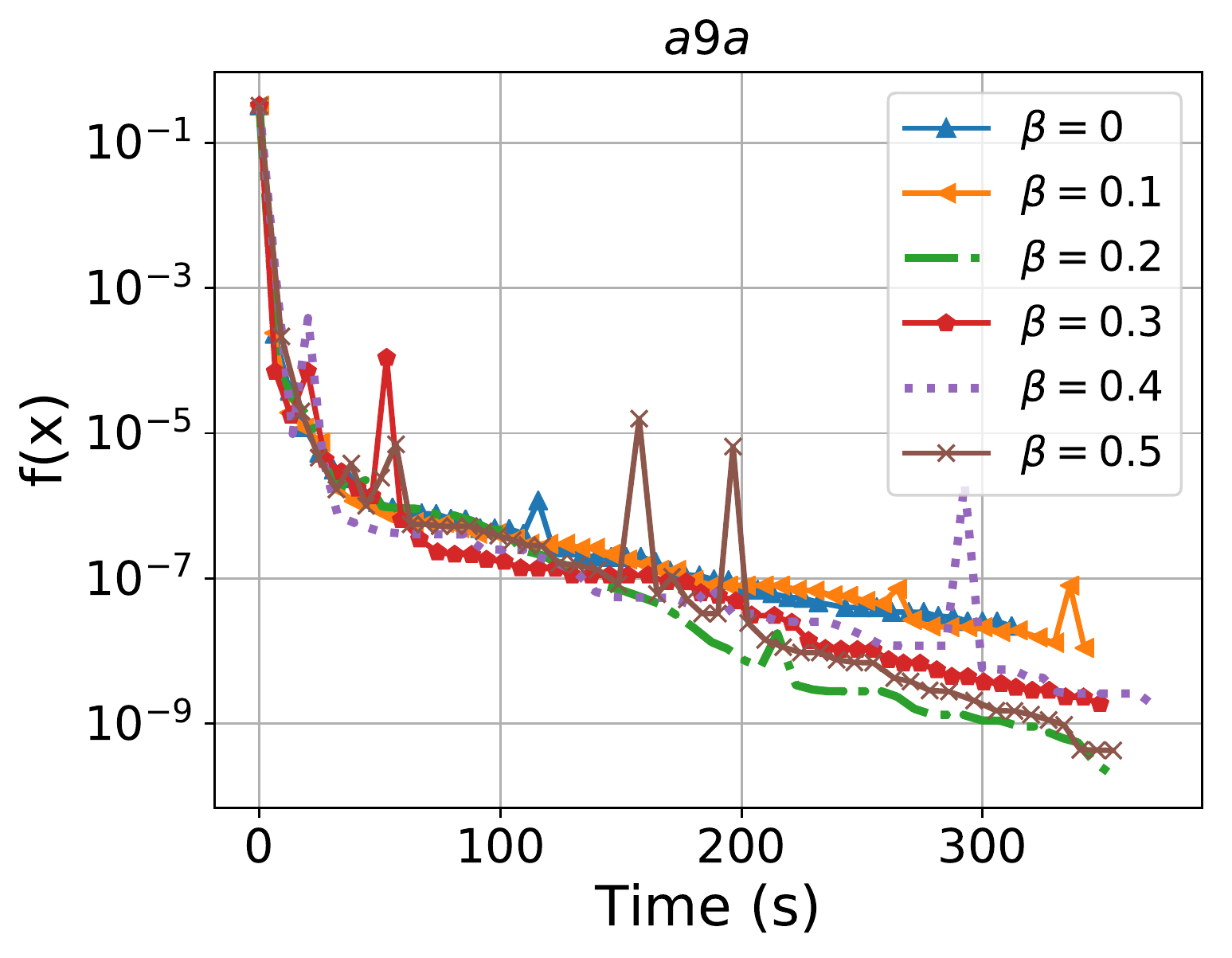}
\end{subfigure}\\
\begin{subfigure}{.23\textwidth}
  \centering
  \includegraphics[width=1\linewidth]{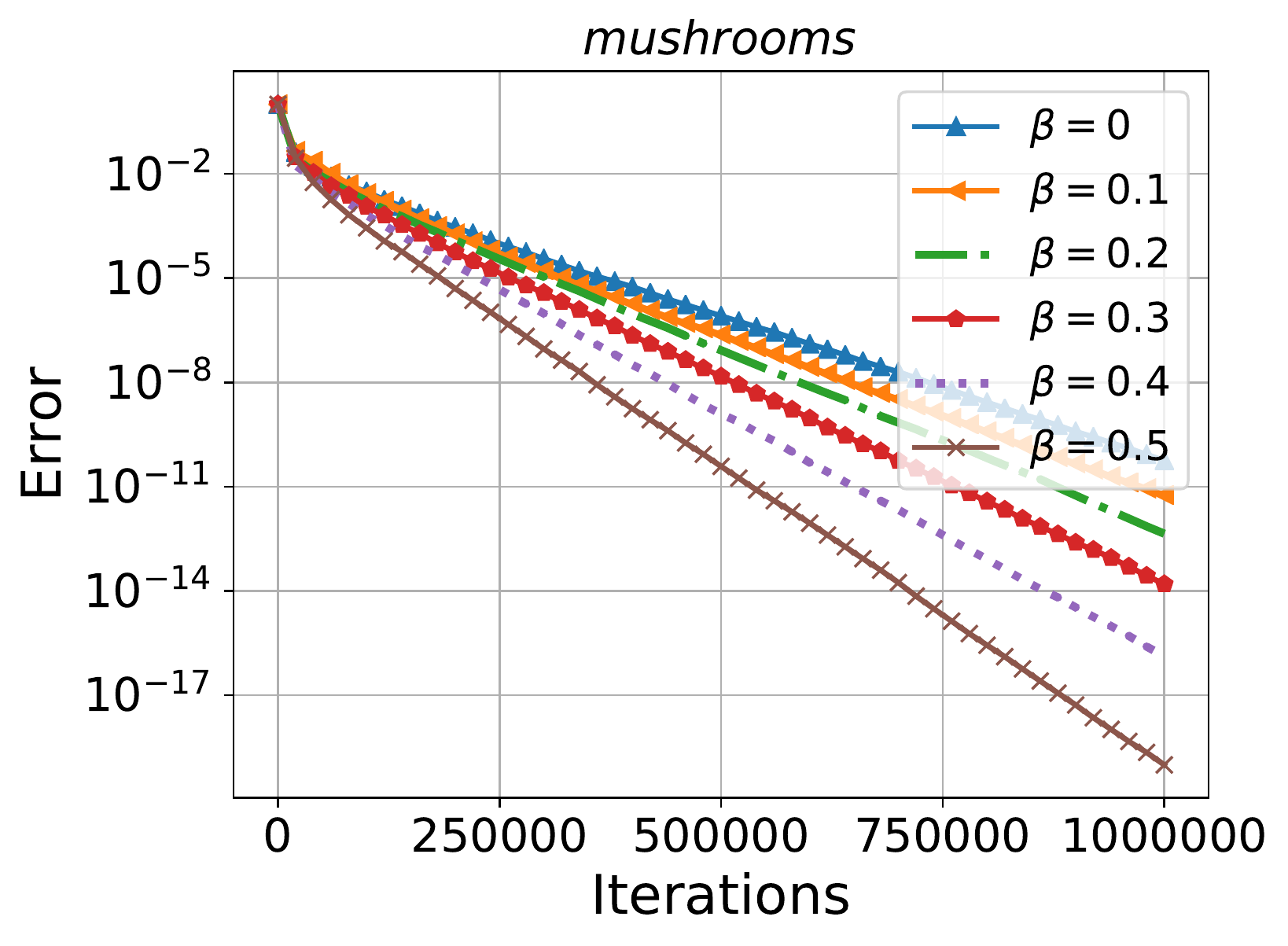}
\end{subfigure}%
\begin{subfigure}{.23\textwidth}
  \centering
  \includegraphics[width=1\linewidth]{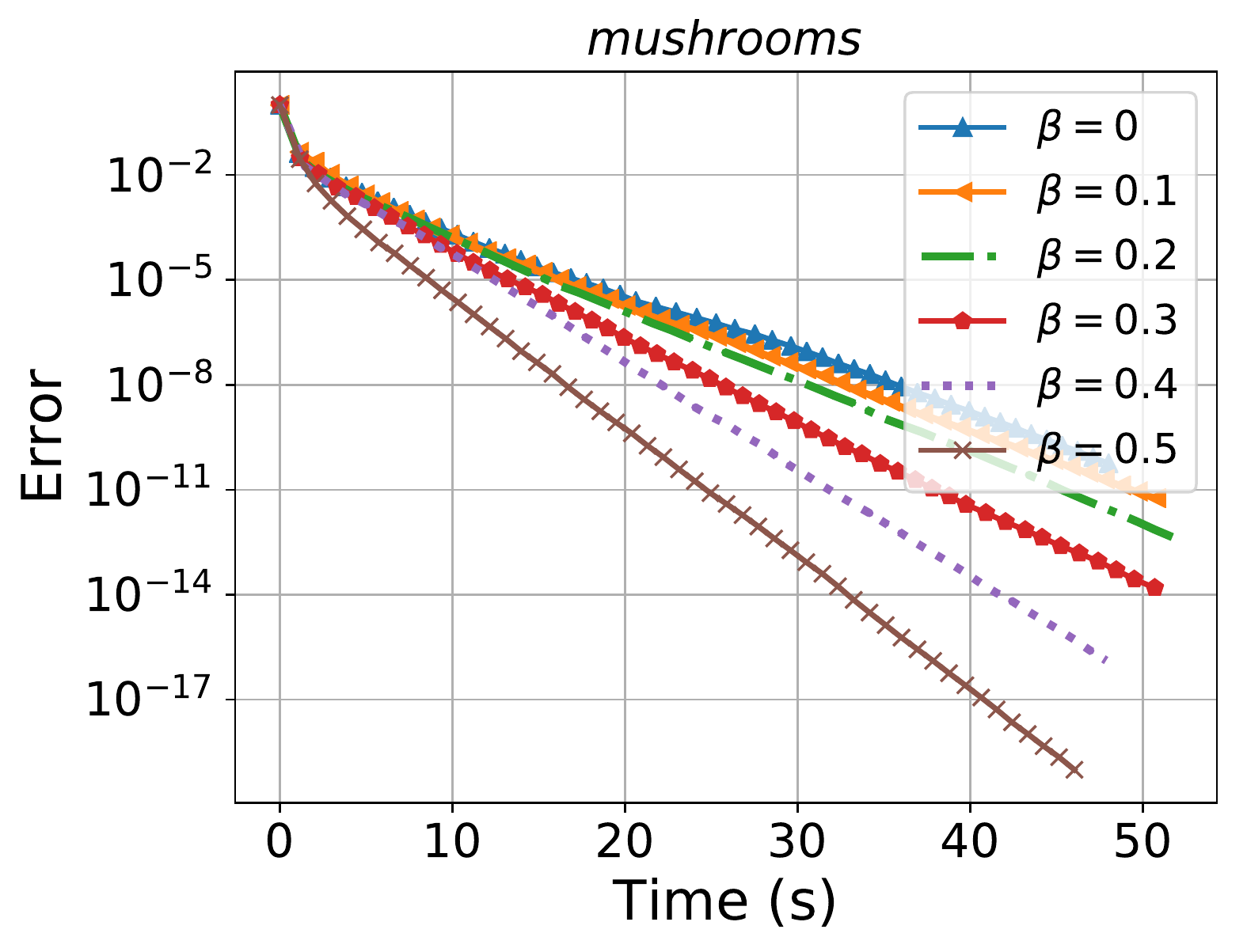}
\end{subfigure}
\begin{subfigure}{.23\textwidth}
  \centering
  \includegraphics[width=1\linewidth]{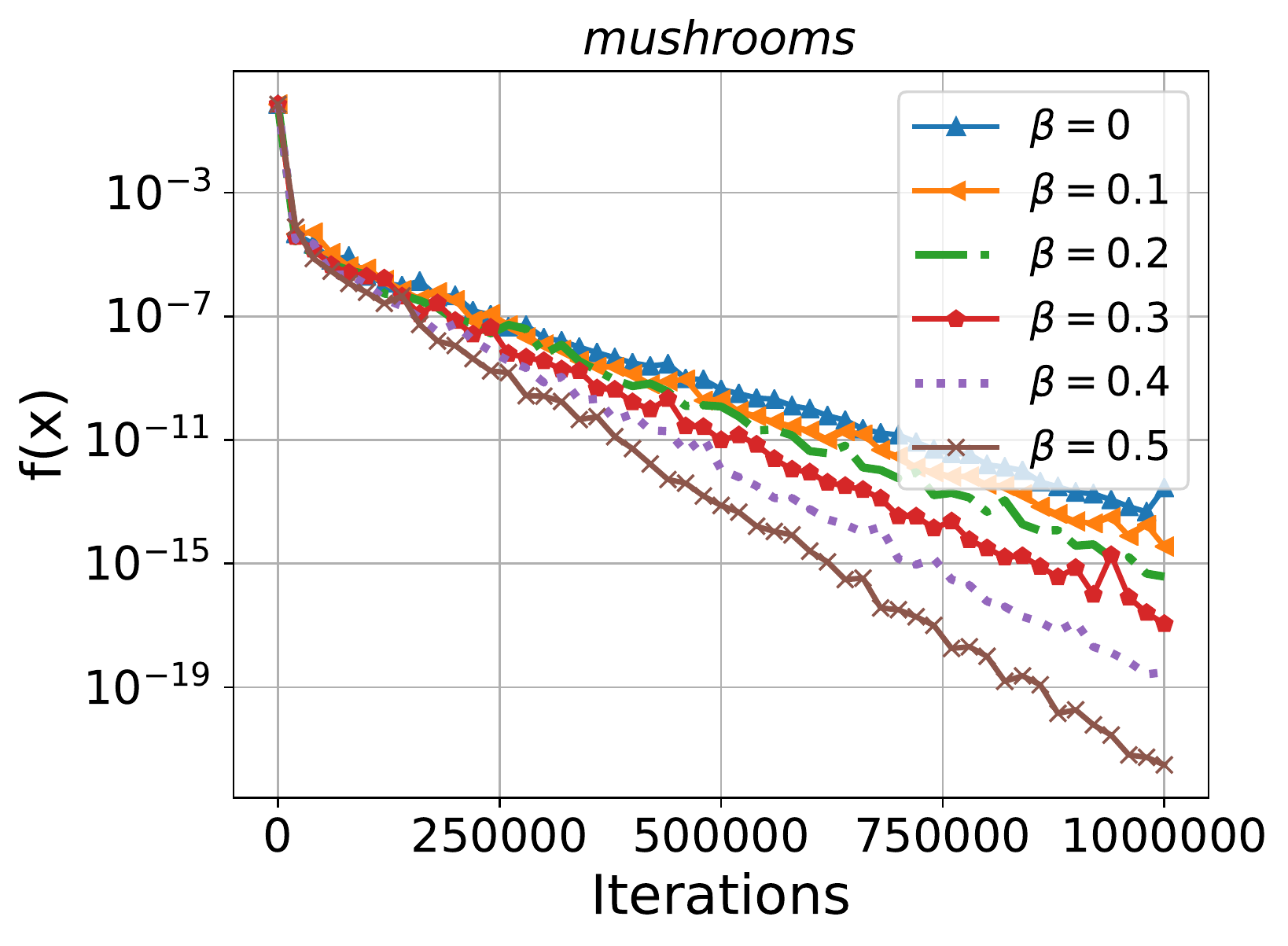}
\end{subfigure}
\begin{subfigure}{.23\textwidth}
  \centering
  \includegraphics[width=1\linewidth]{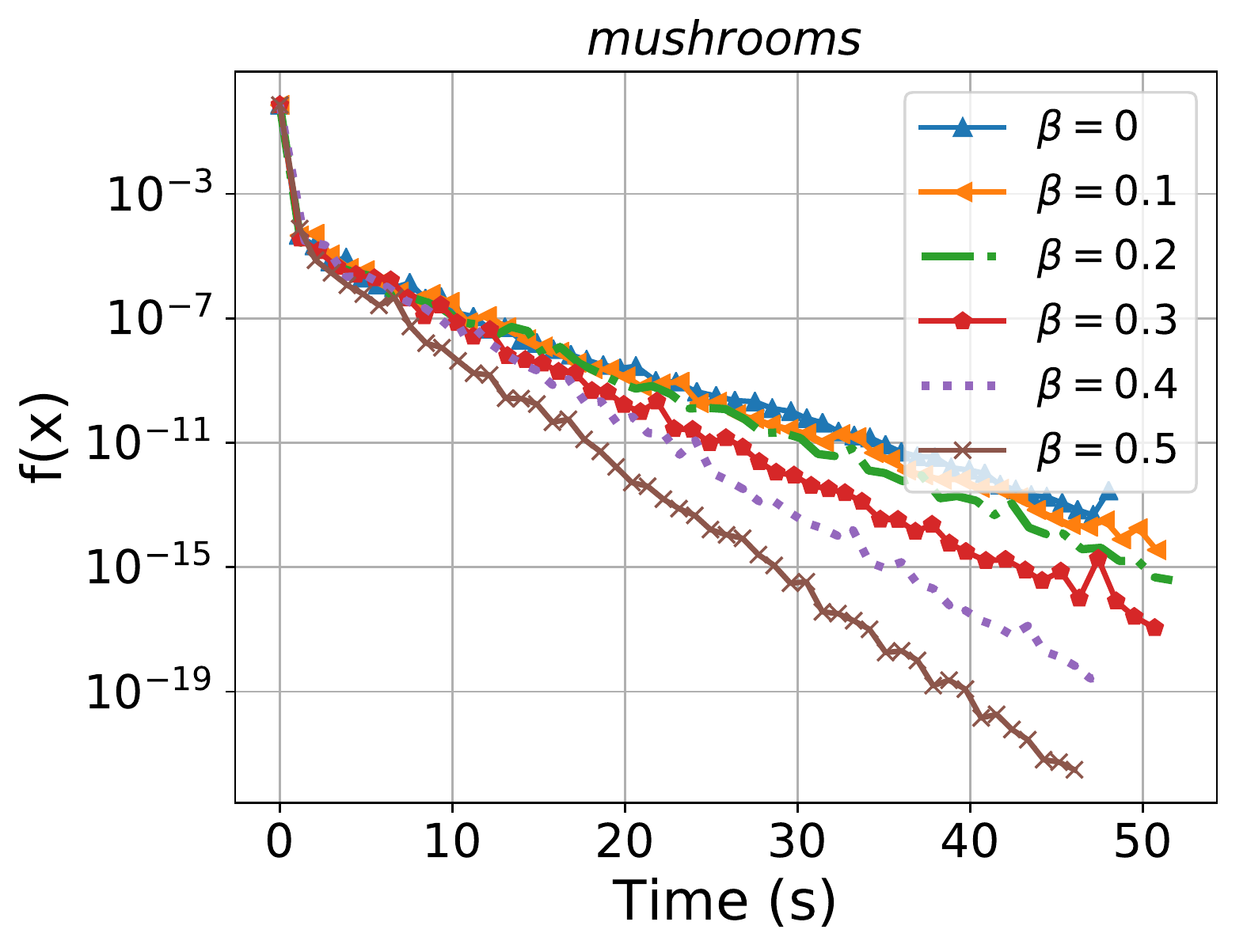}
\end{subfigure}\\
\begin{subfigure}{.23\textwidth}
  \centering
  \includegraphics[width=1\linewidth]{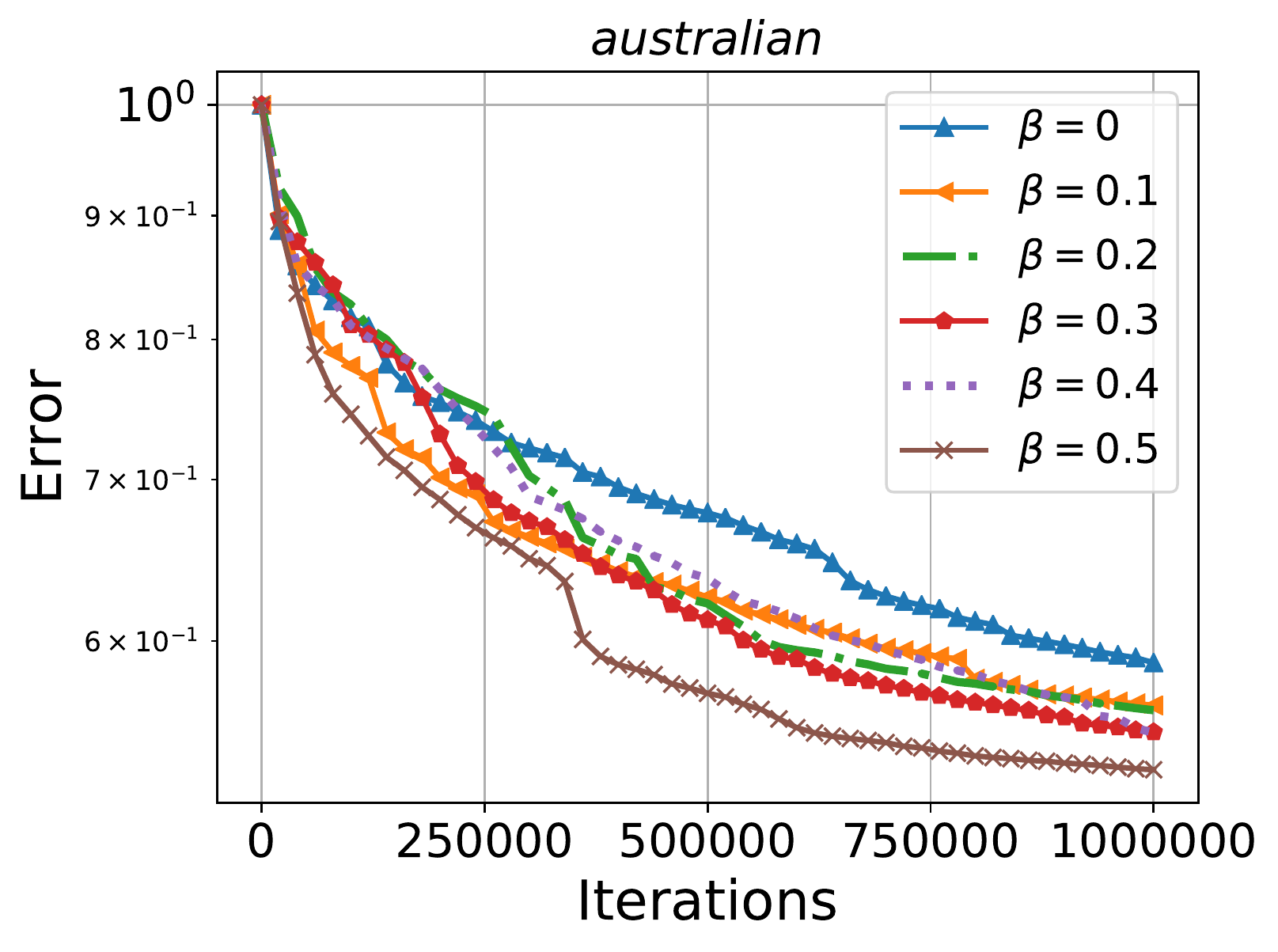}
\end{subfigure}%
\begin{subfigure}{.23\textwidth}
  \centering
  \includegraphics[width=1\linewidth]{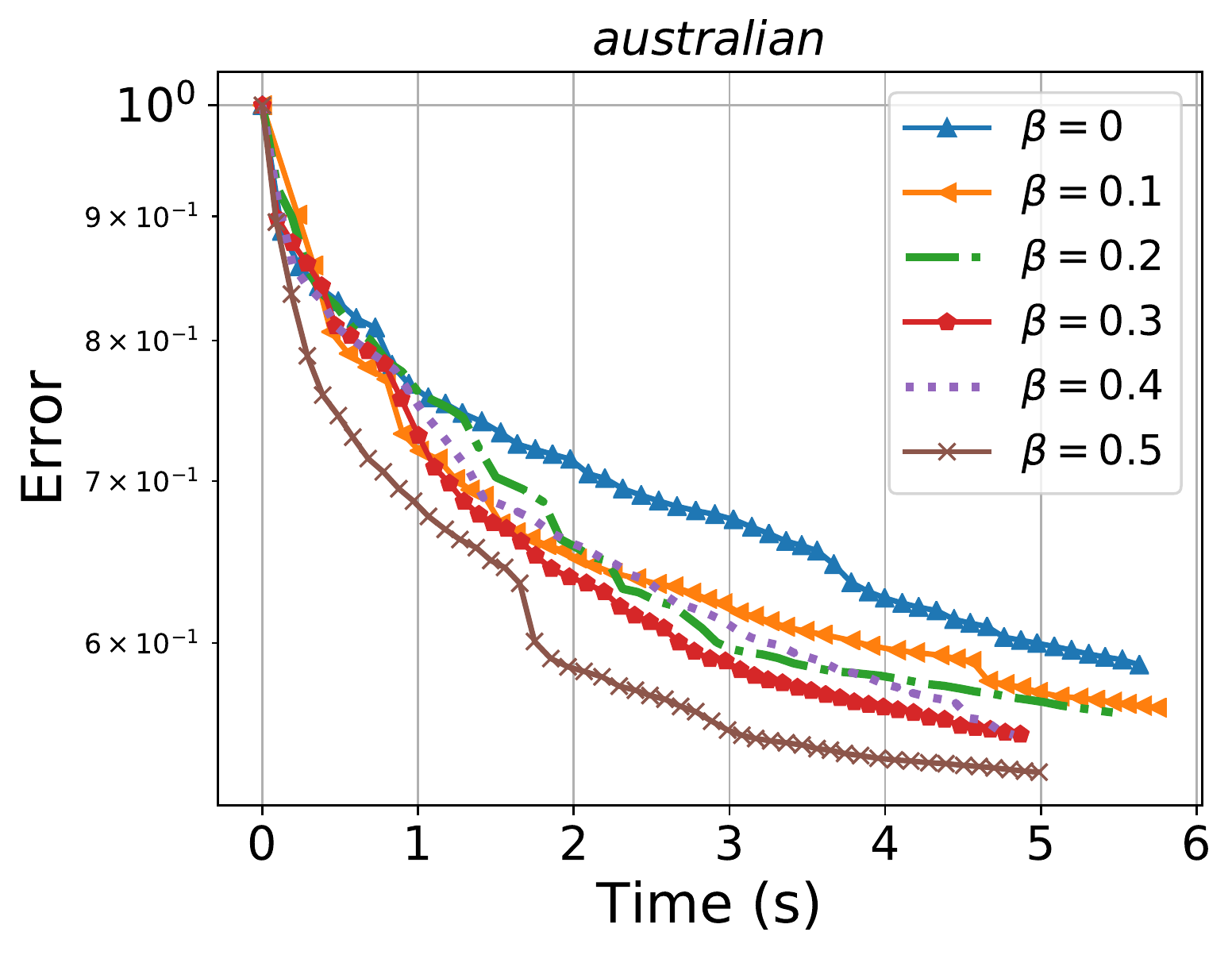}
\end{subfigure}
\begin{subfigure}{.23\textwidth}
  \centering
  \includegraphics[width=1\linewidth]{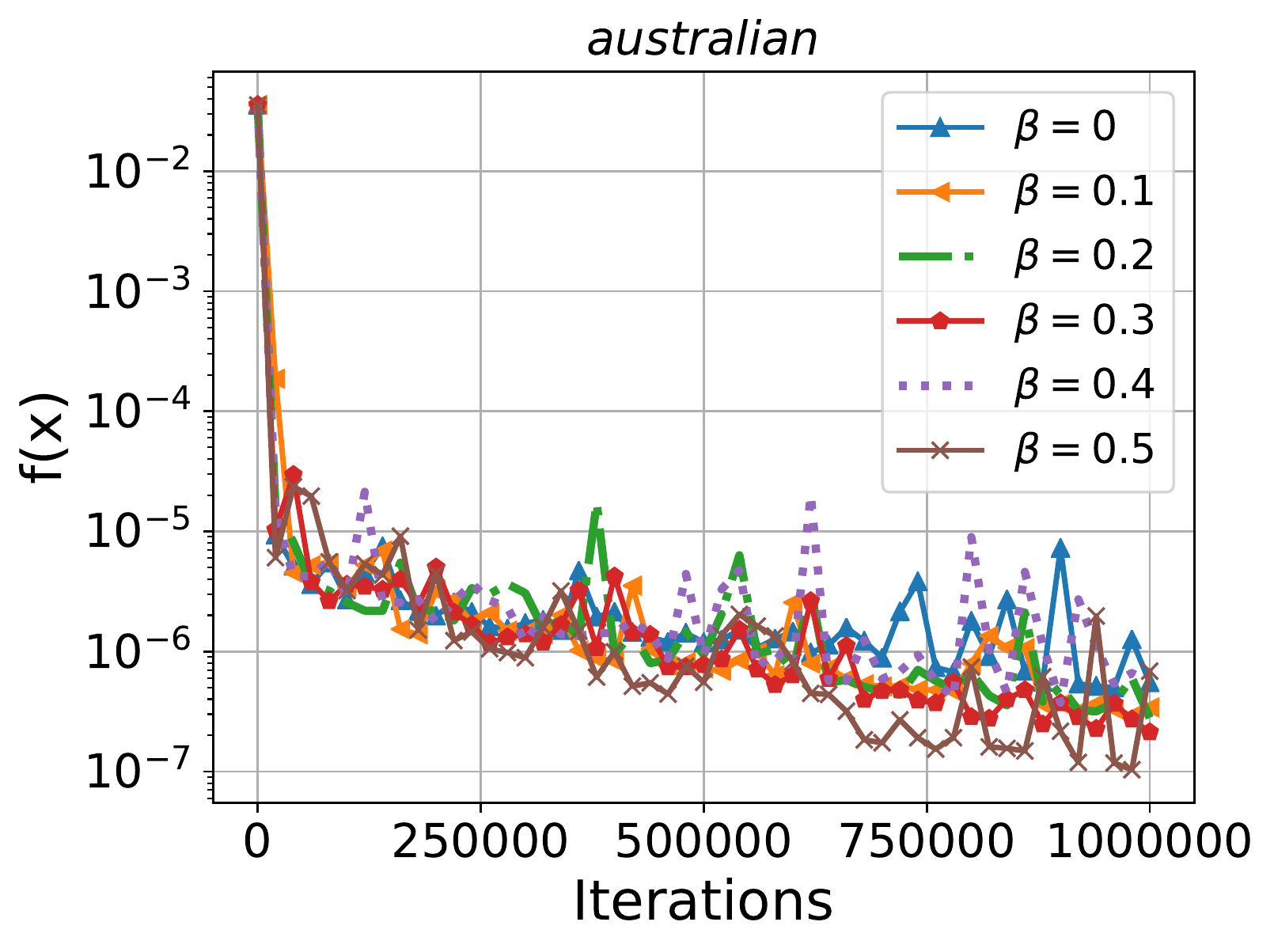}
\end{subfigure}
\begin{subfigure}{.23\textwidth}
  \centering
  \includegraphics[width=1\linewidth]{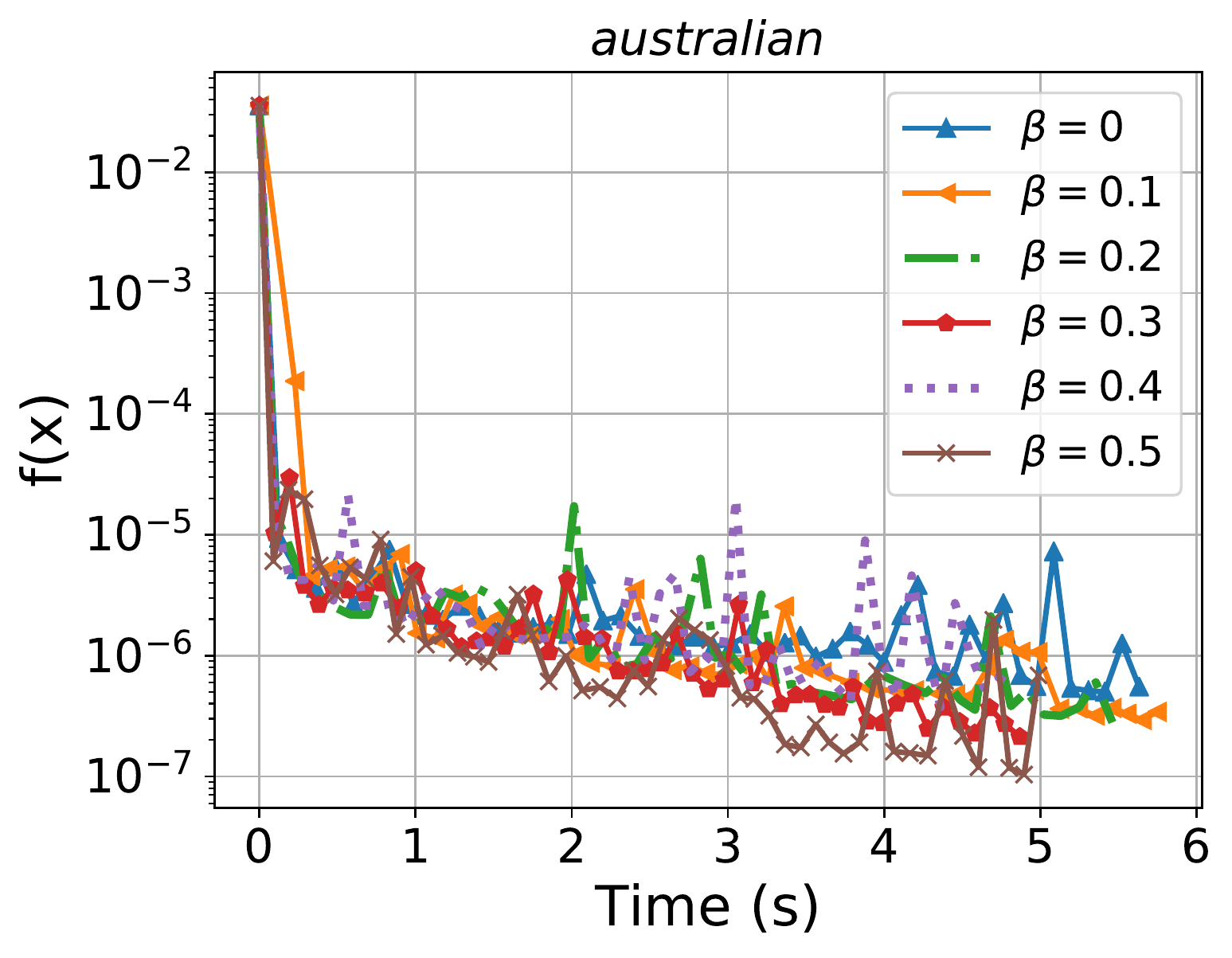}
\end{subfigure}\\
\begin{subfigure}{.23\textwidth}
  \centering
  \includegraphics[width=1\linewidth]{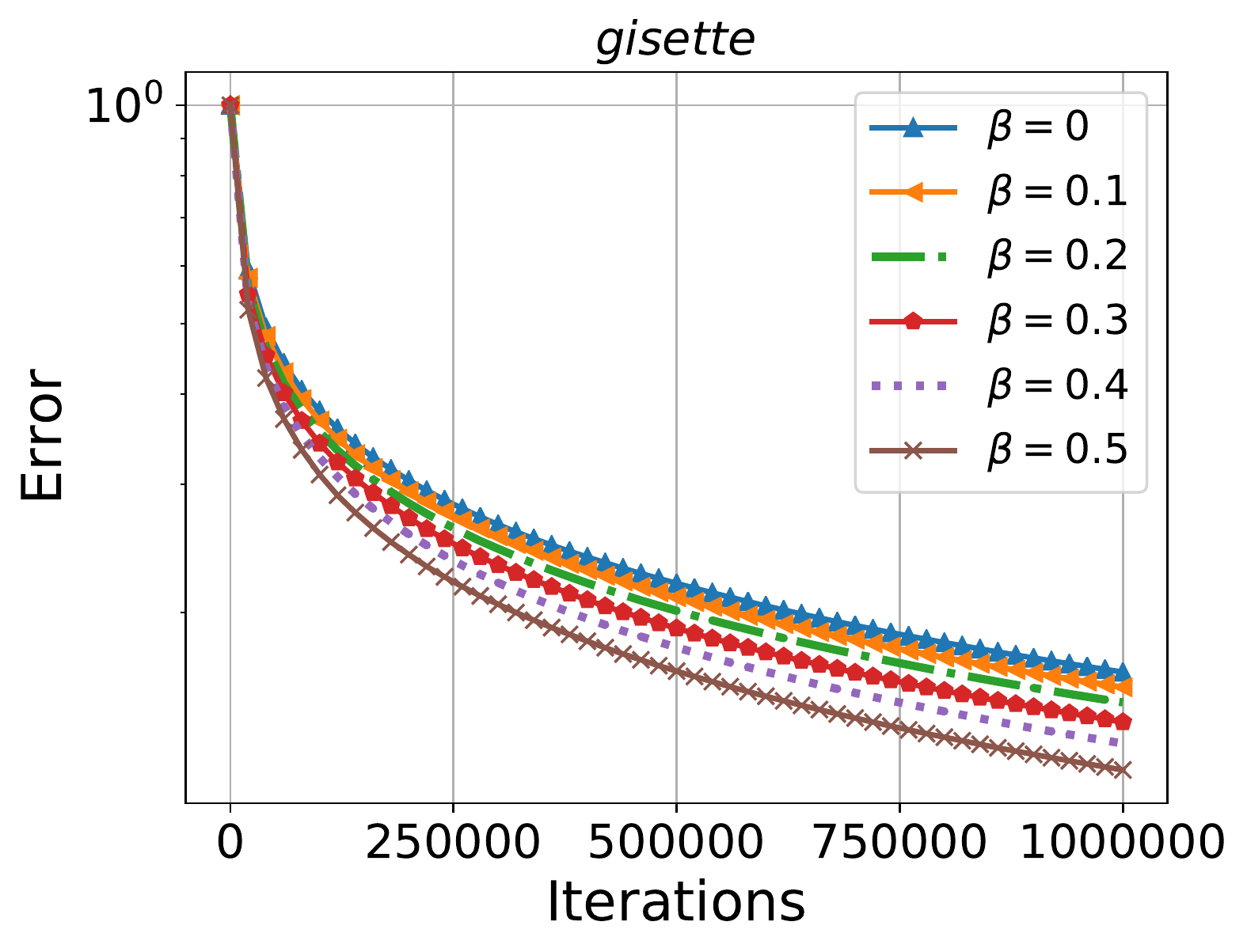}
\end{subfigure}%
\begin{subfigure}{.23\textwidth}
  \centering
  \includegraphics[width=1\linewidth]{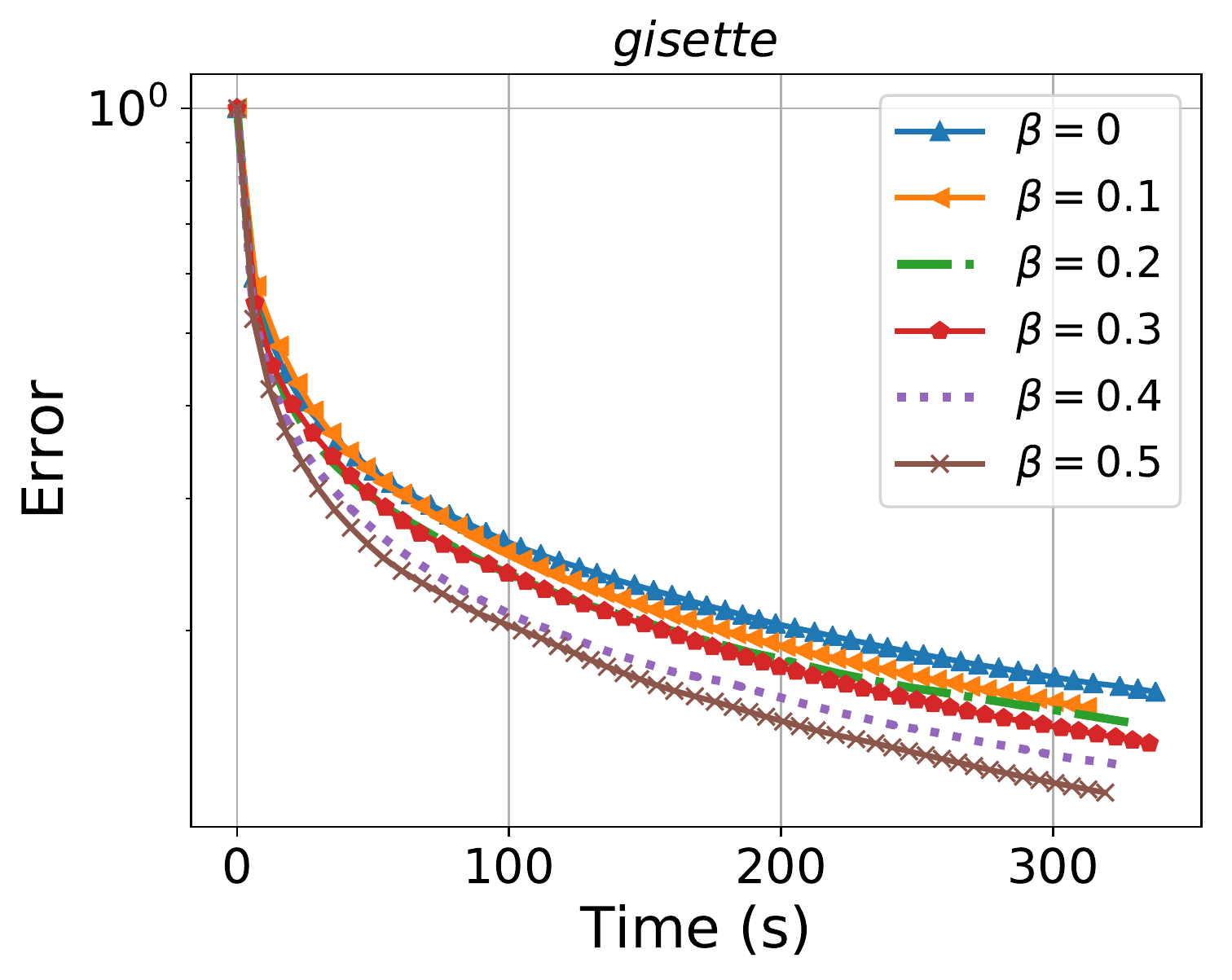}
\end{subfigure}
\begin{subfigure}{.23\textwidth}
  \centering
  \includegraphics[width=1\linewidth]{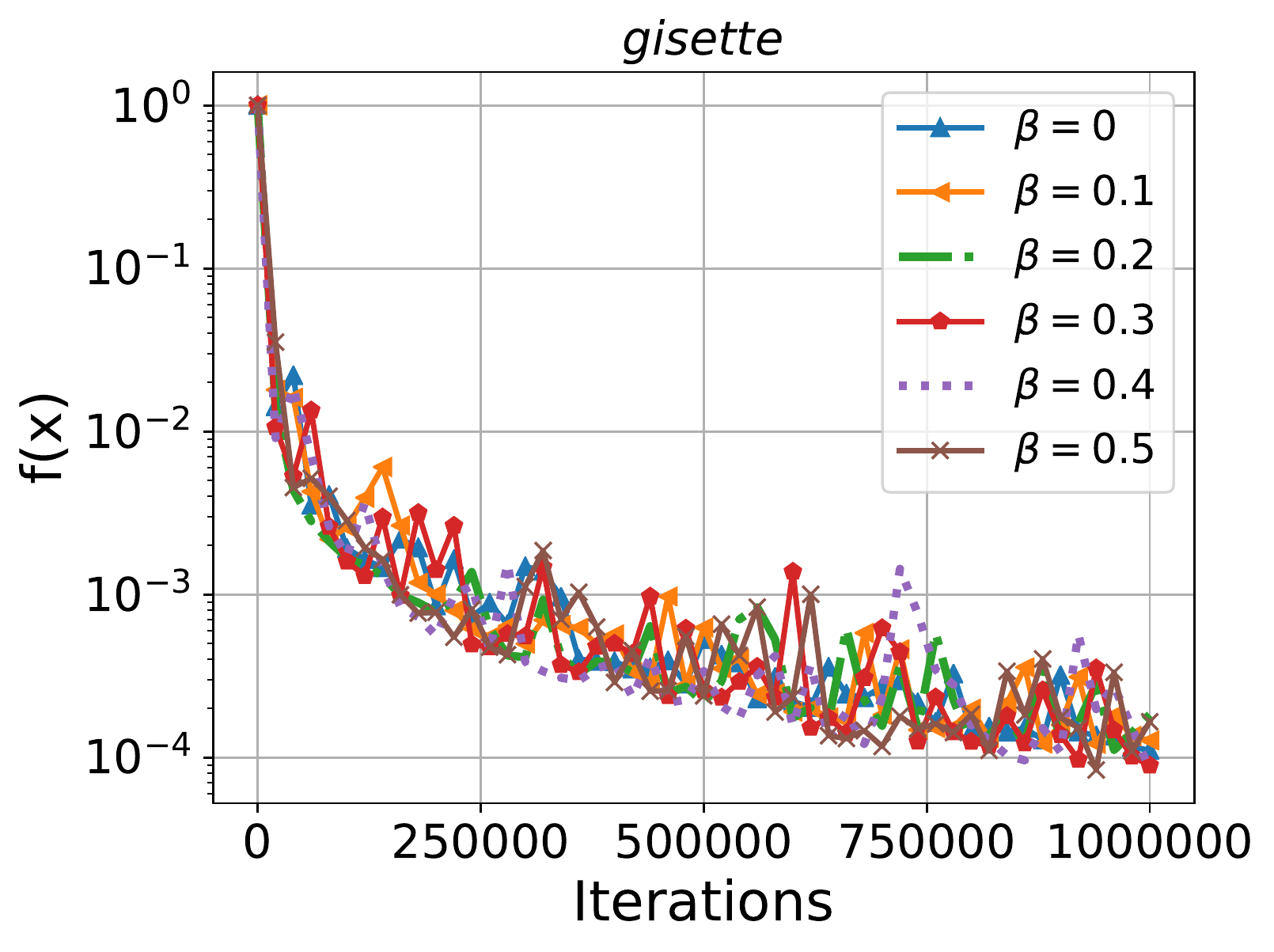}
\end{subfigure}
\begin{subfigure}{.23\textwidth}
  \centering
  \includegraphics[width=1\linewidth]{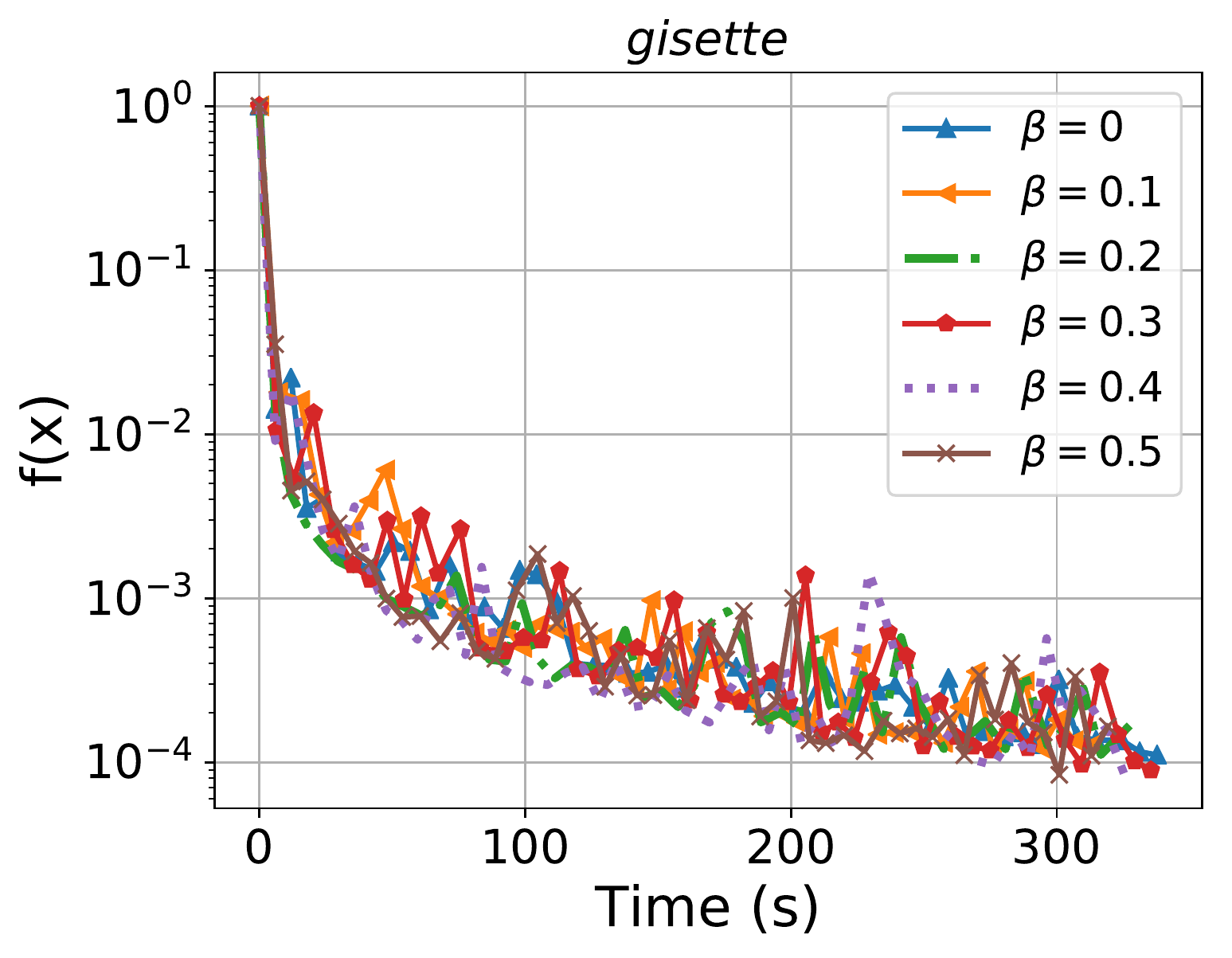}
\end{subfigure}\\
\begin{subfigure}{.23\textwidth}
  \centering
  \includegraphics[width=1\linewidth]{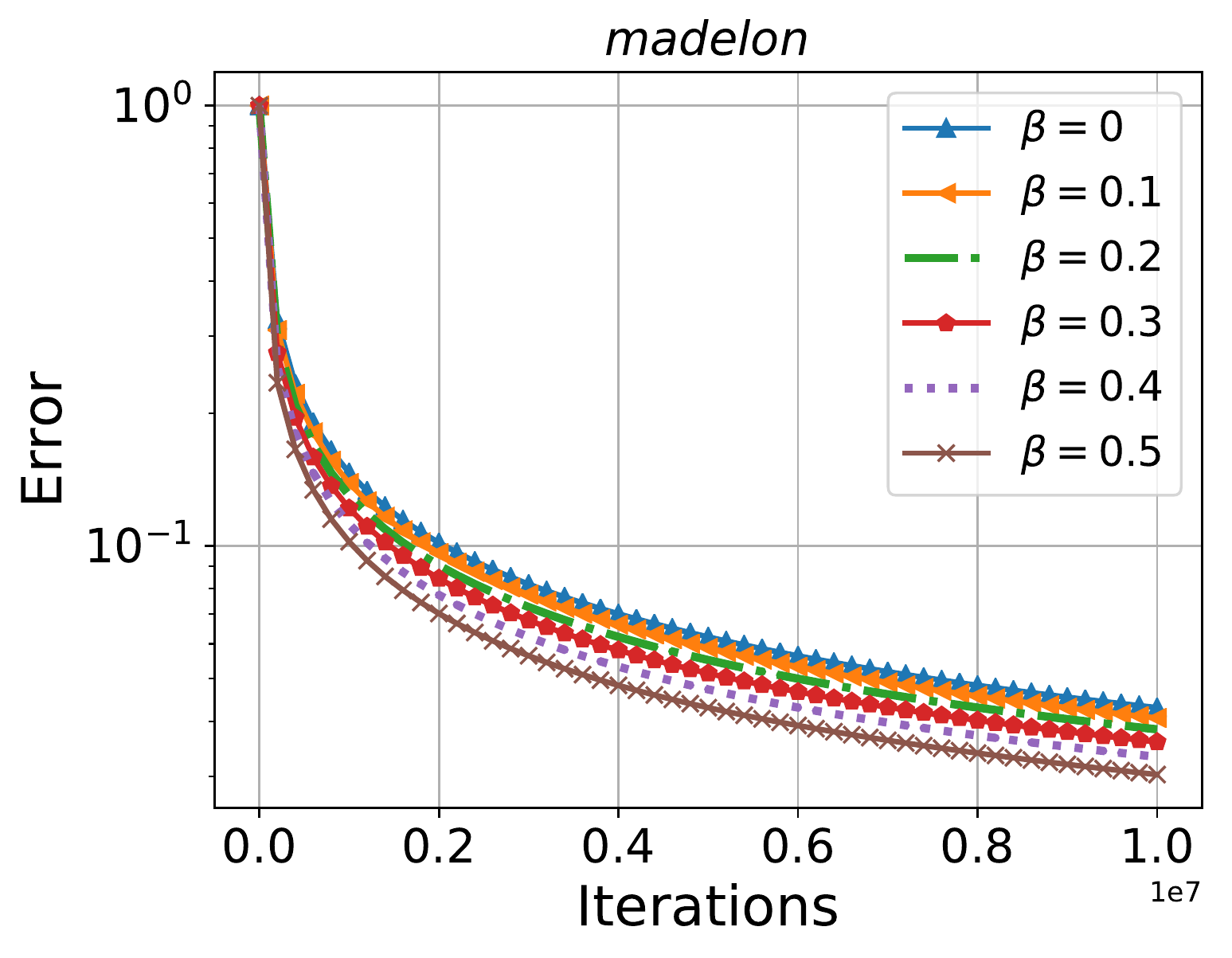}
\end{subfigure}%
\begin{subfigure}{.23\textwidth}
  \centering
  \includegraphics[width=1\linewidth]{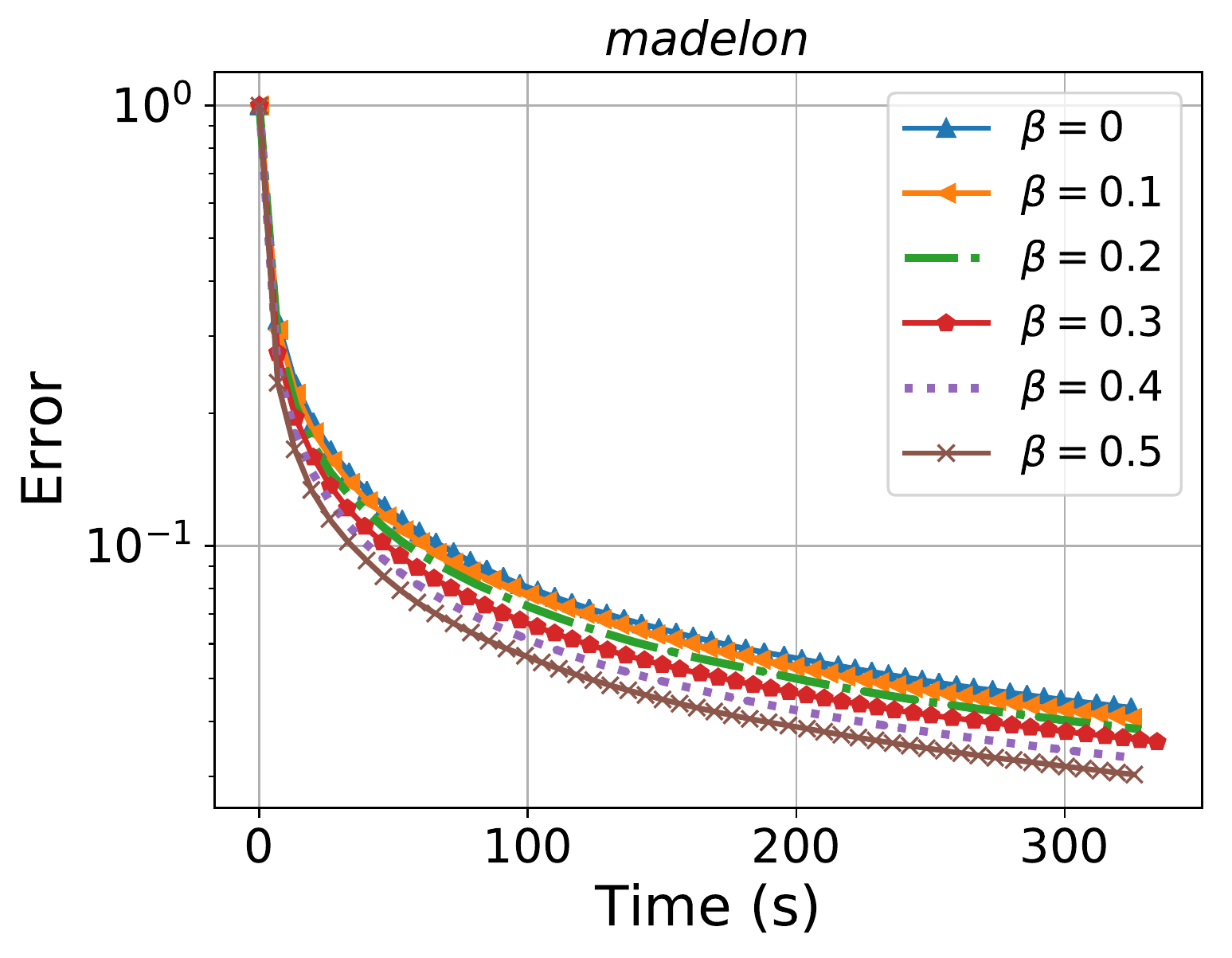}
\end{subfigure}
\begin{subfigure}{.23\textwidth}
  \centering
  \includegraphics[width=1\linewidth]{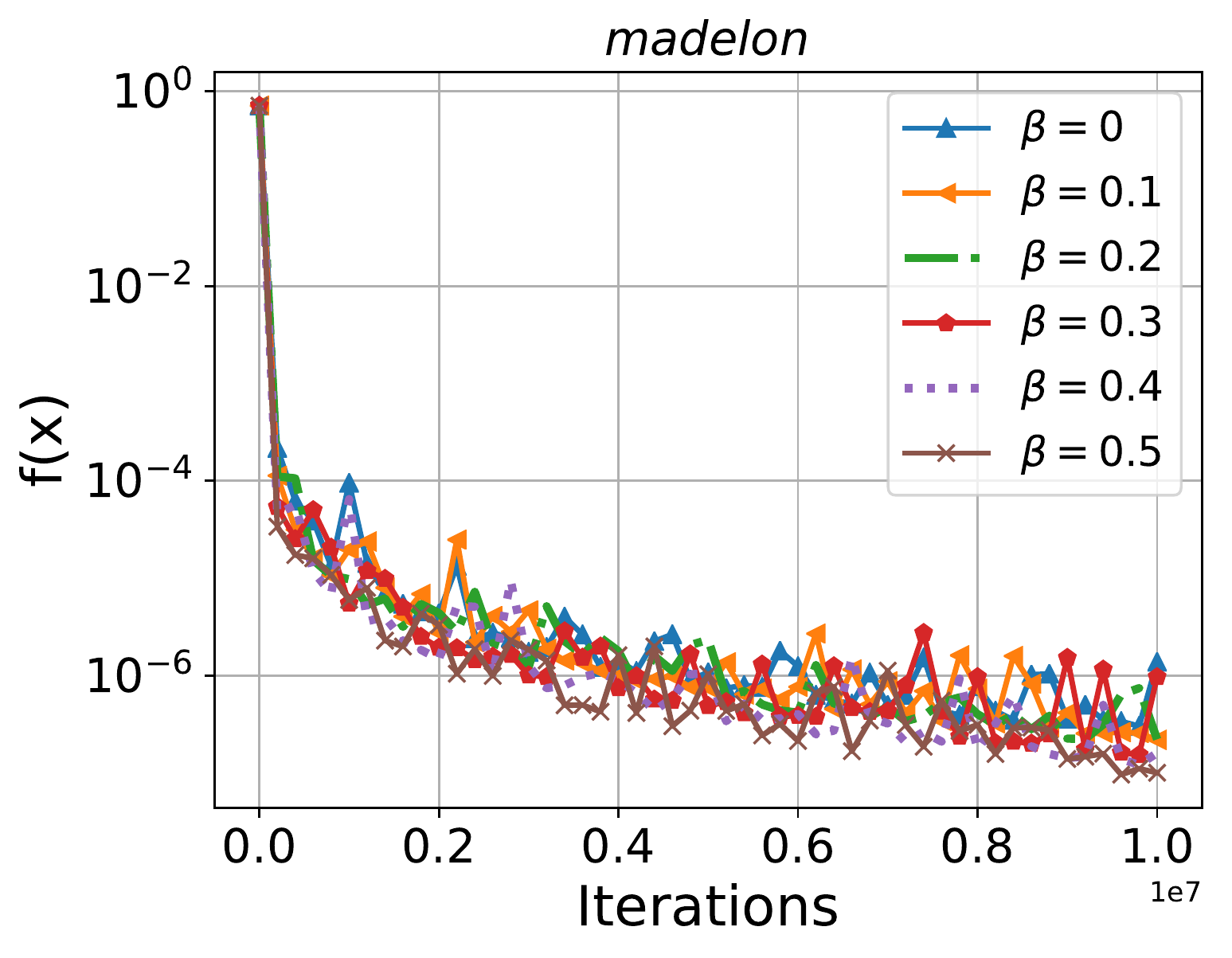}
\end{subfigure}
\begin{subfigure}{.23\textwidth}
  \centering
  \includegraphics[width=1\linewidth]{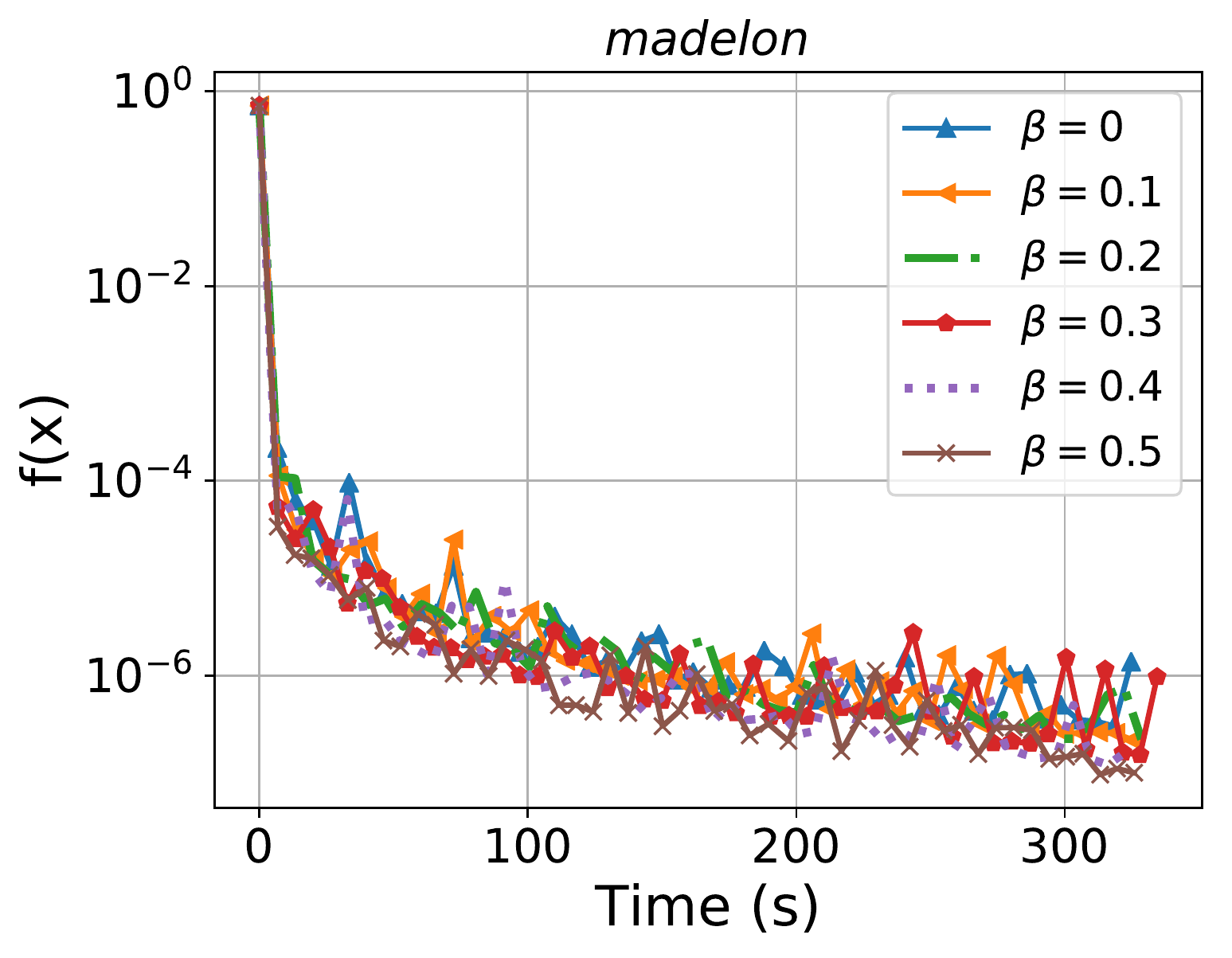}
\end{subfigure}\\
\begin{subfigure}{.23\textwidth}
  \centering
  \includegraphics[width=1\linewidth]{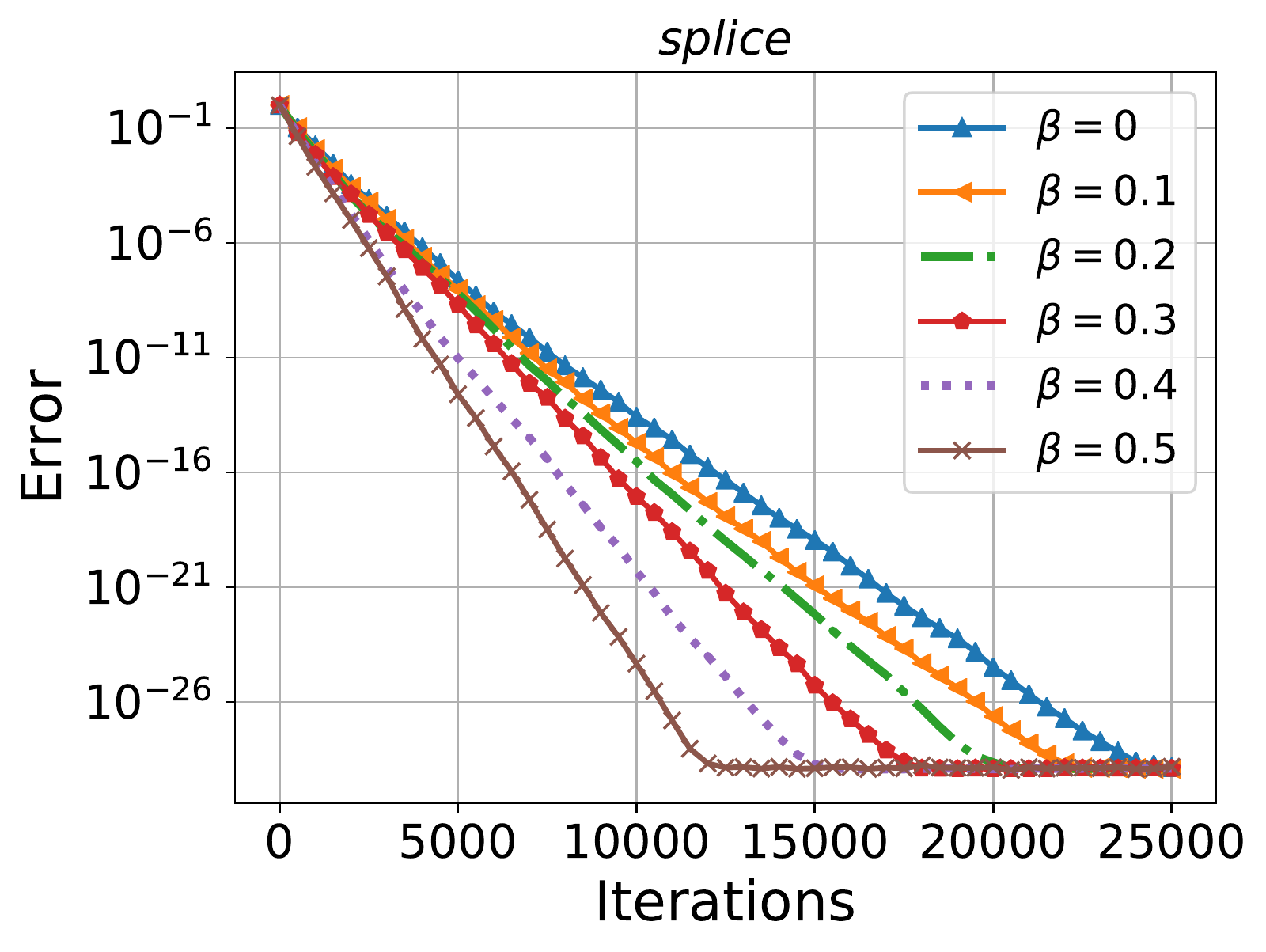}
\end{subfigure}%
\begin{subfigure}{.23\textwidth}
  \centering
  \includegraphics[width=1\linewidth]{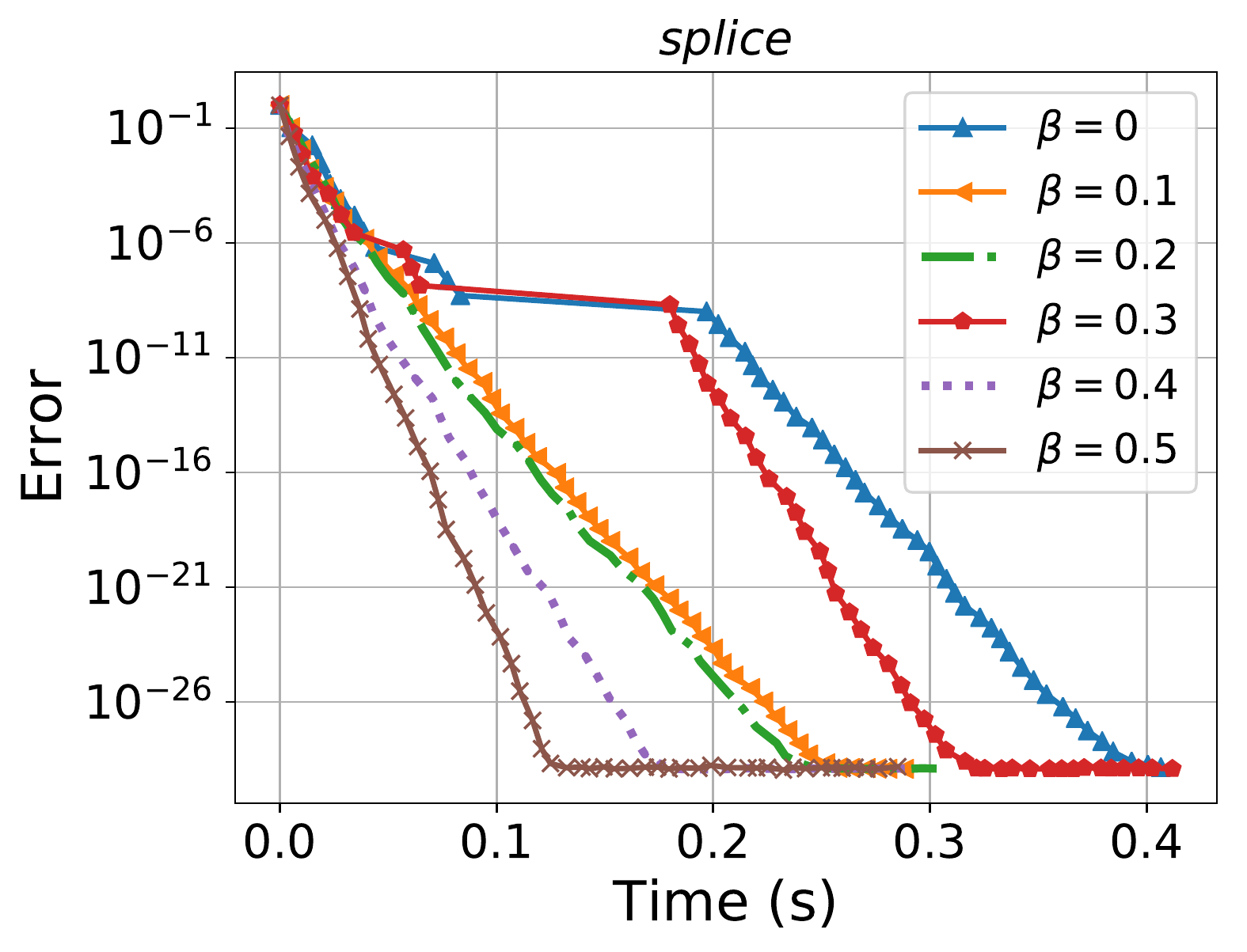}
\end{subfigure}
\begin{subfigure}{.23\textwidth}
  \centering
  \includegraphics[width=1\linewidth]{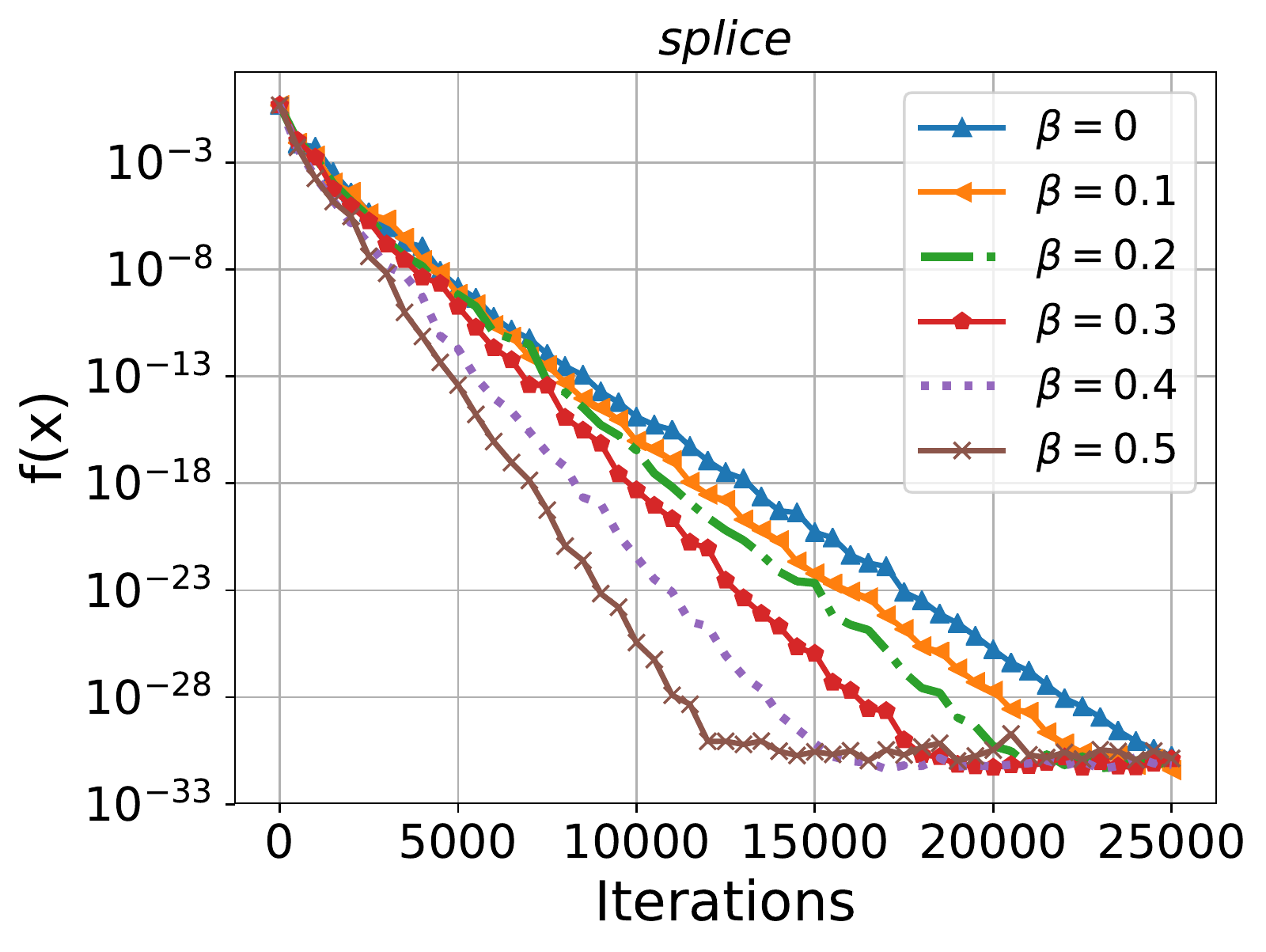}
\end{subfigure}
\begin{subfigure}{.23\textwidth}
  \centering
  \includegraphics[width=1\linewidth]{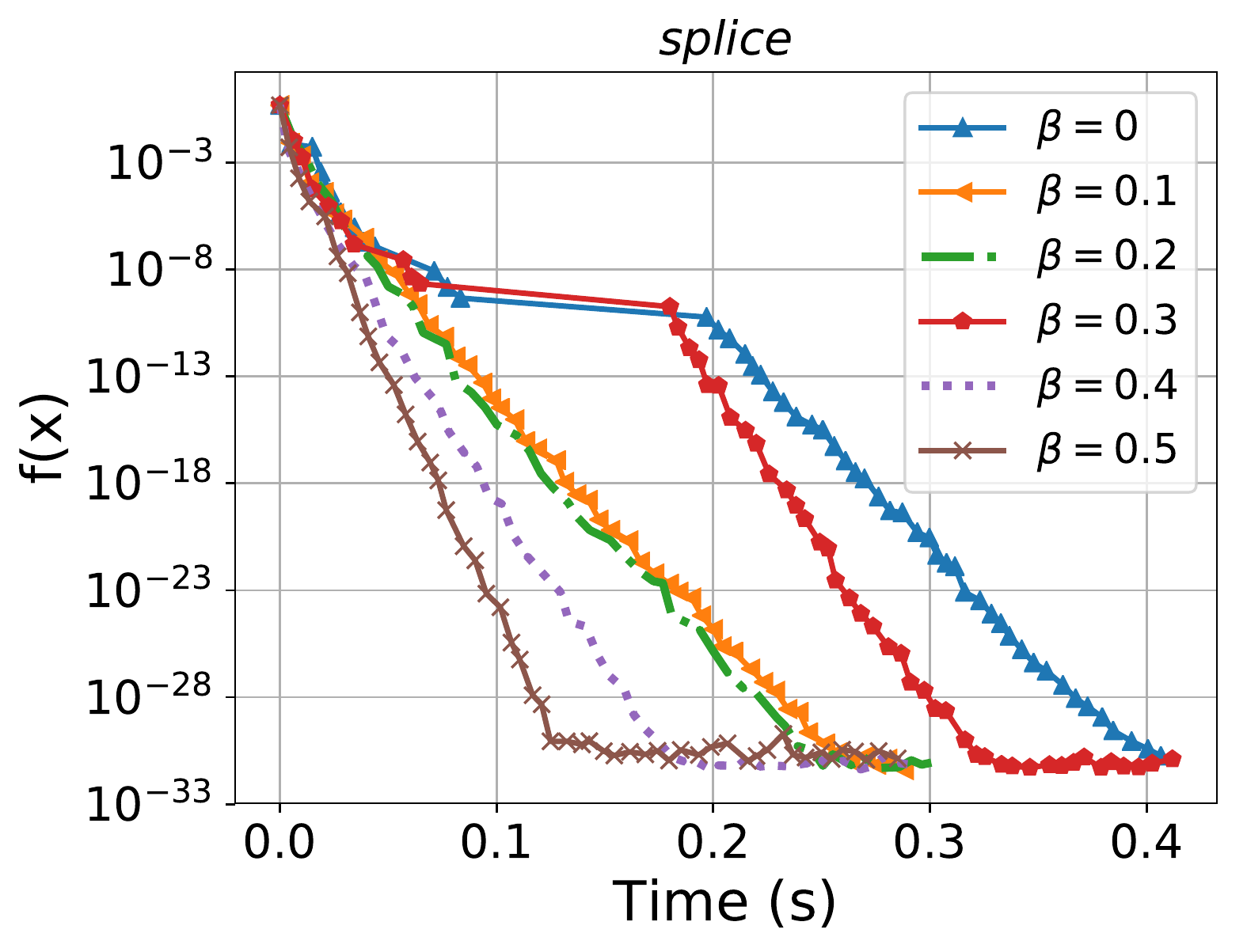}
\end{subfigure}\\
\begin{subfigure}{.23\textwidth}
  \centering
  \includegraphics[width=1\linewidth]{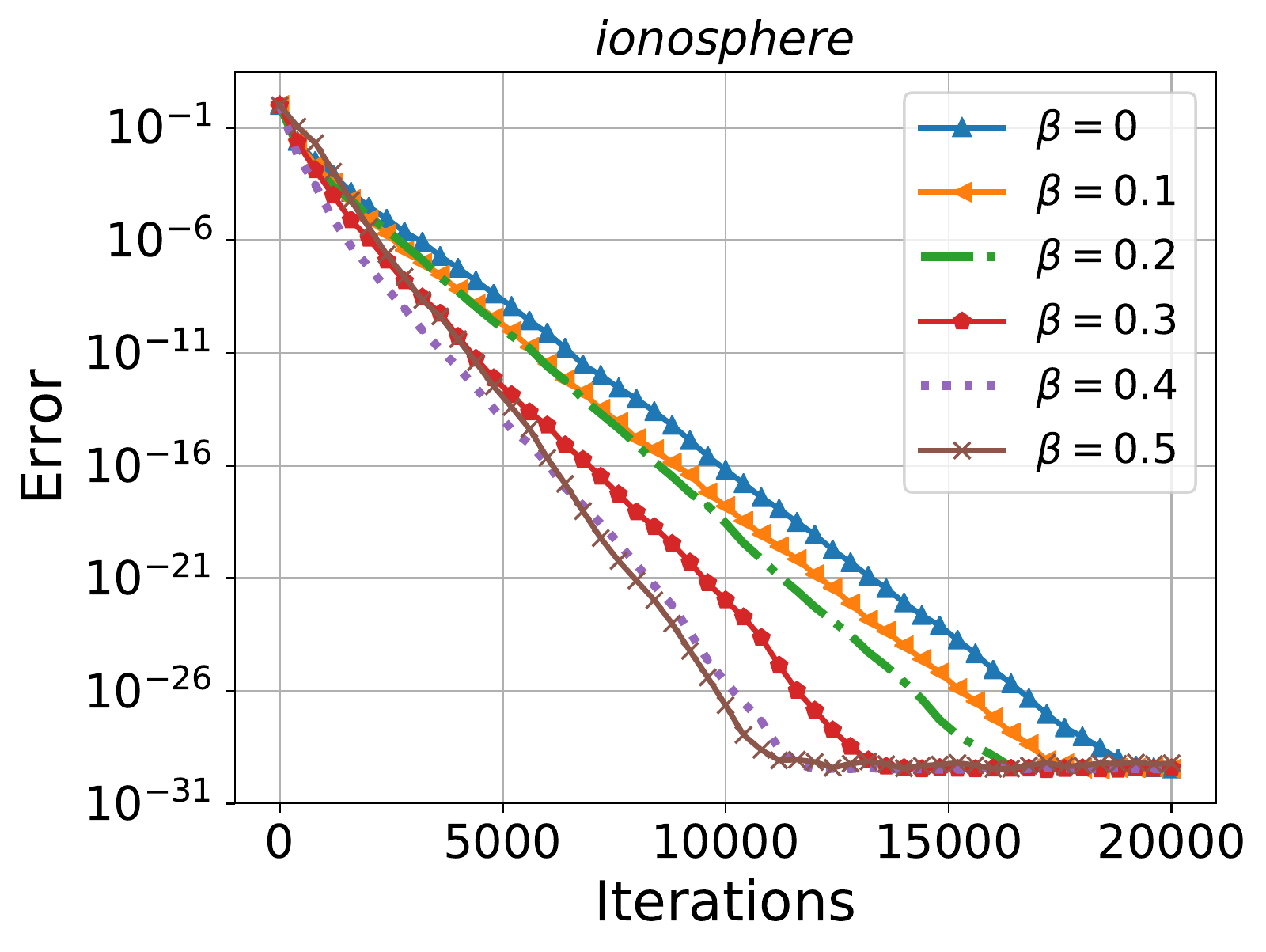}
\end{subfigure}%
\begin{subfigure}{.23\textwidth}
  \centering
  \includegraphics[width=1\linewidth]{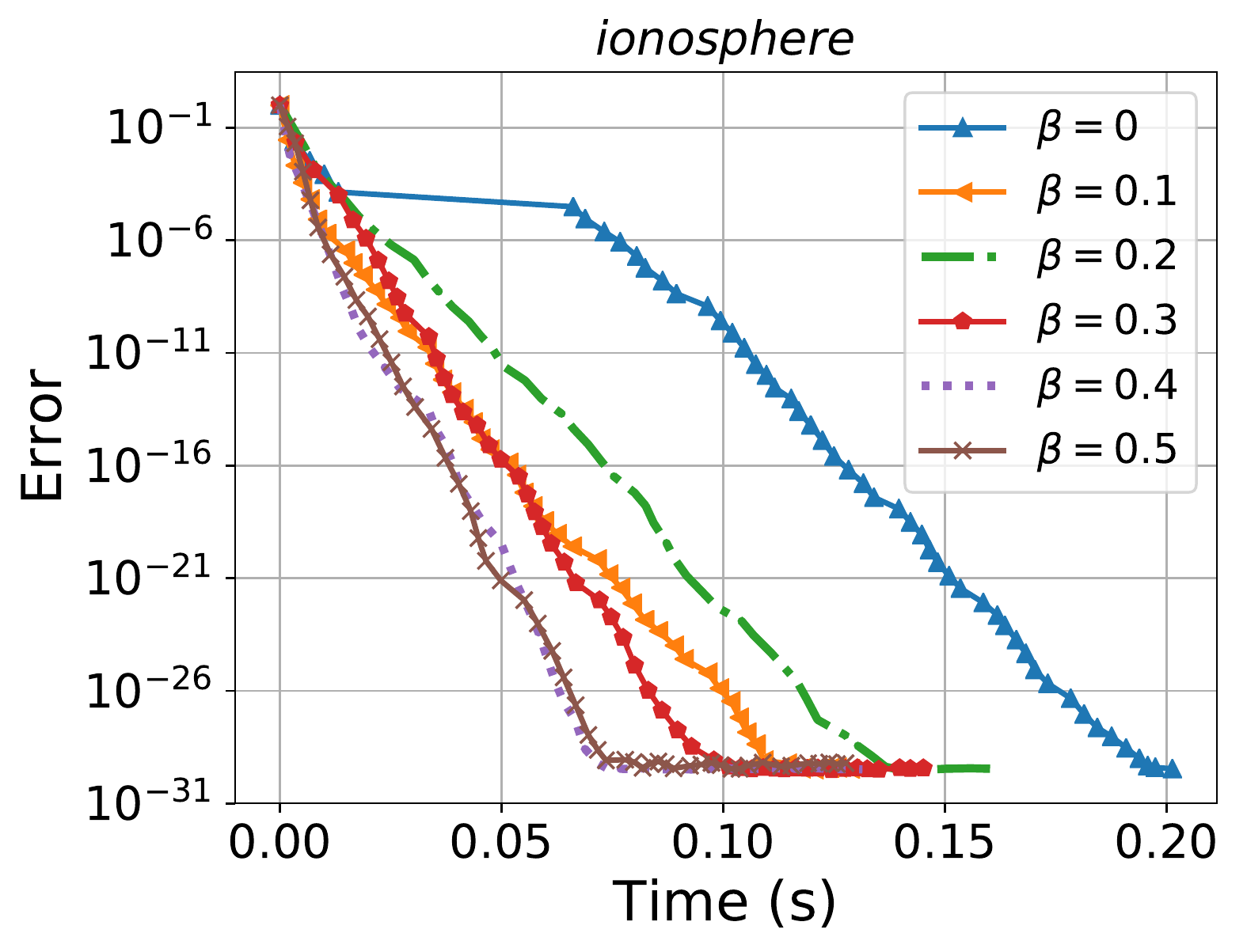}
\end{subfigure}
\begin{subfigure}{.23\textwidth}
  \centering
  \includegraphics[width=1\linewidth]{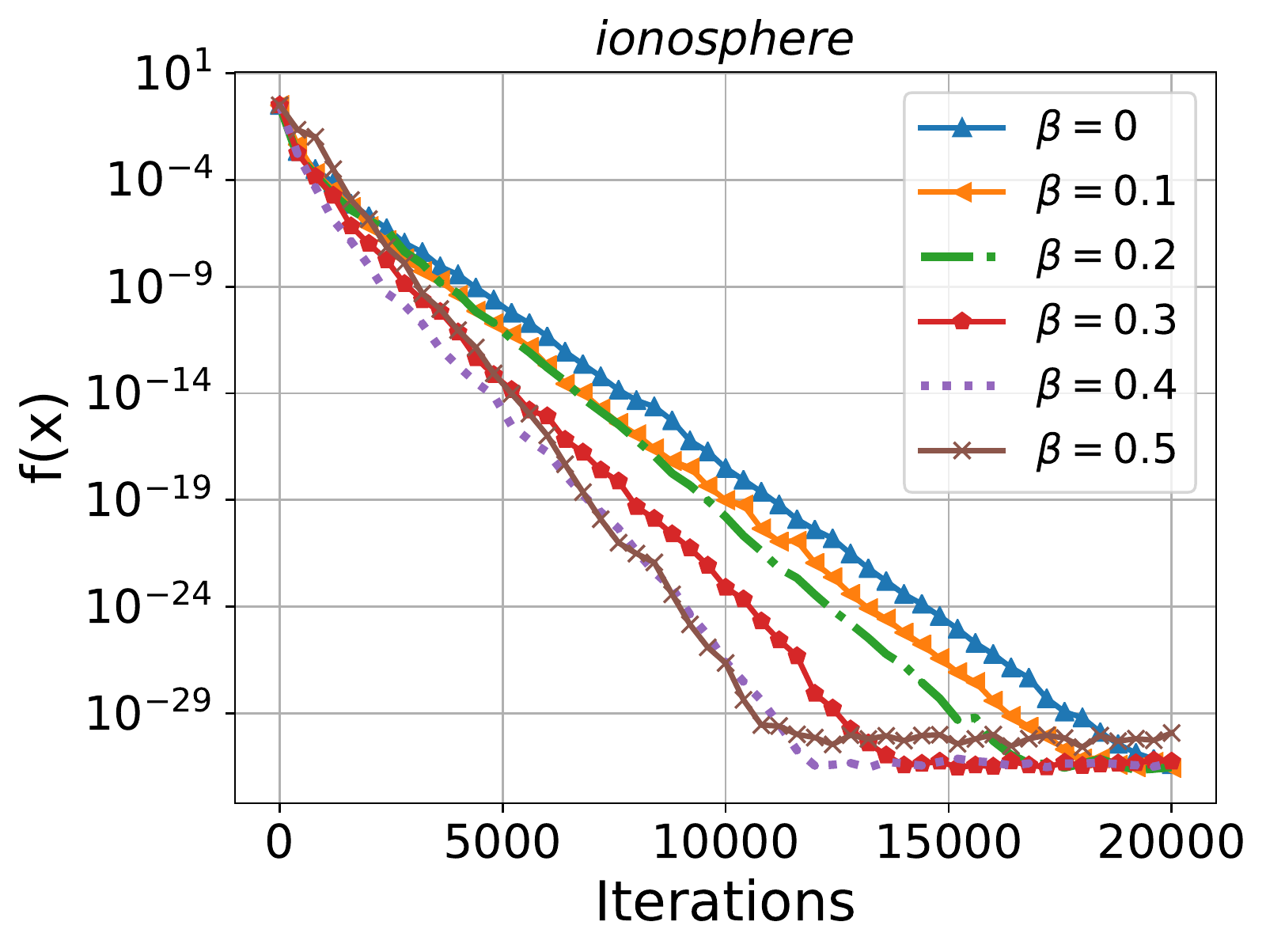}
\end{subfigure}
\begin{subfigure}{.23\textwidth}
  \centering
  \includegraphics[width=1\linewidth]{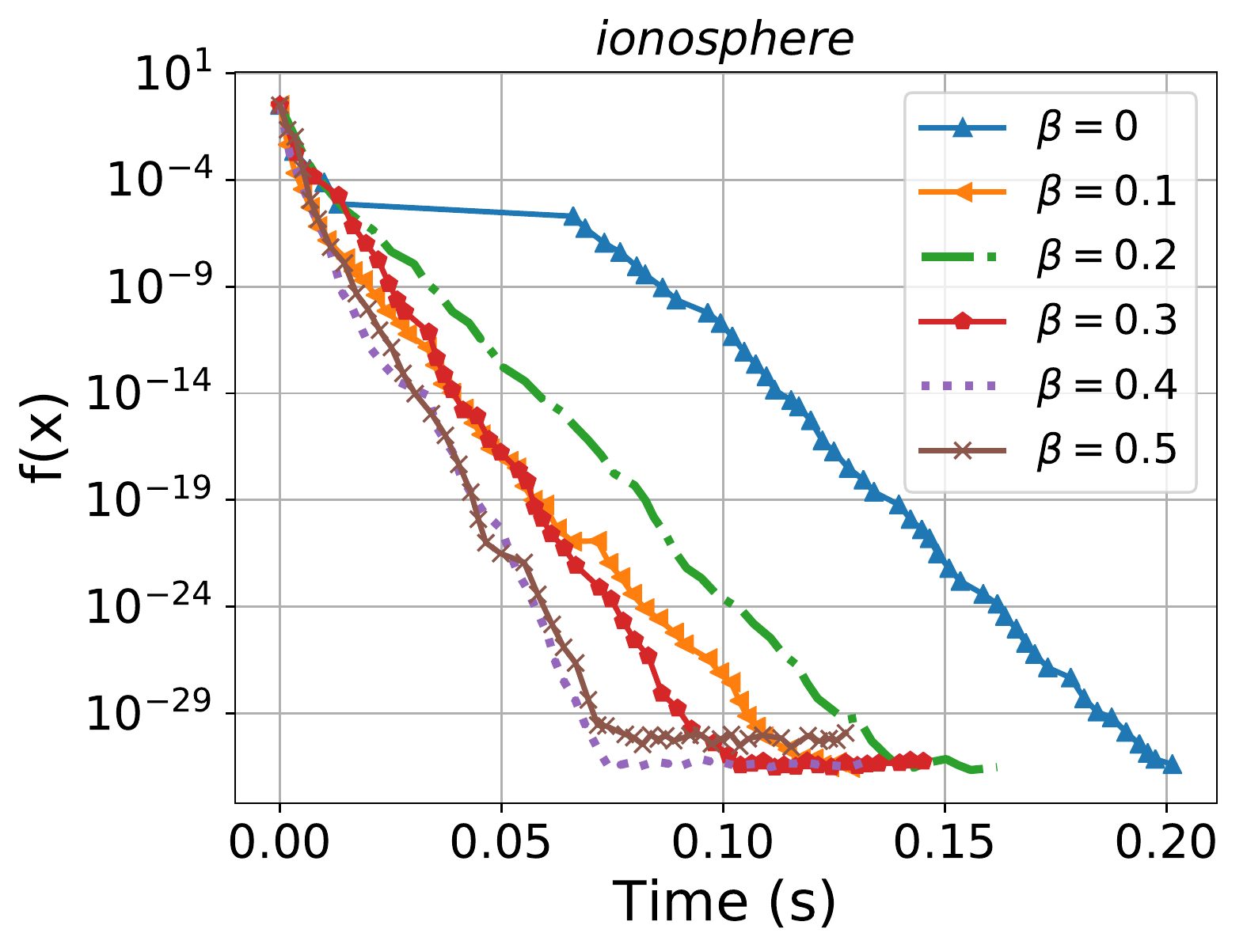}
\end{subfigure}\\
\caption{The performance of mRK for several momentum parameters $\beta$ on real data from LIBSVM \cite{chang2011libsvm}. a9a: $(m,n)=(32561,123)$, mushrooms: $(m,n)=(8124,112)$,  australian: $(m,n)=(690,14)$, gisette: $(m,n)=(6000,5000)$,  madelon: $(m,n)=(2000,500)$, splice: $(m,n)=(1000,60)$, ionosphere: $(m,n)=(351,34)$. The graphs in the first (second) column plot iterations (time) against residual error while those in the third (forth) column plot iterations (time) against function values. The ``Error" on the vertical axis represents the relative error $\|x_k-x_*\|^2_\bB / \|x_0-x_*\|^2_\bB \overset{\bB=\bI, x_0=0}{=}\|x_k-x_*\|^2 / \|x_*\|^2_\bB$ and the function values $f(x_k)$ refer to function~\eqref{functionRK}.}
\label{RealDataplots}
\end{figure}

\subsection{Comparison of momentum \& stochastic momentum}
\label{DSHB comp}
In Theorem \ref{thm:DSHBspeedup}, the total complexities (number of operations needed to achieve a given accuracy) of mSGD and smSGD have been compared and it has been shown that for small momentum parameter $\beta$,
\[C_{\beta}=\frac{C_{\text{mSGD}}(\beta)}{C_{\text{smSGD}}(\beta n)}  \approx  1 + \frac{n}{g},\]  where $C_{\text{mSGD}}$ and $C_{\text{smSGD}}$ represent the total costs of the two methods. The goal of this experiment is to show that this relationship holds also in practice. 

For this experiment we assume that the non-zeros of matrix $\bA$ are not concentrated in certain rows but instead that each row has the same number of non-zero coordinates. We denote by $g$ the number the non-zero elements per row. Having this assumption it can be shown that for the RK method the cost of one projection is equal to $4g$ operations while 
the cost per iteration of the mRK and of the smRK are $4g+3n$ and $4g+1$ respectively. For more details about the cost per iteration of the general mSGD and smSGD check Table~\ref{tableSHBandDSHB}.

As a first step a Gaussian matrix $\bA \in \R^{m \times n}$ is generated.
Then using this matrix several consistent linear systems are obtained as follows. Several values for $g \in [1,n]$ are chosen and for each one of these a matrix $\bA_g \in \R^{m \times n}$ with the same elements as $\bA$ but with $n-g$ zero coordinates per row is produced. 
For every matrix $\bA_g$, a Gaussian vector $z_g \in \R^n$ is drawn and to ensure consistency of the linear system, the right hand side is set to $b_g=\bA_g z$.

We run both mSGD and smSGD with small momentum parameter $\beta=0.0001$ for solving the linear systems $\bA_g x=b_g$ for all selected values of $g \in [1,n]$. The starting point for each run is taken to be $x_0=0 \in \R^n$.  The methods run until $\epsilon=\|x_k-x_*\|<0.001$, where $x_*=\Pi_{\cL_g}(x_0)$ and $\cL_g$ is the solution set of the linear system $\bA_g x=b_g$. In each run the number of operations needed to achieve the accuracy $\epsilon$ have been counted. 
For each linear system the average after $10$ trials of the value $\frac{C_{\text{mSGD}}(\beta)}{C_{\text{smSGD}}(\beta n)}$ is computed. 

\begin{table}[H]
\begin{center}
{\footnotesize
\begin{tabular}{ | p{4cm} | p{4cm} | p{3cm} | p{4cm} | }
 \hline
 Algorithm & Cost per iteration & Cost per Iteration (RK, mRK, smRK) \\
 \hline
 Basic Method ($\beta=0$) & $O(g)$ & $4g$  \\
  \hline
 mSGD & $O(g) +O(n) =O(n)$ &  $4g+3n$  \\
  \hline
smSGD & $O(g) +O(1) =O(g)$ & $4g+1$ \\
 \hline
\end{tabular}
}
\end{center}
\caption{Cost per iteration of the basic, mSGD and smSGD in the general setting and in the special cases of RK,mRK and smRK.}
\label{tableSHBandDSHB}
\end{table}

In Figure \ref{comparisonFigure} the actual ratio $\frac{C_{\text{mSGD}}(\beta)}{C_{\text{smSGD}}(\beta n)}$ and the theoretical approximation $1 + \frac{n}{g}$ are plot and it is shown that they have similar behavior. Thus the theoretical prediction of Theorem~\ref{thm:DSHBspeedup} is numerically confirmed.
In particular in the implementations we use the Gaussian matrices $\bA \in \R^{200 \times 100}$ and $\bA \in \R^{1000 \times 300}$.

\begin{figure}[!]
\centering
\begin{subfigure}{.4\textwidth}
  \centering
  \includegraphics[width=1\linewidth]{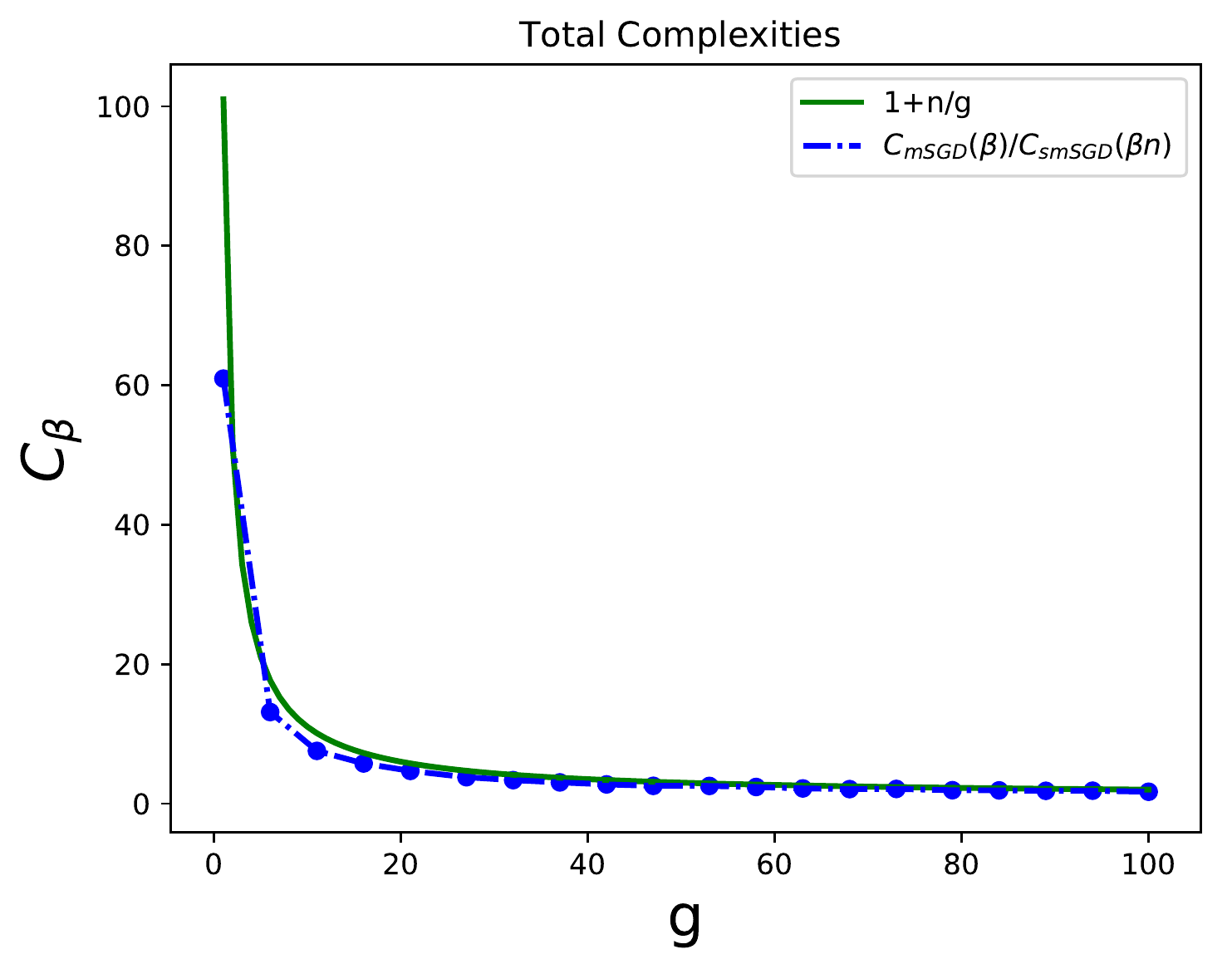}
  \caption{$\bA \in \R^{200 \times 100}$}
\end{subfigure}
\begin{subfigure}{.4\textwidth}
  \centering
  \includegraphics[width=1\linewidth]{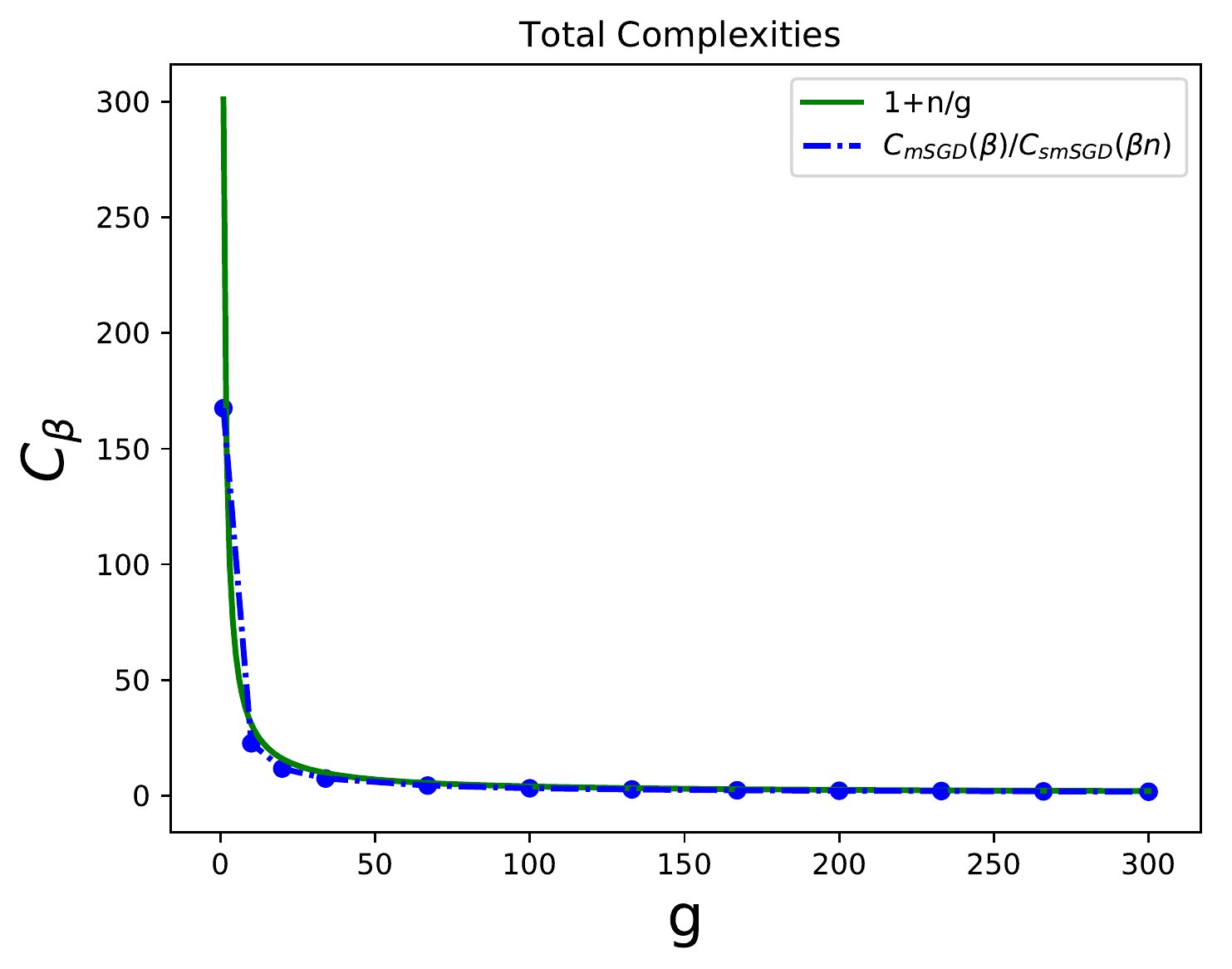}
  \caption{$\bA \in \R^{1000 \times 300}$}
\end{subfigure}
\caption{Comparison of the total complexities of mRK and smRK. The green continuous line denotes the theoretical relationship $1 + \frac{n}{g}$ that we predict in Theorem~\ref{thm:DSHBspeedup}. The blue dotted line shows the ratio of the total complexities $\frac{C_{\text{mSGD}}(\beta)}{C_{\text{smSGD}}(\beta n)} $ for several linear systems $\bA_g x=b_g$ where $g \in [1,n]$. The momentum parameter $\beta=0.0001$ is used for both methods.}
\label{comparisonFigure}
\end{figure}

\subsection{Faster method for average consensus}
\label{consensus}

\subsubsection{Background}
Average consensus (AC) is a fundamental problem in distributed computing and multi-agent systems \cite{dimakis2010gossip, boyd2006randomized}. Consider a connected undirected network $\cG=(\cV,\cE)$ with node set $\cV=\{1,2,\dots,n\}$ and edges $\cE$, ($|\cE|=m$), where each node $i \in \cV$ owns a private value $c_i \in \R$. The goal of the AC problem is each node of the network to compute the average of these private values, $\bar{c}\eqdef\tfrac{1}{n}\sum_i c_i$, via a protocol which allows communication between neighbours only. The problem comes up in many real world applications such as coordination of autonomous agents, estimation, rumour spreading in social networks, PageRank and distributed data fusion on ad-hoc networks and decentralized optimization. 

It was shown recently that several randomized methods for solving linear systems  can be interpreted as randomized gossip algorithms for solving the AC problem when applied to a special system encoding the underlying network \cite{gower2015stochastic, LoizouRichtarik}. As we have already explained both basic method \cite{ASDA} and basic method with momentum (this paper) find the solution of the linear system that is closer to the starting point of the algorithms. That is, both methods converge linearly to $x_*=\Pi^{\mB}_{\cL}(x_0)$; the projection of the initial iterate onto the solution set of the linear system and as a result (check Introduction) can be interpreted as methods for solving the best approximation problem~\eqref{eq:primal}.
In the special case that
\begin{enumerate}
\item the linear system in the constraints of \eqref{eq:primal} is the homogeneous linear system ($\bA x=0$) with matrix $\bA\in \R^ {m \times n}$ being the incidence matrix of the undirected graph $\cG=(\cV,\cE)$, and 
\item the starting point of the method are the initial values of the nodes $x_0=c$,
\end{enumerate} 
it  is straightforward to see that the solution of the best approximation problem is a vector with all components equal to the consensus value $\bar{c}\eqdef\tfrac{1}{n}\sum_i c_i$. Under this setting, the famous randomized pairwise gossip algorithm (randomly pick an edge $e\in E$ and replace the private values of its two nodes to their average) that was first proposed and analyzed in \cite{boyd2006randomized}, is equivalent with the RK method without relaxation ($\omega=1$) \cite{gower2015stochastic, LoizouRichtarik}.


\begin{rem}
In the gossip framework, the condition number of the linear system when RK is used has a simple structure and it depends on the characteristics of the network under study. More specifically, it depends on the number of the edges $m$ and on the Laplacian matrix of the network\footnote{Matrix $\bA$ of the linear system is the incidence matrix of the graph and it is known that the Laplacian matrix is equal to $\bL=\bA^\top \bA$, where $\|\bA\|^2_F=2m$.}:
\begin{equation}
\label{algebconeec}
\frac{1}{\lambda_{\min}^+(\bW)}\overset{\eqref{matrixW}}{=}\frac{1}{\lambda_{\min}^+(\bA^\top \bA/ \|\bA\|^2_F)}\overset{\|\bA\|^2_F=2m}{=}\frac{2m}{\lambda_{\min}^+(\bA^\top \bA)}=\frac{2m}{\lambda_{\min}^+(\bL)},
\end{equation}
where $\bL=\bA^\top \bA$ is the Laplacian matrix of the network and the quantity $\lambda_{\min}^+(\bL)$ is the very well studied \textit{algebraic connectivity} of the graph \cite{de2007old}. 
\end{rem}

\begin{rem} The convergence analysis in this paper holds for any consistent linear system $\bA x= b$ without any assumption on the rank of the matrix $\bA$. The lack of any assumption on the form of matrix $\bA$ allows us to solve the homogeneous linear system $\bA x=0$ where $\bA$ is the incidence matrix of the network which by construction is rank deficient. More specifically, it can be shown that ${\rm rank}(\bA)=n-1$ \cite{LoizouRichtarik}. Note that many existing methods for solving linear systems make the assumption that the matrix $\bA$ of the linear systems is full rank \cite{RK, needell2010randomized, RBK} and as a result can not be used to solve the AC problem.
\end{rem}

\subsubsection{Numerical Setup}
Our goal in this experiment is to show that the addition of the momentum term to the randomized pairwise gossip algorithm (RK in the gossip setting) can lead to faster gossip algorithms and as a result the nodes of the network will converge to the average consensus faster both in number of iterations and in time. We do not intend to analyze the distributed behavior of the method (this is on-going research work). In our implementations we use three of the most popular graph topologies in the literature of wireless sensor networks. These are the line graph, cycle graph and the random geometric graph $G(n,r)$. In practice, $G(n,r)$ consider ideal for modeling wireless sensor networks, because of their particular formulation. In the experiments the $2$-dimensional $G(n,r)$ is used which is formed by placing $n$ nodes uniformly at random in a unit square with edges only between nodes that have euclidean distance less than the given radius $r$. To preserve the connectivity of $G(n, r)$ a radius $r = r(n) =  \log(n)/n$ is used \cite{penrose2003random}. 
The AC problem is solved for the three aforementioned networks for both $n=100$ and $n=200$ number of nodes. We run mRK with several momentum parameters $\beta$ for 10 trials and we plot their average. Our results are available in Figures~\ref{consensus100} and \ref{consensus200}.

Note that the vector of the initial values of the nodes can be chosen arbitrarily, and the proposed algorithms will  find the average  of these values. In Figures~\ref{consensus100} and \ref{consensus200} the initial value of each node is chosen independently at random from the uniform distribution in the interval $(0,1)$.

\subsubsection{Experimental Results}

By observing Figures~\ref{consensus100} and \ref{consensus200}, it is clear that the addition of the momentum term improves the performance of the popular pairwise randomized gossip (PRG) method \cite{boyd2006randomized}.  The choice $\beta=0.4$ as the momentum parameter improves the performance of the vanilla PRG for all networks under study and $\beta=0.5$ is a good choice for the cases of the cycle and line graph. Note that for networks such as the cycle and line graphs there are known closed form expressions for the algebraic connectivity \cite{de2007old}. Thus, using equation \eqref{algebconeec}, we can compute the exact values of the condition number $1/\lambda_{\min}^+$ for these networks. Interestingly, as we can see in Table~\ref{Algebraic Connectvity} for $n=100$ and $n=200$ (number of nodes), the condition number $1/\lambda_{\min}^+$ appearing in the iteration complexity of our methods is not very large. This is in contrast with experimental  observations from Section~\ref{gaussiansyntheric} where it was shown that the choice $\beta=0.5$ is good  for very ill conditioned problems only ($1/\lambda_{\min}^+$ very large). 

\begin{table}[H]
\begin{center}
{\footnotesize
\begin{tabular}{ | p{2cm} | p{4cm} | p{3cm} | p{3cm}|  }
 \hline
Network & Formula for $\lambda_{\min}^+(\bL)$ & $1/\lambda_{\min}^+$ for $n=100$ & $1/\lambda_{\min}^+$ for $n=200$\\
   \hline
Line & $2\left(1-\cos ({\pi}/{n})\right)$& 1013 &4052\\
 \hline
Cycle &  $2\left(1-\cos ({2\pi}/{n})\right)$&253 &1013\\
 \hline
\end{tabular}
}
\end{center}
\caption{Algebraic connectivity of cycle and line graph for $n=100$ and $n=200$}
\label{Algebraic Connectvity}
\end{table}
 
\begin{figure}[!]
\centering
\begin{subfigure}{.23\textwidth}
  \centering
  \includegraphics[width=1\linewidth]{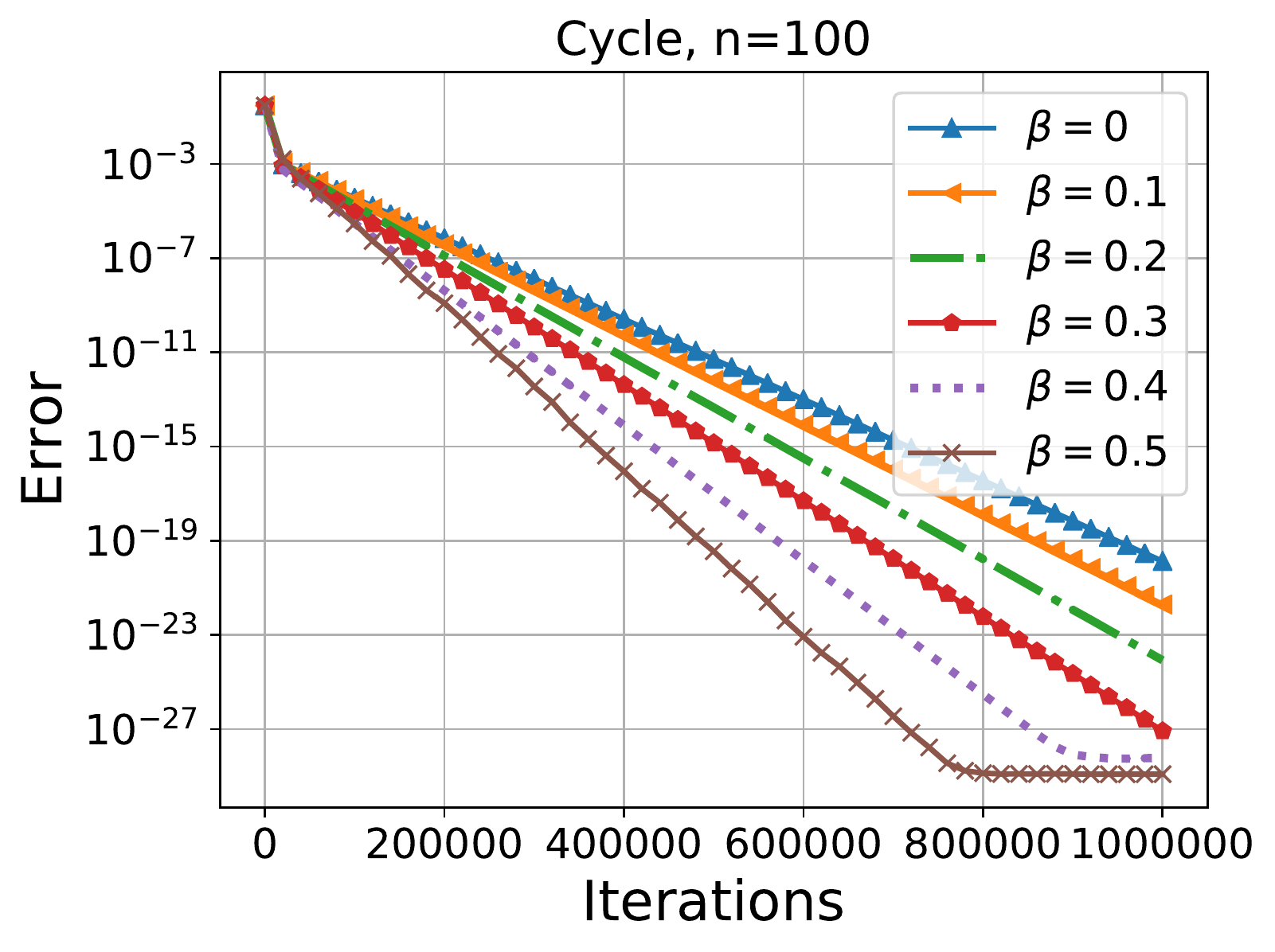}
\end{subfigure}%
\begin{subfigure}{.23\textwidth}
  \centering
  \includegraphics[width=1\linewidth]{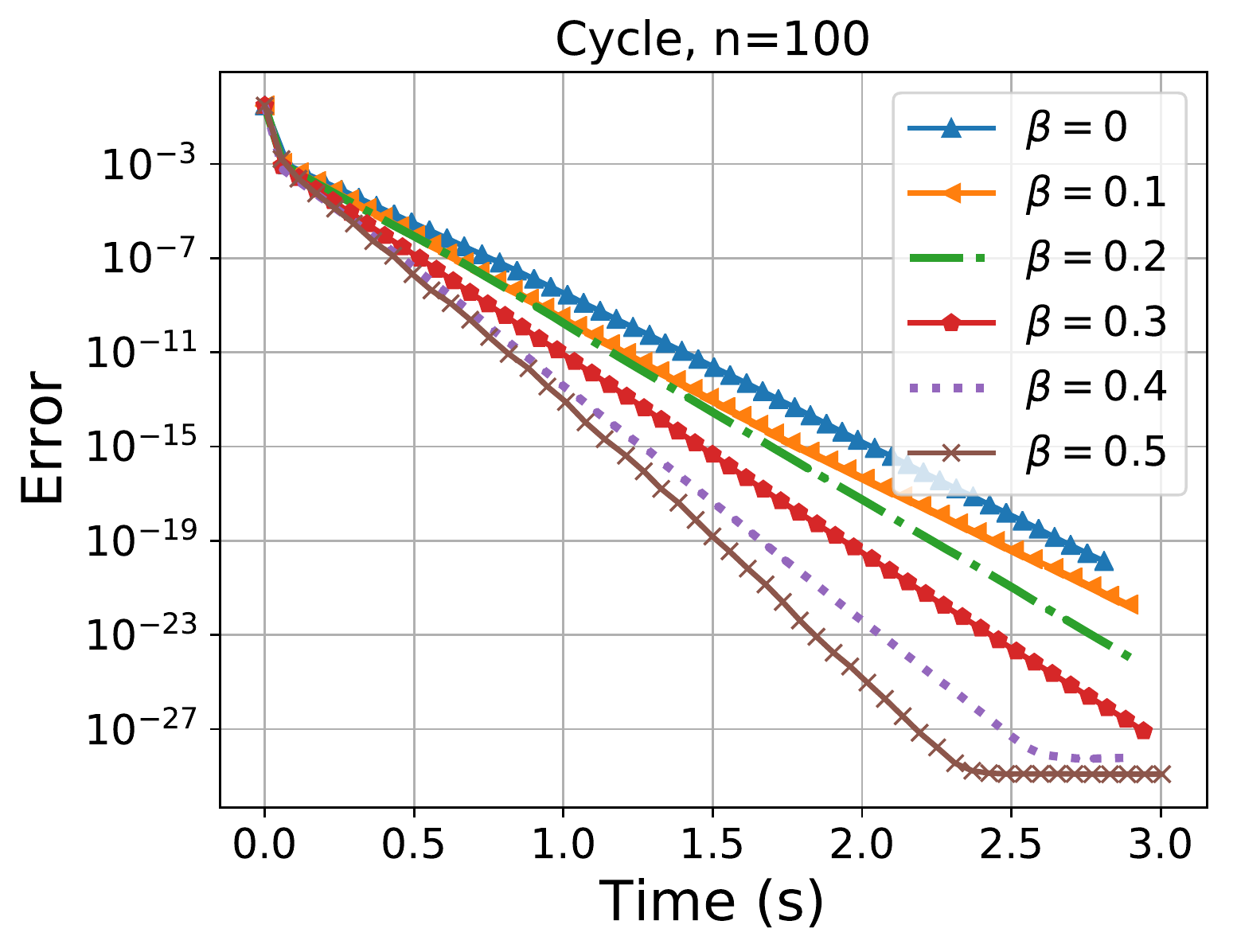}
\end{subfigure}
\begin{subfigure}{.23\textwidth}
  \centering
  \includegraphics[width=1\linewidth]{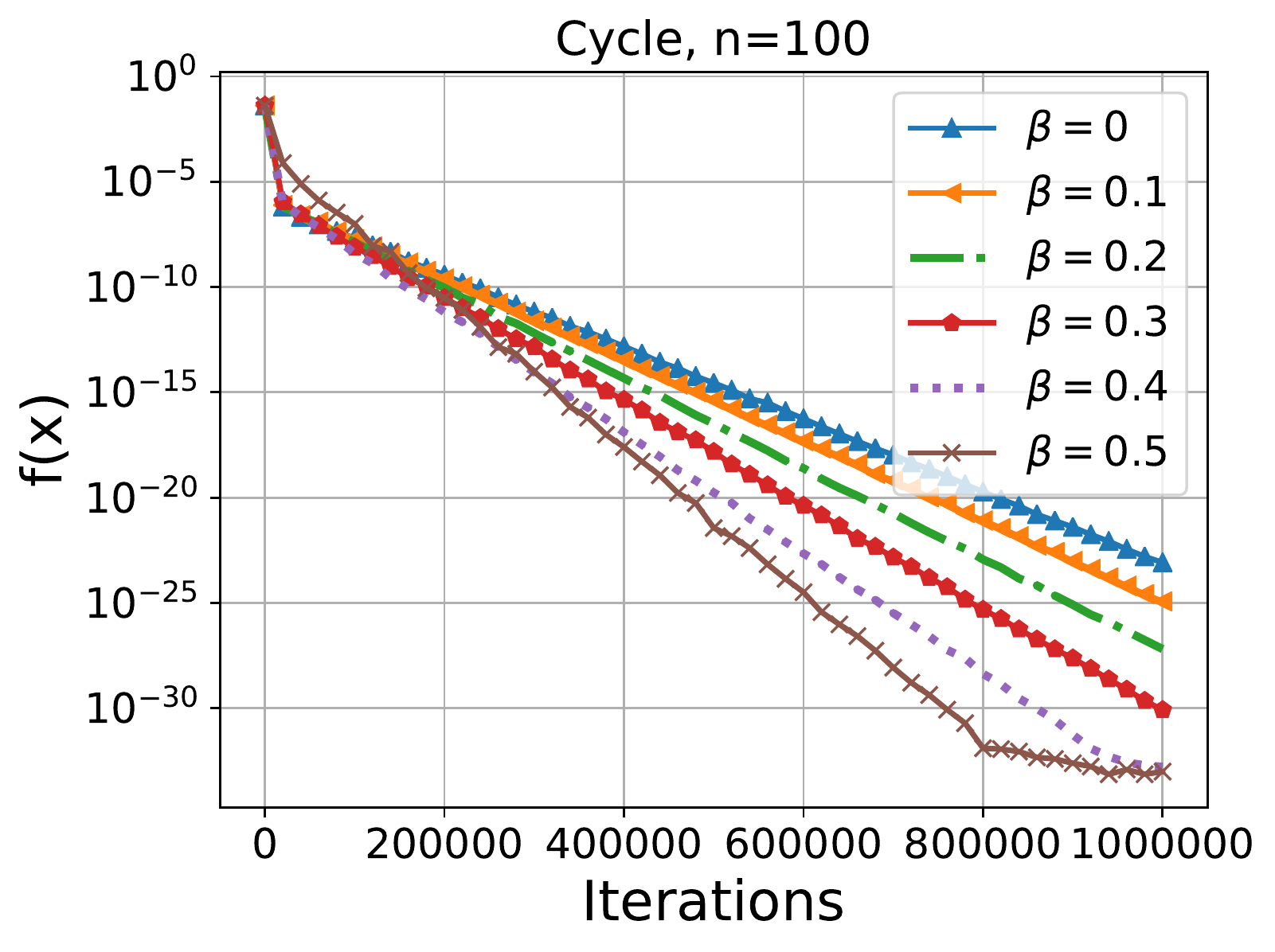}
\end{subfigure}
\begin{subfigure}{.23\textwidth}
  \centering
  \includegraphics[width=1\linewidth]{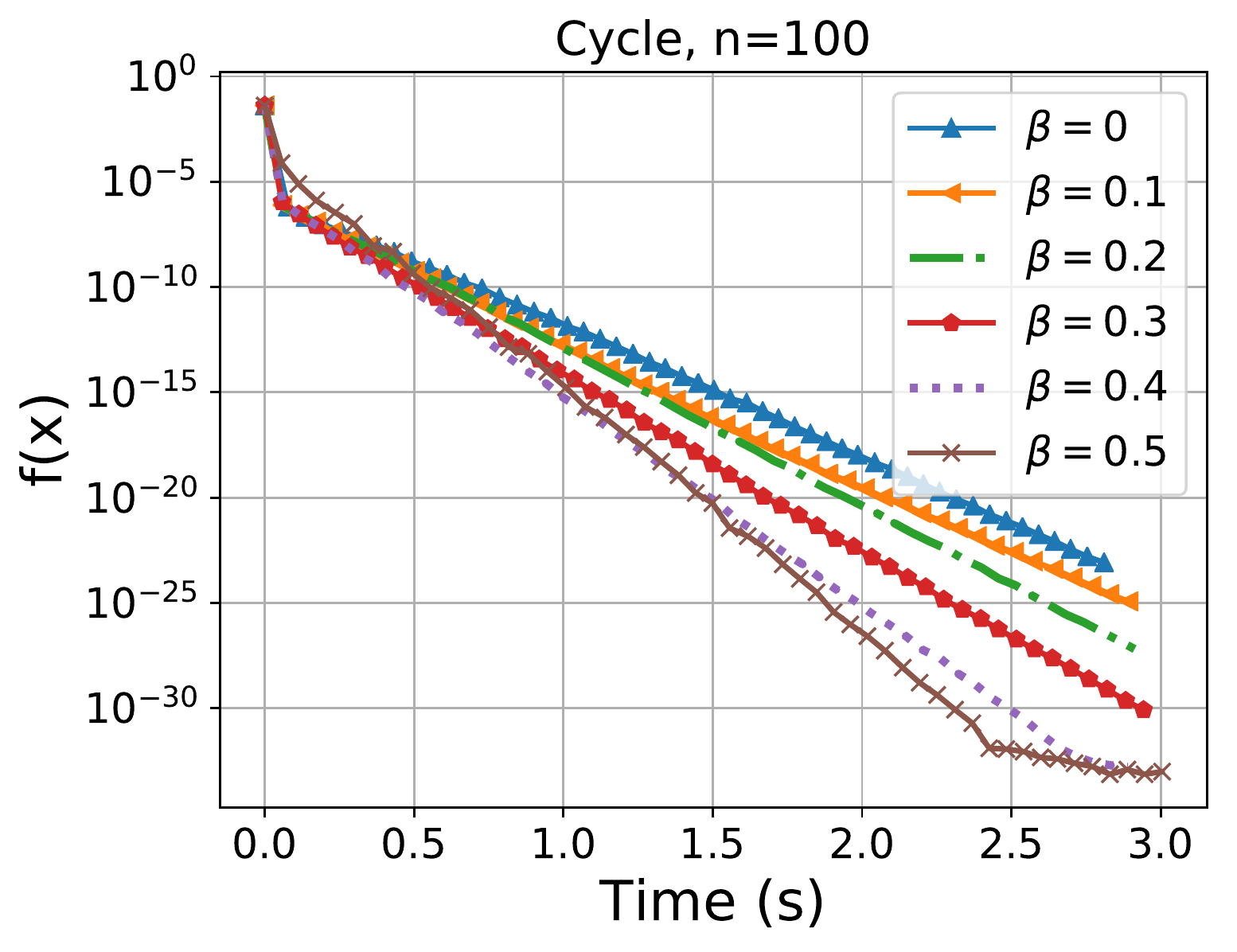}
\end{subfigure}\\
\begin{subfigure}{.23\textwidth}
  \centering
  \includegraphics[width=1\linewidth]{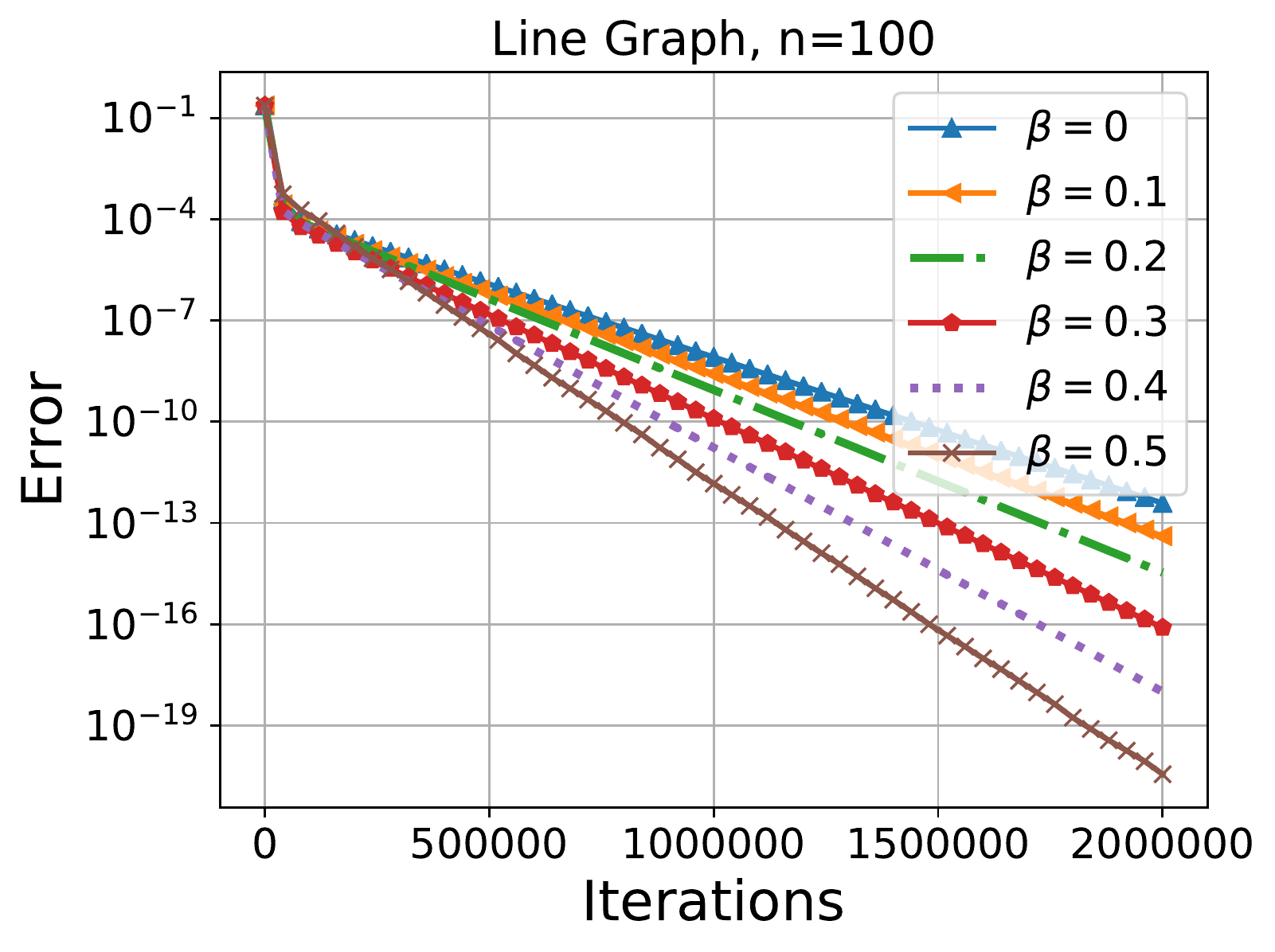}
\end{subfigure}%
\begin{subfigure}{.23\textwidth}
  \centering
  \includegraphics[width=1\linewidth]{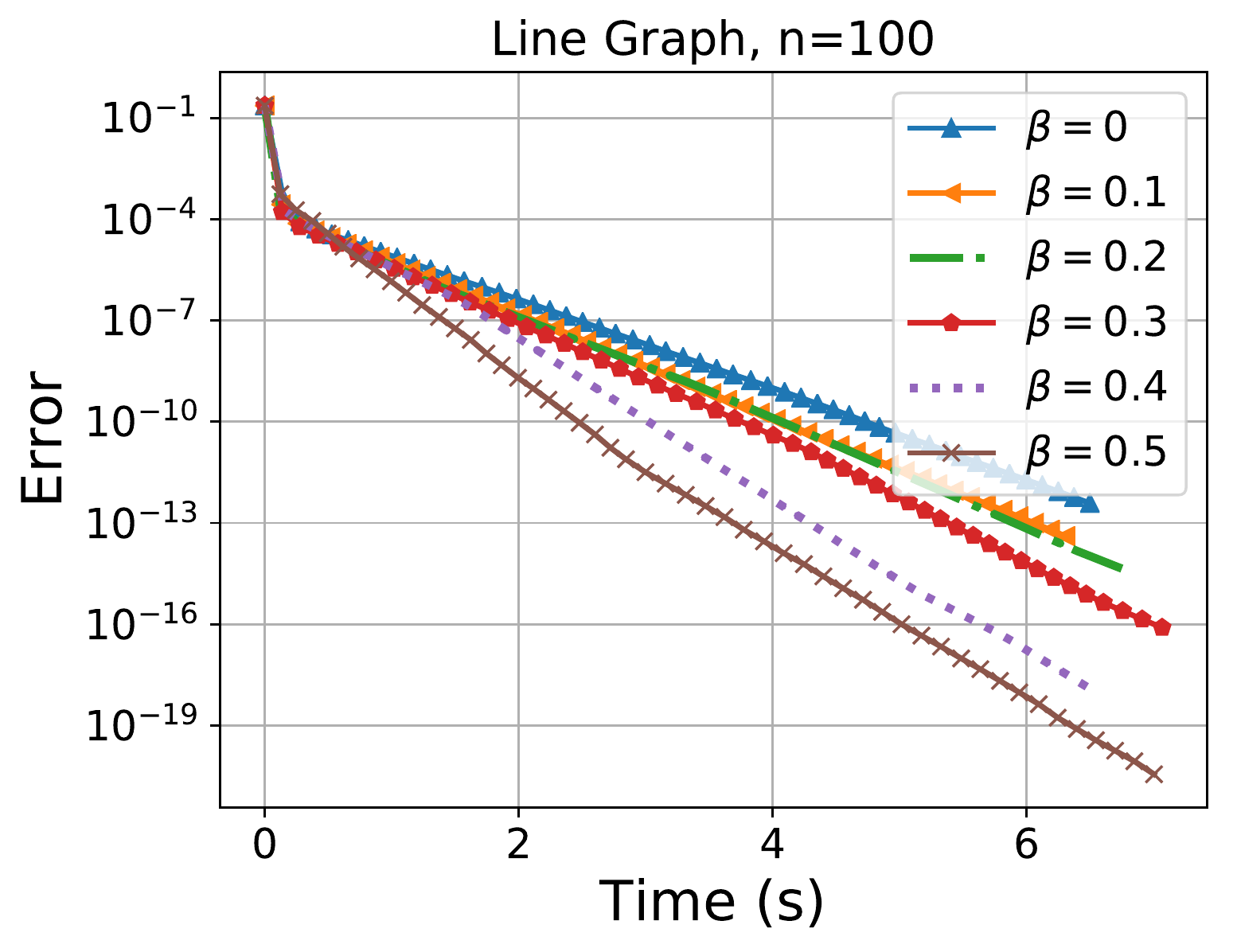}
\end{subfigure}
\begin{subfigure}{.23\textwidth}
  \centering
  \includegraphics[width=1\linewidth]{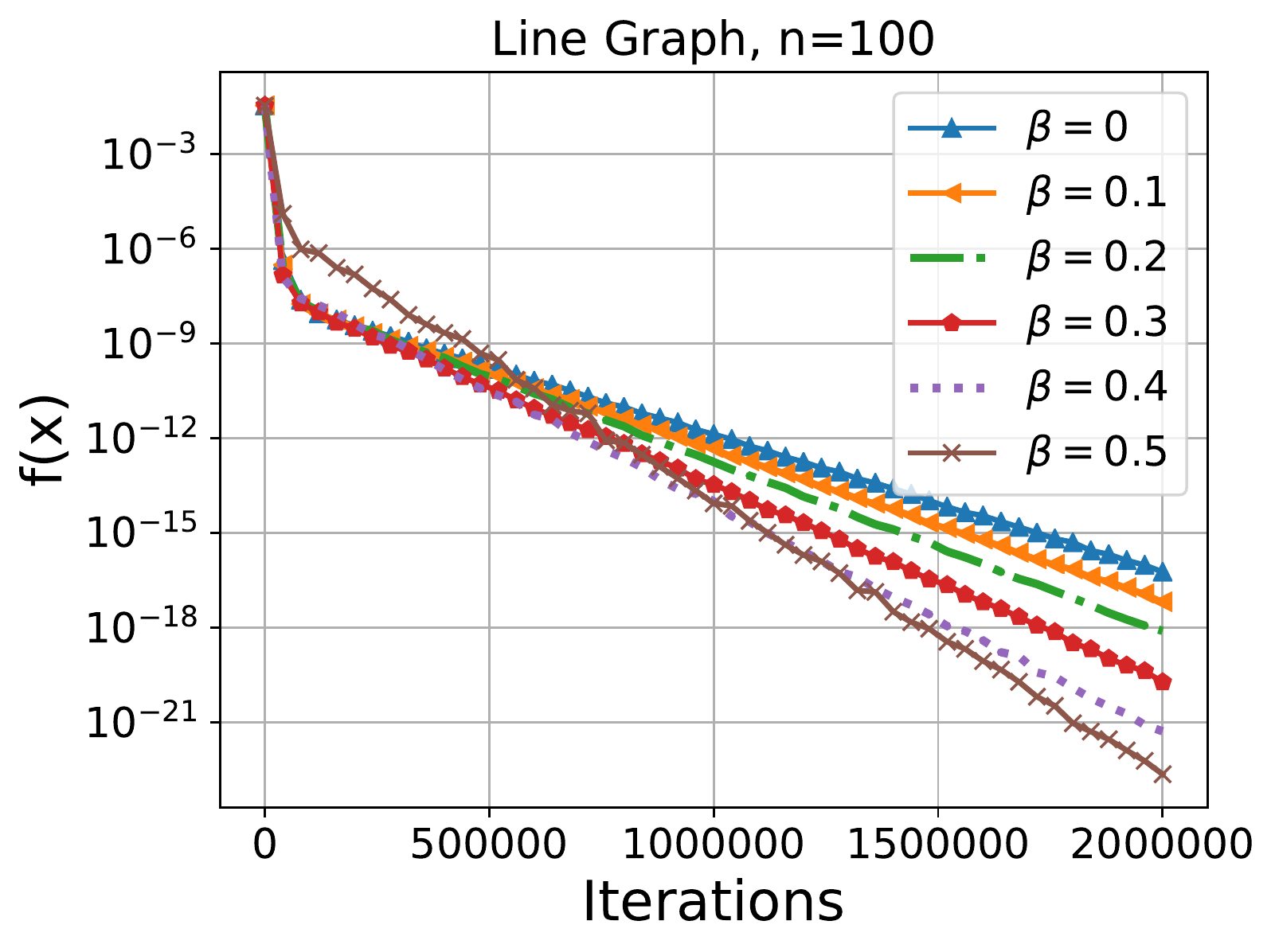}
\end{subfigure}
\begin{subfigure}{.23\textwidth}
  \centering
  \includegraphics[width=1\linewidth]{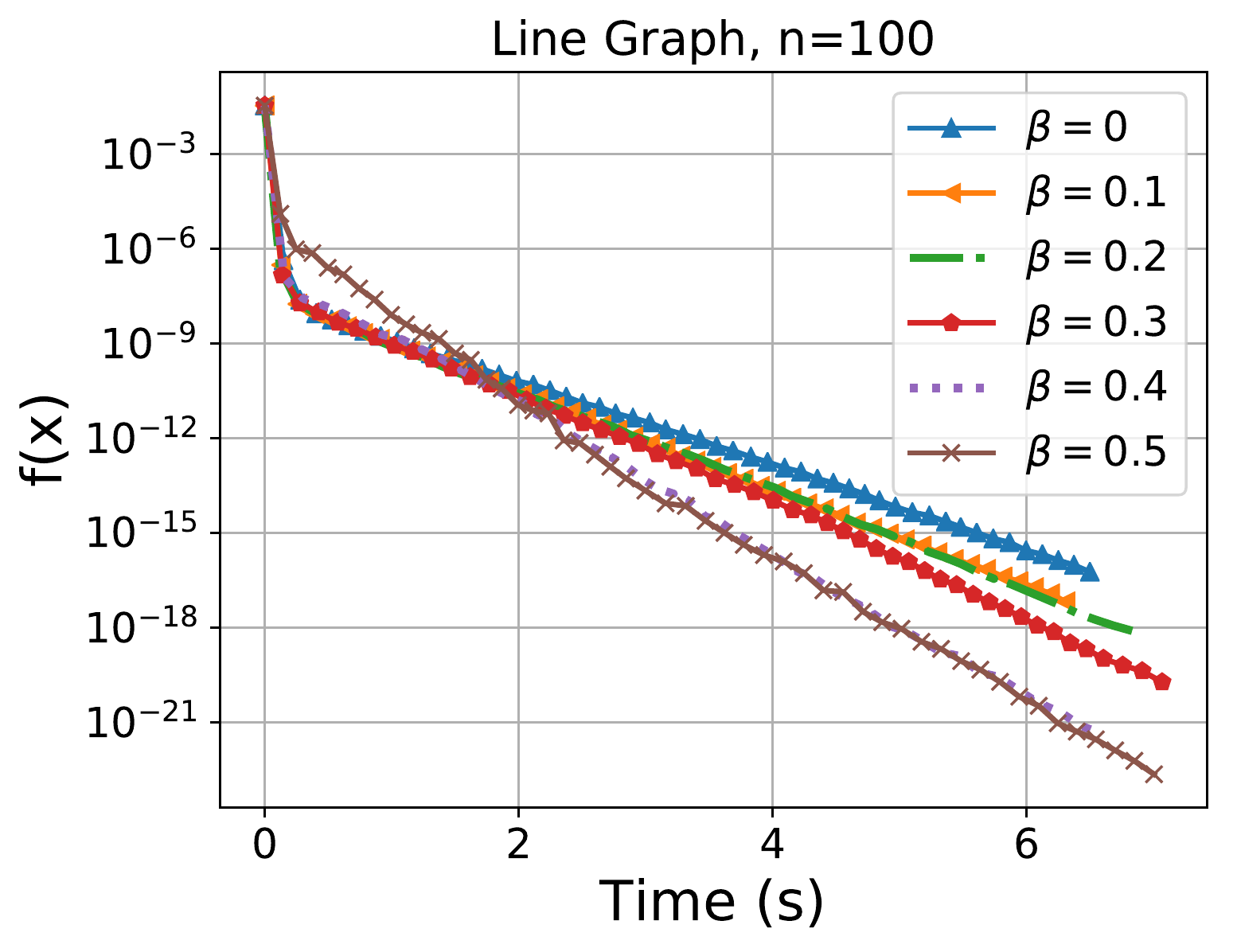}
\end{subfigure}\\
\begin{subfigure}{.23\textwidth}
  \centering
  \includegraphics[width=1\linewidth]{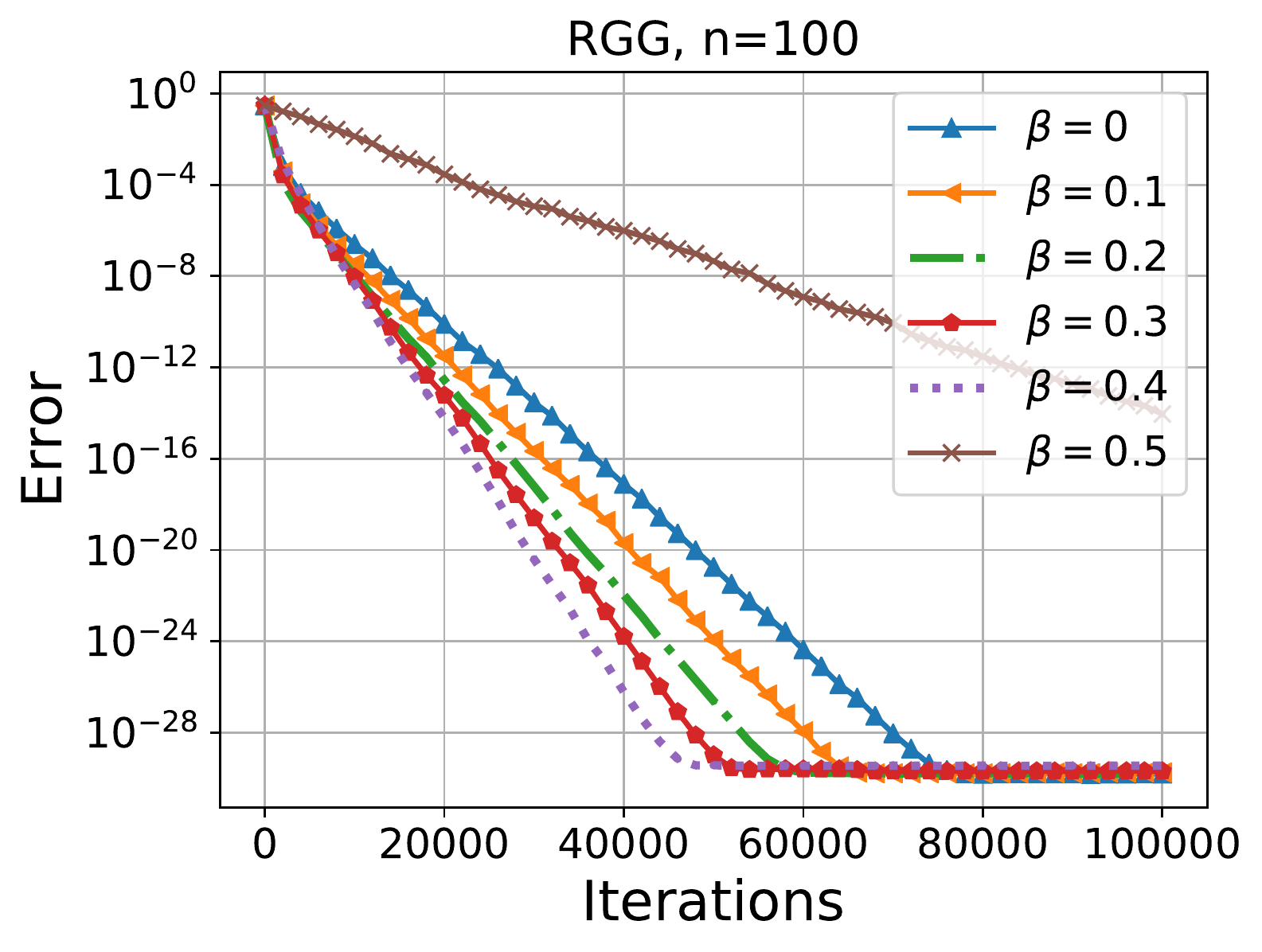}
\end{subfigure}%
\begin{subfigure}{.23\textwidth}
  \centering
  \includegraphics[width=1\linewidth]{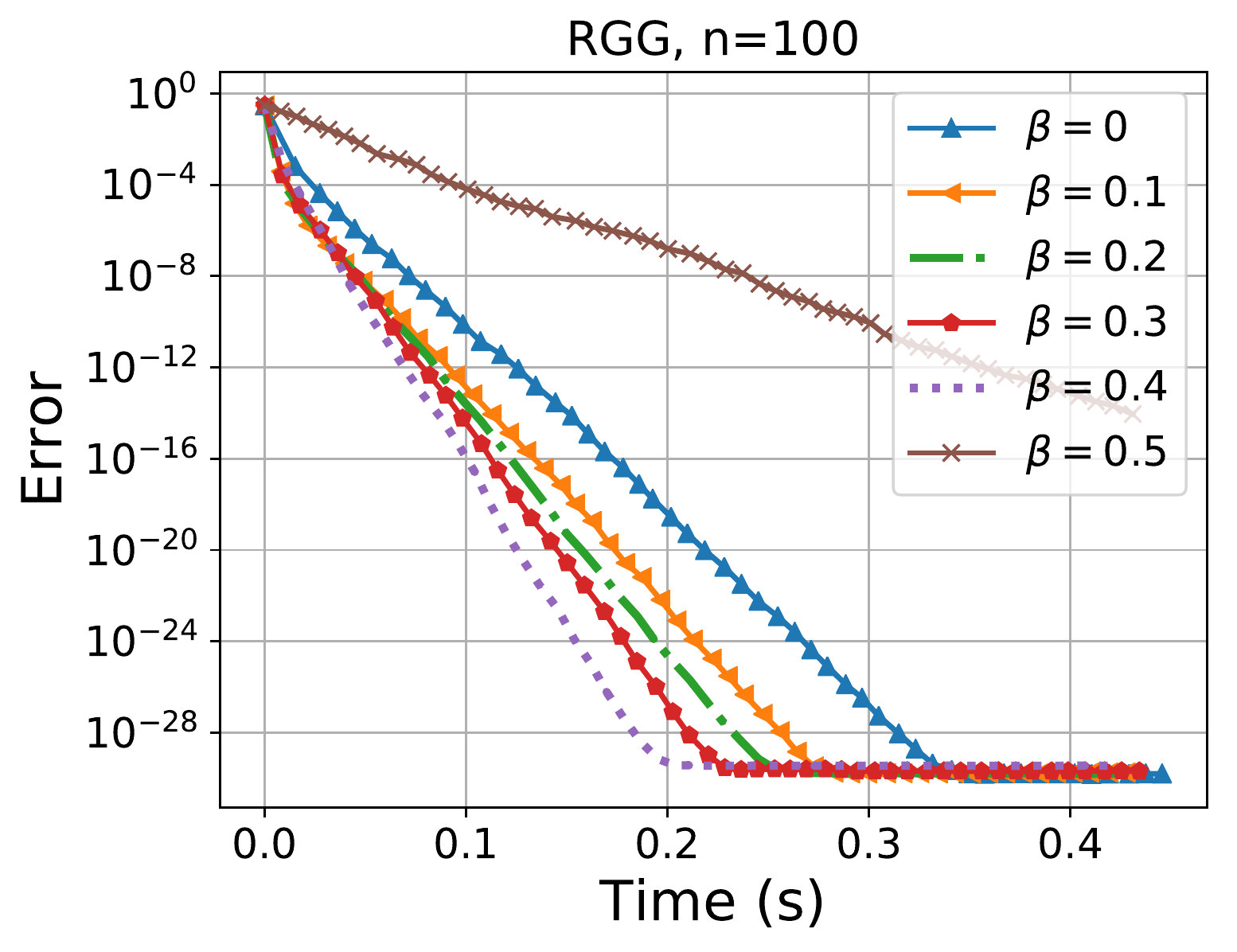}
\end{subfigure}
\begin{subfigure}{.23\textwidth}
  \centering
  \includegraphics[width=1\linewidth]{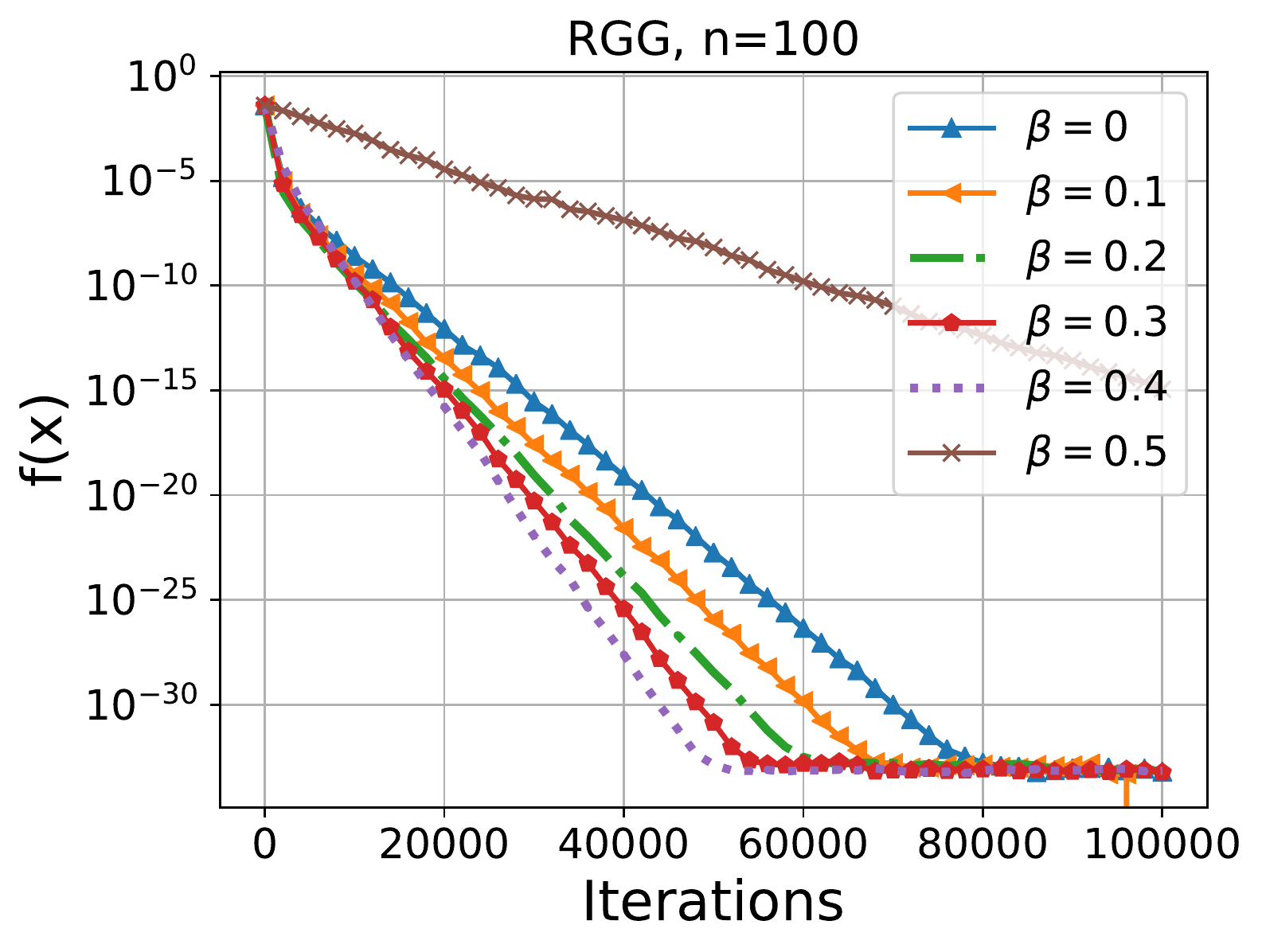}
\end{subfigure}
\begin{subfigure}{.23\textwidth}
  \centering
  \includegraphics[width=1\linewidth]{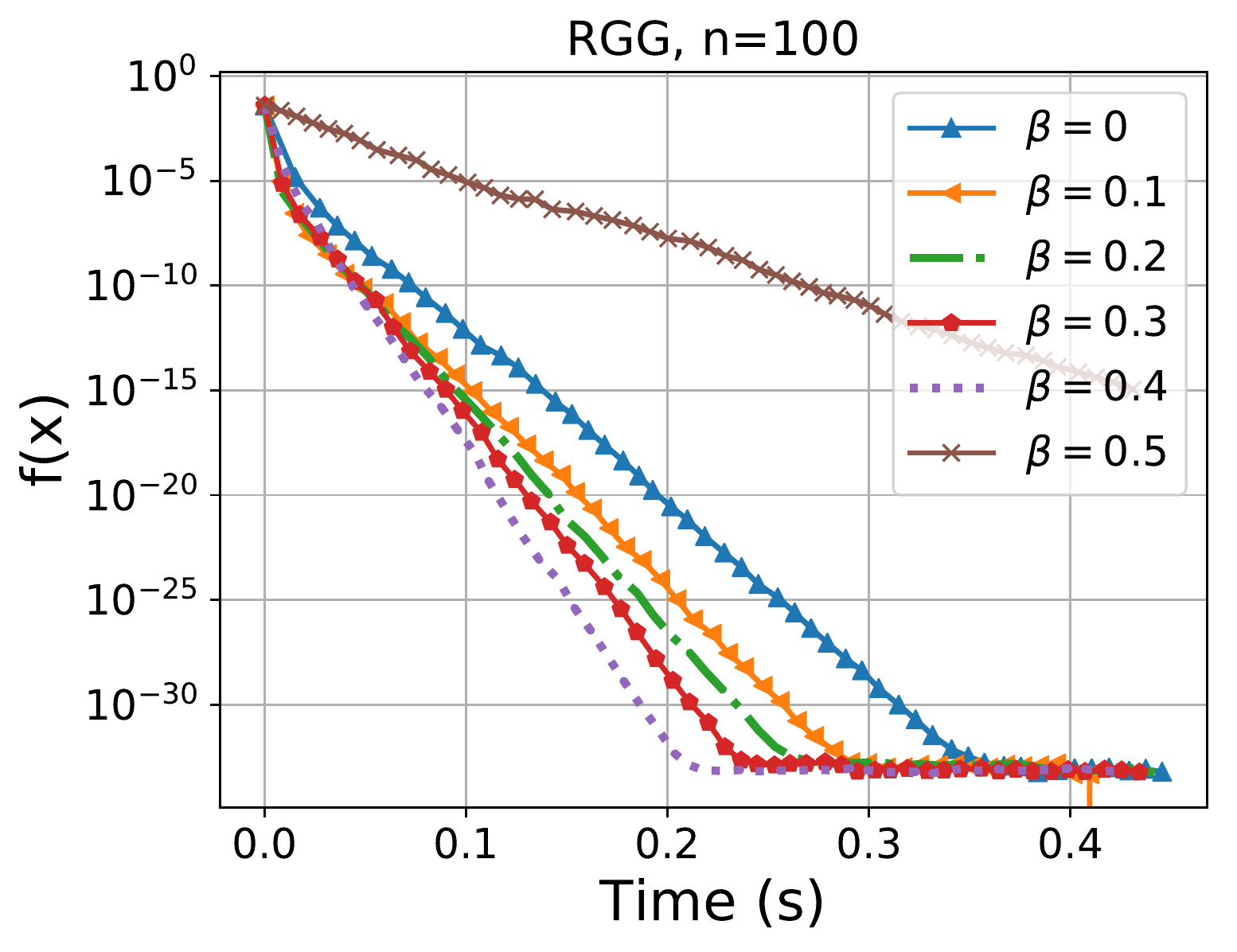}
\end{subfigure}\\
\caption{Performance of mPRG for several momentum parameters $\beta$ for solving the average consensus problem in a cycle graph, line graph and random geometric graph $G(n,r)$ with $n=100$ nodes. For the $G(n,r)$ to ensure connectivity of the network a radius $r=\sqrt{\log(n)/n}$ is used. The graphs in the first (second) column plot iterations (time) against residual error while those in the third (forth) column plot iterations (time) against function values. The ``Error" in the vertical axis represents the relative error $\|x_k-x_*\|^2_\bB / \|x_0-x_*\|^2_\bB \overset{\bB=\bI, x_0=c}{=}\|x_k-x_*\|^2 / \|c-x_*\|^2_\bB$ and the function values $f(x_k)$ refer to function~\eqref{functionRK}.}
\label{consensus100}
\end{figure}

\begin{figure}[!]
\centering
\begin{subfigure}{.23\textwidth}
  \centering
  \includegraphics[width=1\linewidth]{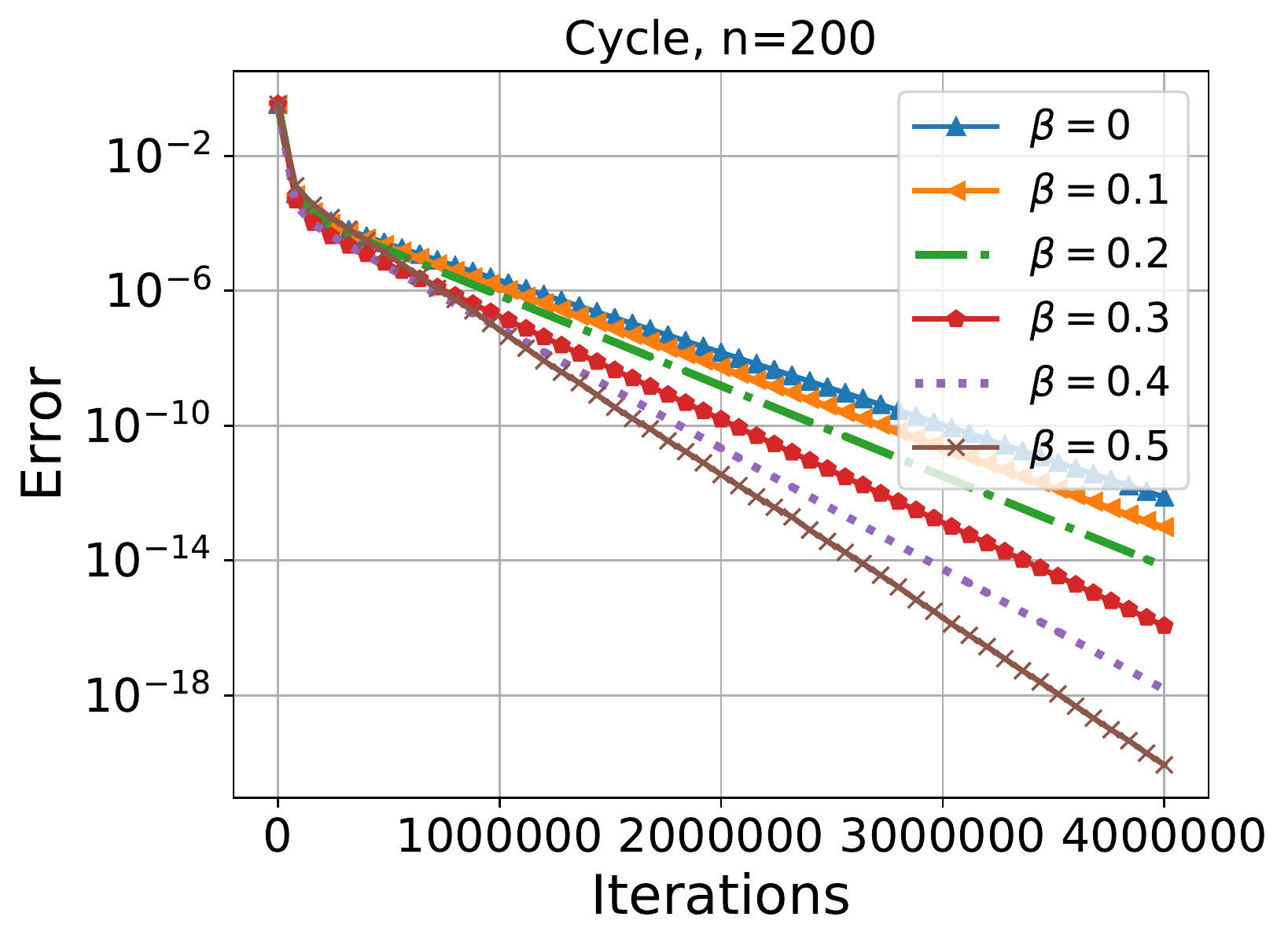}
\end{subfigure}%
\begin{subfigure}{.23\textwidth}
  \centering
  \includegraphics[width=1\linewidth]{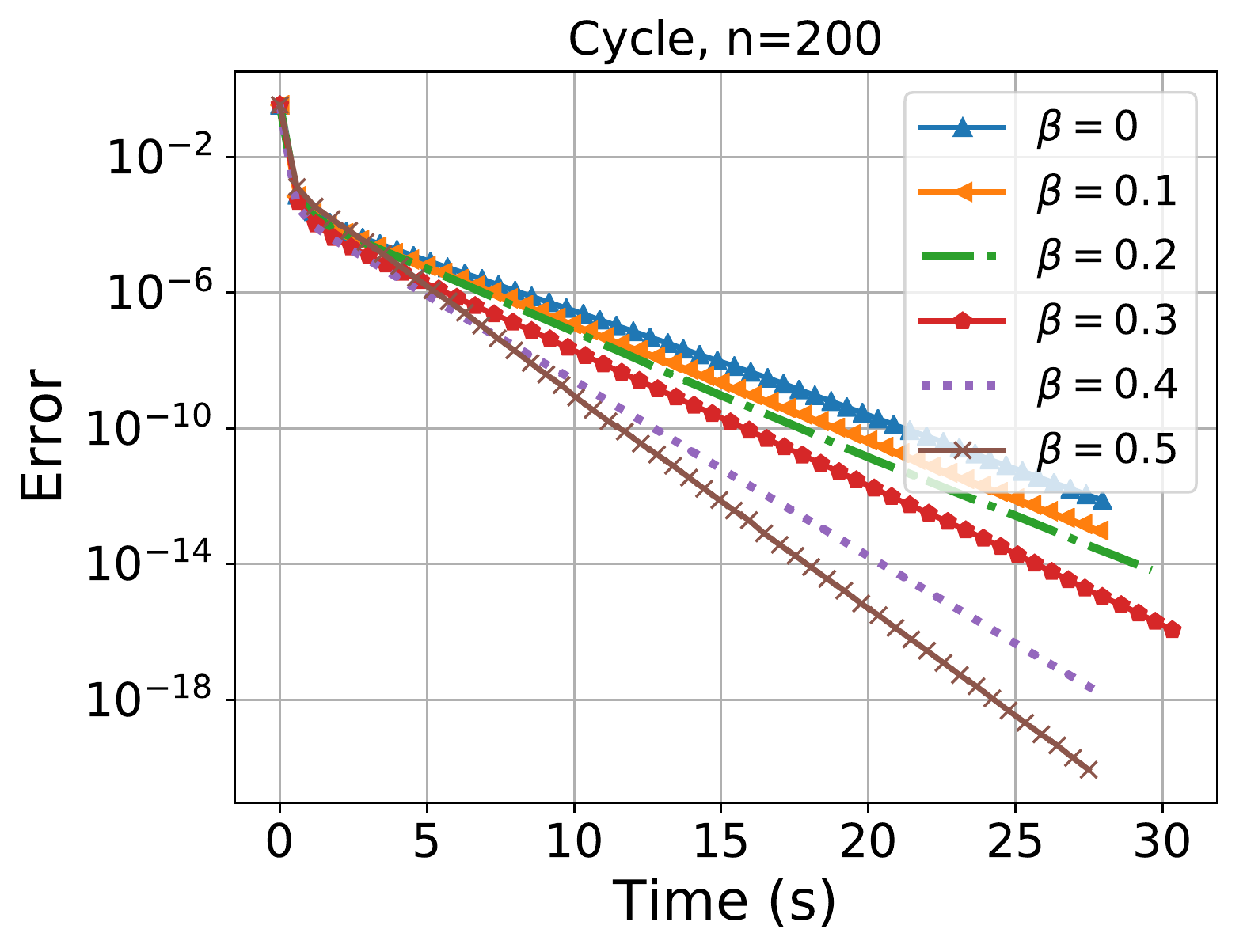}
\end{subfigure}
\begin{subfigure}{.23\textwidth}
  \centering
  \includegraphics[width=1\linewidth]{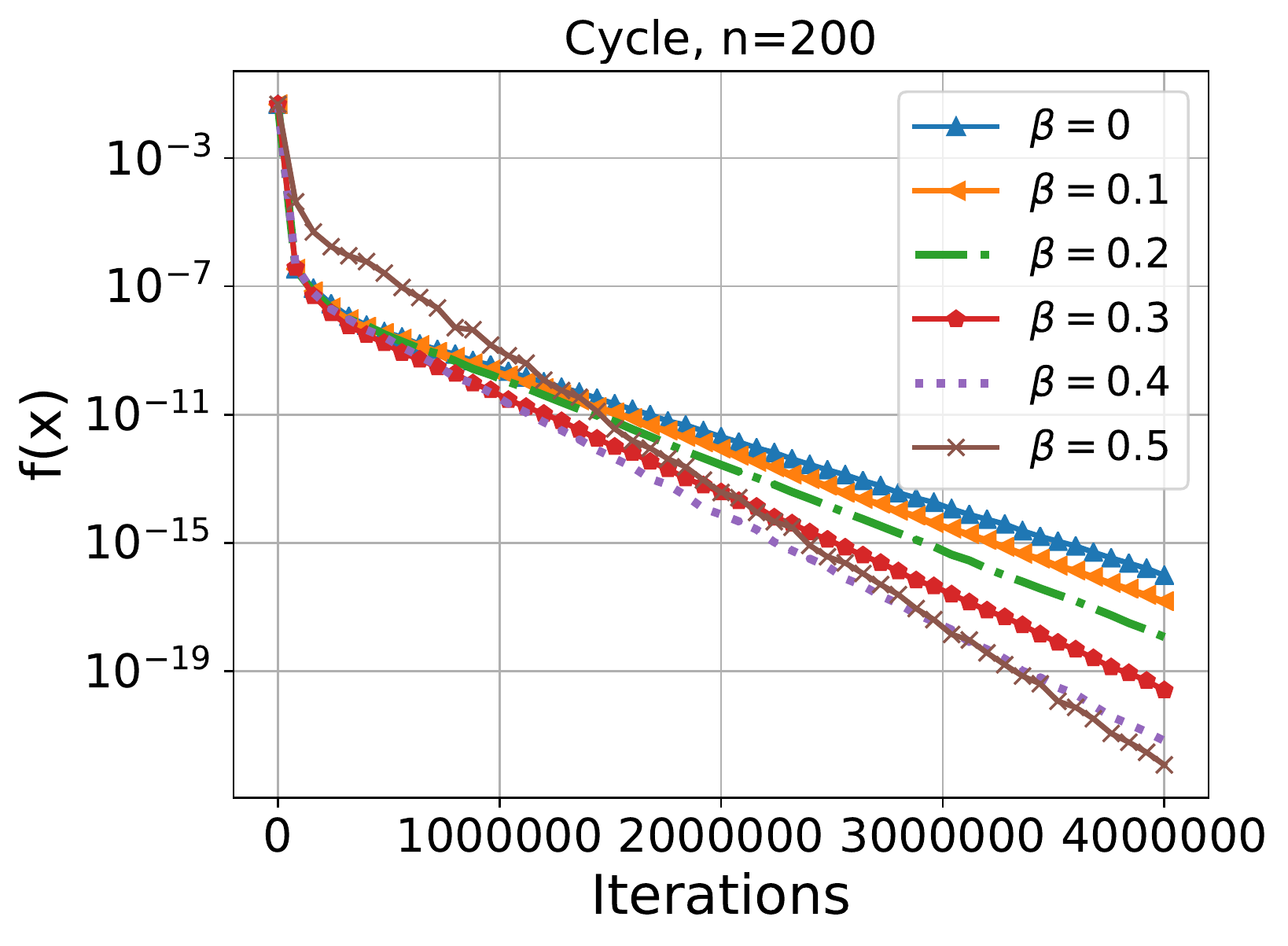}
\end{subfigure}
\begin{subfigure}{.23\textwidth}
  \centering
  \includegraphics[width=1\linewidth]{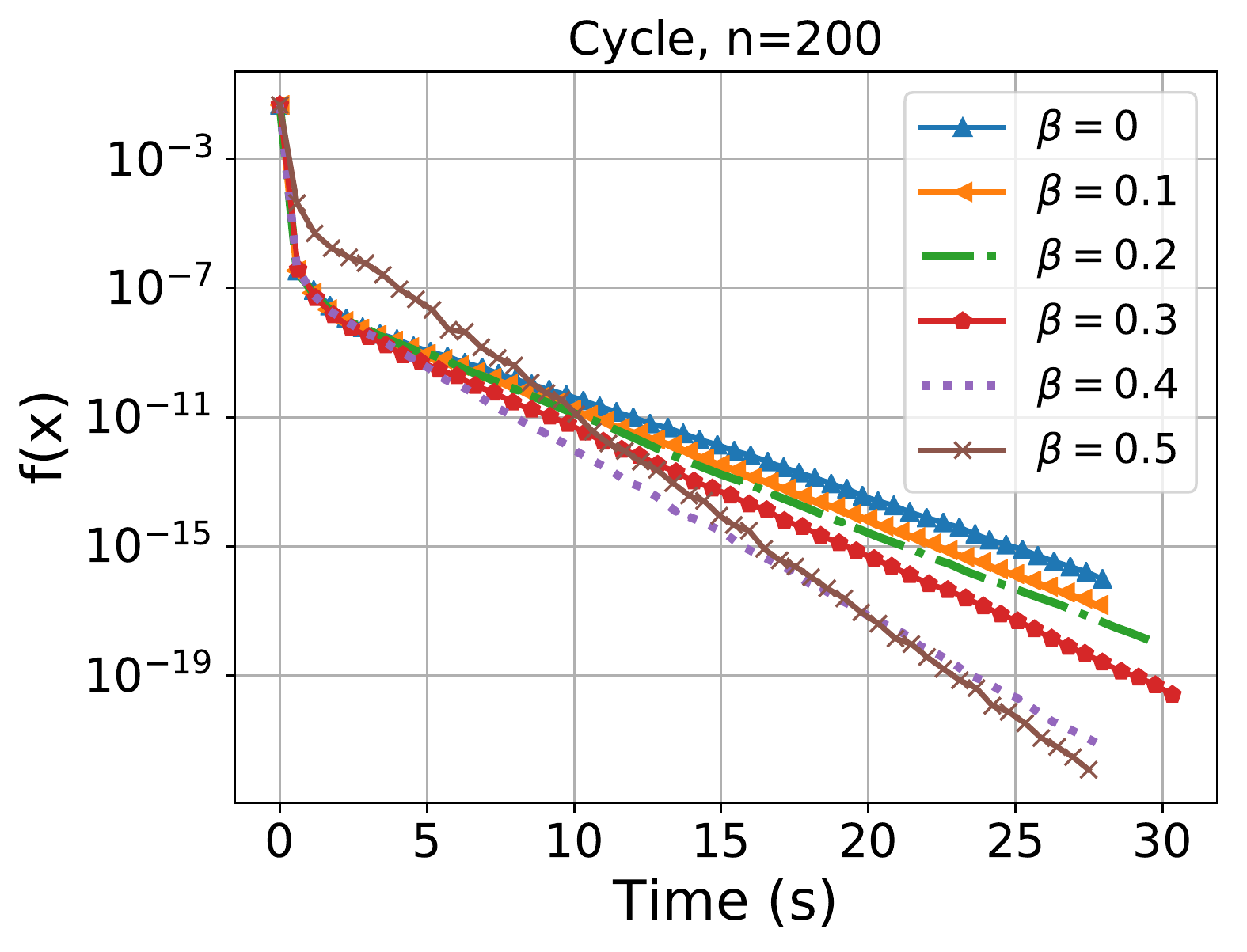}
\end{subfigure}\\
\begin{subfigure}{.23\textwidth}
  \centering
  \includegraphics[width=1\linewidth]{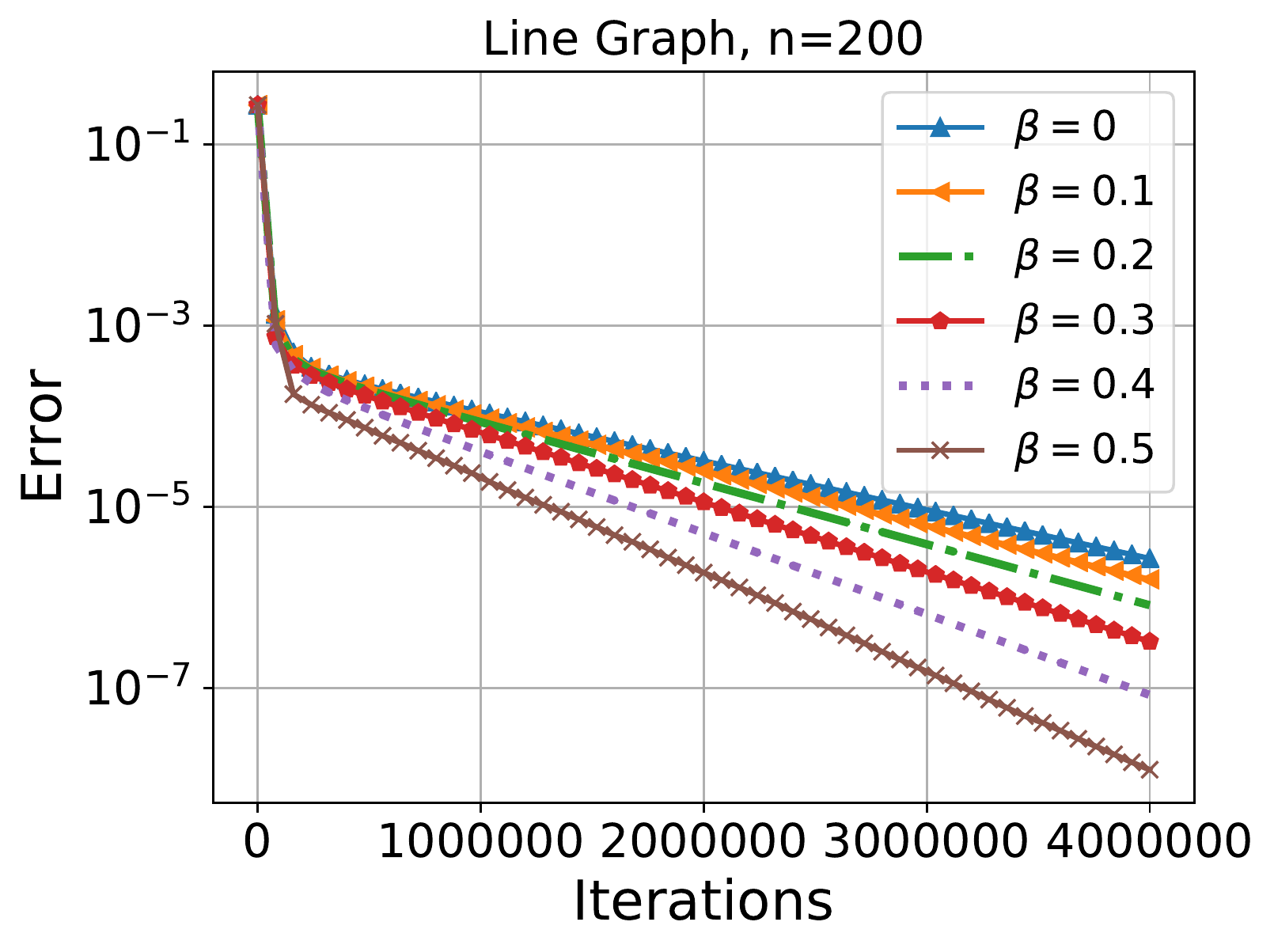}
\end{subfigure}%
\begin{subfigure}{.23\textwidth}
  \centering
  \includegraphics[width=1\linewidth]{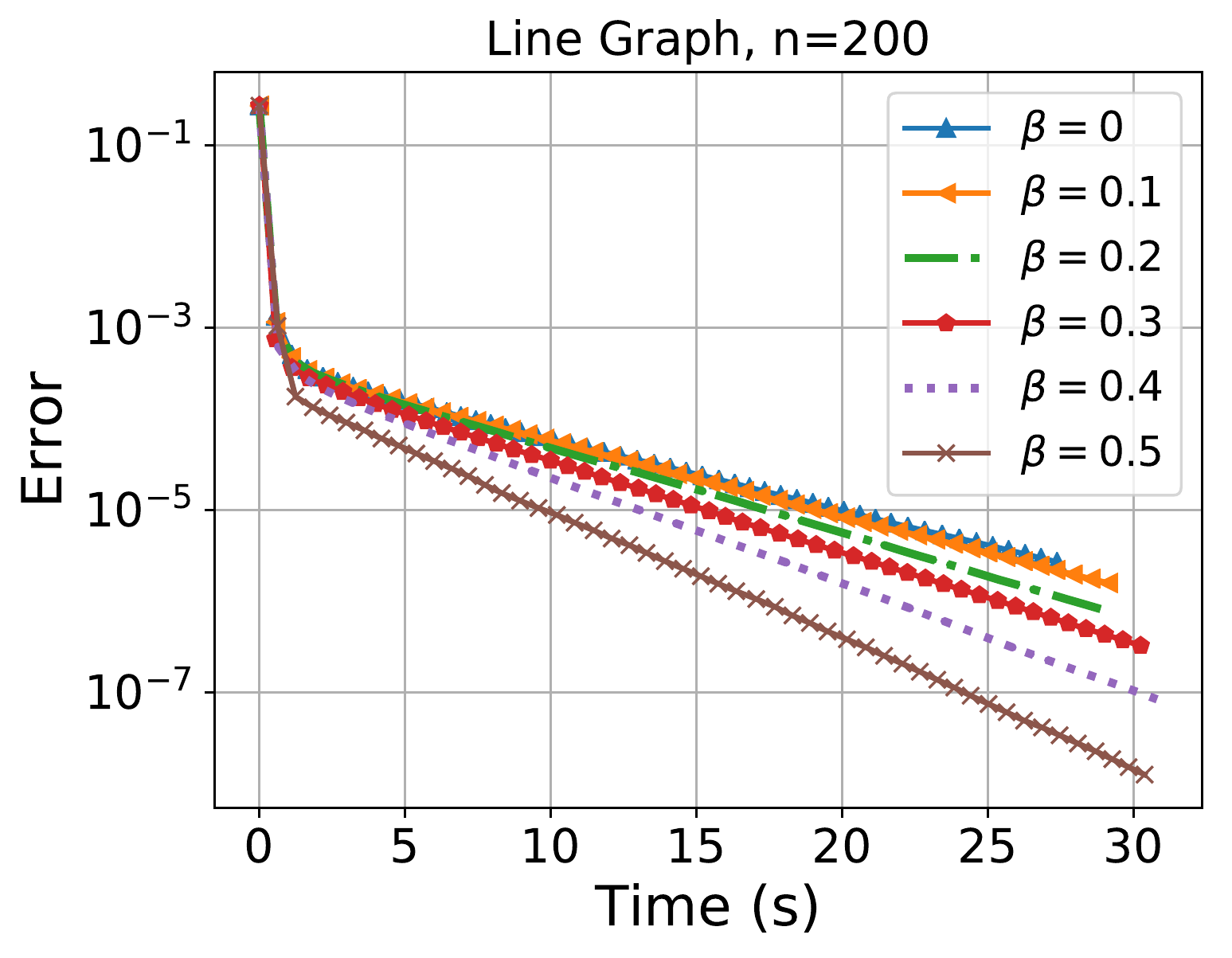}
\end{subfigure}
\begin{subfigure}{.23\textwidth}
  \centering
  \includegraphics[width=1\linewidth]{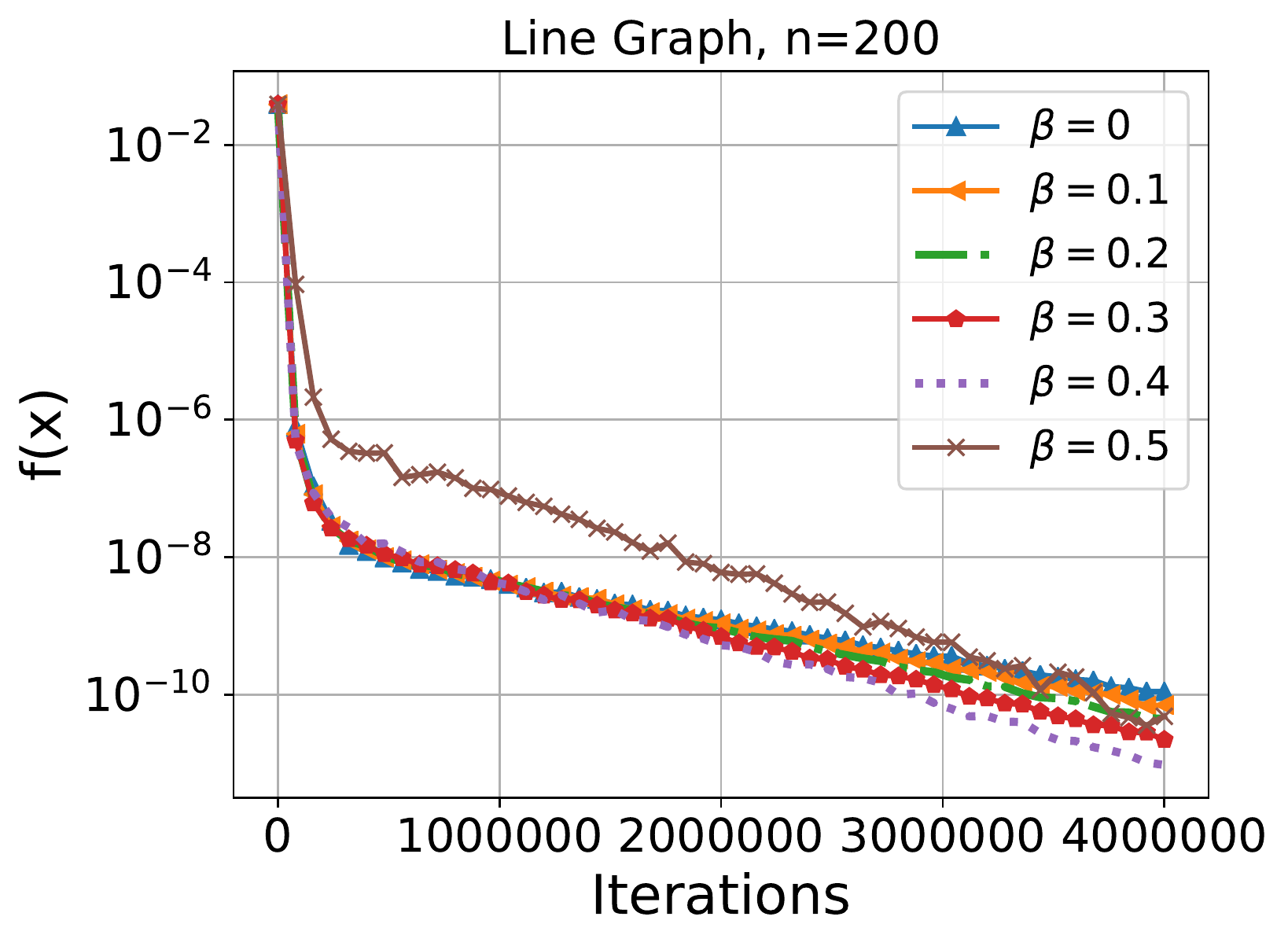}
\end{subfigure}
\begin{subfigure}{.23\textwidth}
  \centering
  \includegraphics[width=1\linewidth]{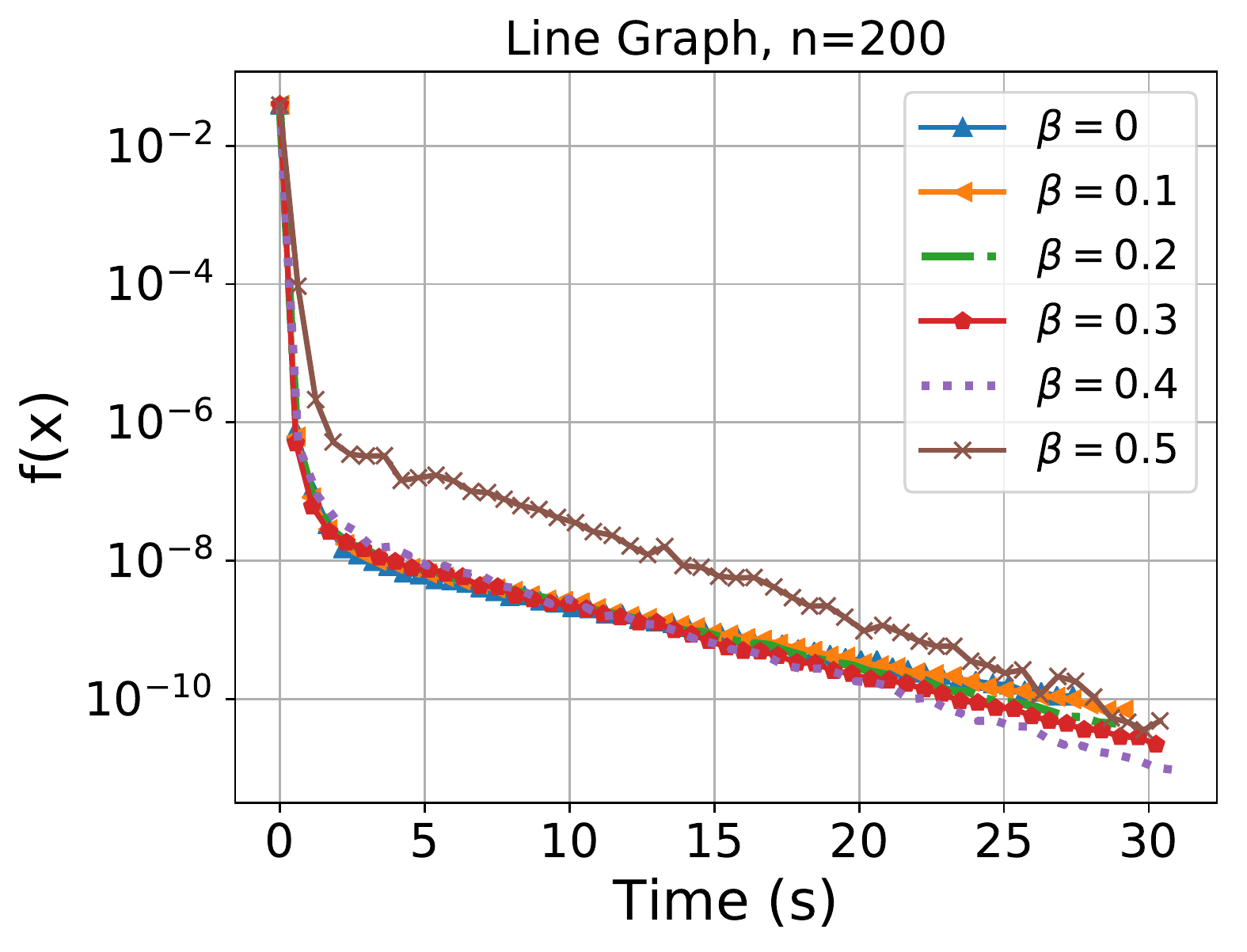}
\end{subfigure}\\
\begin{subfigure}{.23\textwidth}
  \centering
  \includegraphics[width=1\linewidth]{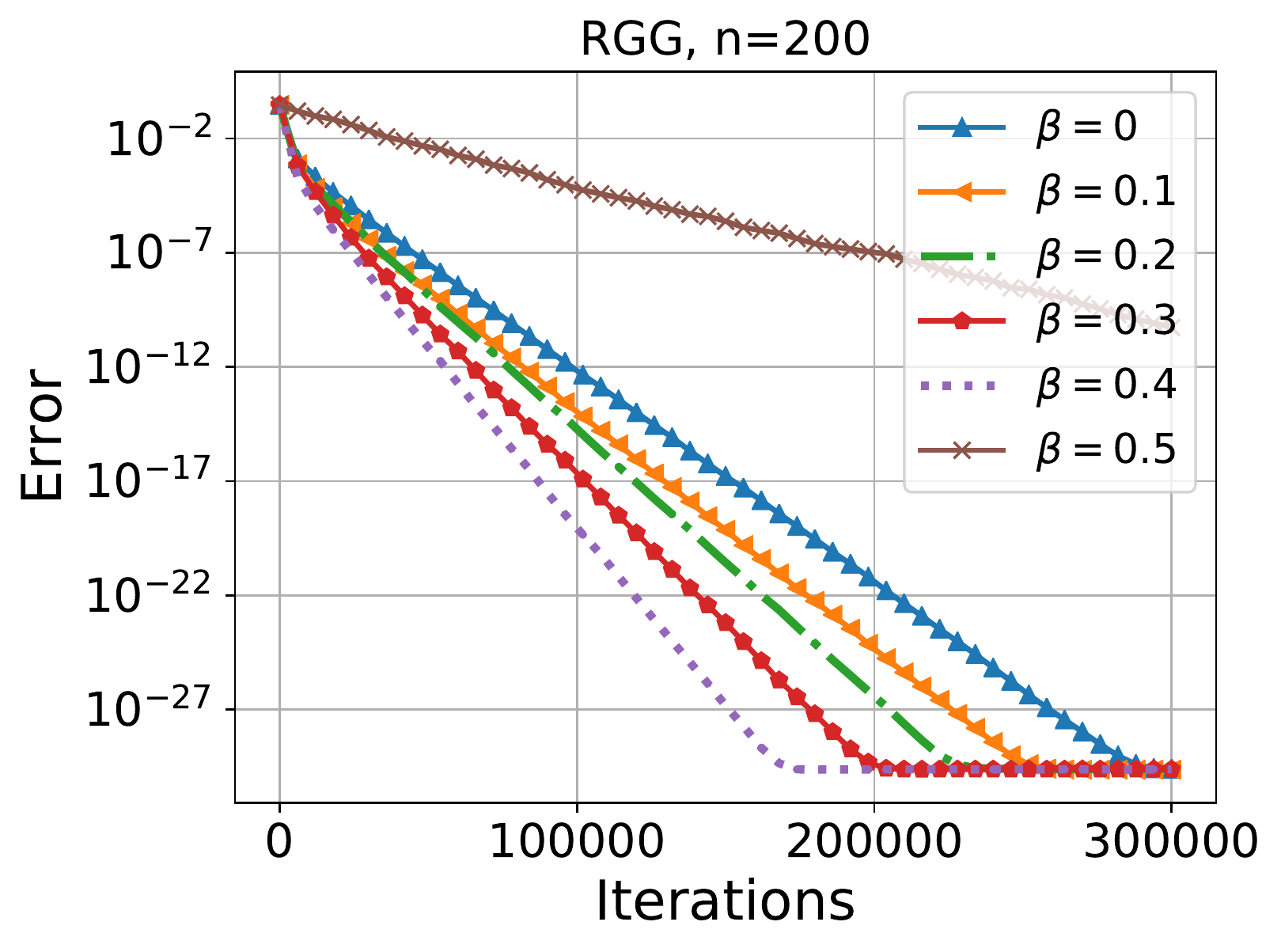}
\end{subfigure}%
\begin{subfigure}{.23\textwidth}
  \centering
  \includegraphics[width=1\linewidth]{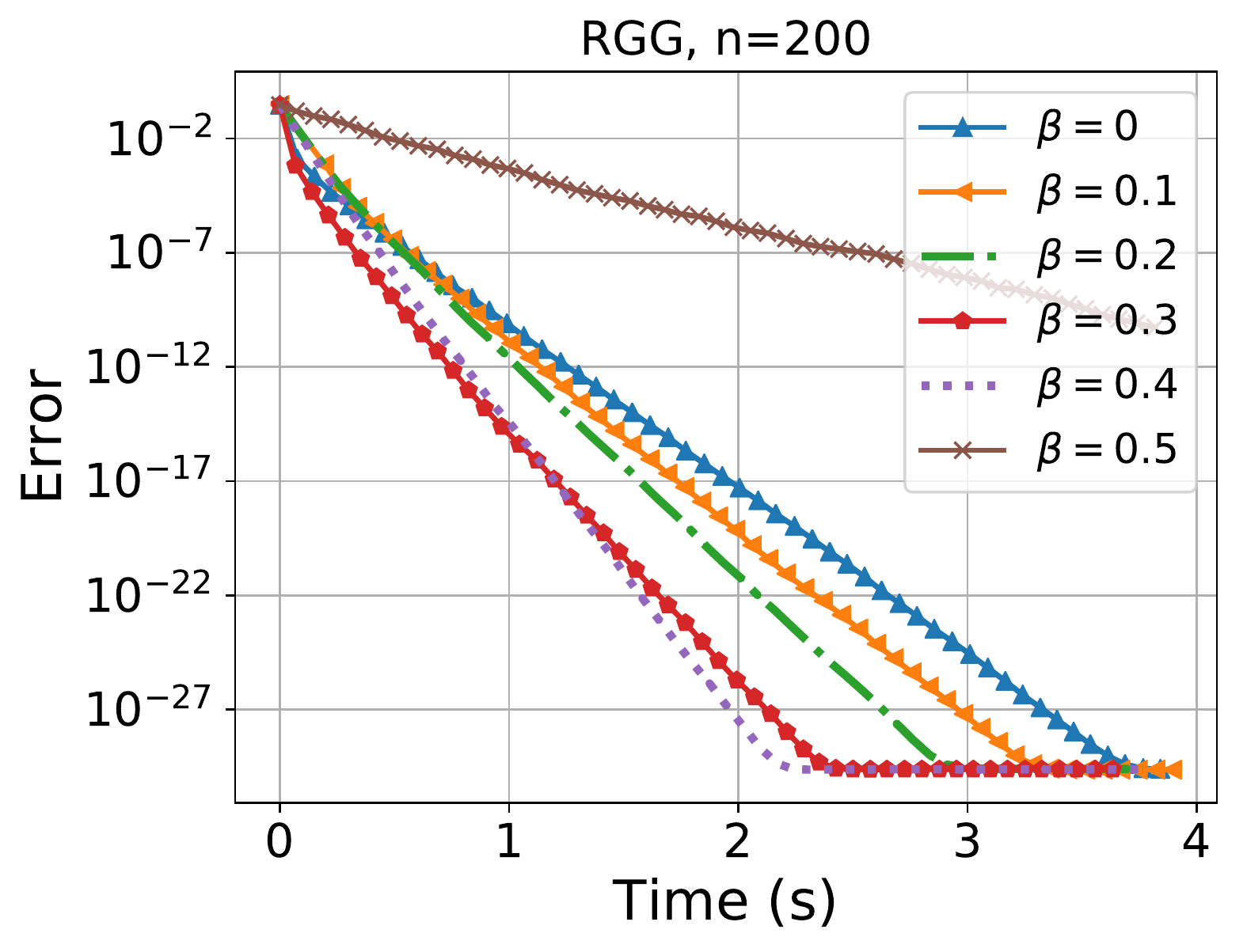}
\end{subfigure}
\begin{subfigure}{.23\textwidth}
  \centering
  \includegraphics[width=1\linewidth]{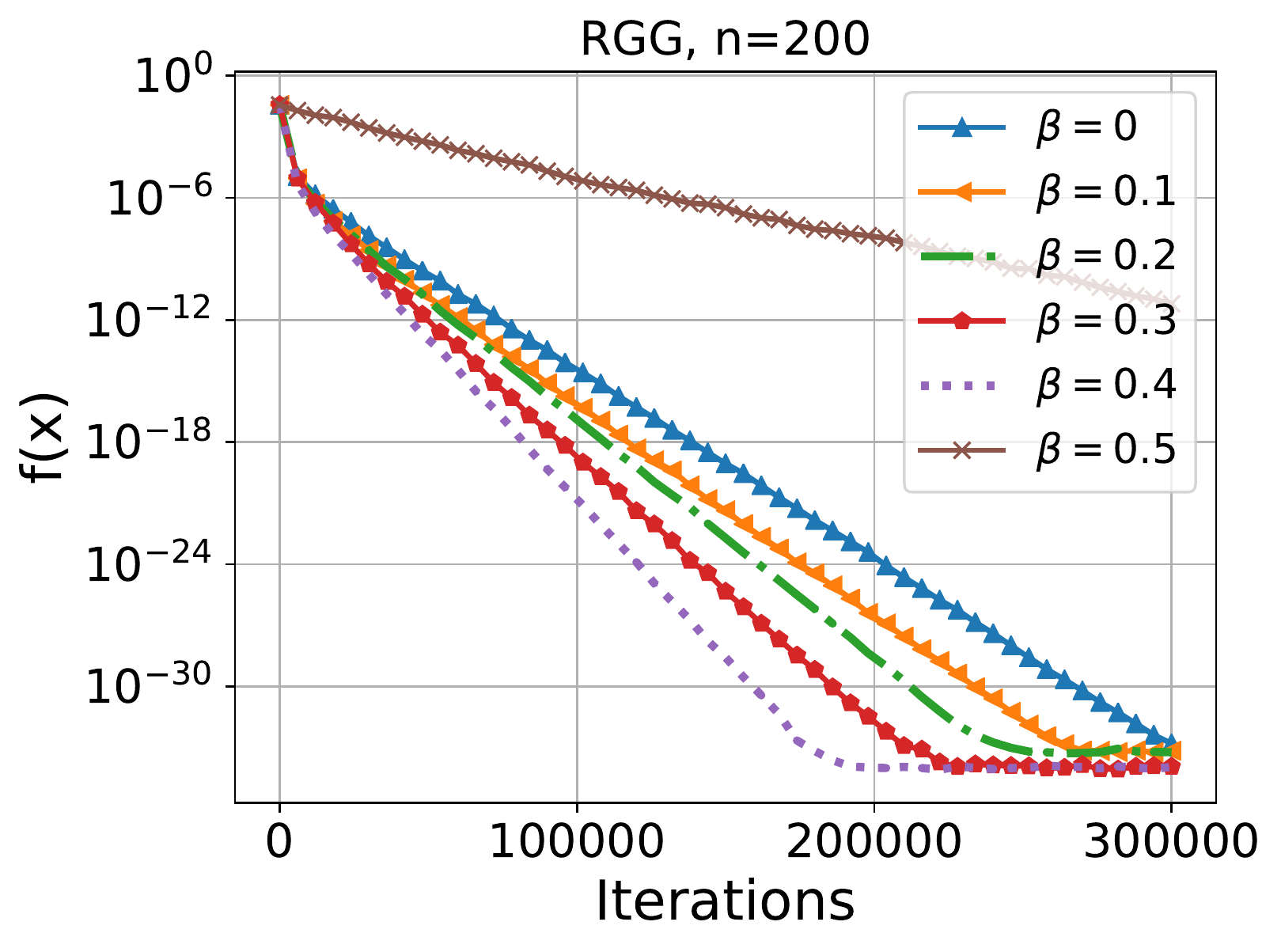}
\end{subfigure}
\begin{subfigure}{.23\textwidth}
  \centering
  \includegraphics[width=1\linewidth]{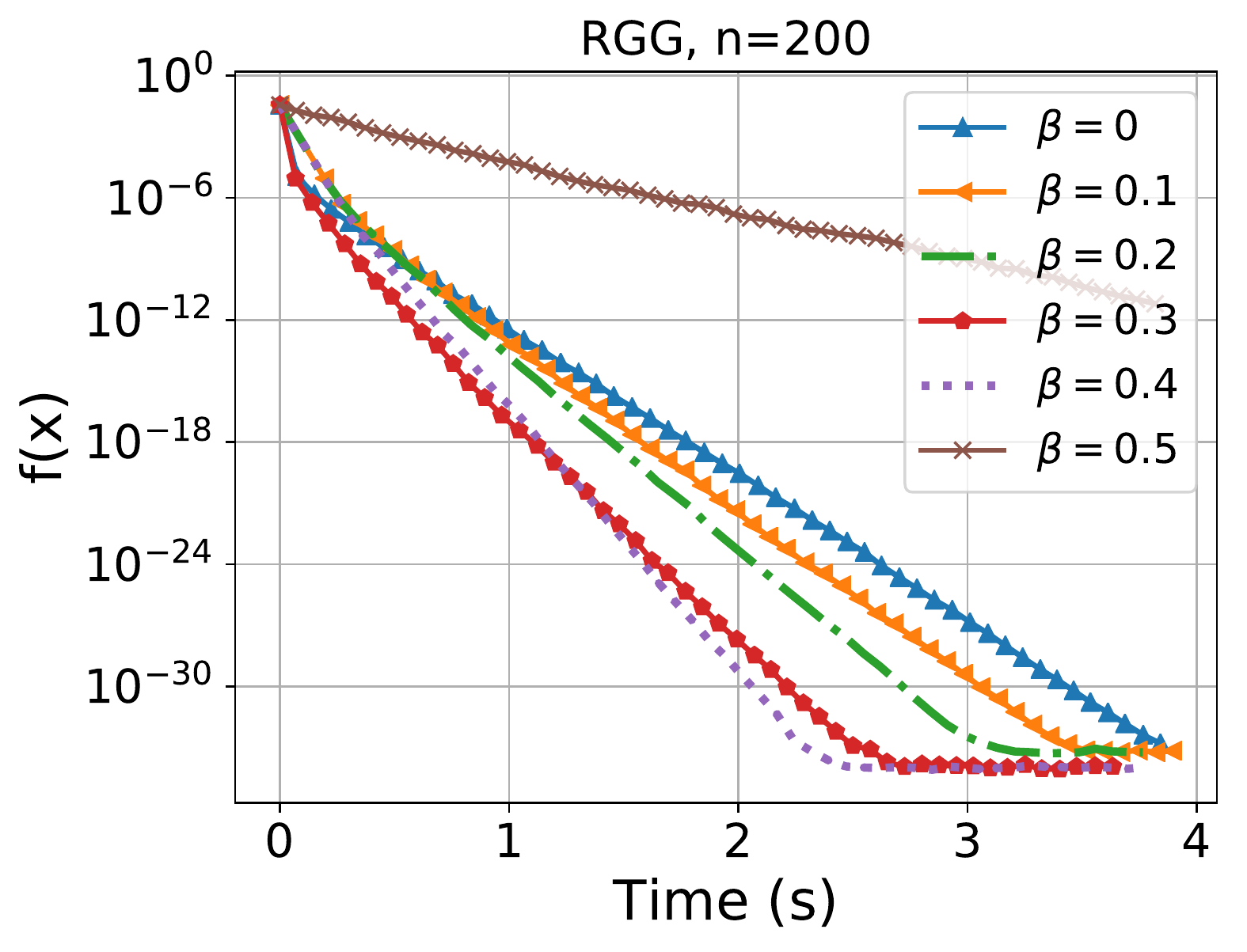}
\end{subfigure}\\
\caption{Performance of mPRG for several momentum parameters $\beta$ for solving the average consensus problem in a cycle graph, line graph and random geometric graph $G(n,r)$ with $n=200$ nodes. For the $G(n,r)$ to ensure connectivity of the network a radius $r=\sqrt{\log(n)/n}$ is used. The graphs in the first (second) column plot iterations (time) against residual error while those in the third (forth) column plot iterations (time) against function values. The ``Error" in the vertical axis represents the relative error $\|x_k-x_*\|^2_\bB / \|x_0-x_*\|^2_\bB \overset{\bB=\bI, x_0=c}{=}\|x_k-x_*\|^2 / \|c-x_*\|^2$ and the function values $f(x_k)$ refer to function~\eqref{functionRK}.}
\label{consensus200}
\end{figure}

\clearpage

\appendix

\section{Proof of Theorem~\ref{L2}} \label{app:1}

\subsection{Lemmas}

We start with a lemma.

\begin{lem}
\label{LemmaGlobal}
Fix $F_1=F_0\geq 0$ and let $\{F_k\}_{k\geq 0}$ be a  sequence of nonnegative  real numbers satisfying the relation 
\begin{equation} \label{eq:809h9s94nbd}F_{k+1}\leq a_1F_k +a_2 F_{k-1}, \qquad \forall k\geq 1,\end{equation} where  $a_2 \geq 0 $,  $ a_1 + a_2 <1$ and at least one of the coefficients $a_1,a_2$ is positive. Then the sequence satisfies the relation
$ F_{k+1}\leq q^{k} (1+ \delta)  F_0$ for all $k\geq 1,$
where $q=\frac{a_1+\sqrt{a_1^2+4a_2}}{2}$  and $\delta=q-a_1\geq 0$. Moreover,
\begin{equation}\label{eq:98sh8hgbf93nd} q \geq a_1 + a_2,\end{equation}
with equality if and only if $a_2=0$ (in which case $q=a_1$ and $\delta=0$).
\end{lem}

\begin{proof}
Choose any $\delta \geq 0$ satisfying $a_2 \leq (a_1+\delta)\delta$. Adding $\delta F_k$  to both sides of \eqref{eq:809h9s94nbd},  we get 
\begin{equation}\label{eq:09h9sh9089ns}
F_{k+1} + \delta F_k  \leq   (a_1+\delta)F_k + a_2 F_{k-1}\\
\leq (a_1+\delta)(F_k + \delta F_{k-1}) = q(F_k+\delta F_{k-1}).
\end{equation}

We now claim that   $\delta = \frac{-a_1+\sqrt{a_1^2+4a_2}}{2}$ satisfies the relations. Non-negativity of $\delta$  follows from $a_2\geq 0$, while the second   relation follows from the fact that $\delta$ satisfies \begin{equation}\label{eq:sbuf78b38bc}(a_1+\delta)\delta - a_2 = 0.\end{equation} Let us now argue that  $0<q<1$. Nonnegativity of $q$ follows from nonnegativity of $a_2$.  Clearly, as long as $a_2>0$, $q$ is positive. If $a_2=0$, then $a_1>0$ by assumption, which implies that $q$ is positive. The inequality $q<1$ follows directly from the assumption $a_1+a_2<1$.
By unrolling the recurrence \eqref{eq:09h9sh9089ns}, we obtain
$F_{k+1} \leq F_{k+1} + \delta F_k \leq q^k (F_1+ \delta F_0) = q^{k}(1+\delta) F_{0}.$

Finally, let us establish \eqref{eq:09h9sh9089ns}. Noting that $a_1 = q-\delta$, and since  in view of \eqref{eq:sbuf78b38bc} we have $a_2=q\delta$, we conclude that $a_1+a_2 = q + \delta(q-1) \leq q$, where the inequality follows from $q <1$.
\end{proof}

The following identities were established in \cite{ASDA}. For completeness, we include different (and somewhat simpler) proofs here.

\begin{lem}[\cite{ASDA}]
For all $x \in \R^n$ we have
\begin{equation}
\label{normbound}
f_{\bS}(x) = \frac{1}{2}\|\nabla f_{\bS}(x)\|^2_{\bB}.
\end{equation}
Moreover, if $x_*\in \cL $ (i.e., if $x_*$ satisfies $\mA x_* =b$), then for all $x\in \R^n$ we have
\begin{equation}
\label{functionequivalence}
f_{\bS}(x) = \frac{1}{2}\langle \nabla f_{\bS}(x),x-x_* \rangle_{\bB}, 
\end{equation}
and
\begin{equation}
\label{asnda}
f(x) = \frac{1}{2}\langle \nabla f(x),x-x_* \rangle_{\bB}.
\end{equation}
\end{lem}

\begin{proof} In view of \eqref{eq:grad_f_S}, and  since $\mZ \mB^{-1} \mZ = \mZ$ (see \cite{ASDA}), we have

\begin{eqnarray*}\|\nabla f_{\bS}(x)\|^2_{\bB} & \overset{\eqref{eq:grad_f_S}}{=}& \|\bB^{-1} \mZ (x-x_*)\|^2_{\bB} \quad = \quad (x-x_*)^\top \mZ \bB^{-1} \bZ (x-x_*) \quad = \quad (x-x_*)^\top \mZ (x-x_*) \\
&\overset{\eqref{eq:Z}}{=} &(x-x_*)^\top \mA^\top \mH \mA (x-x_*) \quad = \quad (\mA x-b)^\top \mH (\mA  x- b ) \quad \overset{\eqref{eq:f_s}}{=}\quad 2f_{\bS}(x).\end{eqnarray*}

Moreover,
\begin{eqnarray*}
\langle \nabla f_{\bS}(x),x-x_* \rangle_{\bB} & \overset{\eqref{eq:grad_f_S}}{=}& \langle \bB^{-1} \bZ (x-x_*),x-x_* \rangle_{\bB}\\
& =& (x-x_*)^\top \bZ \bB^{-1} \bB (x-x_*) \quad = \quad 2f_{\bS}(x).
\end{eqnarray*}
By taking expectations in the last identity with respect to the random matrix $\bS$, we get
$
\langle \nabla f(x),x-x_* \rangle_{\bB}=2f(x).
$
\end{proof}

\begin{lem}[\cite{ASDA}]
\label{bounds}
For all $x \in \R^n$ and $x_* \in \cL$
\begin{equation}
\label{b1}
\lambda_{\min}^+ f(x) \leq \frac{1}{2} \|\nabla f(x) \|^2_{\bB} \leq \lambda_{\max} f(x)
\end{equation}
and 
\begin{equation}
\label{b2}
f(x) \leq  \frac{\lambda_{\max}}{2} \|x-x_*\|^2_{\bB}.
\end{equation}
Moreover, if exactness is satisfied, and we let $x_* =\Pi^{\bB}_{\cL}(x)$, we have 
\begin{equation}
\label{b3}
\frac{\lambda_{\min}^+}{2} \|x-x_*\|^2_{\bB} \leq  f(x)  .
\end{equation}
\end{lem}

\subsection{The Proof}
First, we decompose
\begin{eqnarray}
\|x_{k+1}-x_*\|^2_{\bB} & =& \|x_k-\omega \nabla f_{\bS_k}(x_k)+\beta(x_k-x_{k-1})-x_*\|^2_{\bB} \notag\\
& =& \underbrace{\|x_k-\omega \nabla f_{\bS_k}(x_k)-x_*\|^2_{\bB}}_{\encircle{1}} \notag\\
&& \quad +\underbrace{2\langle x_k-\omega \nabla f_{\bS_k}(x_k)-x_*, \beta(x_k-x_{k-1})  \rangle_{\bB}}_{\encircle{2}} \notag\\
&& \quad +\underbrace{\beta^2\|x_k-x_{k-1}\|^2_{\bB}}_{\encircle{3}}. \label{n0}
\end{eqnarray}
We will now analyze the three expressions \encircle{1}, \encircle{2}, \encircle{3} separately. The first expression can be written as
\begin{eqnarray}
\encircle{1} &=& \|x_k-x_*\|^2_{\bB}-2 \omega \langle x_k-x_*,\nabla f_{\bS_k}(x_k) \rangle_{\bB} +\omega^2 \|\nabla f_{\bS_k}(x_k)\|^2_{\bB} \notag\\
& \overset{\eqref{normbound},\eqref{functionequivalence}}{=} & \|x_k-x_*\|^2_{\bB}-4\omega f_{\bS_k}(x_k)+2\omega^2 f_{\bS_k}(x_k) \notag\\
&=& \|x_k-x_*\|^2_{\bB}-2\omega(2-\omega)f_{\bS_k}(x_k).\label{n1}
\end{eqnarray}
We will now bound the second expression. First, we have
\begin{equation}\label{eq:ibnodh90h}
\begin{aligned}
\encircle{2}
&= 2\beta \langle x_k-x_*, x_k-x_{k-1} \rangle_{\bB} +2\omega \beta \langle \nabla f_{\bS_k}(x_k), x_{k-1} - x_k  \rangle_{\bB}\\ 
&= 2\beta \langle x_k-x_*, x_k-x_{*} \rangle_{\bB} + 2\beta \langle x_k-x_*, x_*-x_{k-1} \rangle_{\bB} +2\omega \beta \langle \nabla f_{\bS_k}(x_k),x_{k-1}- x_k  \rangle_{\bB}\\ 
&= 2\beta \|x_k-x_*\|^2_{\bB}  + 2\beta \langle x_k-x_*, x_*-x_{k-1} \rangle_{\bB} +2\omega \beta \langle \nabla f_{\bS_k}(x_k),x_{k-1}- x_k  \rangle_{\bB}. 
\end{aligned}
\end{equation}
Using the fact that for arbitrary vectors $a,b,c \in \R^n$ we have the identity
$2 \langle a-c,c-b \rangle_{\bB}=\|a-b\|^2_{\bB}-\|c-b\|^2_{\bB}-\|a-c\|^2_{\bB},$
we obtain
$$2 \langle x_k-x_*, x_*-x_{k-1} \rangle_{\bB}=  \|x_k-x_{k-1}\|^2_{\bB}- \|x_{k-1}-x_*\|^2_{\bB}-\|x_k-x_*\|^2_{\bB}.$$
Substituting this into \eqref{eq:ibnodh90h} gives
\begin{equation}
\label{n2}
\begin{aligned}
\encircle{2}& = \beta \|x_k-x_*\|^2_{\bB}+\beta \|x_k-x_{k-1}\|^2_{\bB}-\beta \|x_{k-1}-x_*\|^2_{\bB} + 2\omega \beta \langle \nabla f_{\bS_k}(x_k),x_{k-1}- x_k  \rangle_{\bB}. 
\end{aligned}
\end{equation}

The third expression can be bound as
\begin{equation}
\label{n3}
\encircle{3} =\beta^2\|(x_{k}-x_*)+(x_*-x_{k-1})\|^2_{\bB}  \leq  2\beta^2\|x_{k}-x_*\|^2_{\bB}+2
\beta^2\|x_{k-1}-x_*\|^2_{\bB}.
\end{equation}

By substituting the  bounds \eqref{n1}, \eqref{n2}, \eqref{n3} into  \eqref{n0} we obtain
\begin{eqnarray*}
\|x_{k+1}-x_*\|^2_{\bB} 
& \leq & \|x_k-x_*\|^2_{\bB}-2\omega(2-\omega)f_{\bS_k}(x_k)\\
&& \quad + \beta \|x_k-x_*\|^2_{\bB}+\beta \|x_{k}-x_{k-1}\|^2_{\bB}-\beta \|x_{k-1}-x_*\|^2_{\bB} \\
&& \quad + 2\omega \beta \langle \nabla f_{\bS_k}(x_k),x_{k-1}- x_k  \rangle_{\bB}  +  2\beta^2\|x_{k}-x_*\|^2_{\bB}+2\beta^2\|x_{k-1}-x_*\|^2_{\bB}\\
& \leq & (1+3\beta + 2\beta^2)\|x_k-x_*\|^2_{\bB}+ (\beta + 2\beta^2 )\|x_{k-1}-x_*\|^2_{\bB}-2\omega(2-\omega)f_{\bS_k}(x_k)\\
&& \quad + 2\omega \beta \langle \nabla f_{\bS_k}(x_k),x_{k-1}- x_k  \rangle_{\bB}.
\end{eqnarray*}
Now by first taking expectation with respect to $\mS_k$, we obtain:
\begin{eqnarray*}
\Exp_{\mS_k}[\|x_{k+1}-x_*\|^2_{\bB}] & \leq & (1+3\beta+2\beta^2)\|x_k-x_*\|^2_{\bB}+ (\beta +2\beta^2)\|x_{k-1}-x_*\|^2_{\bB} \\
&& \quad -2\omega(2-\omega)f(x_k) + 2\omega \beta \langle \nabla f(x_k),x_{k-1}- x_k  \rangle_{\bB}\\
 & \leq & (1+3\beta+2\beta^2)\|x_k-x_*\|^2_{\bB}+ (\beta +2\beta^2)\|x_{k-1}-x_*\|^2_{\bB} \\
&& \quad -2\omega(2-\omega)f(x_k) + 2\omega \beta(f(x_{k-1})-f(x_k))\\
 & = & (1+3\beta+2\beta^2)\|x_k-x_*\|^2_{\bB}+ (\beta +2\beta^2)\|x_{k-1}-x_*\|^2_{\bB} \\
&& \quad - (2\omega(2-\omega) +2\omega\beta)f(x_k) + 2\omega \beta f(x_{k-1}).
\end{eqnarray*}
where in the second step we used the inequality
 $\langle \nabla f(x_k),x_{k-1}- x_k \rangle  \leq f(x_{k-1})-f(x_k)$ and the fact that $\omega \beta \geq 0$, which follows from the assumptions. We now apply  inequalities  \eqref{b2} and \eqref{b3}, obtaining

\begin{eqnarray*}
\Exp_{\mS_k}[\|x_{k+1}-x_*\|^2_{\bB}] & \leq &
 \underbrace{(1+3\beta+2\beta^2 - (\omega(2-\omega) +\omega\beta)\lambda_{\min}^+)}_{a_1}\|x_k-x_*\|^2_{\bB} \\
 && \quad + \underbrace{(\beta +2\beta^2 + \omega \beta \lambda_{\max})}_{a_2}\|x_{k-1}-x_*\|^2_{\bB}.
\end{eqnarray*}

By taking expectation again, and letting $F_k\eqdef \Exp[\|x_{k}-x_*\|^2_{\bB}]$, we get  the relation
\begin{equation}
\label{recur}
F_{k+1}  \leq a_1 F_k  + a_2 F_{k-1} .
\end{equation}

It suffices to apply  Lemma~\ref{LemmaGlobal} to the relation  \eqref{recur}. The conditions of the lemma are satisfied. Indeed, $a_2\geq 0$, and if $a_2=0$, then $\beta=0$ and hence $a_1=1-\omega(2-\omega)\lambda_{\min}^+>0$. The condition $a_1+a_2<1$ holds by assumption.

The convergence result in function values, $\Exp[f(x_k)]$, follows as a corollary by applying inequality \eqref{b2} to  \eqref{eq:nfiug582}.

\section{Proof of Theorem~\ref{cesaro}}
\label{app:acc212}

Let $p_t=\frac{\beta}{1-\beta}(x_t-x_{t-1})$ 
and $d_t = \|x_t + p_t -x_*\|_{\mB}^2$. In view of \eqref{eq:SHB-intro},  we can write 
$$x_{t+1}+p_{t+1}= x_t+p_t-\frac{\omega}{1-\beta} \nabla f_{\bS_t}(x_t),$$
and therefore
\begin{eqnarray*}
d_{t+1} & =&  \left\|x_t+p_t-\frac{\omega}{1-\beta} \nabla f_{\bS_t}(x_t) -x_* \right\|^2_{\bB} \  \\
& =& d_t -2 \frac{\omega}{1-\beta} \langle x_t+p_t-x_*,  \nabla f_{\bS_t}(x_t) \rangle_{\bB} + \frac{\omega^2}{(1-\beta)^2} \|\nabla f_{\bS_t}(x_t)\|^2_{\bB}\\
& =& d_t -\frac{2\omega}{1-\beta} \langle x_t-x_*,  \nabla f_{\bS_t}(x_t) \rangle_{\bB} - \frac{2 \omega \beta}{(1-\beta)^2}   \langle x_t-x_{t-1},  \nabla f_{\bS_t}(x_t) \rangle_{\bB}\\ 
& & \quad + \frac{\omega^2}{(1-\beta)^2} \|\nabla f_{\bS_t}(x_t)\|^2_{\bB}.
\end{eqnarray*}

Taking expectation with respect to the random matrix $\bS_t$ we obtain:
\begin{eqnarray*}
\Exp_{\mS_t}[d_{t+1}] & =& \Exp_{\mS_t}[d_t] -\frac{2\omega}{1-\beta} \langle x_t-x_*,  \nabla f(x_t) \rangle_{\bB} - \frac{2 \omega \beta}{(1-\beta)^2}   \langle x_t-x_{t-1},  \nabla f(x_t) \rangle_{\bB} \\
&& \quad + \frac{\omega^2}{(1-\beta)^2} 2 f(x_t) \notag\\
& \overset{\eqref{asnda}}{=} & \Exp_{\mS_t}[d_t] -\frac{4\omega}{1-\beta}  f(x_t) - \frac{2 \omega \beta}{(1-\beta)^2}    \langle x_t-x_{t-1},  \nabla f(x_t) \rangle_{\bB} + \frac{\omega^2}{(1-\beta)^2} 2 f(x_t)\\
& \leq & \Exp_{\mS_t}[d_t] -\frac{4\omega}{1-\beta}  f(x_t) - \frac{2 \omega \beta}{(1-\beta)^2}   [f(x_t)-f(x_{t-1})] + \frac{\omega^2}{(1-\beta)^2} 2 f(x_t)\\
& = & \Exp_{\mS_t}[d_t] + \left[ -\frac{4\omega}{1-\beta} - \frac{2 \omega \beta}{(1-\beta)^2} +\frac{2 \omega^2}{(1-\beta)^2}\right] f(x_t)  +  \frac{2 \omega \beta}{(1-\beta)^2} f(x_{t-1}),
\end{eqnarray*}
where the inequality follows from convexity of $f$.  After rearranging the terms we get
\[
\Exp_{\mS_t}[d_{t+1}] +   \frac{2 \omega \beta}{(1-\beta)^2} f(x_t) + \alpha f(x_t)  \leq \Exp_{\mS_t}[d_t] + \frac{2 \omega \beta}{(1-\beta)^2} f(x_{t-1}),
\]
where $\alpha =   \frac{4\omega}{1-\beta} -\frac{2 \omega^2}{(1-\beta)^2} > 0$. Taking expectations again and using the tower property, we get
\begin{equation}\label{eq:oih89hd8}
\theta_{t+1} + \alpha \Exp[f(x_t)]  \leq \theta_t, \qquad t=1,2,\dots,
\end{equation}
where $\theta_t = \Exp[d_t] + \frac{2 \omega \beta}{(1-\beta)^2}\Exp[ f(x_{t-1})]$. By summing up \eqref{eq:oih89hd8} for $t=1,\dots, k$ we get
\begin{equation}\label{eq:s098h89hffdss}\sum_{t=1}^k \Exp[f(x_t)] \leq \frac{\theta_1-\theta_{k-1}}{\alpha} \leq \frac{\theta_1}{\alpha}.\end{equation}
Finally, using Jensen's inequality, we get
\[\Exp[f(\hat{x}_k)] = \Exp \left[f\left(\frac{1}{k}\sum_{t=1}^k x_t\right)\right] \leq \Exp \left[\frac{1}{k}\sum_{t=1}^k f(x_t)\right] =  \frac{1}{k}\sum_{t=1}^k \Exp[f(x_t)] \overset{\eqref{eq:s098h89hffdss}}{\leq} \frac{\theta_1}{\alpha k}.\]

It remains to note that $\theta_1 = \|x_0-x_*\|_{\mB}^2 + \frac{2\omega \beta}{(1-\beta)^2 }f(x_0).$

\section{Proof of Theorem~\ref{theoremheavyball}} \label{app:acc}

In the proof of Theorem~\ref{theoremheavyball} the following two lemmas are used.

\begin{lem}[\cite{ASDA}]
\label{Forweakconvergence}
Assume exactness. Let $x\in \R^n$ and $x_* = \Pi_\mathcal{L}^\bB(x)$. If $\lambda_i=0$, then $u_i^\top \bB^{1/2} (x-x_*)=0$.

\end{lem}

\begin{lem}[\cite{elaydi2005introduction, fillmore1968linear}]
\label{recurence}
Consider the second degree linear homogeneous recurrence relation:
\begin{equation}
\label{asdasdasd}
 r_{k+1}= a_1r_k+a_2 r_{k-1}
\end{equation}
with initial conditions $r_0,r_1 \in \R$.
Assume that the constant coefficients $a_1$ and $a_2$ satisfy the inequality $a_1^2 +4a_2<0$ (the roots of the characteristic equation $t^2-a_1t-a_2=0$ are imaginary). Then there are complex constants $C_0$ and $ C_1$ (depending on the initial conditions $r_0$ and $r_1$) such that:
$$r_k=2 M^k (C_0 \cos( \theta k) + C_1 \sin(\theta k))$$
where $M= \bigg(\sqrt{\frac{a_1^2}{4}+\frac{(-a_1^2-4a_2)}{4}} \bigg)=\sqrt{-a_2}$ and $\theta$ is such that $a_1=2 M \cos(\theta)$ and $\sqrt{-a_1^2-4a_2}=2 M \sin(\theta)$.
\end{lem}

We can now turn to the proof of Theorem~\ref{theoremheavyball}.
Plugging in the expression for the stochastic gradient, mSGD can be written in the form
\begin{eqnarray}
x_{k+1} & =& x_k -\omega \nabla f_{\bS_k}(x_k) + \beta(x_k - x_{k-1}) \notag \\
&\overset{\eqref{eq:grad_f_S}}{=}& x_k- \omega {\bB}^{-1} \bZ_k(x_k-x_*) + \beta(x_k - x_{k-1}).\label{difernt expad}
\end{eqnarray}

Subtracting  $x_*$ from both sides of \eqref{difernt expad}, we get
\begin{eqnarray*}
x_{k+1}-x_* 
& = & (\bI - \omega {\bB}^{-1} \bZ_k)(x_k-x_*) + \beta(x_k -x_* +x_* - x_{k-1})\\
& =& \left((1+\beta)\bI - \omega {\bB}^{-1} \bZ_k\right)(x_k-x_*) - \beta(x_{k-1}-x_*).
\end{eqnarray*}

Multiplying the last identity from the left by $\bB^{1/2}$, we get
\begin{eqnarray*}
\bB^{1/2} (x_{k+1}-x_*)
&=& \left((1+\beta)\bI - \omega \bB^{-1/2} \bZ_k \bB^{-1/2}\right) \bB^{1/2}(x_{k} -x_*) - \beta \bB^{1/2}(x_{k-1}-x_*).
\end{eqnarray*}

Taking expectations, conditioned on $x_k$ (that is, the expectation is with respect to $\bS_k$):
\begin{eqnarray}
\bB^{1/2} \Exp[x_{k+1} -x_* \;|\; x_k] & = & \left((1+\beta)\bI - \omega \bB^{-1/2}  \Exp[\bZ]  \bB^{-1/2}\right) \bB^{1/2}(x_{k} -x_*) - \beta \bB^{1/2}(x_{k-1}-x_*) . \label{eq:98g8gfsssd}
\end{eqnarray}

Taking expectations again, and using the tower property, we get
\begin{eqnarray*}
\bB^{1/2} \Exp[x_{k+1} -x_*] & = & \bB^{1/2}\Exp\left[\Exp[x_{k+1} -x_* \;|\; x_k]\right]\\
&\overset{\eqref{eq:98g8gfsssd}}{=}& \left((1+\beta)\bI - \omega \bB^{-1/2}  \Exp[\bZ] \bB^{-1/2} \right) \bB^{1/2} \Exp[x_{k} -x_*] - \beta \bB^{1/2} \Exp[x_{k-1}-x_*].
\end{eqnarray*}

Plugging the eigenvalue decomposition ${\bU}\bm{\Lambda} {{\bU}}^\top$ of the matrix $\bW=\bB^{-1/2}  \Exp[\bZ] \bB^{-1/2}$ into the above, and multiplying both sides from the left by ${{\bU}}^\top$, we obtain
\begin{equation}
\label{asdnaskdn}
{{\bU}}^\top \bB^{1/2} \Exp[x_{k+1} -x_*]  = {{\bU}}^\top \left((1+\beta)\bI - \omega {\bU}\bm{\Lambda} {{\bU}}^\top \right)\bB^{1/2} \Exp[x_{k} -x_*] - \beta {{\bU}}^\top \bB^{1/2} \Exp[x_{k-1}-x_*].
\end{equation}

Let us define $s_k\eqdef {{\bU}}^\top \bB^{1/2} \Exp[x_{k} -x_*]  \in \R^n$. Then relation \eqref{asdnaskdn} takes the form of the recursion
$$ s_{k+1}= [(1+\beta)\mI - \omega \bm{\Lambda} ] s_k - \beta s_{k-1},$$
which can be written in a coordinate-by-coordinate form as follows:
\begin{equation}
\label{coordinate}
 s_{k+1}^i= [(1+\beta) - \omega \lambda_i ] s_k^i - \beta s_{k-1}^i   \quad \text{for all} \quad i= 1,2,3,...,n,
\end{equation}
where $s_k^i$ indicates the $i$th coordinate of $s_k$.

We will now fix $i$ and analyze recursion \eqref{coordinate} using Lemma~\ref{recurence}. Note that \eqref{coordinate} is a second degree linear homogeneous recurrence relation of the form \eqref{asdasdasd} with $a_1=1+\beta - \omega \lambda_i $ and $a_2=- \beta$. Recall  that $0\leq\lambda_i \leq1$ for all $i$. Since we assume that $0< \omega \leq 1/\lambda_{\max}$, we know that $0\leq \omega \lambda_i \leq 1$ for all $i$. We now consider two  cases:

\begin{enumerate}
\item $ \lambda_i =0$.   

In this case, \eqref{coordinate} takes the form: \begin{equation} \label{eq:9g898fg9sy}s_{k+1}^i=(1+\beta)s_k^i-\beta s_{k-1}^i .\end{equation} Applying Theorem~\ref{prop:projections}, we know that $x_*=\Pi_\mathcal{L}^\bB(x_0)=\Pi_\mathcal{L}^\bB(x_1)$. Using Lemma \ref{Forweakconvergence} twice, once for $x=x_0$ and then for $x=x_1$, we observe that $s_0^i=u_i^\top \bB^{1/2} (x_0-x_*)=0$ and $s_1^i=u_i^\top \bB^{1/2} (x_1-x_*)=0$. Finally, in view of \eqref{eq:9g898fg9sy} we conclude that \begin{equation}\label{eq:8g98gdu9hhOOh} s_k^i=0 \quad \text{for all} \quad k\geq 0 .\end{equation}

\item $ \lambda_i >0$. 

Since $0<\omega \lambda_i \leq 1$ and $\beta\geq 0$, we have $1+\beta - \omega \lambda_i \geq 0$ and hence
\[a_1^2+4a_2=(1+\beta -\omega \lambda_i)^2-4\beta \leq  (1+\beta -\omega \lambda_{\min}^+)^2-4\beta  < 0,\]
where the last inequality can be shown to hold\footnote{The lower bound on $\beta$ is tight. However,  the upper bound is not. However, we do not care much about the regime of large $\beta$ as $\beta$ is the convergence rate, and hence is only interesting if smaller than 1.} for  $(1-\sqrt{\omega \lambda_{\min}^+})^2 < \beta < 1 $. Applying Lemma~\ref{recurence} the following bound can be deduced
\begin{eqnarray}
s_k^i  &=& 2(-a_2)^{k/2} (C_0  \cos(\theta k) +C_1 \sin(\theta k)) \; \leq \; 2 \beta^{k/2} P_i, \label{eq:uibd880s-pO}
\end{eqnarray}
where $P_i$ is a constant depending  on the initial conditions (we can simply choose $P_i = |C_0| + |C_1|$).

\end{enumerate}

Now putting the two cases together, for all $k\geq0$ we have
\begin{eqnarray*}
\|\Exp[x_{k} -x_*]\|_{\bB}^2&=& \Exp[x_{k} -x_*]^\top \bB \Exp[x_{k} -x_*] \; = \; \Exp[x_{k} -x_*] \bB^{1/2} \bU {\bU}^\top \bB^{1/2} \Exp[x_{k} -x_*]  \\
&=& \|{\bU}^\top \bB^{1/2} \Exp[x_{k} -x_*] \|_2^2 \; = \; \|s_k\|^2 \;= \; \sum_{i=1}^{n} (s_k^i)^2 \\
&= &  \sum_{i: \lambda_i=0} (s_k^i)^2  +  \sum_{i:  \lambda_i >0} (s_k^i)^2 \; \overset{\eqref{eq:8g98gdu9hhOOh}}{=}\;    \sum_{i:  \lambda_i >0} (s_k^i)^2\\
& \overset{\eqref{eq:uibd880s-pO}}{\leq} & \sum_{i:  \lambda_i >0} 4 \beta ^k P_i^2 \\
&=& \beta^k C,
\end{eqnarray*}
where $C=4\sum_{i:  \lambda_i >0}  P_i^2$.

\section{Proof of Theorem~\ref{thm:DSHB-L2} } \label{app:7}

The proof follows a similar pattern to that of Theorem~\ref{L2}. However, stochasticity in the momentum term introduces an additional layer of complexity, which we shall tackle by utilizing   a more involved version of the tower property.

For simplicity, let $i=i_k$ and $r_{k}^i  \eqdef e_i^\top(x_k-x_{k-1})e_i$. First, we decompose
\begin{eqnarray}
\|x_{k+1}-x_*\|^2 
& =& \|x_k-\omega \nabla f_{\bS_k}(x_k)+\beta r_k^i -x_*\|^2 \notag\\
& =& \|x_k-\omega \nabla f_{\bS_k}(x_k)-x_*\|^2+2\langle x_k-\omega \nabla f_{\bS_k}(x_k)-x_*, \beta r_k^i  \rangle + \beta^2\| r_k^i\|^2. \label{eq:iugh98d894}
\end{eqnarray}
 We shall use the tower property in the form
\begin{equation}\label{eq:tower3}\Exp[\Exp[\Exp[ X  \;|\; x_k,  \mS_k] \;|\; x_k]] = \Exp[X],\end{equation}
where $X$ is some random variable. We shall perform the three expectations in order, from the innermost to the outermost.  Applying the inner expectation to the identity \eqref{eq:iugh98d894}, we get
\begin{eqnarray}\Exp[\|x_{k+1}-x_*\|^2 \;|\; x_k, \mS_k ] 
&=&
 \underbrace{\Exp[\|x_k-\omega \nabla f_{\bS_k}(x_k)-x_*\|^2\;|\; x_k, \mS_k]}_{\encircle{1}} \notag\\
 && \quad +\underbrace{\Exp[2\langle x_k-\omega \nabla f_{\bS_k}(x_k)-x_*, \beta r_k^i  \rangle \;|\; x_k, \mS_k]}_{\encircle{2}} \notag\\
&& \quad +\underbrace{\Exp[\beta^2\| r_k^i\|^2\;|\; x_k, \mS_k]}_{\encircle{3}}. \label{eq:098j}
\end{eqnarray}

We will now analyze the three expressions \encircle{1}, \encircle{2}, \encircle{3} separately. The first expression is constant under the expectation, and hence we can write
\begin{eqnarray}
\encircle{1} &=& \|x_k-\omega \nabla f_{\bS_k}(x_k)-x_*\|^2\notag\\
&=& \|x_k-x_*\|^2-2 \omega \langle x_k-x_*,\nabla f_{\bS_k}(x_k) \rangle +\omega^2 \|\nabla f_{\bS_k}(x_k)\|^2\notag\\
& \overset{\eqref{normbound} + \eqref{functionequivalence}}{=} & \|x_k-x_*\|^2-4\omega f_{\bS_k}(x_k)+2\omega^2 f_{\bS_k}(x_k)\notag\\
&=& \|x_k-x_*\|^2-2\omega(2-\omega)f_{\bS_k}(x_k). \label{n1X}
\end{eqnarray}

We will now bound the second expression. Using the identity
\begin{equation}\label{eq:98gf8g8e09}\Exp[r_k^i  \;|\; x_k, \mS_k] = \Exp_i [r_k^i] = \sum_{i=1}^n \frac{1}{n}r_k^i = \frac{1}{n}(x_k-x_{k-1}),\end{equation}
we can write
\begin{eqnarray}
\encircle{2}
&=& \Exp[2\langle x_k-\omega \nabla f_{\bS_k}(x_k)-x_*, \beta r_k^i  \rangle \;|\; x_k, \mS_k]\notag\\
&=& 2\langle x_k-\omega \nabla f_{\bS_k}(x_k)-x_*, \beta \Exp[r_k^i  \;|\; x_k, \mS_k] \rangle\notag\\
&\overset{\eqref{eq:98gf8g8e09}}{=}& 2\langle x_k-\omega \nabla f_{\bS_k}(x_k)-x_*, \tfrac{\beta}{n} (x_k - x_{k-1})  \rangle\notag \\
& =&
2\tfrac{\beta}{n} \langle x_k-x_*, x_k-x_{k-1} \rangle +2\omega \tfrac{\beta}{n} \langle \nabla f_{\bS_k}(x_k), x_{k-1} - x_k  \rangle\notag\\ 
&=& 2\tfrac{\beta}{n} \langle x_k-x_*, x_k-x_{*} \rangle + 2\tfrac{\beta}{n} \langle x_k-x_*, x_*-x_{k-1} \rangle +2\omega \tfrac{\beta}{n} \langle \nabla f_{\bS_k}(x_k),x_{k-1}- x_k  \rangle \notag\\ 
&=& 2\tfrac{\beta}{n} \|x_k-x_*\|^2 + 2\tfrac{\beta}{n} \langle x_k-x_*, x_*-x_{k-1} \rangle +2\omega \tfrac{\beta}{n} \langle \nabla f_{\bS_k}(x_k),x_{k-1}- x_k  \rangle. \label{eq:ibnodh90hX}
\end{eqnarray}
Using the fact that for arbitrary vectors $a,b,c \in \R^n$ we have the identity
$2 \langle a-c,c-b \rangle =\|a-b\|^2-\|c-b\|^2-\|a-c\|^2,$
we obtain
$$2 \langle x_k-x_*, x_*-x_{k-1} \rangle=  \|x_k-x_{k-1}\|^2- \|x_{k-1}-x_*\|^2-\|x_k-x_*\|^2.$$
Substituting this into \eqref{eq:ibnodh90hX} gives
\begin{equation}
\label{n2X}
\begin{aligned}
\encircle{2}& = \tfrac{\beta}{n} \|x_k-x_*\|^2+\tfrac{\beta}{n}\|x_k-x_{k-1}\|^2-\tfrac{\beta}{n} \|x_{k-1}-x_*\|^2 + 2\omega \tfrac{\beta}{n} \langle \nabla f_{\bS_k}(x_k),x_{k-1}- x_k  \rangle. 
\end{aligned}
\end{equation}

The third expression can be bound as
\begin{eqnarray}
\encircle{3} &=& \Exp[\beta^2\| r_k^i\|^2\;|\; x_k, \mS_k] \notag\\
&=& \beta^2 \Exp_i[\| r_k^i\|^2] \notag \\
&=& \beta^2 \sum_{i=1}^n \tfrac{1}{n} (x_k^i -x_{k-1}^i)^2 \notag\\
&=& \tfrac{\beta^2}{n} \|x_k-x_{k-1}\|^2 \notag \\ 
&=&\tfrac{\beta^2}{n}\|(x_{k}-x_*)+(x_*-x_{k-1})\|^2 \notag \\
& \leq &  \tfrac{2\beta^2}{n}\|x_{k}-x_*\|^2+ \tfrac{2\beta^2}{n}\|x_{k-1}-x_*\|^2. \label{n3X}
\end{eqnarray}

By substituting the  bounds \eqref{n1X}, \eqref{n2X}, \eqref{n3X} into  \eqref{eq:098j} we obtain
\begin{eqnarray}
\Exp[\|x_{k+1}-x_*\|^2 \;|\; x_k, \mS_k ] 
& \leq & \|x_k-x_*\|^2-2\omega(2-\omega) f_{\bS_k}(x_k)\notag\\
&& \quad + \tfrac{\beta}{n}  \|x_k-x_*\|^2+ \tfrac{\beta}{n} \|x_{k}-x_{k-1}\|^2 -\tfrac{\beta}{n}  \|x_{k-1}-x_*\|^2 \notag\\
&& \quad + 2\omega \tfrac{\beta}{n}  \langle \nabla f_{\bS_k}(x_k), x_{k-1}- x_k  \rangle  +  2\tfrac{\beta^2}{n} \|x_{k}-x_*\|^2 \\
&& \quad + 2\tfrac{\beta^2}{n}\|x_{k-1}-x_*\|^2 \notag\\
&\overset{\eqref{n3}}{\leq}& \left(1+3\tfrac{\beta}{n} + 2\tfrac{\beta^2}{n}\right)\|x_k-x_*\|^2+ \left(\tfrac{\beta}{n} + 2\tfrac{\beta^2}{n} \right) \|x_{k-1}-x_*\|^2 \notag\\
&& \quad - 2\omega(2-\omega)f_{\bS_k}(x_k) + 2\omega \tfrac{\beta}{n} \langle \nabla f_{\bS_k}(x_k),x_{k-1}- x_k  \rangle.\label{eq:iohih638ygbdd}
\end{eqnarray}
We now take the middle expectation (see \eqref{eq:tower3}) and apply it to inequality \eqref{eq:iohih638ygbdd}:
\begin{eqnarray*}
\Exp[\Exp[\|x_{k+1}-x_*\|^2 \;|\; x_k, \mS_k ] \;|\; x_k]  & \leq & \left(1+3\tfrac{\beta}{n} + 2\tfrac{\beta^2}{n}\right)\|x_k-x_*\|^2+ \left(\tfrac{\beta}{n} + 2\tfrac{\beta^2}{n} \right) \|x_{k-1}-x_*\|^2 \\
&& \quad -2\omega(2-\omega)f(x_k) + 2\omega \tfrac{\beta}{n} \langle \nabla f(x_k),x_{k-1}- x_k  \rangle\\
 & \leq & \left(1+3\tfrac{\beta}{n} + 2\tfrac{\beta^2}{n}\right)\|x_k-x_*\|^2+ \left(\tfrac{\beta}{n} + 2\tfrac{\beta^2}{n} \right) \|x_{k-1}-x_*\|^2 \\
&& \quad -2\omega(2-\omega)f(x_k) + 2\omega \tfrac{\beta}{n}(f(x_{k-1})-f(x_k))\\
 & = &  \left(1+3\tfrac{\beta}{n} + 2\tfrac{\beta^2}{n}\right)\|x_k-x_*\|^2+ \left(\tfrac{\beta}{n} + 2\tfrac{\beta^2}{n} \right) \|x_{k-1}-x_*\|^2 \\
&& \quad - \left(2\omega(2-\omega) +2\omega\tfrac{\beta}{n}\right)f(x_k) + 2\omega \tfrac{\beta}{n} f(x_{k-1}).
\end{eqnarray*}
where in the second step we used the inequality
 $\langle \nabla f(x_k),x_{k-1}- x_k \rangle  \leq f(x_{k-1})-f(x_k)$ and the fact that $\omega \beta \geq 0$, which follows from the assumptions. We now apply  inequalities  \eqref{b2} and \eqref{b3}, obtaining
\begin{eqnarray*}
\Exp[\Exp[\|x_{k+1}-x_*\|^2 \;|\; x_k, \mS_k ] \;|\; x_k]& \leq &
 \underbrace{\left(1+3\tfrac{\beta}{n}+2\tfrac{\beta^2}{n} - \left(\omega(2-\omega) +\omega\tfrac{\beta}{n}\right)\lambda_{\min}^+ \right)}_{a_1}\|x_k-x_*\|^2 \\
 && \quad + \underbrace{\tfrac{1}{n}\left(\beta +2\beta^2 + \omega \beta \lambda_{\max}\right)}_{a_2}\|x_{k-1}-x_*\|^2.
\end{eqnarray*}

By taking expectation again (outermost expectation in the tower rule \eqref{eq:tower3}), and letting $F_k\eqdef \Exp[\|x_{k}-x_*\|^2_{\bB}]$, we get  the relation
\begin{equation}
\label{recurX}
F_{k+1}  \leq a_1 F_k  + a_2 F_{k-1} .
\end{equation}

It suffices to apply  Lemma~\ref{LemmaGlobal} to the relation  \eqref{recur}. The conditions of the lemma are satisfied. Indeed, $a_2\geq 0$, and if $a_2=0$, then $\beta=0$ and hence $a_1=1-\omega(1-\omega)\lambda_{\min}^+>0$. The condition $a_1+a_2<1$ holds by assumption.

The convergence result in function values follows as a corollary by applying inequality \eqref{b2} to  \eqref{eq:nfiug582X}.

\bibliographystyle{plain}
\bibliography{SHB}

\clearpage
\end{document}